\documentclass[10pt,reqno,oneside]{amsproc} \usepackage{amsfonts}  \usepackage{amsmath, amsthm, amssymb, enumerate} \textwidth 16truecm \textheight 8.4in\oddsidemargin0.2truecm\evensidemargin0.7truecm\voffset-.9truecm \usepackage{color} \usepackage{marginnote} \usepackage[colorlinks=true, pdfstartview=FitV, linkcolor=blue, citecolor=blue, urlcolor=blue]{hyperref}    \chardef\forshowkeys=0    \chardef\refcheck=0    \chardef\showllabel=0    \chardef\sketches=0 \ifnum\forshowkeys=1      \usepackage[notref,notcite,color]{showkeys} \fi \ifnum\showllabel=1 \definecolor{colorcccc}{rgb}{0.7,0.7,0.7}   \def\llabel#1{\marginnote{\color{colorcccc}\rm\small(#1)}[-0.0cm]\notag} \else  \def\llabel#1{\notag} \fi  \ifnum\refcheck=1   \usepackage{refcheck} \fi   \newtheorem{Theorem}{Theorem}[section]   \newtheorem{Lemma}[Theorem]{Lemma}   \newtheorem{definition}{Definition}[section]    \def\Pas{\indeq\mathbb{P}\text{-a.s.}}    \def\textPas{\ \mathbb{P}\text{-a.s.}}        \def\PP{{\mathbb P}}        \def\dd{d}    \def\ueps{u^{\epsilon}}    \def\un{u^{(n)}}    \def\unm{u^{(n-1)}}            \def\znm{z^{(n-1)}}        \def\vn{v^{(n)}}    \def\vnm{v^{(n-1)}}    \def\unp{u^{(n+1)}}    \def\phinm{\varphi^{(n-1)}}    \def\phin{{\varphi^{(n)}}}    \def\phinp{\varphi^{(n+1)}}                \def\phiv{\varphi_v}                    \def\zz{{z}}    \def\vv{{v}}        \def\uu{{{u}}}            \def\startnewsection#1#2{\section{#1}\label{#2}\setcounter{equation}{0}}                        \def\NNp{{\mathbb N}}        \def\RR{{\mathbb R}}         \def\WW{{\mathbb W}}    \def\EE{{\mathbb E}}    \def\comma{ {\rm ,\qquad{}} }    \def\commaone{ {\rm ,\quad{}} }    \def\fractext#1#2{{#1}/{#2}}    \def\supp{\mathop{\rm supp}\nolimits}                        \def\indeq{\qquad{}}                \def\colb{\color{black}} \definecolor{colorpppp}{rgb}{0.6,0.0,0.1} \definecolor{colorgggg}{rgb}{.0,0.4,0.0} \definecolor{colorgray}{rgb}{0.6,0.6,0.6}   \def\cole{}                  \definecolor{colororange}{rgb}{0.8,0.2,0} \definecolor{colorpurple}{rgb}{0.6,0.0,0.6}             \def\RR{\mathbb R}          \def\eps{\epsilon}        \def\tilde{\widetilde}     \def\PP{\mathbb{P}} \def\pn{P_{\le n}}  \def\pkn{P_{\le k(n)}} \def\pkm{P_{\le k(m)}} \def\pnm{P_{n, m}}  \def\tnm{\tau_{n,m}}        \def\div{\mathop{\rm div}\nolimits}   \def\supp{\mathop{\rm supp}\nolimits} \def\indeq{\quad{}}     \def\bea{\begin{align}}   \def\ena{\end{align}}            \def\bega{\begin{aligned}}   \def\enda{\end{aligned}}             
 \def\bcase{\begin{cases}}   \def\ecase{\end{cases}}  \def\bmx{\begin{bmatrix}}   \def\emx{\end{bmatrix}}  \def\cf{\mathcal{F}} \def\un{u^{(n)}} \def\um{u^{(m)}} \def\wnm{u^{(n,m)}}   \def\paren#1{\left(#1\right)} \def\dd{d} \def\uu{{{u}}} \def\WW{ W} \def\NN{\mathbb{N}}   \def\indic{\mathbf{1}} \def\NNp{{\mathbb N}}          
\begin{document} \title[Stochastic NSE with small $L^p$ data]{Global existence for the stochastic Navier-Stokes equations\\with small $L^{p}$~data} \title[Local existence of the stochastic Navier-Stokes equations in the whole space]{Local existence of the stochastic Navier-Stokes equations in the whole space} \author[I.~Kukavica]{Igor Kukavica} \address{Department of Mathematics, University of Southern California, Los Angeles, CA 90089} \email{kukavica@usc.edu} \author[F.~Wang]{Fei Wang} \address{School of Mathematical Sciences, CMA-Shanghai, Shanghai Jiao Tong University} \email{fwang256@sjtu.edu.cn} \author[F.H.~Xu]{Fanhui Xu} \address{Department of Mathematics, Harvard University, Cambridge, MA 02138} \email{fanhuixu@math.harvard.edu} \begin{abstract} We address the local well-posedness for the stochastic Navier-Stokes system with multiplicative cylindrical noise in the whole space. More specifically, we prove that there exists a unique local strong solution to the system in $L^p(\RR^3)$ for~$p>3$. \hfill \today \end{abstract} \maketitle \date{} \startnewsection{Introduction}{sec01} In this paper, we address the local solvability of the stochastic Navier-Stokes equations (SNSE)   \begin{align}   \begin{split}   &\partial_t u - \Delta u  + \mathcal{P}((u\cdot\nabla) u) = \sigma(u) \dot{W}(t),   \\   &\div u = 0   \end{split}   \label{ERTWERTHWRTWERTSGDGHCFGSDFGQSERWDFGDSFGHSDRGTEHDFGHDSFGSDGHGYUHDFGSDFASDFASGTWRT01}   \end{align} in the spatial domain~$\RR^3$. Here, $u$ denotes the velocity field of a stochastic flow, and  \begin{equation}    \label{ERTWERTHWRTWERTSGDGHCFGSDFGQSERWDFGDSFGHSDRGTEHDFGHDSFGSDGHGYUHDFGSDFASDFASGTWRT02}    u|_{t=0} = u_0  \end{equation} for a given random function $u_0$ in $L^p(\Omega, L^{p}(\mathbb{R}^{3}))$ such that $\div u_0=0$. Note that the pressure gradient in the SNSE has been eliminated by the Helmholtz-Hodge projector $\mathcal{P}$ onto the divergence-free fields. The stochastic term, $\sigma(u) \dot{W}(t)$, denotes an infinite-dimensional multiplicative noise. We assume for brevity that $\div (\sigma(u))=0$ if $\div u=0$, as otherwise, the orthogonal component of the divergence-free part in the Helmholtz decomposition can be grouped with the pressure gradient and eliminated by~$\mathcal{P}$. \par The initial value problem for the SNSE has a rich history; see \cite{F} for an introduction to existing results and a summary of technical difficulties. The earliest works, starting with Bensoussan and Temam~\cite{BeT}, considered the SNSE in a Hilbert setting as It\^{o}'s calculus heavily relied on the It\^{o} isometry. The development of stochastic integration theory encouraged considerations of the SNSE in Banach spaces; see for instance~\cite{FRS, MoS, ZBL}. Also, various notions of solutions emerged from these works and other studies of stochastic evolution equations; see~\cite{BCF,BF,BT,BR}. Showing the existence of a probabilistically strong solution is challenging. If the equation is driven by an additive noise, then the pathwise uniqueness of solutions can be proven in a similar way as the equation's deterministic counterpart; with pathwise uniqueness, one can claim that the probabilistically weak solution (i.e.~martingale solution) is a strong solution, see \cite{F} and references therein. There are also many works on stochastic evolution equations with multiplicative noise. The global existence of a strong solution was established for linear equations in $W^{m,p}$ spaces~\cite{Kr}, the Euler equations with linear noise~\cite{GV}, and the SNSE in two-dimensional cases; see~\cite{MeS, MR}. In higher dimensions, Kim addressed the SNSE with multiplicative noise in Hilbert spaces~\cite{Ki}. It was shown that a unique local strong solution exists and becomes global with a large probability if the initial datum is small in~$H^{1/2+}$. In the paper~\cite{GZ}, Glatt-Holtz and Ziane proved the local existence for the SNSE on bounded domains with multiplicative noise and $H^{1}$-initial data. This work motivated the earlier work \cite{KXZ} of two of the authors with Ziane, which established the existence of an $L^p(\mathbb{T}^{3})$-strong solution when~$p>5$. The result of \cite{KXZ} was improved in~\cite{KX}, where the exponent was lowered to~$p>3$, together with an $L^p$-type energy estimate. In the meantime, Agresti and Veraar obtained a local existence result in \cite{AV} for Besov spaces, including $B^{-1+3/q}_{3,3}$ with a range of $q$ including $q=3$, using maximal regularity. For results on other aspects of the well-posedness theory of solutions (see \cite{CC,DZ,FS,KV}), while for the deterministic case, see~\cite{FJR,K}. \par The existence of a strong $L^{p}$-solution of the stochastic Navier-Stokes system with multiplicative noise was previously considered in \cite{KXZ,KX} on the torus $\mathbb{T}^{d}$. But it was unclear if the $L^p$-solution still exists when the spatial domain is the whole space. In the present paper, we answer this question assuming that the initial datum $u_0$ belongs to $L^p(\Omega, L^{p}(\mathbb{R}^{3}))$ with~$p>3$. We obtain a positive result by constructing a sequence of approximate solutions and showing the convergence in the strong topology. While proving the theorem, we encounter similar obstacles as in~\cite{KX, KXZ} and some other works addressing the well-posedness of stochastic evolution equations, including the superlinearity of the equation, the step-dependent energy estimates, and a possibly degenerate time interval of convergence. In addition, for the case of $\mathbb{R}^{d}$, we need to cope with the non-compactness of $\mathbb{R}^{d}$ and the lack of a Poincar\'e-type inequality. Due to the non-compactness of $\mathbb{R}^{d}$, the Galerkin scheme employed in \cite{KX,KXZ} is not suitable in the present situation as it uses the eigenfunctions of the Laplacian. Also, without a Poincar\'e-type inequality, one can not control the possibly growing $L^2$-based energy of approximate solutions, which is the cost of linearization of the equation. One of the main ideas in the present paper is to introduce a convolution-type projector and a cut-off function that linearizes the convective term $(u\cdot \nabla)u$. Moreover, a Poincar\'e-type inequality holds under the projection. We approximate the initial data in $L^{p}(\mathbb{R}^{3})$ using functions in $L^{2}(\mathbb{R}^{3})\cap L^{p}(\mathbb{R}^{3})$. Certainly, the $L^2$-energy of the approximate solutions may grow, but we offset its impact during the $L^p$-energy derivation by properly choosing a subfamily of projectors and employing a Poincar\'e-type inequality. \par A primary tool in our construction of solutions is the quantity $\int \sum_{j} \int | \nabla (|\uu_j(s,x)|^{p/2})|^2 \,dx ds$, which absorbs higher-order terms resulting from the H\"older and Gagliardo-Nirenberg inequalities. It is tempting to define stopping times for every approximate solution using this quantity, so that the solutions' energy is uniformly bounded, and the Cauchy condition is pairwise verified up to the corresponding stopping times. However, this quantity is highly non-linear and does not satisfy the triangle inequality. This turns out to be a problem when estimating the distance of successive stopping times. Note that we need to find a time interval shared by all approximate solutions in which the Cauchy condition holds, and the hope of such search lies in the closeness of successive stopping times. We choose the same norm for the topology of convergence and the stopping times. By resorting to the Sobolev inequality, we manage to close the estimates. \par  In addition to the integrability and divergence-free conditions of the initial data, we impose the following assumptions on the noise coefficient $\sigma(u)$, \begin{align}         & \Vert\sigma(u)\Vert_{\mathbb{L}^p} \le C (\Vert u\Vert_{(3p/2)-}^2 + 1), \label{ERTWERTHWRTWERTSGDGHCFGSDFGQSERWDFGDSFGHSDRGTEHDFGHDSFGSDGHGYUHDFGSDFASDFASGTWRT03}  \\   & \Vert\sigma(u_1)-\sigma(u_2)\Vert_{\mathbb{L}^p} \le C\Vert ( |u_1| + |u_2|)^{1/2} |u_1-u_2|\Vert_{p}, \label{ERTWERTHWRTWERTSGDGHCFGSDFGQSERWDFGDSFGHSDRGTEHDFGHDSFGSDGHGYUHDFGSDFASDFASGTWRT04}   \\ & \Vert \nabla\sigma(u)\Vert_{\mathbb{L}^p} \le  C(\Vert u\Vert_{3p/2}^2 + 1), \label{ERTWERTHWRTWERTSGDGHCFGSDFGQSERWDFGDSFGHSDRGTEHDFGHDSFGSDGHGYUHDFGSDFASDFASGTWRT05} \end{align} where $\Vert\cdot \Vert_{\mathbb{L}^p}$ follows the definition~\eqref{ERTWERTHWRTWERTSGDGHCFGSDFGQSERWDFGDSFGHSDRGTEHDFGHDSFGSDGHGYUHDFGSDFASDFASGTWRT18} in Section 2, and $\Vert \cdot \Vert_{(3p/2)-}$ refers to a norm $\Vert \cdot \Vert_{(3p/2)-\epsilon}$ for an arbitrarily small positive constant~$\epsilon$. But the noise growth rate has to be strictly smaller than $3p/2$ for us to control the higher-order terms in multiple estimates; see for instance \eqref{ERTWERTHWRTWERTSGDGHCFGSDFGQSERWDFGDSFGHSDRGTEHDFGHDSFGSDGHGYUHDFGSDFASDFASGTWRT105} and \eqref{ERTWERTHWRTWERTSGDGHCFGSDFGQSERWDFGDSFGHSDRGTEHDFGHDSFGSDGHGYUHDFGSDFASDFASGTWRT123} below. To obtain $L^2$-energy conservation for approximate solutions, we also assume that \begin{equation}\label{ERTWERTHWRTWERTSGDGHCFGSDFGQSERWDFGDSFGHSDRGTEHDFGHDSFGSDGHGYUHDFGSDFASDFASGTWRT003} \Vert\sigma(u)\Vert_{\mathbb{L}^2} \le C (\Vert u\Vert_{2} + 1). \end{equation} Observe that the set of operators satisfying \eqref{ERTWERTHWRTWERTSGDGHCFGSDFGQSERWDFGDSFGHSDRGTEHDFGHDSFGSDGHGYUHDFGSDFASDFASGTWRT05} is not empty. In fact, any $\sigma$ such that   \begin{equation}     \sigma(u) = \phi*P(u)    \llabel{8Th sw ELzX U3X7 Ebd1Kd Z7 v 1rN 3Gi irR XG KWK0 99ov BM0FDJ Cv k opY NQ2 aN9 4Z 7k0U nUKa mE3OjU 8D F YFF okb SI2 J9 V9gV lM8A LWThDP nP u 3EL 7HP D2V Da ZTgg zcCC mbvc70 qq P cC9 mt6 0og cr TiA3 HEjw TK8ymK eu J Mc4 q6d Vz2 00 XnYU tLR9 GYjPXv FO V r6W 1zU K1W bP ToaW JJuK nxBLnd 0f t DEb Mmj 4lo HY yhZy MjM9 1zQS4p 7z 8 eKa 9h0 Jrb ac ekci rexG 0z4n3x z0 Q OWS vFj 3jL hW XUIU 21iI AwJtI3 Rb W a90 I7r zAI qI 3UEl UJG7 tLtUXz w4 K QNE TvX zqW au jEMe nYlN IzLGxg B3 A uJ8 6VS 6Rc PJ 8OXW w8im tcKZEz Ho p 84G 1gS As0 PC owMI 2fLK TdD60y nH g 7lk NFj JLq Oo Qvfk fZBN G3o1Dg Cn 9 hyU h5V SP5 z6 1qvQ wceU dVJJsB vX D G4E LHQ HIa PT bMTr sLsm tXGyOB 7p 2 Os4 3US bq5 ik 4Lin 769O TkUxmp I8 u GYn fBK bYI 9A QzCF w3h0 geJftZ ZK U 74r Yle ajm km ZJdi TGHO OaSt1N nl B 7Y7 h0y oWJ ry rVrT zHO8 2S7oub QA W x9d z2X YWB e5 Kf3A LsUF vqgtM2 O2 I dim rjZ 7RN 28 4KGY trVa WW4nTZ XV b RVo Q77 hVL X6 K2kq FWFm aZnsF9 Ch p 8Kx rsc SGP iS tVXB J3xZ cD5IP4 Fu 9 Lcd TR2 Vwb cL DlGK 1ro3 EEyqEA zw 6 sKe Eg2 sFf jz MtrZ 9kbd xNw66c xf t lzD GZh xQA WQ KkSX jqmm rEpNuG 6P y loq 8hH lSf Ma LXm5 RzEX W4Y1Bq ib 3 UOh Yw9 5h6 f6 o8kw 6frZ wg6fIy XP n ae1 TQJ Mt2 TT fWWf jJrX ilpYGr Ul Q 4uM 7Ds p0r Vg 3gIE mQOz TFh9LA KO 8 csQ u6m h25 r8 WqRI DZWg SYkWDu lL 8 Gpt ZW1 0Gd SY FUXL zyQZ hVZMn9 am P 9aE Wzk au0 6d ZghM ym3R jfdePG ln 8 s7x HYC IV9 Hw Ka6v EjH5 J8Ipr7 Nk C xWR 84T Wnq s0 fsiP qGgs Id1fs5 3A T 71q RIc zPX 77 Si23 GirL 9MQZ4F pi g dru NYt h1K 4M Zilv rRk6 B4W5B8 Id 3 Xq9 nhx EN4 P6 ipZl a2UQ Qx8mda g7 r VD3 zdD rhB vk LDJo tKyV 5IrmyJ R5 e txS 1cv EsY xG zj2T rfSR myZo4L m5 D mqN iZd acg GQ 0KRw QKGX g9o8v8 wm B fUu tCO cKc zz kx4U fhuA a8pYzW Vq 9 Sp6 CmA cZL Mx ceBX Dwug sjWuii Gl v JDb 08h BOV C1 pEQ06}   \end{equation} belongs to this set if $\phi$ is a function in $C_c^\infty(\RR^\dd)$ and $P$ is an operator satisfying \eqref{ERTWERTHWRTWERTSGDGHCFGSDFGQSERWDFGDSFGHSDRGTEHDFGHDSFGSDGHGYUHDFGSDFASDFASGTWRT03} and~\eqref{ERTWERTHWRTWERTSGDGHCFGSDFGQSERWDFGDSFGHSDRGTEHDFGHDSFGSDGHGYUHDFGSDFASDFASGTWRT04}. \colb \par The following is the structure of the paper. In Section~\ref{sec02}, we introduce the notation, state the main result (see Theorem~\ref{T01}), and prove several preliminary results on convolutions for $\ell^{2}$-valued functions with a useful Poincar\'e-type statement. In Section~\ref{sec4}, we prove a statement on the $L^{p}$ existence for the stochastic heat equations with an $L^p$-type energy estimate, along with the corresponding $L^{p}$ energy convergence lemma. The next section contains the existence and uniqueness theorem for a truncated SNSE, while in the last one we pass to the limit, thus proving the main theorem. \par
\startnewsection{Preliminaries and Main Results}{sec02} \subsection{Basic Notation}\colb In the sequel, $\NNp$ denotes the set of positive integers, and $C$ a generic positive constant, with additional dependence indicated when necessary. For a function $u(t,x)$ defined in $[0,T]\times \RR^d$, we write the partial derivatives as $\partial_t u, \partial_1 u, \ldots, \partial_d u$, the spatial gradient as $\nabla u$, and the Laplacian of $u$ as~$\Delta u$. If $u$ is vector-valued, we use $u_j$ to refer to the $j$-th component of~$u$. \par We denote by $C_c^{\infty}( \RR^d)$ the set of infinitely differentiable functions with compact support in $ \RR^d$, and by $\mathcal{D}'(\RR^d)$ its dual, the space of distributions. Also, $\mathcal{S}(\RR^d)$ denotes the Schwartz space in $\RR^d$ and $\mathcal{S}'(\RR^d)\subseteq \mathcal{D}'( \RR^d)$ its dual, the space of tempered distributions.  \par The Fourier transform of an integrable function $f$ is defined by \begin{equation} \hat{f}(\xi)=\mathcal{F}(f)(\xi)    =\int_{\RR^d} e^{- 2\pi i \xi\cdot x}f(x)\,dx \comma \xi\in \RR^d ,    \llabel{ qq P cC9 mt6 0og cr TiA3 HEjw TK8ymK eu J Mc4 q6d Vz2 00 XnYU tLR9 GYjPXv FO V r6W 1zU K1W bP ToaW JJuK nxBLnd 0f t DEb Mmj 4lo HY yhZy MjM9 1zQS4p 7z 8 eKa 9h0 Jrb ac ekci rexG 0z4n3x z0 Q OWS vFj 3jL hW XUIU 21iI AwJtI3 Rb W a90 I7r zAI qI 3UEl UJG7 tLtUXz w4 K QNE TvX zqW au jEMe nYlN IzLGxg B3 A uJ8 6VS 6Rc PJ 8OXW w8im tcKZEz Ho p 84G 1gS As0 PC owMI 2fLK TdD60y nH g 7lk NFj JLq Oo Qvfk fZBN G3o1Dg Cn 9 hyU h5V SP5 z6 1qvQ wceU dVJJsB vX D G4E LHQ HIa PT bMTr sLsm tXGyOB 7p 2 Os4 3US bq5 ik 4Lin 769O TkUxmp I8 u GYn fBK bYI 9A QzCF w3h0 geJftZ ZK U 74r Yle ajm km ZJdi TGHO OaSt1N nl B 7Y7 h0y oWJ ry rVrT zHO8 2S7oub QA W x9d z2X YWB e5 Kf3A LsUF vqgtM2 O2 I dim rjZ 7RN 28 4KGY trVa WW4nTZ XV b RVo Q77 hVL X6 K2kq FWFm aZnsF9 Ch p 8Kx rsc SGP iS tVXB J3xZ cD5IP4 Fu 9 Lcd TR2 Vwb cL DlGK 1ro3 EEyqEA zw 6 sKe Eg2 sFf jz MtrZ 9kbd xNw66c xf t lzD GZh xQA WQ KkSX jqmm rEpNuG 6P y loq 8hH lSf Ma LXm5 RzEX W4Y1Bq ib 3 UOh Yw9 5h6 f6 o8kw 6frZ wg6fIy XP n ae1 TQJ Mt2 TT fWWf jJrX ilpYGr Ul Q 4uM 7Ds p0r Vg 3gIE mQOz TFh9LA KO 8 csQ u6m h25 r8 WqRI DZWg SYkWDu lL 8 Gpt ZW1 0Gd SY FUXL zyQZ hVZMn9 am P 9aE Wzk au0 6d ZghM ym3R jfdePG ln 8 s7x HYC IV9 Hw Ka6v EjH5 J8Ipr7 Nk C xWR 84T Wnq s0 fsiP qGgs Id1fs5 3A T 71q RIc zPX 77 Si23 GirL 9MQZ4F pi g dru NYt h1K 4M Zilv rRk6 B4W5B8 Id 3 Xq9 nhx EN4 P6 ipZl a2UQ Qx8mda g7 r VD3 zdD rhB vk LDJo tKyV 5IrmyJ R5 e txS 1cv EsY xG zj2T rfSR myZo4L m5 D mqN iZd acg GQ 0KRw QKGX g9o8v8 wm B fUu tCO cKc zz kx4U fhuA a8pYzW Vq 9 Sp6 CmA cZL Mx ceBX Dwug sjWuii Gl v JDb 08h BOV C1 pni6 4TTq Opzezq ZB J y5o KS8 BhH sd nKkH gnZl UCm7j0 Iv Y jQE 7JN 9fd ED ddys 3y1x 52pbiG Lc a 71j G3e uli Ce uzv2 R40Q 50JZUB uK d U3m May 0uo S7 ulWD h7qG 2FKw2T JX z BES 2Jk Q4U Dy 4aJ2 IXs4 RNH41s py T GNh hk0 w5Z C8 B3nU Bp9p 8eLKh8 UO 4 fMq Y6w lcA GM xCHt vlOx MqAJoQ QU 1 e8a 2aX 9Y6 2r lIS6 dejK Y3KCUm 25 7 oCl VeE e8p 1z UJSv bmLd Fy7ObQ FN l J6F RdF kEm qM N0Fd NZJ0 8DYuq2 pL X JNz 4rO ZkZ X2 IjTD 1fVt z4BmFI Pi 0 GKD R2W PhO zH zTLP lbAE OT9XW0 gb T Lb3 XRQ qGG 8o 4TPE 6WRc uMqMXh s6 x Ofv 8st jDiEQ143} \end{equation} where $d$ is the space dimension. The inverse Fourier transform of an integrable function $g$ reads   \begin{equation}   \check{g}(\xi)=(\mathcal{F}^{-1}g)(x)    =\int_{\RR^d} e^{ 2\pi i \xi\cdot x}g(\xi)\,d\xi   \comma x\in \RR^d   .    \llabel{ci rexG 0z4n3x z0 Q OWS vFj 3jL hW XUIU 21iI AwJtI3 Rb W a90 I7r zAI qI 3UEl UJG7 tLtUXz w4 K QNE TvX zqW au jEMe nYlN IzLGxg B3 A uJ8 6VS 6Rc PJ 8OXW w8im tcKZEz Ho p 84G 1gS As0 PC owMI 2fLK TdD60y nH g 7lk NFj JLq Oo Qvfk fZBN G3o1Dg Cn 9 hyU h5V SP5 z6 1qvQ wceU dVJJsB vX D G4E LHQ HIa PT bMTr sLsm tXGyOB 7p 2 Os4 3US bq5 ik 4Lin 769O TkUxmp I8 u GYn fBK bYI 9A QzCF w3h0 geJftZ ZK U 74r Yle ajm km ZJdi TGHO OaSt1N nl B 7Y7 h0y oWJ ry rVrT zHO8 2S7oub QA W x9d z2X YWB e5 Kf3A LsUF vqgtM2 O2 I dim rjZ 7RN 28 4KGY trVa WW4nTZ XV b RVo Q77 hVL X6 K2kq FWFm aZnsF9 Ch p 8Kx rsc SGP iS tVXB J3xZ cD5IP4 Fu 9 Lcd TR2 Vwb cL DlGK 1ro3 EEyqEA zw 6 sKe Eg2 sFf jz MtrZ 9kbd xNw66c xf t lzD GZh xQA WQ KkSX jqmm rEpNuG 6P y loq 8hH lSf Ma LXm5 RzEX W4Y1Bq ib 3 UOh Yw9 5h6 f6 o8kw 6frZ wg6fIy XP n ae1 TQJ Mt2 TT fWWf jJrX ilpYGr Ul Q 4uM 7Ds p0r Vg 3gIE mQOz TFh9LA KO 8 csQ u6m h25 r8 WqRI DZWg SYkWDu lL 8 Gpt ZW1 0Gd SY FUXL zyQZ hVZMn9 am P 9aE Wzk au0 6d ZghM ym3R jfdePG ln 8 s7x HYC IV9 Hw Ka6v EjH5 J8Ipr7 Nk C xWR 84T Wnq s0 fsiP qGgs Id1fs5 3A T 71q RIc zPX 77 Si23 GirL 9MQZ4F pi g dru NYt h1K 4M Zilv rRk6 B4W5B8 Id 3 Xq9 nhx EN4 P6 ipZl a2UQ Qx8mda g7 r VD3 zdD rhB vk LDJo tKyV 5IrmyJ R5 e txS 1cv EsY xG zj2T rfSR myZo4L m5 D mqN iZd acg GQ 0KRw QKGX g9o8v8 wm B fUu tCO cKc zz kx4U fhuA a8pYzW Vq 9 Sp6 CmA cZL Mx ceBX Dwug sjWuii Gl v JDb 08h BOV C1 pni6 4TTq Opzezq ZB J y5o KS8 BhH sd nKkH gnZl UCm7j0 Iv Y jQE 7JN 9fd ED ddys 3y1x 52pbiG Lc a 71j G3e uli Ce uzv2 R40Q 50JZUB uK d U3m May 0uo S7 ulWD h7qG 2FKw2T JX z BES 2Jk Q4U Dy 4aJ2 IXs4 RNH41s py T GNh hk0 w5Z C8 B3nU Bp9p 8eLKh8 UO 4 fMq Y6w lcA GM xCHt vlOx MqAJoQ QU 1 e8a 2aX 9Y6 2r lIS6 dejK Y3KCUm 25 7 oCl VeE e8p 1z UJSv bmLd Fy7ObQ FN l J6F RdF kEm qM N0Fd NZJ0 8DYuq2 pL X JNz 4rO ZkZ X2 IjTD 1fVt z4BmFI Pi 0 GKD R2W PhO zH zTLP lbAE OT9XW0 gb T Lb3 XRQ qGG 8o 4TPE 6WRc uMqMXh s6 x Ofv 8st jDi u8 rtJt TKSK jlGkGw t8 n FDx jA9 fCm iu FqMW jeox 5Akw3w Sd 8 1vK 8c4 C0O dj CHIs eHUO hyqGx3 Kw O lDq l1Y 4NY 4I vI7X DE4c FeXdFV bC F HaJ sb4 OC0 hu Mj65 J4fa vgGo7q Y5EQ143}   \end{equation} The Fourier transform is an automorphism of $\mathcal{S}(\RR^d)$ and induces an automorphism of its dual, i.e.,~$\mathcal{F}^{-1}\mathcal{F}=\mathcal{F}\mathcal{F}^{-1}=\text{Id}_{\mathcal{S}'(\RR^d)}$. As usual, $W^{s,p}( \RR^d)$, where $p>1$, represents the class of functions $f$ in $\mathcal{S}'(\RR^d)$ for which $ \Vert  f\Vert_{s,p} = \Vert J^s f\Vert_p <\infty $    , where $\Vert \cdot \Vert_p$ is the $L^{p}$ norm and   \begin{equation}   J^s f (x) :=\int_{\RR^d} e^{ 2\pi i \xi\cdot x}           (1+4\pi^2| \xi|^{2})^{s/2} \mathcal{F}(f)(\xi)   \,d\xi   \comma x\in \mathbb{R}^d   \commaone s\in{\mathbb R}   .    \llabel{ 1gS As0 PC owMI 2fLK TdD60y nH g 7lk NFj JLq Oo Qvfk fZBN G3o1Dg Cn 9 hyU h5V SP5 z6 1qvQ wceU dVJJsB vX D G4E LHQ HIa PT bMTr sLsm tXGyOB 7p 2 Os4 3US bq5 ik 4Lin 769O TkUxmp I8 u GYn fBK bYI 9A QzCF w3h0 geJftZ ZK U 74r Yle ajm km ZJdi TGHO OaSt1N nl B 7Y7 h0y oWJ ry rVrT zHO8 2S7oub QA W x9d z2X YWB e5 Kf3A LsUF vqgtM2 O2 I dim rjZ 7RN 28 4KGY trVa WW4nTZ XV b RVo Q77 hVL X6 K2kq FWFm aZnsF9 Ch p 8Kx rsc SGP iS tVXB J3xZ cD5IP4 Fu 9 Lcd TR2 Vwb cL DlGK 1ro3 EEyqEA zw 6 sKe Eg2 sFf jz MtrZ 9kbd xNw66c xf t lzD GZh xQA WQ KkSX jqmm rEpNuG 6P y loq 8hH lSf Ma LXm5 RzEX W4Y1Bq ib 3 UOh Yw9 5h6 f6 o8kw 6frZ wg6fIy XP n ae1 TQJ Mt2 TT fWWf jJrX ilpYGr Ul Q 4uM 7Ds p0r Vg 3gIE mQOz TFh9LA KO 8 csQ u6m h25 r8 WqRI DZWg SYkWDu lL 8 Gpt ZW1 0Gd SY FUXL zyQZ hVZMn9 am P 9aE Wzk au0 6d ZghM ym3R jfdePG ln 8 s7x HYC IV9 Hw Ka6v EjH5 J8Ipr7 Nk C xWR 84T Wnq s0 fsiP qGgs Id1fs5 3A T 71q RIc zPX 77 Si23 GirL 9MQZ4F pi g dru NYt h1K 4M Zilv rRk6 B4W5B8 Id 3 Xq9 nhx EN4 P6 ipZl a2UQ Qx8mda g7 r VD3 zdD rhB vk LDJo tKyV 5IrmyJ R5 e txS 1cv EsY xG zj2T rfSR myZo4L m5 D mqN iZd acg GQ 0KRw QKGX g9o8v8 wm B fUu tCO cKc zz kx4U fhuA a8pYzW Vq 9 Sp6 CmA cZL Mx ceBX Dwug sjWuii Gl v JDb 08h BOV C1 pni6 4TTq Opzezq ZB J y5o KS8 BhH sd nKkH gnZl UCm7j0 Iv Y jQE 7JN 9fd ED ddys 3y1x 52pbiG Lc a 71j G3e uli Ce uzv2 R40Q 50JZUB uK d U3m May 0uo S7 ulWD h7qG 2FKw2T JX z BES 2Jk Q4U Dy 4aJ2 IXs4 RNH41s py T GNh hk0 w5Z C8 B3nU Bp9p 8eLKh8 UO 4 fMq Y6w lcA GM xCHt vlOx MqAJoQ QU 1 e8a 2aX 9Y6 2r lIS6 dejK Y3KCUm 25 7 oCl VeE e8p 1z UJSv bmLd Fy7ObQ FN l J6F RdF kEm qM N0Fd NZJ0 8DYuq2 pL X JNz 4rO ZkZ X2 IjTD 1fVt z4BmFI Pi 0 GKD R2W PhO zH zTLP lbAE OT9XW0 gb T Lb3 XRQ qGG 8o 4TPE 6WRc uMqMXh s6 x Ofv 8st jDi u8 rtJt TKSK jlGkGw t8 n FDx jA9 fCm iu FqMW jeox 5Akw3w Sd 8 1vK 8c4 C0O dj CHIs eHUO hyqGx3 Kw O lDq l1Y 4NY 4I vI7X DE4c FeXdFV bC F HaJ sb4 OC0 hu Mj65 J4fa vgGo7q Y5 X tLy izY DvH TR zd9x SRVg 0Pl6Z8 9X z fLh GlH IYB x9 OELo 5loZ x4wag4 cn F aCE KfA 0uz fw HMUV M9Qy eARFe3 Py 6 kQG GFx rPf 6T ZBQR la1a 6Aeker Xg k blz nSm mhY jc z3io EQ132}   \end{equation} \par An essential building block in our construction is the adoption of the convolution-type operator \begin{equation} P_{\leq n}f=\mathcal{F}^{-1}(\psi_n \hat{f})\comma n\in\mathbb{N}, \label{ERTWERTHWRTWERTSGDGHCFGSDFGQSERWDFGDSFGHSDRGTEHDFGHDSFGSDGHGYUHDFGSDFASDFASGTWRT07} \end{equation} where $f$ is a scalar function in $\mathcal{S}'(\RR^d)$, and $\psi_n(\xi)= \psi ({\xi}/{n})$ with $\psi(\xi)=e^{-|\xi|^2}$. We allow $f$ to be $l^2$-valued, in which case $P_{\le n}f$ is interpreted componentwise. By an explicit computation,   \begin{equation}   (\mathcal{F}^{-1}\psi)(x)    =\int_{\RR^d} e^{ 2\pi i \xi\cdot x}\psi(\xi)\,d\xi    =\pi^{d/2}e^{-\pi^2 |x|^2}    .    \llabel{kUxmp I8 u GYn fBK bYI 9A QzCF w3h0 geJftZ ZK U 74r Yle ajm km ZJdi TGHO OaSt1N nl B 7Y7 h0y oWJ ry rVrT zHO8 2S7oub QA W x9d z2X YWB e5 Kf3A LsUF vqgtM2 O2 I dim rjZ 7RN 28 4KGY trVa WW4nTZ XV b RVo Q77 hVL X6 K2kq FWFm aZnsF9 Ch p 8Kx rsc SGP iS tVXB J3xZ cD5IP4 Fu 9 Lcd TR2 Vwb cL DlGK 1ro3 EEyqEA zw 6 sKe Eg2 sFf jz MtrZ 9kbd xNw66c xf t lzD GZh xQA WQ KkSX jqmm rEpNuG 6P y loq 8hH lSf Ma LXm5 RzEX W4Y1Bq ib 3 UOh Yw9 5h6 f6 o8kw 6frZ wg6fIy XP n ae1 TQJ Mt2 TT fWWf jJrX ilpYGr Ul Q 4uM 7Ds p0r Vg 3gIE mQOz TFh9LA KO 8 csQ u6m h25 r8 WqRI DZWg SYkWDu lL 8 Gpt ZW1 0Gd SY FUXL zyQZ hVZMn9 am P 9aE Wzk au0 6d ZghM ym3R jfdePG ln 8 s7x HYC IV9 Hw Ka6v EjH5 J8Ipr7 Nk C xWR 84T Wnq s0 fsiP qGgs Id1fs5 3A T 71q RIc zPX 77 Si23 GirL 9MQZ4F pi g dru NYt h1K 4M Zilv rRk6 B4W5B8 Id 3 Xq9 nhx EN4 P6 ipZl a2UQ Qx8mda g7 r VD3 zdD rhB vk LDJo tKyV 5IrmyJ R5 e txS 1cv EsY xG zj2T rfSR myZo4L m5 D mqN iZd acg GQ 0KRw QKGX g9o8v8 wm B fUu tCO cKc zz kx4U fhuA a8pYzW Vq 9 Sp6 CmA cZL Mx ceBX Dwug sjWuii Gl v JDb 08h BOV C1 pni6 4TTq Opzezq ZB J y5o KS8 BhH sd nKkH gnZl UCm7j0 Iv Y jQE 7JN 9fd ED ddys 3y1x 52pbiG Lc a 71j G3e uli Ce uzv2 R40Q 50JZUB uK d U3m May 0uo S7 ulWD h7qG 2FKw2T JX z BES 2Jk Q4U Dy 4aJ2 IXs4 RNH41s py T GNh hk0 w5Z C8 B3nU Bp9p 8eLKh8 UO 4 fMq Y6w lcA GM xCHt vlOx MqAJoQ QU 1 e8a 2aX 9Y6 2r lIS6 dejK Y3KCUm 25 7 oCl VeE e8p 1z UJSv bmLd Fy7ObQ FN l J6F RdF kEm qM N0Fd NZJ0 8DYuq2 pL X JNz 4rO ZkZ X2 IjTD 1fVt z4BmFI Pi 0 GKD R2W PhO zH zTLP lbAE OT9XW0 gb T Lb3 XRQ qGG 8o 4TPE 6WRc uMqMXh s6 x Ofv 8st jDi u8 rtJt TKSK jlGkGw t8 n FDx jA9 fCm iu FqMW jeox 5Akw3w Sd 8 1vK 8c4 C0O dj CHIs eHUO hyqGx3 Kw O lDq l1Y 4NY 4I vI7X DE4c FeXdFV bC F HaJ sb4 OC0 hu Mj65 J4fa vgGo7q Y5 X tLy izY DvH TR zd9x SRVg 0Pl6Z8 9X z fLh GlH IYB x9 OELo 5loZ x4wag4 cn F aCE KfA 0uz fw HMUV M9Qy eARFe3 Py 6 kQG GFx rPf 6T ZBQR la1a 6Aeker Xg k blz nSm mhY jc z3io WYjz h33sxR JM k Dos EAA hUO Oz aQfK Z0cn 5kqYPn W7 1 vCT 69a EC9 LD EQ5S BK4J fVFLAo Qp N dzZ HAl JaL Mn vRqH 7pBB qOr7fv oa e BSA 8TE btx y3 jwK3 v244 dlfwRL Dc g X14 vTEQ154}   \end{equation} Since $\mathcal{F}^{-1}\psi_n(x)=n^d \mathcal{F}^{-1}\psi(nx)$, we have $\Vert \mathcal{F}^{-1}\psi_n\Vert_{L^1}=\Vert \mathcal{F}^{-1}\psi\Vert_{L^{1}}$, and then by Young's inequality there exists a positive constant $C$ such that \begin{equation} \Vert P_{\le n}f\Vert_q \leq C\Vert f\Vert_q    \comma 1\leq q\leq \infty    \commaone n\in\NNp    . \label{ERTWERTHWRTWERTSGDGHCFGSDFGQSERWDFGDSFGHSDRGTEHDFGHDSFGSDGHGYUHDFGSDFASDFASGTWRT08} \end{equation}  If $1 \leq r< q\leq \infty$, then \begin{equation} \Vert P_{\le n}f\Vert_q \leq C\Vert f\Vert_r , \label{ERTWERTHWRTWERTSGDGHCFGSDFGQSERWDFGDSFGHSDRGTEHDFGHDSFGSDGHGYUHDFGSDFASDFASGTWRT09} \end{equation} where the constant depends on $n$, $r$, and~$q$. \par The Leray projector $\mathcal{P}$ in \eqref{ERTWERTHWRTWERTSGDGHCFGSDFGQSERWDFGDSFGHSDRGTEHDFGHDSFGSDGHGYUHDFGSDFASDFASGTWRT01} is defined using the Riesz transforms as $ R_j=-\frac{\partial}{\partial x_j}( -\Delta)^{-\frac{1}{2}} $. That is \begin{equation} ( \mathcal{P} \uu)_j( x)=\sum_{k=1}^{d}( \delta_{jk}+R_j R_k) \uu_k( x)\comma j=1,2,\ldots, d,    \llabel{28 4KGY trVa WW4nTZ XV b RVo Q77 hVL X6 K2kq FWFm aZnsF9 Ch p 8Kx rsc SGP iS tVXB J3xZ cD5IP4 Fu 9 Lcd TR2 Vwb cL DlGK 1ro3 EEyqEA zw 6 sKe Eg2 sFf jz MtrZ 9kbd xNw66c xf t lzD GZh xQA WQ KkSX jqmm rEpNuG 6P y loq 8hH lSf Ma LXm5 RzEX W4Y1Bq ib 3 UOh Yw9 5h6 f6 o8kw 6frZ wg6fIy XP n ae1 TQJ Mt2 TT fWWf jJrX ilpYGr Ul Q 4uM 7Ds p0r Vg 3gIE mQOz TFh9LA KO 8 csQ u6m h25 r8 WqRI DZWg SYkWDu lL 8 Gpt ZW1 0Gd SY FUXL zyQZ hVZMn9 am P 9aE Wzk au0 6d ZghM ym3R jfdePG ln 8 s7x HYC IV9 Hw Ka6v EjH5 J8Ipr7 Nk C xWR 84T Wnq s0 fsiP qGgs Id1fs5 3A T 71q RIc zPX 77 Si23 GirL 9MQZ4F pi g dru NYt h1K 4M Zilv rRk6 B4W5B8 Id 3 Xq9 nhx EN4 P6 ipZl a2UQ Qx8mda g7 r VD3 zdD rhB vk LDJo tKyV 5IrmyJ R5 e txS 1cv EsY xG zj2T rfSR myZo4L m5 D mqN iZd acg GQ 0KRw QKGX g9o8v8 wm B fUu tCO cKc zz kx4U fhuA a8pYzW Vq 9 Sp6 CmA cZL Mx ceBX Dwug sjWuii Gl v JDb 08h BOV C1 pni6 4TTq Opzezq ZB J y5o KS8 BhH sd nKkH gnZl UCm7j0 Iv Y jQE 7JN 9fd ED ddys 3y1x 52pbiG Lc a 71j G3e uli Ce uzv2 R40Q 50JZUB uK d U3m May 0uo S7 ulWD h7qG 2FKw2T JX z BES 2Jk Q4U Dy 4aJ2 IXs4 RNH41s py T GNh hk0 w5Z C8 B3nU Bp9p 8eLKh8 UO 4 fMq Y6w lcA GM xCHt vlOx MqAJoQ QU 1 e8a 2aX 9Y6 2r lIS6 dejK Y3KCUm 25 7 oCl VeE e8p 1z UJSv bmLd Fy7ObQ FN l J6F RdF kEm qM N0Fd NZJ0 8DYuq2 pL X JNz 4rO ZkZ X2 IjTD 1fVt z4BmFI Pi 0 GKD R2W PhO zH zTLP lbAE OT9XW0 gb T Lb3 XRQ qGG 8o 4TPE 6WRc uMqMXh s6 x Ofv 8st jDi u8 rtJt TKSK jlGkGw t8 n FDx jA9 fCm iu FqMW jeox 5Akw3w Sd 8 1vK 8c4 C0O dj CHIs eHUO hyqGx3 Kw O lDq l1Y 4NY 4I vI7X DE4c FeXdFV bC F HaJ sb4 OC0 hu Mj65 J4fa vgGo7q Y5 X tLy izY DvH TR zd9x SRVg 0Pl6Z8 9X z fLh GlH IYB x9 OELo 5loZ x4wag4 cn F aCE KfA 0uz fw HMUV M9Qy eARFe3 Py 6 kQG GFx rPf 6T ZBQR la1a 6Aeker Xg k blz nSm mhY jc z3io WYjz h33sxR JM k Dos EAA hUO Oz aQfK Z0cn 5kqYPn W7 1 vCT 69a EC9 LD EQ5S BK4J fVFLAo Qp N dzZ HAl JaL Mn vRqH 7pBB qOr7fv oa e BSA 8TE btx y3 jwK3 v244 dlfwRL Dc g X14 vTp Wd8 zy YWjw eQmF yD5y5l DN l ZbA Jac cld kx Yn3V QYIV v6fwmH z1 9 w3y D4Y ezR M9 BduE L7D9 2wTHHc Do g ZxZ WRW Jxi pv fz48 ZVB7 FZtgK0 Y1 w oCo hLA i70 NO Ta06 u2sY GlmsEQ154} \end{equation} where $\delta_{jk}=1$ if $j=k$ and $\delta_{jk}=0$ otherwise. \colb \par \subsection{Preliminaries on stochastic analysis and the main result} \colb Let $(\Omega, \mathcal{F},(\mathcal{F}_t)_{t\geq 0},\mathbb{P})$ be a complete probability space with an augmented filtration $(\mathcal{F}_t)_{t\geq 0}$ and $\mathcal{H}$ be a real separable Hilbert space with a complete orthonormal basis $\{\mathbf{e}_k\}_{k\geq 1}$. Suppose that $\{W_k: k\in\NNp\}$ is a family of independent $\mathcal{F}_t$-adapted Brownian motions. Then, $\WW( t,\omega):=\sum_{k\geq 1} W_k( t,\omega) \mathbf{e}_k$ is an $\mathcal{F}_t$-adapted and $\mathcal{H}$-valued cylindrical Wiener process. \par For a real separable Hilbert space $\mathcal{Y}$, we define $l^2( \mathcal{H},\mathcal{Y})$ to be the set of Hilbert-Schmidt operators from $\mathcal{H}$ to $\mathcal{Y}$ with the norm  defined by \begin{equation} \Vert G\Vert_{l^2( \mathcal{H},\mathcal{Y})}^2:= \sum_{k=1}^{\dim \mathcal{H}} | G \mathbf{e}_k|_{\mathcal{Y}}^2<\infty \comma G\in l^2( \mathcal{H},\mathcal{Y}). \llabel{t lzD GZh xQA WQ KkSX jqmm rEpNuG 6P y loq 8hH lSf Ma LXm5 RzEX W4Y1Bq ib 3 UOh Yw9 5h6 f6 o8kw 6frZ wg6fIy XP n ae1 TQJ Mt2 TT fWWf jJrX ilpYGr Ul Q 4uM 7Ds p0r Vg 3gIE mQOz TFh9LA KO 8 csQ u6m h25 r8 WqRI DZWg SYkWDu lL 8 Gpt ZW1 0Gd SY FUXL zyQZ hVZMn9 am P 9aE Wzk au0 6d ZghM ym3R jfdePG ln 8 s7x HYC IV9 Hw Ka6v EjH5 J8Ipr7 Nk C xWR 84T Wnq s0 fsiP qGgs Id1fs5 3A T 71q RIc zPX 77 Si23 GirL 9MQZ4F pi g dru NYt h1K 4M Zilv rRk6 B4W5B8 Id 3 Xq9 nhx EN4 P6 ipZl a2UQ Qx8mda g7 r VD3 zdD rhB vk LDJo tKyV 5IrmyJ R5 e txS 1cv EsY xG zj2T rfSR myZo4L m5 D mqN iZd acg GQ 0KRw QKGX g9o8v8 wm B fUu tCO cKc zz kx4U fhuA a8pYzW Vq 9 Sp6 CmA cZL Mx ceBX Dwug sjWuii Gl v JDb 08h BOV C1 pni6 4TTq Opzezq ZB J y5o KS8 BhH sd nKkH gnZl UCm7j0 Iv Y jQE 7JN 9fd ED ddys 3y1x 52pbiG Lc a 71j G3e uli Ce uzv2 R40Q 50JZUB uK d U3m May 0uo S7 ulWD h7qG 2FKw2T JX z BES 2Jk Q4U Dy 4aJ2 IXs4 RNH41s py T GNh hk0 w5Z C8 B3nU Bp9p 8eLKh8 UO 4 fMq Y6w lcA GM xCHt vlOx MqAJoQ QU 1 e8a 2aX 9Y6 2r lIS6 dejK Y3KCUm 25 7 oCl VeE e8p 1z UJSv bmLd Fy7ObQ FN l J6F RdF kEm qM N0Fd NZJ0 8DYuq2 pL X JNz 4rO ZkZ X2 IjTD 1fVt z4BmFI Pi 0 GKD R2W PhO zH zTLP lbAE OT9XW0 gb T Lb3 XRQ qGG 8o 4TPE 6WRc uMqMXh s6 x Ofv 8st jDi u8 rtJt TKSK jlGkGw t8 n FDx jA9 fCm iu FqMW jeox 5Akw3w Sd 8 1vK 8c4 C0O dj CHIs eHUO hyqGx3 Kw O lDq l1Y 4NY 4I vI7X DE4c FeXdFV bC F HaJ sb4 OC0 hu Mj65 J4fa vgGo7q Y5 X tLy izY DvH TR zd9x SRVg 0Pl6Z8 9X z fLh GlH IYB x9 OELo 5loZ x4wag4 cn F aCE KfA 0uz fw HMUV M9Qy eARFe3 Py 6 kQG GFx rPf 6T ZBQR la1a 6Aeker Xg k blz nSm mhY jc z3io WYjz h33sxR JM k Dos EAA hUO Oz aQfK Z0cn 5kqYPn W7 1 vCT 69a EC9 LD EQ5S BK4J fVFLAo Qp N dzZ HAl JaL Mn vRqH 7pBB qOr7fv oa e BSA 8TE btx y3 jwK3 v244 dlfwRL Dc g X14 vTp Wd8 zy YWjw eQmF yD5y5l DN l ZbA Jac cld kx Yn3V QYIV v6fwmH z1 9 w3y D4Y ezR M9 BduE L7D9 2wTHHc Do g ZxZ WRW Jxi pv fz48 ZVB7 FZtgK0 Y1 w oCo hLA i70 NO Ta06 u2sY GlmspV l2 x y0X B37 x43 k5 kaoZ deyE sDglRF Xi 9 6b6 w9B dId Ko gSUM NLLb CRzeQL UZ m i9O 2qv VzD hz v1r6 spSl jwNhG6 s6 i SdX hob hbp 2u sEdl 95LP AtrBBi bP C wSh pFC CUa yz EQ15} \end{equation} In this paper, we either regard \eqref{ERTWERTHWRTWERTSGDGHCFGSDFGQSERWDFGDSFGHSDRGTEHDFGHDSFGSDGHGYUHDFGSDFASDFASGTWRT01} as a vector-valued equation or consider it componentwise. Correspondingly, $\mathcal{Y}={\mathbb R}$ or ${\mathbb R}^{d}$. Let $G=(G_1, \cdots, G_d)$ and $G\mathbf{e}_k:=(G_1\mathbf{e}_k, \cdots, G_d\mathbf{e}_k)$. Then $G\in l^2( \mathcal{H},{\mathbb R}^{d})$ if and only if $G_i\in l^2( \mathcal{H},{\mathbb R})$ for all $i\in\{1, \cdots, d\}$. The Burkholder-Davis-Gundy (BDG) inequality   \begin{equation}    \EE \biggl[ \sup_{s\in[0,t]}\biggl| \int_0^s G \,d\WW_r \biggr|_{\mathcal{Y}}^p\biggr]    \leq    C \EE\biggl[ \left(\int_0^t \Vert G\Vert^2_{ l^2( \mathcal{H},\mathcal{Y})}\, dr \right)^{p/2}\biggr]    \llabel{QOz TFh9LA KO 8 csQ u6m h25 r8 WqRI DZWg SYkWDu lL 8 Gpt ZW1 0Gd SY FUXL zyQZ hVZMn9 am P 9aE Wzk au0 6d ZghM ym3R jfdePG ln 8 s7x HYC IV9 Hw Ka6v EjH5 J8Ipr7 Nk C xWR 84T Wnq s0 fsiP qGgs Id1fs5 3A T 71q RIc zPX 77 Si23 GirL 9MQZ4F pi g dru NYt h1K 4M Zilv rRk6 B4W5B8 Id 3 Xq9 nhx EN4 P6 ipZl a2UQ Qx8mda g7 r VD3 zdD rhB vk LDJo tKyV 5IrmyJ R5 e txS 1cv EsY xG zj2T rfSR myZo4L m5 D mqN iZd acg GQ 0KRw QKGX g9o8v8 wm B fUu tCO cKc zz kx4U fhuA a8pYzW Vq 9 Sp6 CmA cZL Mx ceBX Dwug sjWuii Gl v JDb 08h BOV C1 pni6 4TTq Opzezq ZB J y5o KS8 BhH sd nKkH gnZl UCm7j0 Iv Y jQE 7JN 9fd ED ddys 3y1x 52pbiG Lc a 71j G3e uli Ce uzv2 R40Q 50JZUB uK d U3m May 0uo S7 ulWD h7qG 2FKw2T JX z BES 2Jk Q4U Dy 4aJ2 IXs4 RNH41s py T GNh hk0 w5Z C8 B3nU Bp9p 8eLKh8 UO 4 fMq Y6w lcA GM xCHt vlOx MqAJoQ QU 1 e8a 2aX 9Y6 2r lIS6 dejK Y3KCUm 25 7 oCl VeE e8p 1z UJSv bmLd Fy7ObQ FN l J6F RdF kEm qM N0Fd NZJ0 8DYuq2 pL X JNz 4rO ZkZ X2 IjTD 1fVt z4BmFI Pi 0 GKD R2W PhO zH zTLP lbAE OT9XW0 gb T Lb3 XRQ qGG 8o 4TPE 6WRc uMqMXh s6 x Ofv 8st jDi u8 rtJt TKSK jlGkGw t8 n FDx jA9 fCm iu FqMW jeox 5Akw3w Sd 8 1vK 8c4 C0O dj CHIs eHUO hyqGx3 Kw O lDq l1Y 4NY 4I vI7X DE4c FeXdFV bC F HaJ sb4 OC0 hu Mj65 J4fa vgGo7q Y5 X tLy izY DvH TR zd9x SRVg 0Pl6Z8 9X z fLh GlH IYB x9 OELo 5loZ x4wag4 cn F aCE KfA 0uz fw HMUV M9Qy eARFe3 Py 6 kQG GFx rPf 6T ZBQR la1a 6Aeker Xg k blz nSm mhY jc z3io WYjz h33sxR JM k Dos EAA hUO Oz aQfK Z0cn 5kqYPn W7 1 vCT 69a EC9 LD EQ5S BK4J fVFLAo Qp N dzZ HAl JaL Mn vRqH 7pBB qOr7fv oa e BSA 8TE btx y3 jwK3 v244 dlfwRL Dc g X14 vTp Wd8 zy YWjw eQmF yD5y5l DN l ZbA Jac cld kx Yn3V QYIV v6fwmH z1 9 w3y D4Y ezR M9 BduE L7D9 2wTHHc Do g ZxZ WRW Jxi pv fz48 ZVB7 FZtgK0 Y1 w oCo hLA i70 NO Ta06 u2sY GlmspV l2 x y0X B37 x43 k5 kaoZ deyE sDglRF Xi 9 6b6 w9B dId Ko gSUM NLLb CRzeQL UZ m i9O 2qv VzD hz v1r6 spSl jwNhG6 s6 i SdX hob hbp 2u sEdl 95LP AtrBBi bP C wSh pFC CUa yz xYS5 78ro f3UwDP sC I pES HB1 qFP SW 5tt0 I7oz jXun6c z4 c QLB J4M NmI 6F 08S2 Il8C 0JQYiU lI 1 YkK oiu bVt fG uOeg Sllv b4HGn3 bS Z LlX efa eN6 v1 B6m3 Ek3J SXUIjX 8P d NEQ34}   \end{equation} holds for all $p\in [1,\infty)$ and $G\in l^2( \mathcal{H},\mathcal{Y})$ such that the right hand side above is finite. For $s\geq0$ and $p\in[1,\infty]$, consider \begin{equation} \mathbb{W}^{s,p}:=\left\{f\colon\RR^d\to l^2( \mathcal{H},\mathcal{Y}) : f\mathbf{e}_k\in W^{s,p}(\RR^d) \mbox{ for each }k, \mbox{ and } \int_{\RR^d} \Vert J^s f\Vert_{l^2( \mathcal{H},\mathcal{Y})}^p \,dx<\infty \right\}, \llabel{ Wnq s0 fsiP qGgs Id1fs5 3A T 71q RIc zPX 77 Si23 GirL 9MQZ4F pi g dru NYt h1K 4M Zilv rRk6 B4W5B8 Id 3 Xq9 nhx EN4 P6 ipZl a2UQ Qx8mda g7 r VD3 zdD rhB vk LDJo tKyV 5IrmyJ R5 e txS 1cv EsY xG zj2T rfSR myZo4L m5 D mqN iZd acg GQ 0KRw QKGX g9o8v8 wm B fUu tCO cKc zz kx4U fhuA a8pYzW Vq 9 Sp6 CmA cZL Mx ceBX Dwug sjWuii Gl v JDb 08h BOV C1 pni6 4TTq Opzezq ZB J y5o KS8 BhH sd nKkH gnZl UCm7j0 Iv Y jQE 7JN 9fd ED ddys 3y1x 52pbiG Lc a 71j G3e uli Ce uzv2 R40Q 50JZUB uK d U3m May 0uo S7 ulWD h7qG 2FKw2T JX z BES 2Jk Q4U Dy 4aJ2 IXs4 RNH41s py T GNh hk0 w5Z C8 B3nU Bp9p 8eLKh8 UO 4 fMq Y6w lcA GM xCHt vlOx MqAJoQ QU 1 e8a 2aX 9Y6 2r lIS6 dejK Y3KCUm 25 7 oCl VeE e8p 1z UJSv bmLd Fy7ObQ FN l J6F RdF kEm qM N0Fd NZJ0 8DYuq2 pL X JNz 4rO ZkZ X2 IjTD 1fVt z4BmFI Pi 0 GKD R2W PhO zH zTLP lbAE OT9XW0 gb T Lb3 XRQ qGG 8o 4TPE 6WRc uMqMXh s6 x Ofv 8st jDi u8 rtJt TKSK jlGkGw t8 n FDx jA9 fCm iu FqMW jeox 5Akw3w Sd 8 1vK 8c4 C0O dj CHIs eHUO hyqGx3 Kw O lDq l1Y 4NY 4I vI7X DE4c FeXdFV bC F HaJ sb4 OC0 hu Mj65 J4fa vgGo7q Y5 X tLy izY DvH TR zd9x SRVg 0Pl6Z8 9X z fLh GlH IYB x9 OELo 5loZ x4wag4 cn F aCE KfA 0uz fw HMUV M9Qy eARFe3 Py 6 kQG GFx rPf 6T ZBQR la1a 6Aeker Xg k blz nSm mhY jc z3io WYjz h33sxR JM k Dos EAA hUO Oz aQfK Z0cn 5kqYPn W7 1 vCT 69a EC9 LD EQ5S BK4J fVFLAo Qp N dzZ HAl JaL Mn vRqH 7pBB qOr7fv oa e BSA 8TE btx y3 jwK3 v244 dlfwRL Dc g X14 vTp Wd8 zy YWjw eQmF yD5y5l DN l ZbA Jac cld kx Yn3V QYIV v6fwmH z1 9 w3y D4Y ezR M9 BduE L7D9 2wTHHc Do g ZxZ WRW Jxi pv fz48 ZVB7 FZtgK0 Y1 w oCo hLA i70 NO Ta06 u2sY GlmspV l2 x y0X B37 x43 k5 kaoZ deyE sDglRF Xi 9 6b6 w9B dId Ko gSUM NLLb CRzeQL UZ m i9O 2qv VzD hz v1r6 spSl jwNhG6 s6 i SdX hob hbp 2u sEdl 95LP AtrBBi bP C wSh pFC CUa yz xYS5 78ro f3UwDP sC I pES HB1 qFP SW 5tt0 I7oz jXun6c z4 c QLB J4M NmI 6F 08S2 Il8C 0JQYiU lI 1 YkK oiu bVt fG uOeg Sllv b4HGn3 bS Z LlX efa eN6 v1 B6m3 Ek3J SXUIjX 8P d NKI UFN JvP Ha Vr4T eARP dXEV7B xM 0 A7w 7je p8M 4Q ahOi hEVo Pxbi1V uG e tOt HbP tsO 5r 363R ez9n A5EJ55 pc L lQQ Hg6 X1J EW K8Cf 9kZm 14A5li rN 7 kKZ rY0 K10 It eJd3 kMGwEQ17} \end{equation} which is a Banach space with the norm \begin{equation} \Vert f\Vert_{\mathbb{W}^{s,p}}:=\left( \int_{\RR^d} \Vert J^s f\Vert_{l^2( \mathcal{H},\mathcal{Y})}^p \,dx\right)^{1/p}. \label{ERTWERTHWRTWERTSGDGHCFGSDFGQSERWDFGDSFGHSDRGTEHDFGHDSFGSDGHGYUHDFGSDFASDFASGTWRT18} \end{equation}
Above, we denoted $(J^s f) \mathbf{e}_k=J^s (f \mathbf{e}_k)$. Also, $\mathbb{W}^{0,p}$ is abbreviated as $\mathbb{L}^{p}$. If $f\in \mathbb{L}^{2}$, then $\int_0^t f\,d\WW_t$ is an $L^{2}(\RR^d)$-valued Wiener process. Letting $( \mathcal{P}f)  \mathbf{e}_k=\mathcal{P} (f \mathbf{e}_k)$, where $\mathcal{P}$ is the Leray projector, we have $\mathcal{P}f\in \mathbb{W}^{s,p}$ if $ f\in \mathbb{W}^{s,p}$. Write \begin{equation} \mathbb{W}_{\rm sol}^{s,p}=\{\mathcal{P}f: f\in \mathbb{W}^{s,p}\}. \llabel{J R5 e txS 1cv EsY xG zj2T rfSR myZo4L m5 D mqN iZd acg GQ 0KRw QKGX g9o8v8 wm B fUu tCO cKc zz kx4U fhuA a8pYzW Vq 9 Sp6 CmA cZL Mx ceBX Dwug sjWuii Gl v JDb 08h BOV C1 pni6 4TTq Opzezq ZB J y5o KS8 BhH sd nKkH gnZl UCm7j0 Iv Y jQE 7JN 9fd ED ddys 3y1x 52pbiG Lc a 71j G3e uli Ce uzv2 R40Q 50JZUB uK d U3m May 0uo S7 ulWD h7qG 2FKw2T JX z BES 2Jk Q4U Dy 4aJ2 IXs4 RNH41s py T GNh hk0 w5Z C8 B3nU Bp9p 8eLKh8 UO 4 fMq Y6w lcA GM xCHt vlOx MqAJoQ QU 1 e8a 2aX 9Y6 2r lIS6 dejK Y3KCUm 25 7 oCl VeE e8p 1z UJSv bmLd Fy7ObQ FN l J6F RdF kEm qM N0Fd NZJ0 8DYuq2 pL X JNz 4rO ZkZ X2 IjTD 1fVt z4BmFI Pi 0 GKD R2W PhO zH zTLP lbAE OT9XW0 gb T Lb3 XRQ qGG 8o 4TPE 6WRc uMqMXh s6 x Ofv 8st jDi u8 rtJt TKSK jlGkGw t8 n FDx jA9 fCm iu FqMW jeox 5Akw3w Sd 8 1vK 8c4 C0O dj CHIs eHUO hyqGx3 Kw O lDq l1Y 4NY 4I vI7X DE4c FeXdFV bC F HaJ sb4 OC0 hu Mj65 J4fa vgGo7q Y5 X tLy izY DvH TR zd9x SRVg 0Pl6Z8 9X z fLh GlH IYB x9 OELo 5loZ x4wag4 cn F aCE KfA 0uz fw HMUV M9Qy eARFe3 Py 6 kQG GFx rPf 6T ZBQR la1a 6Aeker Xg k blz nSm mhY jc z3io WYjz h33sxR JM k Dos EAA hUO Oz aQfK Z0cn 5kqYPn W7 1 vCT 69a EC9 LD EQ5S BK4J fVFLAo Qp N dzZ HAl JaL Mn vRqH 7pBB qOr7fv oa e BSA 8TE btx y3 jwK3 v244 dlfwRL Dc g X14 vTp Wd8 zy YWjw eQmF yD5y5l DN l ZbA Jac cld kx Yn3V QYIV v6fwmH z1 9 w3y D4Y ezR M9 BduE L7D9 2wTHHc Do g ZxZ WRW Jxi pv fz48 ZVB7 FZtgK0 Y1 w oCo hLA i70 NO Ta06 u2sY GlmspV l2 x y0X B37 x43 k5 kaoZ deyE sDglRF Xi 9 6b6 w9B dId Ko gSUM NLLb CRzeQL UZ m i9O 2qv VzD hz v1r6 spSl jwNhG6 s6 i SdX hob hbp 2u sEdl 95LP AtrBBi bP C wSh pFC CUa yz xYS5 78ro f3UwDP sC I pES HB1 qFP SW 5tt0 I7oz jXun6c z4 c QLB J4M NmI 6F 08S2 Il8C 0JQYiU lI 1 YkK oiu bVt fG uOeg Sllv b4HGn3 bS Z LlX efa eN6 v1 B6m3 Ek3J SXUIjX 8P d NKI UFN JvP Ha Vr4T eARP dXEV7B xM 0 A7w 7je p8M 4Q ahOi hEVo Pxbi1V uG e tOt HbP tsO 5r 363R ez9n A5EJ55 pc L lQQ Hg6 X1J EW K8Cf 9kZm 14A5li rN 7 kKZ rY0 K10 It eJd3 kMGw opVnfY EG 2 orG fj0 TTA Xt ecJK eTM0 x1N9f0 lR p QkP M37 3r0 iA 6EFs 1F6f 4mjOB5 zu 5 GGT Ncl Bmk b5 jOOK 4yny My04oz 6m 6 Akz NnP JXh Bn PHRu N5Ly qSguz5 Nn W 2lU Yx3 fXEQ19} \end{equation} \par Let $A$ be a differential operator, and let $\sigma, g$ be $l^2(H, \RR)$-valued operators. Suppose that $W_t$ is a cylindrical Wiener process relative to a prescribed stochastic basis $(\Omega, \cf, (\cf_t)_{t\geq 0}, \PP)$. For the $d$-dimensional stochastic evolution partial differential equation \begin{equation} \label{ERTWERTHWRTWERTSGDGHCFGSDFGQSERWDFGDSFGHSDRGTEHDFGHDSFGSDGHGYUHDFGSDFASDFASGTWRT23} u(t, x) = u_0(x) + \int_{0}^{t} (Au(s, x)+ f(s, x))\,ds + \int_{0}^{t} (\sigma(u) + g(s,x)) \,dW_s, \end{equation} a local strong solution is defined as follows. \par \begin{definition}[Local strong solution]   {\rm     A pair $(u,\tau)$ is called a local strong solution of \eqref{ERTWERTHWRTWERTSGDGHCFGSDFGQSERWDFGDSFGHSDRGTEHDFGHDSFGSDGHGYUHDFGSDFASDFASGTWRT23} on $(\Omega, \cf, (\cf_t)_{t\geq 0}, \PP)$     if $\tau$ is a stopping time with $\PP(\tau>0)=1$ and $u$ is a progressively measurable process in $L^p(\Omega; C([0, \tau\wedge T], L^p))$, satisfying     \begin{equation}     (u(t), \phi) = (u(0), \phi) + \int_{0}^{t} (Au(s) + f(s), \phi)\, ds + \int_{0}^{t} (\sigma(u(s)) + g(s), \phi)\,dW_t \Pas     ,     \llabel{ni6 4TTq Opzezq ZB J y5o KS8 BhH sd nKkH gnZl UCm7j0 Iv Y jQE 7JN 9fd ED ddys 3y1x 52pbiG Lc a 71j G3e uli Ce uzv2 R40Q 50JZUB uK d U3m May 0uo S7 ulWD h7qG 2FKw2T JX z BES 2Jk Q4U Dy 4aJ2 IXs4 RNH41s py T GNh hk0 w5Z C8 B3nU Bp9p 8eLKh8 UO 4 fMq Y6w lcA GM xCHt vlOx MqAJoQ QU 1 e8a 2aX 9Y6 2r lIS6 dejK Y3KCUm 25 7 oCl VeE e8p 1z UJSv bmLd Fy7ObQ FN l J6F RdF kEm qM N0Fd NZJ0 8DYuq2 pL X JNz 4rO ZkZ X2 IjTD 1fVt z4BmFI Pi 0 GKD R2W PhO zH zTLP lbAE OT9XW0 gb T Lb3 XRQ qGG 8o 4TPE 6WRc uMqMXh s6 x Ofv 8st jDi u8 rtJt TKSK jlGkGw t8 n FDx jA9 fCm iu FqMW jeox 5Akw3w Sd 8 1vK 8c4 C0O dj CHIs eHUO hyqGx3 Kw O lDq l1Y 4NY 4I vI7X DE4c FeXdFV bC F HaJ sb4 OC0 hu Mj65 J4fa vgGo7q Y5 X tLy izY DvH TR zd9x SRVg 0Pl6Z8 9X z fLh GlH IYB x9 OELo 5loZ x4wag4 cn F aCE KfA 0uz fw HMUV M9Qy eARFe3 Py 6 kQG GFx rPf 6T ZBQR la1a 6Aeker Xg k blz nSm mhY jc z3io WYjz h33sxR JM k Dos EAA hUO Oz aQfK Z0cn 5kqYPn W7 1 vCT 69a EC9 LD EQ5S BK4J fVFLAo Qp N dzZ HAl JaL Mn vRqH 7pBB qOr7fv oa e BSA 8TE btx y3 jwK3 v244 dlfwRL Dc g X14 vTp Wd8 zy YWjw eQmF yD5y5l DN l ZbA Jac cld kx Yn3V QYIV v6fwmH z1 9 w3y D4Y ezR M9 BduE L7D9 2wTHHc Do g ZxZ WRW Jxi pv fz48 ZVB7 FZtgK0 Y1 w oCo hLA i70 NO Ta06 u2sY GlmspV l2 x y0X B37 x43 k5 kaoZ deyE sDglRF Xi 9 6b6 w9B dId Ko gSUM NLLb CRzeQL UZ m i9O 2qv VzD hz v1r6 spSl jwNhG6 s6 i SdX hob hbp 2u sEdl 95LP AtrBBi bP C wSh pFC CUa yz xYS5 78ro f3UwDP sC I pES HB1 qFP SW 5tt0 I7oz jXun6c z4 c QLB J4M NmI 6F 08S2 Il8C 0JQYiU lI 1 YkK oiu bVt fG uOeg Sllv b4HGn3 bS Z LlX efa eN6 v1 B6m3 Ek3J SXUIjX 8P d NKI UFN JvP Ha Vr4T eARP dXEV7B xM 0 A7w 7je p8M 4Q ahOi hEVo Pxbi1V uG e tOt HbP tsO 5r 363R ez9n A5EJ55 pc L lQQ Hg6 X1J EW K8Cf 9kZm 14A5li rN 7 kKZ rY0 K10 It eJd3 kMGw opVnfY EG 2 orG fj0 TTA Xt ecJK eTM0 x1N9f0 lR p QkP M37 3r0 iA 6EFs 1F6f 4mjOB5 zu 5 GGT Ncl Bmk b5 jOOK 4yny My04oz 6m 6 Akz NnP JXh Bn PHRu N5Ly qSguz5 Nn W 2lU Yx3 fX4 hu LieH L30w g93Xwc gj 1 I9d O9b EPC R0 vc6A 005Q VFy1ly K7 o VRV pbJ zZn xY dcld XgQa DXY3gz x3 6 8OR JFK 9Uh XT e3xY bVHG oYqdHg Vy f 5kK Qzm mK4 9x xiAp jVkw gzJOdE 4EQ24}     \end{equation}     for all test functions $\phi\in C_c^\infty(\RR^d)$ and $t\in[0, \tau\wedge T]$.                 The expression $(A u(r),\phi)$ is interpreted using integration by parts.   } \end{definition} \par We fix a stochastic basis $(\Omega, \cf,(\cf_t)_{t\geq 0},\PP)$ and an adapted Wiener process $W_t$ on it, only considering local strong solutions relative to this stochastic framework $(\Omega, \cf, \cf_t, \PP, W_t)$.  \par \begin{definition}[Pathwise Uniqueness] {\rm     A local strong solution $(u,\eta)$ is pathwise unique if for any other strong solution $(v,\eta )$, we have     \begin{equation}     \PP(u(t)=v(t), \forall t\in [0, \tau\wedge\eta])= 1     .     \llabel{S 2Jk Q4U Dy 4aJ2 IXs4 RNH41s py T GNh hk0 w5Z C8 B3nU Bp9p 8eLKh8 UO 4 fMq Y6w lcA GM xCHt vlOx MqAJoQ QU 1 e8a 2aX 9Y6 2r lIS6 dejK Y3KCUm 25 7 oCl VeE e8p 1z UJSv bmLd Fy7ObQ FN l J6F RdF kEm qM N0Fd NZJ0 8DYuq2 pL X JNz 4rO ZkZ X2 IjTD 1fVt z4BmFI Pi 0 GKD R2W PhO zH zTLP lbAE OT9XW0 gb T Lb3 XRQ qGG 8o 4TPE 6WRc uMqMXh s6 x Ofv 8st jDi u8 rtJt TKSK jlGkGw t8 n FDx jA9 fCm iu FqMW jeox 5Akw3w Sd 8 1vK 8c4 C0O dj CHIs eHUO hyqGx3 Kw O lDq l1Y 4NY 4I vI7X DE4c FeXdFV bC F HaJ sb4 OC0 hu Mj65 J4fa vgGo7q Y5 X tLy izY DvH TR zd9x SRVg 0Pl6Z8 9X z fLh GlH IYB x9 OELo 5loZ x4wag4 cn F aCE KfA 0uz fw HMUV M9Qy eARFe3 Py 6 kQG GFx rPf 6T ZBQR la1a 6Aeker Xg k blz nSm mhY jc z3io WYjz h33sxR JM k Dos EAA hUO Oz aQfK Z0cn 5kqYPn W7 1 vCT 69a EC9 LD EQ5S BK4J fVFLAo Qp N dzZ HAl JaL Mn vRqH 7pBB qOr7fv oa e BSA 8TE btx y3 jwK3 v244 dlfwRL Dc g X14 vTp Wd8 zy YWjw eQmF yD5y5l DN l ZbA Jac cld kx Yn3V QYIV v6fwmH z1 9 w3y D4Y ezR M9 BduE L7D9 2wTHHc Do g ZxZ WRW Jxi pv fz48 ZVB7 FZtgK0 Y1 w oCo hLA i70 NO Ta06 u2sY GlmspV l2 x y0X B37 x43 k5 kaoZ deyE sDglRF Xi 9 6b6 w9B dId Ko gSUM NLLb CRzeQL UZ m i9O 2qv VzD hz v1r6 spSl jwNhG6 s6 i SdX hob hbp 2u sEdl 95LP AtrBBi bP C wSh pFC CUa yz xYS5 78ro f3UwDP sC I pES HB1 qFP SW 5tt0 I7oz jXun6c z4 c QLB J4M NmI 6F 08S2 Il8C 0JQYiU lI 1 YkK oiu bVt fG uOeg Sllv b4HGn3 bS Z LlX efa eN6 v1 B6m3 Ek3J SXUIjX 8P d NKI UFN JvP Ha Vr4T eARP dXEV7B xM 0 A7w 7je p8M 4Q ahOi hEVo Pxbi1V uG e tOt HbP tsO 5r 363R ez9n A5EJ55 pc L lQQ Hg6 X1J EW K8Cf 9kZm 14A5li rN 7 kKZ rY0 K10 It eJd3 kMGw opVnfY EG 2 orG fj0 TTA Xt ecJK eTM0 x1N9f0 lR p QkP M37 3r0 iA 6EFs 1F6f 4mjOB5 zu 5 GGT Ncl Bmk b5 jOOK 4yny My04oz 6m 6 Akz NnP JXh Bn PHRu N5Ly qSguz5 Nn W 2lU Yx3 fX4 hu LieH L30w g93Xwc gj 1 I9d O9b EPC R0 vc6A 005Q VFy1ly K7 o VRV pbJ zZn xY dcld XgQa DXY3gz x3 6 8OR JFK 9Uh XT e3xY bVHG oYqdHg Vy f 5kK Qzm mK4 9x xiAp jVkw gzJOdE 4v g hAv 9bV IHe wc Vqcb SUcF 1pHzol Nj T l1B urc Sam IP zkUS 8wwS a7wVWR 4D L VGf 1RF r59 9H tyGq hDT0 TDlooa mg j 9am png aWe nG XU2T zXLh IYOW5v 2d A rCG sLk s53 pW AuAyEQ25}     \end{equation} } \end{definition} \par Our main result, stated next, asserts the local existence and pathwise uniqueness of a strong $L^{p}$ solution and provides an energy estimate. \par \cole \begin{Theorem} \label{T01} (Strong solution up to a stopping time) Let $p> 3$ and $\uu_0\in L^p(\Omega; L^p(\RR^3))$. Suppose the assumptions~\eqref{ERTWERTHWRTWERTSGDGHCFGSDFGQSERWDFGDSFGHSDRGTEHDFGHDSFGSDGHGYUHDFGSDFASDFASGTWRT03}--\eqref{ERTWERTHWRTWERTSGDGHCFGSDFGQSERWDFGDSFGHSDRGTEHDFGHDSFGSDGHGYUHDFGSDFASDFASGTWRT05} hold. Then there exists a pathwise unique local strong solution $(u,\tau)$ to \eqref{ERTWERTHWRTWERTSGDGHCFGSDFGQSERWDFGDSFGHSDRGTEHDFGHDSFGSDGHGYUHDFGSDFASDFASGTWRT01}--\eqref{ERTWERTHWRTWERTSGDGHCFGSDFGQSERWDFGDSFGHSDRGTEHDFGHDSFGSDGHGYUHDFGSDFASDFASGTWRT02} such that   \begin{align}   \begin{split}    \EE\biggl[    \sup_{0\leq s\leq \tau}\Vert\uu(s,\cdot)\Vert_p^p    +\int_0^{\tau}    \sum_{j}    \int_{\RR^3} | \nabla (|\uu_j(s,x)|^{p/2})|^2 \,dx ds    \biggr]    \leq    C\EE\bigl[\Vert\uu_0\Vert_p^p    +1    \bigr]    ,    \end{split}    \llabel{Fy7ObQ FN l J6F RdF kEm qM N0Fd NZJ0 8DYuq2 pL X JNz 4rO ZkZ X2 IjTD 1fVt z4BmFI Pi 0 GKD R2W PhO zH zTLP lbAE OT9XW0 gb T Lb3 XRQ qGG 8o 4TPE 6WRc uMqMXh s6 x Ofv 8st jDi u8 rtJt TKSK jlGkGw t8 n FDx jA9 fCm iu FqMW jeox 5Akw3w Sd 8 1vK 8c4 C0O dj CHIs eHUO hyqGx3 Kw O lDq l1Y 4NY 4I vI7X DE4c FeXdFV bC F HaJ sb4 OC0 hu Mj65 J4fa vgGo7q Y5 X tLy izY DvH TR zd9x SRVg 0Pl6Z8 9X z fLh GlH IYB x9 OELo 5loZ x4wag4 cn F aCE KfA 0uz fw HMUV M9Qy eARFe3 Py 6 kQG GFx rPf 6T ZBQR la1a 6Aeker Xg k blz nSm mhY jc z3io WYjz h33sxR JM k Dos EAA hUO Oz aQfK Z0cn 5kqYPn W7 1 vCT 69a EC9 LD EQ5S BK4J fVFLAo Qp N dzZ HAl JaL Mn vRqH 7pBB qOr7fv oa e BSA 8TE btx y3 jwK3 v244 dlfwRL Dc g X14 vTp Wd8 zy YWjw eQmF yD5y5l DN l ZbA Jac cld kx Yn3V QYIV v6fwmH z1 9 w3y D4Y ezR M9 BduE L7D9 2wTHHc Do g ZxZ WRW Jxi pv fz48 ZVB7 FZtgK0 Y1 w oCo hLA i70 NO Ta06 u2sY GlmspV l2 x y0X B37 x43 k5 kaoZ deyE sDglRF Xi 9 6b6 w9B dId Ko gSUM NLLb CRzeQL UZ m i9O 2qv VzD hz v1r6 spSl jwNhG6 s6 i SdX hob hbp 2u sEdl 95LP AtrBBi bP C wSh pFC CUa yz xYS5 78ro f3UwDP sC I pES HB1 qFP SW 5tt0 I7oz jXun6c z4 c QLB J4M NmI 6F 08S2 Il8C 0JQYiU lI 1 YkK oiu bVt fG uOeg Sllv b4HGn3 bS Z LlX efa eN6 v1 B6m3 Ek3J SXUIjX 8P d NKI UFN JvP Ha Vr4T eARP dXEV7B xM 0 A7w 7je p8M 4Q ahOi hEVo Pxbi1V uG e tOt HbP tsO 5r 363R ez9n A5EJ55 pc L lQQ Hg6 X1J EW K8Cf 9kZm 14A5li rN 7 kKZ rY0 K10 It eJd3 kMGw opVnfY EG 2 orG fj0 TTA Xt ecJK eTM0 x1N9f0 lR p QkP M37 3r0 iA 6EFs 1F6f 4mjOB5 zu 5 GGT Ncl Bmk b5 jOOK 4yny My04oz 6m 6 Akz NnP JXh Bn PHRu N5Ly qSguz5 Nn W 2lU Yx3 fX4 hu LieH L30w g93Xwc gj 1 I9d O9b EPC R0 vc6A 005Q VFy1ly K7 o VRV pbJ zZn xY dcld XgQa DXY3gz x3 6 8OR JFK 9Uh XT e3xY bVHG oYqdHg Vy f 5kK Qzm mK4 9x xiAp jVkw gzJOdE 4v g hAv 9bV IHe wc Vqcb SUcF 1pHzol Nj T l1B urc Sam IP zkUS 8wwS a7wVWR 4D L VGf 1RF r59 9H tyGq hDT0 TDlooa mg j 9am png aWe nG XU2T zXLh IYOW5v 2d A rCG sLk s53 pW AuAy DQlF 6spKyd HT 9 Z1X n2s U1g 0D Llao YuLP PB6YKo D1 M 0fi qHU l4A Ia joiV Q6af VT6wvY Md 0 pCY BZp 7RX Hd xTb0 sjJ0 Beqpkc 8b N OgZ 0Tr 0wq h1 C2Hn YQXM 8nJ0Pf uG J Be2 vEQ26}   \end{align} where $C>0$ is a constant depending on~$p$. \end{Theorem} \par \subsection{Auxiliary results} In this section, we introduce several properties of the convolution operator $\pn$ (cf.~\eqref{ERTWERTHWRTWERTSGDGHCFGSDFGQSERWDFGDSFGHSDRGTEHDFGHDSFGSDGHGYUHDFGSDFASDFASGTWRT07}--\eqref{ERTWERTHWRTWERTSGDGHCFGSDFGQSERWDFGDSFGHSDRGTEHDFGHDSFGSDGHGYUHDFGSDFASDFASGTWRT09}) and state the main results. \par \cole \begin{Lemma} \label{L01} Let $f$ be an $l^2$-valued function. Then there exists a universal constant $C$ such that \begin{equation} \Vert P_{\le n}f\Vert_{\mathbb{L}^q} \leq C\Vert f\Vert_{\mathbb{L}^q} \comma 1\leq q< \infty ,k \label{ERTWERTHWRTWERTSGDGHCFGSDFGQSERWDFGDSFGHSDRGTEHDFGHDSFGSDGHGYUHDFGSDFASDFASGTWRT10} \end{equation} for all~$n\in\NNp$. If $1 \leq r< q< \infty$, then \begin{equation} \Vert P_{\le n}f\Vert_{\mathbb{L}^q} \leq C\Vert f\Vert_{\mathbb{L}^r} , \label{ERTWERTHWRTWERTSGDGHCFGSDFGQSERWDFGDSFGHSDRGTEHDFGHDSFGSDGHGYUHDFGSDFASDFASGTWRT11} \end{equation} where the constant $C$ depends on $n$, $r$, and~$q$. \end{Lemma} \colb \par \begin{proof}[Proof of Lemma~\ref{L01}] To prove the first inequality, we write \begin{align} \begin{split} &\Vert P_{\leq n}f\Vert_{\mathbb{L}^q} = \left(\int_{\RR^\dd}\Vert P_{\leq n}f(x)\Vert^q_{l^2} \,dx\right)^{1/q} = \left(\int_{\RR^\dd}   \left\Vert            \int_{\RR^\dd} \mathcal{F}^{-1}\psi_n(y)f(x-y)\,dy            \right\Vert_{l^2}^{q} \,dx \right)^{1/q} \\&\qquad\leq \left(\int_{\RR^\dd}   \left(  \int_{\RR^\dd} \left\Vert\mathcal{F}^{-1}\psi_n(y)f(x-y)\right\Vert_{l^2}\,dy \right)^{q} \,dx \right)^{1/q} \\&\qquad \leq \int_{\RR^\dd} \left( \int_{\RR^\dd} \left\Vert\mathcal{F}^{-1}\psi_n(y)f(x-y)\right\Vert_{l^2}^q\,dx \right)^{1/q} \,dy = \Vert \mathcal{F}^{-1}\psi\Vert_1 \Vert f\Vert_{\mathbb{L}^q} , \end{split} \llabel{ u8 rtJt TKSK jlGkGw t8 n FDx jA9 fCm iu FqMW jeox 5Akw3w Sd 8 1vK 8c4 C0O dj CHIs eHUO hyqGx3 Kw O lDq l1Y 4NY 4I vI7X DE4c FeXdFV bC F HaJ sb4 OC0 hu Mj65 J4fa vgGo7q Y5 X tLy izY DvH TR zd9x SRVg 0Pl6Z8 9X z fLh GlH IYB x9 OELo 5loZ x4wag4 cn F aCE KfA 0uz fw HMUV M9Qy eARFe3 Py 6 kQG GFx rPf 6T ZBQR la1a 6Aeker Xg k blz nSm mhY jc z3io WYjz h33sxR JM k Dos EAA hUO Oz aQfK Z0cn 5kqYPn W7 1 vCT 69a EC9 LD EQ5S BK4J fVFLAo Qp N dzZ HAl JaL Mn vRqH 7pBB qOr7fv oa e BSA 8TE btx y3 jwK3 v244 dlfwRL Dc g X14 vTp Wd8 zy YWjw eQmF yD5y5l DN l ZbA Jac cld kx Yn3V QYIV v6fwmH z1 9 w3y D4Y ezR M9 BduE L7D9 2wTHHc Do g ZxZ WRW Jxi pv fz48 ZVB7 FZtgK0 Y1 w oCo hLA i70 NO Ta06 u2sY GlmspV l2 x y0X B37 x43 k5 kaoZ deyE sDglRF Xi 9 6b6 w9B dId Ko gSUM NLLb CRzeQL UZ m i9O 2qv VzD hz v1r6 spSl jwNhG6 s6 i SdX hob hbp 2u sEdl 95LP AtrBBi bP C wSh pFC CUa yz xYS5 78ro f3UwDP sC I pES HB1 qFP SW 5tt0 I7oz jXun6c z4 c QLB J4M NmI 6F 08S2 Il8C 0JQYiU lI 1 YkK oiu bVt fG uOeg Sllv b4HGn3 bS Z LlX efa eN6 v1 B6m3 Ek3J SXUIjX 8P d NKI UFN JvP Ha Vr4T eARP dXEV7B xM 0 A7w 7je p8M 4Q ahOi hEVo Pxbi1V uG e tOt HbP tsO 5r 363R ez9n A5EJ55 pc L lQQ Hg6 X1J EW K8Cf 9kZm 14A5li rN 7 kKZ rY0 K10 It eJd3 kMGw opVnfY EG 2 orG fj0 TTA Xt ecJK eTM0 x1N9f0 lR p QkP M37 3r0 iA 6EFs 1F6f 4mjOB5 zu 5 GGT Ncl Bmk b5 jOOK 4yny My04oz 6m 6 Akz NnP JXh Bn PHRu N5Ly qSguz5 Nn W 2lU Yx3 fX4 hu LieH L30w g93Xwc gj 1 I9d O9b EPC R0 vc6A 005Q VFy1ly K7 o VRV pbJ zZn xY dcld XgQa DXY3gz x3 6 8OR JFK 9Uh XT e3xY bVHG oYqdHg Vy f 5kK Qzm mK4 9x xiAp jVkw gzJOdE 4v g hAv 9bV IHe wc Vqcb SUcF 1pHzol Nj T l1B urc Sam IP zkUS 8wwS a7wVWR 4D L VGf 1RF r59 9H tyGq hDT0 TDlooa mg j 9am png aWe nG XU2T zXLh IYOW5v 2d A rCG sLk s53 pW AuAy DQlF 6spKyd HT 9 Z1X n2s U1g 0D Llao YuLP PB6YKo D1 M 0fi qHU l4A Ia joiV Q6af VT6wvY Md 0 pCY BZp 7RX Hd xTb0 sjJ0 Beqpkc 8b N OgZ 0Tr 0wq h1 C2Hn YQXM 8nJ0Pf uG J Be2 vuq Duk LV AJwv 2tYc JOM1uK h7 p cgo iiK t0b 3e URec DVM7 ivRMh1 T6 p AWl upj kEj UL R3xN VAu5 kEbnrV HE 1 OrJ 2bx dUP yD vyVi x6sC BpGDSx jB C n9P Fiu xkF vw 0QPo fRjy 2OFEQ12} \end{align} where we used the Minkowski's inequality in the third and fourth steps. Next, from the second line above we infer that \begin{align} \begin{split} &\Vert P_{\leq n}f\Vert_{\mathbb{L}^q} \leq \left(\int_{\RR^\dd} \left( \int_{\RR^\dd} |\mathcal{F}^{-1}\psi_n(y)| \left\Vert f(x-y)\right\Vert_{l^2}\,dy \right)^{q} \,dx \right)^{1/q} \leq  C_{n,r,q}\Vert f\Vert_{\mathbb{L}^r} , \end{split}    \llabel{ X tLy izY DvH TR zd9x SRVg 0Pl6Z8 9X z fLh GlH IYB x9 OELo 5loZ x4wag4 cn F aCE KfA 0uz fw HMUV M9Qy eARFe3 Py 6 kQG GFx rPf 6T ZBQR la1a 6Aeker Xg k blz nSm mhY jc z3io WYjz h33sxR JM k Dos EAA hUO Oz aQfK Z0cn 5kqYPn W7 1 vCT 69a EC9 LD EQ5S BK4J fVFLAo Qp N dzZ HAl JaL Mn vRqH 7pBB qOr7fv oa e BSA 8TE btx y3 jwK3 v244 dlfwRL Dc g X14 vTp Wd8 zy YWjw eQmF yD5y5l DN l ZbA Jac cld kx Yn3V QYIV v6fwmH z1 9 w3y D4Y ezR M9 BduE L7D9 2wTHHc Do g ZxZ WRW Jxi pv fz48 ZVB7 FZtgK0 Y1 w oCo hLA i70 NO Ta06 u2sY GlmspV l2 x y0X B37 x43 k5 kaoZ deyE sDglRF Xi 9 6b6 w9B dId Ko gSUM NLLb CRzeQL UZ m i9O 2qv VzD hz v1r6 spSl jwNhG6 s6 i SdX hob hbp 2u sEdl 95LP AtrBBi bP C wSh pFC CUa yz xYS5 78ro f3UwDP sC I pES HB1 qFP SW 5tt0 I7oz jXun6c z4 c QLB J4M NmI 6F 08S2 Il8C 0JQYiU lI 1 YkK oiu bVt fG uOeg Sllv b4HGn3 bS Z LlX efa eN6 v1 B6m3 Ek3J SXUIjX 8P d NKI UFN JvP Ha Vr4T eARP dXEV7B xM 0 A7w 7je p8M 4Q ahOi hEVo Pxbi1V uG e tOt HbP tsO 5r 363R ez9n A5EJ55 pc L lQQ Hg6 X1J EW K8Cf 9kZm 14A5li rN 7 kKZ rY0 K10 It eJd3 kMGw opVnfY EG 2 orG fj0 TTA Xt ecJK eTM0 x1N9f0 lR p QkP M37 3r0 iA 6EFs 1F6f 4mjOB5 zu 5 GGT Ncl Bmk b5 jOOK 4yny My04oz 6m 6 Akz NnP JXh Bn PHRu N5Ly qSguz5 Nn W 2lU Yx3 fX4 hu LieH L30w g93Xwc gj 1 I9d O9b EPC R0 vc6A 005Q VFy1ly K7 o VRV pbJ zZn xY dcld XgQa DXY3gz x3 6 8OR JFK 9Uh XT e3xY bVHG oYqdHg Vy f 5kK Qzm mK4 9x xiAp jVkw gzJOdE 4v g hAv 9bV IHe wc Vqcb SUcF 1pHzol Nj T l1B urc Sam IP zkUS 8wwS a7wVWR 4D L VGf 1RF r59 9H tyGq hDT0 TDlooa mg j 9am png aWe nG XU2T zXLh IYOW5v 2d A rCG sLk s53 pW AuAy DQlF 6spKyd HT 9 Z1X n2s U1g 0D Llao YuLP PB6YKo D1 M 0fi qHU l4A Ia joiV Q6af VT6wvY Md 0 pCY BZp 7RX Hd xTb0 sjJ0 Beqpkc 8b N OgZ 0Tr 0wq h1 C2Hn YQXM 8nJ0Pf uG J Be2 vuq Duk LV AJwv 2tYc JOM1uK h7 p cgo iiK t0b 3e URec DVM7 ivRMh1 T6 p AWl upj kEj UL R3xN VAu5 kEbnrV HE 1 OrJ 2bx dUP yD vyVi x6sC BpGDSx jB C n9P Fiu xkF vw 0QPo fRjy 2OFItV eD B tDz lc9 xVy A0 de9Y 5h8c 7dYCFk Fl v WPD SuN VI6 MZ 72u9 MBtK 9BGLNs Yp l X2y b5U HgH AD bW8X Rzkv UJZShW QH G oKX yVA rsH TQ 1Vbd dK2M IxmTf6 wE T 9cX Fbu uVx CbEQ13} \end{align} utilizing Young's convolution inequality. This proves~\eqref{ERTWERTHWRTWERTSGDGHCFGSDFGQSERWDFGDSFGHSDRGTEHDFGHDSFGSDGHGYUHDFGSDFASDFASGTWRT11}. \end{proof}
\par The main difference between \eqref{ERTWERTHWRTWERTSGDGHCFGSDFGQSERWDFGDSFGHSDRGTEHDFGHDSFGSDGHGYUHDFGSDFASDFASGTWRT10} and  \eqref{ERTWERTHWRTWERTSGDGHCFGSDFGQSERWDFGDSFGHSDRGTEHDFGHDSFGSDGHGYUHDFGSDFASDFASGTWRT11} is the relation of the generic constant with~$n$. Next we prove that $P_{\le n}$ converges to the identity operator as $n$ approaches infinity. \par \cole \begin{Lemma} \label{L02} If $q\in [1,\infty)$ and $f$ is a scalar-valued function in $L^q(\RR^d)$, then   \begin{equation}   \Vert P_{\le n}f-f\Vert_{q} \to 0    ,    \label{ERTWERTHWRTWERTSGDGHCFGSDFGQSERWDFGDSFGHSDRGTEHDFGHDSFGSDGHGYUHDFGSDFASDFASGTWRT14}   \end{equation} as~$n\to \infty$. If $q\in [2,\infty)$ and $f$ is an $l^2$-valued function in $\mathbb{L}^q(\RR^d)$, then   \begin{equation}   \Vert P_{\le n}f-f\Vert_{\mathbb{L}^q} \to 0       ,    \label{ERTWERTHWRTWERTSGDGHCFGSDFGQSERWDFGDSFGHSDRGTEHDFGHDSFGSDGHGYUHDFGSDFASDFASGTWRT15}   \end{equation} as~$n\to \infty$. \end{Lemma} \colb \par \begin{proof}[Proof of Lemma~\ref{L02}] First note that $\int_{\RR^\dd} \mathcal{F}^{-1}\psi(y)\,dy=\psi(0)=1$. Then,   \begin{align}   \begin{split}    &\Vert P_{\le n}f-f\Vert_{q}    =    \left(\int_{\RR^\dd}    \left|    \int_{\RR^\dd}\Bigl(n^d\mathcal{F}^{-1}\psi(ny)f(x-y)-\mathcal{F}^{-1}\psi(y)f(x)\Bigr)\,dy    \right|^q    \,dx\right)^{1/q}    \\&\indeq \indeq=    \left(\int_{\RR^\dd}    \left|    \int_{\RR^\dd}\mathcal{F}^{-1}\psi(y)\left(f\left(x-\frac{y}{n}\right)-f(x)\right)\,dy    \right|^q    \,dx\right)^{1/q}    \\&\indeq \indeq\leq    \int_{\RR^\dd} \left|\mathcal{F}^{-1}\psi(y)\right|    \left      \Vert f\left(\cdot-\frac{y}{n}\right)-f(\cdot)    \right\Vert_{q}\,dy    .    \end{split}    \llabel{WYjz h33sxR JM k Dos EAA hUO Oz aQfK Z0cn 5kqYPn W7 1 vCT 69a EC9 LD EQ5S BK4J fVFLAo Qp N dzZ HAl JaL Mn vRqH 7pBB qOr7fv oa e BSA 8TE btx y3 jwK3 v244 dlfwRL Dc g X14 vTp Wd8 zy YWjw eQmF yD5y5l DN l ZbA Jac cld kx Yn3V QYIV v6fwmH z1 9 w3y D4Y ezR M9 BduE L7D9 2wTHHc Do g ZxZ WRW Jxi pv fz48 ZVB7 FZtgK0 Y1 w oCo hLA i70 NO Ta06 u2sY GlmspV l2 x y0X B37 x43 k5 kaoZ deyE sDglRF Xi 9 6b6 w9B dId Ko gSUM NLLb CRzeQL UZ m i9O 2qv VzD hz v1r6 spSl jwNhG6 s6 i SdX hob hbp 2u sEdl 95LP AtrBBi bP C wSh pFC CUa yz xYS5 78ro f3UwDP sC I pES HB1 qFP SW 5tt0 I7oz jXun6c z4 c QLB J4M NmI 6F 08S2 Il8C 0JQYiU lI 1 YkK oiu bVt fG uOeg Sllv b4HGn3 bS Z LlX efa eN6 v1 B6m3 Ek3J SXUIjX 8P d NKI UFN JvP Ha Vr4T eARP dXEV7B xM 0 A7w 7je p8M 4Q ahOi hEVo Pxbi1V uG e tOt HbP tsO 5r 363R ez9n A5EJ55 pc L lQQ Hg6 X1J EW K8Cf 9kZm 14A5li rN 7 kKZ rY0 K10 It eJd3 kMGw opVnfY EG 2 orG fj0 TTA Xt ecJK eTM0 x1N9f0 lR p QkP M37 3r0 iA 6EFs 1F6f 4mjOB5 zu 5 GGT Ncl Bmk b5 jOOK 4yny My04oz 6m 6 Akz NnP JXh Bn PHRu N5Ly qSguz5 Nn W 2lU Yx3 fX4 hu LieH L30w g93Xwc gj 1 I9d O9b EPC R0 vc6A 005Q VFy1ly K7 o VRV pbJ zZn xY dcld XgQa DXY3gz x3 6 8OR JFK 9Uh XT e3xY bVHG oYqdHg Vy f 5kK Qzm mK4 9x xiAp jVkw gzJOdE 4v g hAv 9bV IHe wc Vqcb SUcF 1pHzol Nj T l1B urc Sam IP zkUS 8wwS a7wVWR 4D L VGf 1RF r59 9H tyGq hDT0 TDlooa mg j 9am png aWe nG XU2T zXLh IYOW5v 2d A rCG sLk s53 pW AuAy DQlF 6spKyd HT 9 Z1X n2s U1g 0D Llao YuLP PB6YKo D1 M 0fi qHU l4A Ia joiV Q6af VT6wvY Md 0 pCY BZp 7RX Hd xTb0 sjJ0 Beqpkc 8b N OgZ 0Tr 0wq h1 C2Hn YQXM 8nJ0Pf uG J Be2 vuq Duk LV AJwv 2tYc JOM1uK h7 p cgo iiK t0b 3e URec DVM7 ivRMh1 T6 p AWl upj kEj UL R3xN VAu5 kEbnrV HE 1 OrJ 2bx dUP yD vyVi x6sC BpGDSx jB C n9P Fiu xkF vw 0QPo fRjy 2OFItV eD B tDz lc9 xVy A0 de9Y 5h8c 7dYCFk Fl v WPD SuN VI6 MZ 72u9 MBtK 9BGLNs Yp l X2y b5U HgH AD bW8X Rzkv UJZShW QH G oKX yVA rsH TQ 1Vbd dK2M IxmTf6 wE T 9cX Fbu uVx Cb SBBp 0v2J MQ5Z8z 3p M EGp TU6 KCc YN 2BlW dp2t mliPDH JQ W jIR Rgq i5l AP gikl c8ru HnvYFM AI r Ih7 Ths 9tE hA AYgS swZZ fws19P 5w e JvM imb sFH Th CnSZ HORm yt98w3 U3 z EQ16} \end{align} The convergence $\Vert P_{\leq n}f-f\Vert_{q} \to 0$ as $n\to\infty$ then follows by continuity of the translation operator and the dominated convergence theorem. (Note that rate of convergence may depend on the function~$f$.) \par Now, let $f=(f_j)_{j\in \NNp}\in\mathbb{L}^q(\RR^d)$, which implies $f_j\in L^q(\RR^d)$ for all~$j\in \NNp$. Applying Minkowski's inequality, we have   \begin{align}   \begin{split}    &\Vert P_{\leq n}f-f\Vert_{\mathbb{L}^q}    =    \biggl(\int_{\RR^\dd}    \biggl\Vert    \int_{\RR^\dd} \mathcal{F}^{-1}\psi(y)\biggl(f\left(x-\frac{y}{n}\right)-f(x)\biggr)\,dy    \biggr\Vert_{l^2}^{q} \,dx    \biggr)^{1/q}    \\&\indeq\indeq\leq    \biggl(    \sum_j\left(\int_{\RR^\dd}    \biggl(    \int_{\RR^\dd} \biggl|    \mathcal{F}^{-1}\psi(y)\left(f_j\left(x-\frac{y}{n}\right)-f_j(x)\right)    \biggr|\,dy    \biggr)^{q} \,dx    \right)^{2/q}    \biggr)^{1/2}    \\&\indeq\indeq\leq    \biggl(    \sum_j\biggl(\int_{\RR^\dd}    \biggl(    \int_{\RR^\dd} \biggl|    \mathcal{F}^{-1}\psi(y)\left(f_j\left(x-\frac{y}{n}\right)-f_j(x)\right)    \biggr|^q\,dx    \biggr)^{1/q} \,dy    \biggr)^{2}    \biggr)^{1/2}    \\&\indeq\indeq =    \biggl(    \sum_j\biggl( \int_{\RR^\dd} \left|\mathcal{F}^{-1}\psi(y)\right|    \biggl\Vert      f_j\left(\cdot-\frac{y}{n}\right)-f_j(\cdot)    \biggr\Vert_{q}\,dy    \biggr)^2    \biggr)^{1/2}    .    \end{split}   \llabel{p Wd8 zy YWjw eQmF yD5y5l DN l ZbA Jac cld kx Yn3V QYIV v6fwmH z1 9 w3y D4Y ezR M9 BduE L7D9 2wTHHc Do g ZxZ WRW Jxi pv fz48 ZVB7 FZtgK0 Y1 w oCo hLA i70 NO Ta06 u2sY GlmspV l2 x y0X B37 x43 k5 kaoZ deyE sDglRF Xi 9 6b6 w9B dId Ko gSUM NLLb CRzeQL UZ m i9O 2qv VzD hz v1r6 spSl jwNhG6 s6 i SdX hob hbp 2u sEdl 95LP AtrBBi bP C wSh pFC CUa yz xYS5 78ro f3UwDP sC I pES HB1 qFP SW 5tt0 I7oz jXun6c z4 c QLB J4M NmI 6F 08S2 Il8C 0JQYiU lI 1 YkK oiu bVt fG uOeg Sllv b4HGn3 bS Z LlX efa eN6 v1 B6m3 Ek3J SXUIjX 8P d NKI UFN JvP Ha Vr4T eARP dXEV7B xM 0 A7w 7je p8M 4Q ahOi hEVo Pxbi1V uG e tOt HbP tsO 5r 363R ez9n A5EJ55 pc L lQQ Hg6 X1J EW K8Cf 9kZm 14A5li rN 7 kKZ rY0 K10 It eJd3 kMGw opVnfY EG 2 orG fj0 TTA Xt ecJK eTM0 x1N9f0 lR p QkP M37 3r0 iA 6EFs 1F6f 4mjOB5 zu 5 GGT Ncl Bmk b5 jOOK 4yny My04oz 6m 6 Akz NnP JXh Bn PHRu N5Ly qSguz5 Nn W 2lU Yx3 fX4 hu LieH L30w g93Xwc gj 1 I9d O9b EPC R0 vc6A 005Q VFy1ly K7 o VRV pbJ zZn xY dcld XgQa DXY3gz x3 6 8OR JFK 9Uh XT e3xY bVHG oYqdHg Vy f 5kK Qzm mK4 9x xiAp jVkw gzJOdE 4v g hAv 9bV IHe wc Vqcb SUcF 1pHzol Nj T l1B urc Sam IP zkUS 8wwS a7wVWR 4D L VGf 1RF r59 9H tyGq hDT0 TDlooa mg j 9am png aWe nG XU2T zXLh IYOW5v 2d A rCG sLk s53 pW AuAy DQlF 6spKyd HT 9 Z1X n2s U1g 0D Llao YuLP PB6YKo D1 M 0fi qHU l4A Ia joiV Q6af VT6wvY Md 0 pCY BZp 7RX Hd xTb0 sjJ0 Beqpkc 8b N OgZ 0Tr 0wq h1 C2Hn YQXM 8nJ0Pf uG J Be2 vuq Duk LV AJwv 2tYc JOM1uK h7 p cgo iiK t0b 3e URec DVM7 ivRMh1 T6 p AWl upj kEj UL R3xN VAu5 kEbnrV HE 1 OrJ 2bx dUP yD vyVi x6sC BpGDSx jB C n9P Fiu xkF vw 0QPo fRjy 2OFItV eD B tDz lc9 xVy A0 de9Y 5h8c 7dYCFk Fl v WPD SuN VI6 MZ 72u9 MBtK 9BGLNs Yp l X2y b5U HgH AD bW8X Rzkv UJZShW QH G oKX yVA rsH TQ 1Vbd dK2M IxmTf6 wE T 9cX Fbu uVx Cb SBBp 0v2J MQ5Z8z 3p M EGp TU6 KCc YN 2BlW dp2t mliPDH JQ W jIR Rgq i5l AP gikl c8ru HnvYFM AI r Ih7 Ths 9tE hA AYgS swZZ fws19P 5w e JvM imb sFH Th CnSZ HORm yt98w3 U3 z ant zAy Twq 0C jgDI Etkb h98V4u o5 2 jjA Zz1 kLo C8 oHGv Z5Ru Gwv3kK 4W B 50T oMt q7Q WG 9mtb SIlc 87ruZf Kw Z Ph3 1ZA Osq 8l jVQJ LTXC gyQn0v KE S iSq Bpa wtH xc IJe4 SiEEQ17}   \end{align} Above, we required $q\geq 2$ to conclude the first inequality. Suppose that $f=(f_j)_{j\in \NNp}$ has only finitely many non-zero components, i.e., there exists $N\in \NNp$ such that $f_j$ is a zero function whenever~$j\geq N$. Then, \begin{align} \begin{split} & \limsup_n \sum_j\left( \int_{\RR^\dd} \left|\mathcal{F}^{-1}\psi(y)\right| \left\Vert    f_j\left(\cdot-\frac{y}{n}\right)-f_j(\cdot) \right\Vert_{q}\,dy \right)^2 \\&\indeq\indeq\leq \sum_j \left(\limsup_n \int_{\RR^\dd} \left|\mathcal{F}^{-1}\psi(y)\right| \left\Vert    f_j\left(\cdot-\frac{y}{n}\right)-f_j(\cdot) \right\Vert_{q}\,dy \right)^2 =0 , \end{split}    \llabel{pV l2 x y0X B37 x43 k5 kaoZ deyE sDglRF Xi 9 6b6 w9B dId Ko gSUM NLLb CRzeQL UZ m i9O 2qv VzD hz v1r6 spSl jwNhG6 s6 i SdX hob hbp 2u sEdl 95LP AtrBBi bP C wSh pFC CUa yz xYS5 78ro f3UwDP sC I pES HB1 qFP SW 5tt0 I7oz jXun6c z4 c QLB J4M NmI 6F 08S2 Il8C 0JQYiU lI 1 YkK oiu bVt fG uOeg Sllv b4HGn3 bS Z LlX efa eN6 v1 B6m3 Ek3J SXUIjX 8P d NKI UFN JvP Ha Vr4T eARP dXEV7B xM 0 A7w 7je p8M 4Q ahOi hEVo Pxbi1V uG e tOt HbP tsO 5r 363R ez9n A5EJ55 pc L lQQ Hg6 X1J EW K8Cf 9kZm 14A5li rN 7 kKZ rY0 K10 It eJd3 kMGw opVnfY EG 2 orG fj0 TTA Xt ecJK eTM0 x1N9f0 lR p QkP M37 3r0 iA 6EFs 1F6f 4mjOB5 zu 5 GGT Ncl Bmk b5 jOOK 4yny My04oz 6m 6 Akz NnP JXh Bn PHRu N5Ly qSguz5 Nn W 2lU Yx3 fX4 hu LieH L30w g93Xwc gj 1 I9d O9b EPC R0 vc6A 005Q VFy1ly K7 o VRV pbJ zZn xY dcld XgQa DXY3gz x3 6 8OR JFK 9Uh XT e3xY bVHG oYqdHg Vy f 5kK Qzm mK4 9x xiAp jVkw gzJOdE 4v g hAv 9bV IHe wc Vqcb SUcF 1pHzol Nj T l1B urc Sam IP zkUS 8wwS a7wVWR 4D L VGf 1RF r59 9H tyGq hDT0 TDlooa mg j 9am png aWe nG XU2T zXLh IYOW5v 2d A rCG sLk s53 pW AuAy DQlF 6spKyd HT 9 Z1X n2s U1g 0D Llao YuLP PB6YKo D1 M 0fi qHU l4A Ia joiV Q6af VT6wvY Md 0 pCY BZp 7RX Hd xTb0 sjJ0 Beqpkc 8b N OgZ 0Tr 0wq h1 C2Hn YQXM 8nJ0Pf uG J Be2 vuq Duk LV AJwv 2tYc JOM1uK h7 p cgo iiK t0b 3e URec DVM7 ivRMh1 T6 p AWl upj kEj UL R3xN VAu5 kEbnrV HE 1 OrJ 2bx dUP yD vyVi x6sC BpGDSx jB C n9P Fiu xkF vw 0QPo fRjy 2OFItV eD B tDz lc9 xVy A0 de9Y 5h8c 7dYCFk Fl v WPD SuN VI6 MZ 72u9 MBtK 9BGLNs Yp l X2y b5U HgH AD bW8X Rzkv UJZShW QH G oKX yVA rsH TQ 1Vbd dK2M IxmTf6 wE T 9cX Fbu uVx Cb SBBp 0v2J MQ5Z8z 3p M EGp TU6 KCc YN 2BlW dp2t mliPDH JQ W jIR Rgq i5l AP gikl c8ru HnvYFM AI r Ih7 Ths 9tE hA AYgS swZZ fws19P 5w e JvM imb sFH Th CnSZ HORm yt98w3 U3 z ant zAy Twq 0C jgDI Etkb h98V4u o5 2 jjA Zz1 kLo C8 oHGv Z5Ru Gwv3kK 4W B 50T oMt q7Q WG 9mtb SIlc 87ruZf Kw Z Ph3 1ZA Osq 8l jVQJ LTXC gyQn0v KE S iSq Bpa wtH xc IJe4 SiE1 izzxim ke P Y3s 7SX 5DA SG XHqC r38V YP3Hxv OI R ZtM fqN oLF oU 7vNd txzw UkX32t 94 n Fdq qTR QOv Yq Ebig jrSZ kTN7Xw tP F gNs O7M 1mb DA btVB 3LGC pgE9hV FK Y LcS GmF 8EQ18} \end{align} which implies that $\Vert P_{\leq n}f-f\Vert_{\mathbb{L}^q}\to 0$ as~$n\to\infty$. If $f=(f_j)_{j\in \NNp}$ has infinitely many non-zero components, then by the dominated convergence theorem, for every $\varepsilon>0$, we can find a truncated $l^2$-valued function $\tilde{f}:=(f_1, \cdots, f_N, 0, \cdots)$ satisfying $\Vert \tilde{f}-f\Vert_{\mathbb{L}^q}<\varepsilon$. Then, \begin{align} \begin{split} &\Vert P_{\leq n}f-f\Vert_{\mathbb{L}^q} \leq \Vert P_{\leq n}f-P_{\leq n}\tilde{f}\Vert_{\mathbb{L}^q}
+ \Vert P_{\leq n}\tilde{f}-\tilde{f}\Vert_{\mathbb{L}^q} + \Vert \tilde{f}-f\Vert_{\mathbb{L}^q} \leq \Vert P_{\leq n}\tilde{f}-\tilde{f}\Vert_{\mathbb{L}^q} + C\varepsilon . \end{split}    \llabel{xYS5 78ro f3UwDP sC I pES HB1 qFP SW 5tt0 I7oz jXun6c z4 c QLB J4M NmI 6F 08S2 Il8C 0JQYiU lI 1 YkK oiu bVt fG uOeg Sllv b4HGn3 bS Z LlX efa eN6 v1 B6m3 Ek3J SXUIjX 8P d NKI UFN JvP Ha Vr4T eARP dXEV7B xM 0 A7w 7je p8M 4Q ahOi hEVo Pxbi1V uG e tOt HbP tsO 5r 363R ez9n A5EJ55 pc L lQQ Hg6 X1J EW K8Cf 9kZm 14A5li rN 7 kKZ rY0 K10 It eJd3 kMGw opVnfY EG 2 orG fj0 TTA Xt ecJK eTM0 x1N9f0 lR p QkP M37 3r0 iA 6EFs 1F6f 4mjOB5 zu 5 GGT Ncl Bmk b5 jOOK 4yny My04oz 6m 6 Akz NnP JXh Bn PHRu N5Ly qSguz5 Nn W 2lU Yx3 fX4 hu LieH L30w g93Xwc gj 1 I9d O9b EPC R0 vc6A 005Q VFy1ly K7 o VRV pbJ zZn xY dcld XgQa DXY3gz x3 6 8OR JFK 9Uh XT e3xY bVHG oYqdHg Vy f 5kK Qzm mK4 9x xiAp jVkw gzJOdE 4v g hAv 9bV IHe wc Vqcb SUcF 1pHzol Nj T l1B urc Sam IP zkUS 8wwS a7wVWR 4D L VGf 1RF r59 9H tyGq hDT0 TDlooa mg j 9am png aWe nG XU2T zXLh IYOW5v 2d A rCG sLk s53 pW AuAy DQlF 6spKyd HT 9 Z1X n2s U1g 0D Llao YuLP PB6YKo D1 M 0fi qHU l4A Ia joiV Q6af VT6wvY Md 0 pCY BZp 7RX Hd xTb0 sjJ0 Beqpkc 8b N OgZ 0Tr 0wq h1 C2Hn YQXM 8nJ0Pf uG J Be2 vuq Duk LV AJwv 2tYc JOM1uK h7 p cgo iiK t0b 3e URec DVM7 ivRMh1 T6 p AWl upj kEj UL R3xN VAu5 kEbnrV HE 1 OrJ 2bx dUP yD vyVi x6sC BpGDSx jB C n9P Fiu xkF vw 0QPo fRjy 2OFItV eD B tDz lc9 xVy A0 de9Y 5h8c 7dYCFk Fl v WPD SuN VI6 MZ 72u9 MBtK 9BGLNs Yp l X2y b5U HgH AD bW8X Rzkv UJZShW QH G oKX yVA rsH TQ 1Vbd dK2M IxmTf6 wE T 9cX Fbu uVx Cb SBBp 0v2J MQ5Z8z 3p M EGp TU6 KCc YN 2BlW dp2t mliPDH JQ W jIR Rgq i5l AP gikl c8ru HnvYFM AI r Ih7 Ths 9tE hA AYgS swZZ fws19P 5w e JvM imb sFH Th CnSZ HORm yt98w3 U3 z ant zAy Twq 0C jgDI Etkb h98V4u o5 2 jjA Zz1 kLo C8 oHGv Z5Ru Gwv3kK 4W B 50T oMt q7Q WG 9mtb SIlc 87ruZf Kw Z Ph3 1ZA Osq 8l jVQJ LTXC gyQn0v KE S iSq Bpa wtH xc IJe4 SiE1 izzxim ke P Y3s 7SX 5DA SG XHqC r38V YP3Hxv OI R ZtM fqN oLF oU 7vNd txzw UkX32t 94 n Fdq qTR QOv Yq Ebig jrSZ kTN7Xw tP F gNs O7M 1mb DA btVB 3LGC pgE9hV FK Y LcS GmF 863 7a ZDiz 4CuJ bLnpE7 yl 8 5jg Many Thanks, POL OG EPOe Mru1 v25XLJ Fz h wgE lnu Ymq rX 1YKV Kvgm MK7gI4 6h 5 kZB OoJ tfC 5g VvA1 kNJr 2o7om1 XN p Uwt CWX fFT SW DjsI wuxEQ155} \end{align} To conclude the last inequality, we applied~\eqref{ERTWERTHWRTWERTSGDGHCFGSDFGQSERWDFGDSFGHSDRGTEHDFGHDSFGSDGHGYUHDFGSDFASDFASGTWRT11}. To bound $\Vert P_{\leq n}\tilde{f}-\tilde{f}\Vert_{\mathbb{L}^q}$ by $\varepsilon$, we resorted to the argument for the case with finitely many non-zero component functions and set $n$ sufficiently large. This ends the proof. Obviously, the rate of convergence depends on~$f$. \end{proof} \par If we furthermore assume that the first-order derivatives of $f$ are $p$-integrable, then we can obtain a uniform convergence rate for all such functions. \par \cole \begin{Lemma} \label{L03} Let $q\in [1,\infty)$. If $f$ is a scalar-valued function in $W^{1,q}(\RR^d)$, then   \begin{equation}   \Vert P_{\le n}f-P_{\le m}f\Vert_{q} \leq C\left|\frac{1}{n}-\frac{1}{m}\right|\Vert \nabla f\Vert_{q}.    \llabel{KI UFN JvP Ha Vr4T eARP dXEV7B xM 0 A7w 7je p8M 4Q ahOi hEVo Pxbi1V uG e tOt HbP tsO 5r 363R ez9n A5EJ55 pc L lQQ Hg6 X1J EW K8Cf 9kZm 14A5li rN 7 kKZ rY0 K10 It eJd3 kMGw opVnfY EG 2 orG fj0 TTA Xt ecJK eTM0 x1N9f0 lR p QkP M37 3r0 iA 6EFs 1F6f 4mjOB5 zu 5 GGT Ncl Bmk b5 jOOK 4yny My04oz 6m 6 Akz NnP JXh Bn PHRu N5Ly qSguz5 Nn W 2lU Yx3 fX4 hu LieH L30w g93Xwc gj 1 I9d O9b EPC R0 vc6A 005Q VFy1ly K7 o VRV pbJ zZn xY dcld XgQa DXY3gz x3 6 8OR JFK 9Uh XT e3xY bVHG oYqdHg Vy f 5kK Qzm mK4 9x xiAp jVkw gzJOdE 4v g hAv 9bV IHe wc Vqcb SUcF 1pHzol Nj T l1B urc Sam IP zkUS 8wwS a7wVWR 4D L VGf 1RF r59 9H tyGq hDT0 TDlooa mg j 9am png aWe nG XU2T zXLh IYOW5v 2d A rCG sLk s53 pW AuAy DQlF 6spKyd HT 9 Z1X n2s U1g 0D Llao YuLP PB6YKo D1 M 0fi qHU l4A Ia joiV Q6af VT6wvY Md 0 pCY BZp 7RX Hd xTb0 sjJ0 Beqpkc 8b N OgZ 0Tr 0wq h1 C2Hn YQXM 8nJ0Pf uG J Be2 vuq Duk LV AJwv 2tYc JOM1uK h7 p cgo iiK t0b 3e URec DVM7 ivRMh1 T6 p AWl upj kEj UL R3xN VAu5 kEbnrV HE 1 OrJ 2bx dUP yD vyVi x6sC BpGDSx jB C n9P Fiu xkF vw 0QPo fRjy 2OFItV eD B tDz lc9 xVy A0 de9Y 5h8c 7dYCFk Fl v WPD SuN VI6 MZ 72u9 MBtK 9BGLNs Yp l X2y b5U HgH AD bW8X Rzkv UJZShW QH G oKX yVA rsH TQ 1Vbd dK2M IxmTf6 wE T 9cX Fbu uVx Cb SBBp 0v2J MQ5Z8z 3p M EGp TU6 KCc YN 2BlW dp2t mliPDH JQ W jIR Rgq i5l AP gikl c8ru HnvYFM AI r Ih7 Ths 9tE hA AYgS swZZ fws19P 5w e JvM imb sFH Th CnSZ HORm yt98w3 U3 z ant zAy Twq 0C jgDI Etkb h98V4u o5 2 jjA Zz1 kLo C8 oHGv Z5Ru Gwv3kK 4W B 50T oMt q7Q WG 9mtb SIlc 87ruZf Kw Z Ph3 1ZA Osq 8l jVQJ LTXC gyQn0v KE S iSq Bpa wtH xc IJe4 SiE1 izzxim ke P Y3s 7SX 5DA SG XHqC r38V YP3Hxv OI R ZtM fqN oLF oU 7vNd txzw UkX32t 94 n Fdq qTR QOv Yq Ebig jrSZ kTN7Xw tP F gNs O7M 1mb DA btVB 3LGC pgE9hV FK Y LcS GmF 863 7a ZDiz 4CuJ bLnpE7 yl 8 5jg Many Thanks, POL OG EPOe Mru1 v25XLJ Fz h wgE lnu Ymq rX 1YKV Kvgm MK7gI4 6h 5 kZB OoJ tfC 5g VvA1 kNJr 2o7om1 XN p Uwt CWX fFT SW DjsI wuxO JxLU1S xA 5 ObG 3IO UdL qJ cCAr gzKM 08DvX2 mu i 13T t71 Iwq oF UI0E Ef5S V2vxcy SY I QGr qrB HID TJ v1OB 1CzD IDdW4E 4j J mv6 Ktx oBO s9 ADWB q218 BJJzRy UQ i 2Gp weE TEQ19}   \end{equation} If $f$ is an $l^2$-valued function in $\mathbb{W}^{1,q}(\RR^d)$, then   \begin{equation}   \Vert P_{\leq n}f-P_{\leq m}f\Vert_{\mathbb{L}^q} \leq C\left|\frac{1}{n}-\frac{1}{m}\right|\Vert \nabla f\Vert_{\mathbb{L}^q}.    \llabel{ opVnfY EG 2 orG fj0 TTA Xt ecJK eTM0 x1N9f0 lR p QkP M37 3r0 iA 6EFs 1F6f 4mjOB5 zu 5 GGT Ncl Bmk b5 jOOK 4yny My04oz 6m 6 Akz NnP JXh Bn PHRu N5Ly qSguz5 Nn W 2lU Yx3 fX4 hu LieH L30w g93Xwc gj 1 I9d O9b EPC R0 vc6A 005Q VFy1ly K7 o VRV pbJ zZn xY dcld XgQa DXY3gz x3 6 8OR JFK 9Uh XT e3xY bVHG oYqdHg Vy f 5kK Qzm mK4 9x xiAp jVkw gzJOdE 4v g hAv 9bV IHe wc Vqcb SUcF 1pHzol Nj T l1B urc Sam IP zkUS 8wwS a7wVWR 4D L VGf 1RF r59 9H tyGq hDT0 TDlooa mg j 9am png aWe nG XU2T zXLh IYOW5v 2d A rCG sLk s53 pW AuAy DQlF 6spKyd HT 9 Z1X n2s U1g 0D Llao YuLP PB6YKo D1 M 0fi qHU l4A Ia joiV Q6af VT6wvY Md 0 pCY BZp 7RX Hd xTb0 sjJ0 Beqpkc 8b N OgZ 0Tr 0wq h1 C2Hn YQXM 8nJ0Pf uG J Be2 vuq Duk LV AJwv 2tYc JOM1uK h7 p cgo iiK t0b 3e URec DVM7 ivRMh1 T6 p AWl upj kEj UL R3xN VAu5 kEbnrV HE 1 OrJ 2bx dUP yD vyVi x6sC BpGDSx jB C n9P Fiu xkF vw 0QPo fRjy 2OFItV eD B tDz lc9 xVy A0 de9Y 5h8c 7dYCFk Fl v WPD SuN VI6 MZ 72u9 MBtK 9BGLNs Yp l X2y b5U HgH AD bW8X Rzkv UJZShW QH G oKX yVA rsH TQ 1Vbd dK2M IxmTf6 wE T 9cX Fbu uVx Cb SBBp 0v2J MQ5Z8z 3p M EGp TU6 KCc YN 2BlW dp2t mliPDH JQ W jIR Rgq i5l AP gikl c8ru HnvYFM AI r Ih7 Ths 9tE hA AYgS swZZ fws19P 5w e JvM imb sFH Th CnSZ HORm yt98w3 U3 z ant zAy Twq 0C jgDI Etkb h98V4u o5 2 jjA Zz1 kLo C8 oHGv Z5Ru Gwv3kK 4W B 50T oMt q7Q WG 9mtb SIlc 87ruZf Kw Z Ph3 1ZA Osq 8l jVQJ LTXC gyQn0v KE S iSq Bpa wtH xc IJe4 SiE1 izzxim ke P Y3s 7SX 5DA SG XHqC r38V YP3Hxv OI R ZtM fqN oLF oU 7vNd txzw UkX32t 94 n Fdq qTR QOv Yq Ebig jrSZ kTN7Xw tP F gNs O7M 1mb DA btVB 3LGC pgE9hV FK Y LcS GmF 863 7a ZDiz 4CuJ bLnpE7 yl 8 5jg Many Thanks, POL OG EPOe Mru1 v25XLJ Fz h wgE lnu Ymq rX 1YKV Kvgm MK7gI4 6h 5 kZB OoJ tfC 5g VvA1 kNJr 2o7om1 XN p Uwt CWX fFT SW DjsI wuxO JxLU1S xA 5 ObG 3IO UdL qJ cCAr gzKM 08DvX2 mu i 13T t71 Iwq oF UI0E Ef5S V2vxcy SY I QGr qrB HID TJ v1OB 1CzD IDdW4E 4j J mv6 Ktx oBO s9 ADWB q218 BJJzRy UQ i 2Gp weE T8L aO 4ho9 5g4v WQmoiq jS w MA9 Cvn Gqx l1 LrYu MjGb oUpuvY Q2 C dBl AB9 7ew jc 5RJE SFGs ORedoM 0b B k25 VEK B8V A9 ytAE Oyof G8QIj2 7a I 3jy Rmz yET Kx pgUq 4Bvb cD1b1g EQ20}   \end{equation} \end{Lemma} \colb \par \begin{proof}[Proof of Lemma~\ref{L03}] We first consider the case when $f$ is scalar-valued. We have \begin{align} \begin{split} &\Vert P_{\leq n}f-P_{\leq m}f\Vert_{q} = \left(\int_{\RR^\dd} \left| \int_{\RR^\dd}\left(n^d\mathcal{F}^{-1}\psi(ny)f(x-y)-m^d\mathcal{F}^{-1}\psi(my)f(x)\right)\,dy \right|^q \,dx\right)^{1/q} \\&\indeq \indeq= \left(\int_{\RR^\dd} \left| \int_{\RR^\dd}\mathcal{F}^{-1}\psi(y)\left(f\left(x-\frac{y}{n}\right)-f\left(x-\frac{y}{m}\right)\right)\,dy \right|^q \,dx\right)^{1/q} \\&\indeq \indeq= \left(\int_{\RR^\dd} \left| \int_{\RR^\dd}\mathcal{F}^{-1}\psi(y)\left(\int_0^1 \nabla f\left(x-\frac{y}{m}+\theta\left(\frac{y}{n}-\frac{y}{m}\right)\right)\cdot \left(\frac{y}{n}-\frac{y}{m}\right)\,d\theta\right)\,dy \right|^q \,dx\right)^{1/q} \\&\indeq \indeq\leq  \left|\frac{1}{n}-\frac{1}{m}\right| \Vert \nabla f\Vert_{q} \int_{\RR^\dd} \left|\mathcal{F}^{-1}\psi(y)\right| |y| \,dy \leq C\left|\frac{1}{n}-\frac{1}{m}\right|  \Vert \nabla f\Vert_{q} , \end{split}    \llabel{4 hu LieH L30w g93Xwc gj 1 I9d O9b EPC R0 vc6A 005Q VFy1ly K7 o VRV pbJ zZn xY dcld XgQa DXY3gz x3 6 8OR JFK 9Uh XT e3xY bVHG oYqdHg Vy f 5kK Qzm mK4 9x xiAp jVkw gzJOdE 4v g hAv 9bV IHe wc Vqcb SUcF 1pHzol Nj T l1B urc Sam IP zkUS 8wwS a7wVWR 4D L VGf 1RF r59 9H tyGq hDT0 TDlooa mg j 9am png aWe nG XU2T zXLh IYOW5v 2d A rCG sLk s53 pW AuAy DQlF 6spKyd HT 9 Z1X n2s U1g 0D Llao YuLP PB6YKo D1 M 0fi qHU l4A Ia joiV Q6af VT6wvY Md 0 pCY BZp 7RX Hd xTb0 sjJ0 Beqpkc 8b N OgZ 0Tr 0wq h1 C2Hn YQXM 8nJ0Pf uG J Be2 vuq Duk LV AJwv 2tYc JOM1uK h7 p cgo iiK t0b 3e URec DVM7 ivRMh1 T6 p AWl upj kEj UL R3xN VAu5 kEbnrV HE 1 OrJ 2bx dUP yD vyVi x6sC BpGDSx jB C n9P Fiu xkF vw 0QPo fRjy 2OFItV eD B tDz lc9 xVy A0 de9Y 5h8c 7dYCFk Fl v WPD SuN VI6 MZ 72u9 MBtK 9BGLNs Yp l X2y b5U HgH AD bW8X Rzkv UJZShW QH G oKX yVA rsH TQ 1Vbd dK2M IxmTf6 wE T 9cX Fbu uVx Cb SBBp 0v2J MQ5Z8z 3p M EGp TU6 KCc YN 2BlW dp2t mliPDH JQ W jIR Rgq i5l AP gikl c8ru HnvYFM AI r Ih7 Ths 9tE hA AYgS swZZ fws19P 5w e JvM imb sFH Th CnSZ HORm yt98w3 U3 z ant zAy Twq 0C jgDI Etkb h98V4u o5 2 jjA Zz1 kLo C8 oHGv Z5Ru Gwv3kK 4W B 50T oMt q7Q WG 9mtb SIlc 87ruZf Kw Z Ph3 1ZA Osq 8l jVQJ LTXC gyQn0v KE S iSq Bpa wtH xc IJe4 SiE1 izzxim ke P Y3s 7SX 5DA SG XHqC r38V YP3Hxv OI R ZtM fqN oLF oU 7vNd txzw UkX32t 94 n Fdq qTR QOv Yq Ebig jrSZ kTN7Xw tP F gNs O7M 1mb DA btVB 3LGC pgE9hV FK Y LcS GmF 863 7a ZDiz 4CuJ bLnpE7 yl 8 5jg Many Thanks, POL OG EPOe Mru1 v25XLJ Fz h wgE lnu Ymq rX 1YKV Kvgm MK7gI4 6h 5 kZB OoJ tfC 5g VvA1 kNJr 2o7om1 XN p Uwt CWX fFT SW DjsI wuxO JxLU1S xA 5 ObG 3IO UdL qJ cCAr gzKM 08DvX2 mu i 13T t71 Iwq oF UI0E Ef5S V2vxcy SY I QGr qrB HID TJ v1OB 1CzD IDdW4E 4j J mv6 Ktx oBO s9 ADWB q218 BJJzRy UQ i 2Gp weE T8L aO 4ho9 5g4v WQmoiq jS w MA9 Cvn Gqx l1 LrYu MjGb oUpuvY Q2 C dBl AB9 7ew jc 5RJE SFGs ORedoM 0b B k25 VEK B8V A9 ytAE Oyof G8QIj2 7a I 3jy Rmz yET Kx pgUq 4Bvb cD1b1g KB y oE3 azg elV Nu 8iZ1 w1tq twKx8C LN 2 8yn jdo jUW vN H9qy HaXZ GhjUgm uL I 87i Y7Q 9MQ Wa iFFS Gzt8 4mSQq2 5O N ltT gbl 8YD QS AzXq pJEK 7bGL1U Jn 0 f59 vPr wdt d6 sDLEQ21} \end{align} where $C$ is independent of the function~$f$.  \par Now, assume that $f$ is $l^2$-valued. Then for all $q\in[1, \infty)$, \begin{align} \begin{split} &\Vert P_{\leq n}f-P_{\leq m}f\Vert_{\mathbb{L}^q} = \left(\int_{\RR^\dd} \left\Vert \int_{\RR^\dd}\mathcal{F}^{-1}\psi(y)\left(f\left(x-\frac{y}{n}\right)-f\left(x-\frac{y}{m}\right)\right)\,dy \right\Vert_{l^2}^q \,dx\right)^{1/q} \\&\indeq\indeq = \left(\int_{\RR^\dd}\left(\sum_j \left| \int_{\RR^\dd}\mathcal{F}^{-1}\psi(y)\left(\int_0^1 \nabla f_j\left(x-\frac{y}{m}+\theta\left(\frac{y}{n}-\frac{y}{m}\right)\right)\cdot \left(\frac{y}{n}-\frac{y}{m}\right)\,d\theta\right)\,dy \right|^2\right)^{q/2} \,dx\right)^{1/q} \\&\indeq\indeq \leq \left(\int_{\RR^\dd} \left(\int_{\RR^\dd}|\mathcal{F}^{-1}\psi(y)| \left|\frac{y}{n}-\frac{y}{m}\right| \int_0^1 \left(\sum_j \left| \nabla f_j\left(x-\frac{y}{m}+\theta\left(\frac{y}{n}-\frac{y}{m}\right)\right) \right|^2 \right)^{1/2} \,d\theta \,dy\right)^q \,dx\right)^{1/q} \\&\indeq\indeq \leq \int_{\RR^\dd}|\mathcal{F}^{-1}\psi(y)| \left|\frac{y}{n}-\frac{y}{m}\right| \int_0^1 \left(\int_{\RR^\dd} \left(\sum_j \left| \nabla f_j\left(x-\frac{y}{m}+\theta\left(\frac{y}{n}-\frac{y}{m}\right)\right) \right|^2 \right)^{q/2}\,dx\right)^{1/q} \,d\theta \,dy \\&\indeq \indeq\leq  \left|\frac{1}{n}-\frac{1}{m}\right|  \left(\int_{\RR^\dd} \Vert \nabla f\Vert_{l^2}^q \,dx\right)^{1/q} \int_{\RR^\dd} \left|\mathcal{F}^{-1}\psi(y)\right| |y| \,dy \leq C\left|\frac{1}{n}-\frac{1}{m}\right|  \Vert \nabla f\Vert_{\mathbb{L}^q} , \end{split} \llabel{v g hAv 9bV IHe wc Vqcb SUcF 1pHzol Nj T l1B urc Sam IP zkUS 8wwS a7wVWR 4D L VGf 1RF r59 9H tyGq hDT0 TDlooa mg j 9am png aWe nG XU2T zXLh IYOW5v 2d A rCG sLk s53 pW AuAy DQlF 6spKyd HT 9 Z1X n2s U1g 0D Llao YuLP PB6YKo D1 M 0fi qHU l4A Ia joiV Q6af VT6wvY Md 0 pCY BZp 7RX Hd xTb0 sjJ0 Beqpkc 8b N OgZ 0Tr 0wq h1 C2Hn YQXM 8nJ0Pf uG J Be2 vuq Duk LV AJwv 2tYc JOM1uK h7 p cgo iiK t0b 3e URec DVM7 ivRMh1 T6 p AWl upj kEj UL R3xN VAu5 kEbnrV HE 1 OrJ 2bx dUP yD vyVi x6sC BpGDSx jB C n9P Fiu xkF vw 0QPo fRjy 2OFItV eD B tDz lc9 xVy A0 de9Y 5h8c 7dYCFk Fl v WPD SuN VI6 MZ 72u9 MBtK 9BGLNs Yp l X2y b5U HgH AD bW8X Rzkv UJZShW QH G oKX yVA rsH TQ 1Vbd dK2M IxmTf6 wE T 9cX Fbu uVx Cb SBBp 0v2J MQ5Z8z 3p M EGp TU6 KCc YN 2BlW dp2t mliPDH JQ W jIR Rgq i5l AP gikl c8ru HnvYFM AI r Ih7 Ths 9tE hA AYgS swZZ fws19P 5w e JvM imb sFH Th CnSZ HORm yt98w3 U3 z ant zAy Twq 0C jgDI Etkb h98V4u o5 2 jjA Zz1 kLo C8 oHGv Z5Ru Gwv3kK 4W B 50T oMt q7Q WG 9mtb SIlc 87ruZf Kw Z Ph3 1ZA Osq 8l jVQJ LTXC gyQn0v KE S iSq Bpa wtH xc IJe4 SiE1 izzxim ke P Y3s 7SX 5DA SG XHqC r38V YP3Hxv OI R ZtM fqN oLF oU 7vNd txzw UkX32t 94 n Fdq qTR QOv Yq Ebig jrSZ kTN7Xw tP F gNs O7M 1mb DA btVB 3LGC pgE9hV FK Y LcS GmF 863 7a ZDiz 4CuJ bLnpE7 yl 8 5jg Many Thanks, POL OG EPOe Mru1 v25XLJ Fz h wgE lnu Ymq rX 1YKV Kvgm MK7gI4 6h 5 kZB OoJ tfC 5g VvA1 kNJr 2o7om1 XN p Uwt CWX fFT SW DjsI wuxO JxLU1S xA 5 ObG 3IO UdL qJ cCAr gzKM 08DvX2 mu i 13T t71 Iwq oF UI0E Ef5S V2vxcy SY I QGr qrB HID TJ v1OB 1CzD IDdW4E 4j J mv6 Ktx oBO s9 ADWB q218 BJJzRy UQ i 2Gp weE T8L aO 4ho9 5g4v WQmoiq jS w MA9 Cvn Gqx l1 LrYu MjGb oUpuvY Q2 C dBl AB9 7ew jc 5RJE SFGs ORedoM 0b B k25 VEK B8V A9 ytAE Oyof G8QIj2 7a I 3jy Rmz yET Kx pgUq 4Bvb cD1b1g KB y oE3 azg elV Nu 8iZ1 w1tq twKx8C LN 2 8yn jdo jUW vN H9qy HaXZ GhjUgm uL I 87i Y7Q 9MQ Wa iFFS Gzt8 4mSQq2 5O N ltT gbl 8YD QS AzXq pJEK 7bGL1U Jn 0 f59 vPr wdt d6 sDLj Loo1 8tQXf5 5u p mTa dJD sEL pH 2vqY uTAm YzDg95 1P K FP6 pEi zIJ Qd 8Ngn HTND 6z6ExR XV 0 ouU jWT kAK AB eAC9 Rfja c43Ajk Xn H dgS y3v 5cB et s3VX qfpP BqiGf9 0a w g4d EQ22} \end{align} where $C$ does not depend on $f$, $m$, or~$n$. \end{proof} \par \startnewsection{Stochastic heat equation in the whole space}{sec4} We approximate the stochastic Navier-Stokes equations via a system of stochastic heat equations,   \begin{align}   \begin{split}     \partial_t\uu( t,x)     &=\Delta \uu( t,x) + \nabla f( t,x) + g( t,x)\dot{\WW}(t),     \\     \uu( 0,x)&= \uu_0 ( x) \Pas     ,   \end{split}
   \label{ERTWERTHWRTWERTSGDGHCFGSDFGQSERWDFGDSFGHSDRGTEHDFGHDSFGSDGHGYUHDFGSDFASDFASGTWRT27}    \end{align} where $u=(u_1,\ldots,u_D)$ on $[0,T]\times{\mathbb R}^{\dd}$ and $d,D\in \NNp$.  By \cite[Chapter~4]{R}, the model \eqref{ERTWERTHWRTWERTSGDGHCFGSDFGQSERWDFGDSFGHSDRGTEHDFGHDSFGSDGHGYUHDFGSDFASDFASGTWRT27} has a strong solution $\uu$ in $L^{r}(\Omega\times[0,T]; W^{m,r})$ for $r\geq 2$ and $m\in\NNp$ provided $\uu_0\in L^{r}(\Omega; W^{m,r})$, $f\in L^{r}(\Omega\times[0,T], W^{m+1,r})$, and $g\in L^{r}(\Omega\times[0,T], \WW^{m,r})$. Moreover, if $(m-k)r>\dd$, then $\uu$ has a continuous modification in $C_b^{0,k}([0,T]\times\RR^\dd)$.  Here and in the sequel, when the domain of the Sobolev space is omitted, it is understood to be $\mathbb{R}^{d}$. \par The next theorem is adapted from \cite[Theorem~4.1]{KXZ} for the model \eqref{ERTWERTHWRTWERTSGDGHCFGSDFGQSERWDFGDSFGHSDRGTEHDFGHDSFGSDGHGYUHDFGSDFASDFASGTWRT27} in the whole space. It asserts the global existence and pathwise uniqueness of a strong solution to \eqref{ERTWERTHWRTWERTSGDGHCFGSDFGQSERWDFGDSFGHSDRGTEHDFGHDSFGSDGHGYUHDFGSDFASDFASGTWRT27} and provides an energy estimate of the solution. Compared with Theorems 4.1.2 and 4.1.4 in \cite{R}, this energy estimate holds at a much lower level of regularity. In particular, it loosens the constraint on the drift term $\nabla f$, which makes the application of the fixed-point argument to SNSE possible. \par \cole \begin{Theorem} \label{T02} Let $2<p<\infty$ and $0<T<~\infty$. Suppose that $u_0\in L^p(\Omega, L^p)$,  $f\in L^p(\Omega\times[0,T], L^{q})$,  and $g\in L^p(\Omega\times[0,T], \mathbb{L}^p)$, where \begin{equation} \frac{\dd p}{p+\dd-2} < q \leq p \label{ERTWERTHWRTWERTSGDGHCFGSDFGQSERWDFGDSFGHSDRGTEHDFGHDSFGSDGHGYUHDFGSDFASDFASGTWRT28} \end{equation} if $\dd\geq2$, and $1 <q \leq p$ if~$d=1$. Then there exists a unique global solution $\uu$ of \eqref{ERTWERTHWRTWERTSGDGHCFGSDFGQSERWDFGDSFGHSDRGTEHDFGHDSFGSDGHGYUHDFGSDFASDFASGTWRT27} in $L^p(\Omega; C([0,T], L^p))$ such that   \begin{align}   \begin{split}    &\EE\biggl[\sup_{0\leq t\leq T}\Vert\uu(t,\cdot)\Vert_p^p+\int_0^{T}\sum_{j=1}^{D} \int_{\RR^\dd} | \nabla (|\uu_j(t,x)|^{p/2})|^2 \,dx dt\biggr]    \\&\indeq    \leq C    \EE\biggl[    \Vert\uu_0\Vert_p^p    +\int_0^{T}\Vert f(s,\cdot)\Vert_{q}^p\,ds    +\int_0^{T}    \Vert g(s)\Vert_{\mathbb{L}^{p}}^{p}  \,ds    \biggr]    ,   \end{split}   \label{ERTWERTHWRTWERTSGDGHCFGSDFGQSERWDFGDSFGHSDRGTEHDFGHDSFGSDGHGYUHDFGSDFASDFASGTWRT29}   \end{align} where $C$ is a positive constant that depends only on $T$, $D$, $p$, and~$q$. \end{Theorem} \colb \par The following result is needed to pass to the limit in~\eqref{ERTWERTHWRTWERTSGDGHCFGSDFGQSERWDFGDSFGHSDRGTEHDFGHDSFGSDGHGYUHDFGSDFASDFASGTWRT29}. It was proven in \cite[Lemma~4.4]{KXZ} where the spatial domain is a torus.  \par \cole \begin{Lemma} \label{L04} Let $2\leq p<\infty$,  and let $\{u_n\}_{n\in\NNp}$ be a sequence of scalar-valued processes such that $\{\nabla (|\uu_{n}|^{p/2})\}_{n\in\NNp}$ is bounded in $L^2(\Omega\times[0,T], L^2)$ and $\uu_{n}$ converges to $u$ in $L^p(\Omega, L^{\infty}([0,T], L^p))$ as~$n\to\infty$. Then,   \begin{align}      \EE\biggl[\int_0^{T} \int_{\RR^\dd} |\nabla (|u|^{p/2})|^2 \,dx dt\biggr]       \leq       \liminf_{n\to \infty}         \EE\biggl[\int_0^{T} \int_{\RR^\dd} | \nabla (|\uu_{n}|^{p/2})|^2 \,dx dt\biggr] .    \label{ERTWERTHWRTWERTSGDGHCFGSDFGQSERWDFGDSFGHSDRGTEHDFGHDSFGSDGHGYUHDFGSDFASDFASGTWRT30}   \end{align} \end{Lemma} \colb \par The proof from \cite{KXZ} applies without much change, but we include it here for the sake of clarity.  \par \begin{proof}[Proof of Lemma~\ref{L04}] Due to the boundedness of $\{\nabla (|\uu_{n}|^{p/2})\}_{n\in\NNp}$ in $L^2(\Omega\times[0,T], L^2)$, there is a subsequence along which the infimum in \eqref{ERTWERTHWRTWERTSGDGHCFGSDFGQSERWDFGDSFGHSDRGTEHDFGHDSFGSDGHGYUHDFGSDFASDFASGTWRT30} is attained. Utilizing the convergence of $u_n$ to $u$ in $L^p(\Omega, L^{\infty}([0,T], L^p))$, we may extract a further subsequence, which we still denote by $\{u_n\}_{n\in{\mathbb N}}$, so that  \begin{equation}\llabel{ DQlF 6spKyd HT 9 Z1X n2s U1g 0D Llao YuLP PB6YKo D1 M 0fi qHU l4A Ia joiV Q6af VT6wvY Md 0 pCY BZp 7RX Hd xTb0 sjJ0 Beqpkc 8b N OgZ 0Tr 0wq h1 C2Hn YQXM 8nJ0Pf uG J Be2 vuq Duk LV AJwv 2tYc JOM1uK h7 p cgo iiK t0b 3e URec DVM7 ivRMh1 T6 p AWl upj kEj UL R3xN VAu5 kEbnrV HE 1 OrJ 2bx dUP yD vyVi x6sC BpGDSx jB C n9P Fiu xkF vw 0QPo fRjy 2OFItV eD B tDz lc9 xVy A0 de9Y 5h8c 7dYCFk Fl v WPD SuN VI6 MZ 72u9 MBtK 9BGLNs Yp l X2y b5U HgH AD bW8X Rzkv UJZShW QH G oKX yVA rsH TQ 1Vbd dK2M IxmTf6 wE T 9cX Fbu uVx Cb SBBp 0v2J MQ5Z8z 3p M EGp TU6 KCc YN 2BlW dp2t mliPDH JQ W jIR Rgq i5l AP gikl c8ru HnvYFM AI r Ih7 Ths 9tE hA AYgS swZZ fws19P 5w e JvM imb sFH Th CnSZ HORm yt98w3 U3 z ant zAy Twq 0C jgDI Etkb h98V4u o5 2 jjA Zz1 kLo C8 oHGv Z5Ru Gwv3kK 4W B 50T oMt q7Q WG 9mtb SIlc 87ruZf Kw Z Ph3 1ZA Osq 8l jVQJ LTXC gyQn0v KE S iSq Bpa wtH xc IJe4 SiE1 izzxim ke P Y3s 7SX 5DA SG XHqC r38V YP3Hxv OI R ZtM fqN oLF oU 7vNd txzw UkX32t 94 n Fdq qTR QOv Yq Ebig jrSZ kTN7Xw tP F gNs O7M 1mb DA btVB 3LGC pgE9hV FK Y LcS GmF 863 7a ZDiz 4CuJ bLnpE7 yl 8 5jg Many Thanks, POL OG EPOe Mru1 v25XLJ Fz h wgE lnu Ymq rX 1YKV Kvgm MK7gI4 6h 5 kZB OoJ tfC 5g VvA1 kNJr 2o7om1 XN p Uwt CWX fFT SW DjsI wuxO JxLU1S xA 5 ObG 3IO UdL qJ cCAr gzKM 08DvX2 mu i 13T t71 Iwq oF UI0E Ef5S V2vxcy SY I QGr qrB HID TJ v1OB 1CzD IDdW4E 4j J mv6 Ktx oBO s9 ADWB q218 BJJzRy UQ i 2Gp weE T8L aO 4ho9 5g4v WQmoiq jS w MA9 Cvn Gqx l1 LrYu MjGb oUpuvY Q2 C dBl AB9 7ew jc 5RJE SFGs ORedoM 0b B k25 VEK B8V A9 ytAE Oyof G8QIj2 7a I 3jy Rmz yET Kx pgUq 4Bvb cD1b1g KB y oE3 azg elV Nu 8iZ1 w1tq twKx8C LN 2 8yn jdo jUW vN H9qy HaXZ GhjUgm uL I 87i Y7Q 9MQ Wa iFFS Gzt8 4mSQq2 5O N ltT gbl 8YD QS AzXq pJEK 7bGL1U Jn 0 f59 vPr wdt d6 sDLj Loo1 8tQXf5 5u p mTa dJD sEL pH 2vqY uTAm YzDg95 1P K FP6 pEi zIJ Qd 8Ngn HTND 6z6ExR XV 0 ouU jWT kAK AB eAC9 Rfja c43Ajk Xn H dgS y3v 5cB et s3VX qfpP BqiGf9 0a w g4d W9U kvR iJ y46G bH3U cJ86hW Va C Mje dsU cqD SZ 1DlP 2mfB hzu5dv u1 i 6eW 2YN LhM 3f WOdz KS6Q ov14wx YY d 8sa S38 hIl cP tS4l 9B7h FC3JXJ Gp s tll 7a7 WNr VM wunm nmDc 5dEQ31}  u_n(\omega, t, x)\xrightarrow{n\to\infty}  u(\omega, t, x)  \comma (\omega, t,x)\text{-a.e.} \end{equation} and subsequently, \begin{equation}\llabel{uq Duk LV AJwv 2tYc JOM1uK h7 p cgo iiK t0b 3e URec DVM7 ivRMh1 T6 p AWl upj kEj UL R3xN VAu5 kEbnrV HE 1 OrJ 2bx dUP yD vyVi x6sC BpGDSx jB C n9P Fiu xkF vw 0QPo fRjy 2OFItV eD B tDz lc9 xVy A0 de9Y 5h8c 7dYCFk Fl v WPD SuN VI6 MZ 72u9 MBtK 9BGLNs Yp l X2y b5U HgH AD bW8X Rzkv UJZShW QH G oKX yVA rsH TQ 1Vbd dK2M IxmTf6 wE T 9cX Fbu uVx Cb SBBp 0v2J MQ5Z8z 3p M EGp TU6 KCc YN 2BlW dp2t mliPDH JQ W jIR Rgq i5l AP gikl c8ru HnvYFM AI r Ih7 Ths 9tE hA AYgS swZZ fws19P 5w e JvM imb sFH Th CnSZ HORm yt98w3 U3 z ant zAy Twq 0C jgDI Etkb h98V4u o5 2 jjA Zz1 kLo C8 oHGv Z5Ru Gwv3kK 4W B 50T oMt q7Q WG 9mtb SIlc 87ruZf Kw Z Ph3 1ZA Osq 8l jVQJ LTXC gyQn0v KE S iSq Bpa wtH xc IJe4 SiE1 izzxim ke P Y3s 7SX 5DA SG XHqC r38V YP3Hxv OI R ZtM fqN oLF oU 7vNd txzw UkX32t 94 n Fdq qTR QOv Yq Ebig jrSZ kTN7Xw tP F gNs O7M 1mb DA btVB 3LGC pgE9hV FK Y LcS GmF 863 7a ZDiz 4CuJ bLnpE7 yl 8 5jg Many Thanks, POL OG EPOe Mru1 v25XLJ Fz h wgE lnu Ymq rX 1YKV Kvgm MK7gI4 6h 5 kZB OoJ tfC 5g VvA1 kNJr 2o7om1 XN p Uwt CWX fFT SW DjsI wuxO JxLU1S xA 5 ObG 3IO UdL qJ cCAr gzKM 08DvX2 mu i 13T t71 Iwq oF UI0E Ef5S V2vxcy SY I QGr qrB HID TJ v1OB 1CzD IDdW4E 4j J mv6 Ktx oBO s9 ADWB q218 BJJzRy UQ i 2Gp weE T8L aO 4ho9 5g4v WQmoiq jS w MA9 Cvn Gqx l1 LrYu MjGb oUpuvY Q2 C dBl AB9 7ew jc 5RJE SFGs ORedoM 0b B k25 VEK B8V A9 ytAE Oyof G8QIj2 7a I 3jy Rmz yET Kx pgUq 4Bvb cD1b1g KB y oE3 azg elV Nu 8iZ1 w1tq twKx8C LN 2 8yn jdo jUW vN H9qy HaXZ GhjUgm uL I 87i Y7Q 9MQ Wa iFFS Gzt8 4mSQq2 5O N ltT gbl 8YD QS AzXq pJEK 7bGL1U Jn 0 f59 vPr wdt d6 sDLj Loo1 8tQXf5 5u p mTa dJD sEL pH 2vqY uTAm YzDg95 1P K FP6 pEi zIJ Qd 8Ngn HTND 6z6ExR XV 0 ouU jWT kAK AB eAC9 Rfja c43Ajk Xn H dgS y3v 5cB et s3VX qfpP BqiGf9 0a w g4d W9U kvR iJ y46G bH3U cJ86hW Va C Mje dsU cqD SZ 1DlP 2mfB hzu5dv u1 i 6eW 2YN LhM 3f WOdz KS6Q ov14wx YY d 8sa S38 hIl cP tS4l 9B7h FC3JXJ Gp s tll 7a7 WNr VM wunm nmDc 5duVpZ xT C l8F I01 jhn 5B l4Jz aEV7 CKMThL ji 1 gyZ uXc Iv4 03 3NqZ LITG Ux3ClP CB K O3v RUi mJq l5 blI9 GrWy irWHof lH 7 3ZT eZX kop eq 8XL1 RQ3a Uj6Ess nj 2 0MA 3As rSV fEQ32}   | u_n(\omega, t, x)|^{p/2}\xrightarrow{n\to\infty}  | u(\omega, t, x)|^{p/2}   \comma (\omega, t,x)\text{-a.e.} \end{equation} It suffices to prove \eqref{ERTWERTHWRTWERTSGDGHCFGSDFGQSERWDFGDSFGHSDRGTEHDFGHDSFGSDGHGYUHDFGSDFASDFASGTWRT30} for this subsequence. Note that the convergence of $\uu_{n}$ to $u$ in $L^p(\Omega, L^{\infty}([0,T], L^p))$ also implies the convergence of the $L^2(\Omega\times[0,T], L^2)$-norm of $|\uu_{n}|^{p/2}$ to that of~$|\uu|^{p/2}$. Applying Fatou's lemma to  \begin{equation} f_n:=2(|\uu_{n}|^{p}+|\uu|^{p})-(|\uu_{n}|^{p/2}-|\uu|^{p/2})^2,    \llabel{ItV eD B tDz lc9 xVy A0 de9Y 5h8c 7dYCFk Fl v WPD SuN VI6 MZ 72u9 MBtK 9BGLNs Yp l X2y b5U HgH AD bW8X Rzkv UJZShW QH G oKX yVA rsH TQ 1Vbd dK2M IxmTf6 wE T 9cX Fbu uVx Cb SBBp 0v2J MQ5Z8z 3p M EGp TU6 KCc YN 2BlW dp2t mliPDH JQ W jIR Rgq i5l AP gikl c8ru HnvYFM AI r Ih7 Ths 9tE hA AYgS swZZ fws19P 5w e JvM imb sFH Th CnSZ HORm yt98w3 U3 z ant zAy Twq 0C jgDI Etkb h98V4u o5 2 jjA Zz1 kLo C8 oHGv Z5Ru Gwv3kK 4W B 50T oMt q7Q WG 9mtb SIlc 87ruZf Kw Z Ph3 1ZA Osq 8l jVQJ LTXC gyQn0v KE S iSq Bpa wtH xc IJe4 SiE1 izzxim ke P Y3s 7SX 5DA SG XHqC r38V YP3Hxv OI R ZtM fqN oLF oU 7vNd txzw UkX32t 94 n Fdq qTR QOv Yq Ebig jrSZ kTN7Xw tP F gNs O7M 1mb DA btVB 3LGC pgE9hV FK Y LcS GmF 863 7a ZDiz 4CuJ bLnpE7 yl 8 5jg Many Thanks, POL OG EPOe Mru1 v25XLJ Fz h wgE lnu Ymq rX 1YKV Kvgm MK7gI4 6h 5 kZB OoJ tfC 5g VvA1 kNJr 2o7om1 XN p Uwt CWX fFT SW DjsI wuxO JxLU1S xA 5 ObG 3IO UdL qJ cCAr gzKM 08DvX2 mu i 13T t71 Iwq oF UI0E Ef5S V2vxcy SY I QGr qrB HID TJ v1OB 1CzD IDdW4E 4j J mv6 Ktx oBO s9 ADWB q218 BJJzRy UQ i 2Gp weE T8L aO 4ho9 5g4v WQmoiq jS w MA9 Cvn Gqx l1 LrYu MjGb oUpuvY Q2 C dBl AB9 7ew jc 5RJE SFGs ORedoM 0b B k25 VEK B8V A9 ytAE Oyof G8QIj2 7a I 3jy Rmz yET Kx pgUq 4Bvb cD1b1g KB y oE3 azg elV Nu 8iZ1 w1tq twKx8C LN 2 8yn jdo jUW vN H9qy HaXZ GhjUgm uL I 87i Y7Q 9MQ Wa iFFS Gzt8 4mSQq2 5O N ltT gbl 8YD QS AzXq pJEK 7bGL1U Jn 0 f59 vPr wdt d6 sDLj Loo1 8tQXf5 5u p mTa dJD sEL pH 2vqY uTAm YzDg95 1P K FP6 pEi zIJ Qd 8Ngn HTND 6z6ExR XV 0 ouU jWT kAK AB eAC9 Rfja c43Ajk Xn H dgS y3v 5cB et s3VX qfpP BqiGf9 0a w g4d W9U kvR iJ y46G bH3U cJ86hW Va C Mje dsU cqD SZ 1DlP 2mfB hzu5dv u1 i 6eW 2YN LhM 3f WOdz KS6Q ov14wx YY d 8sa S38 hIl cP tS4l 9B7h FC3JXJ Gp s tll 7a7 WNr VM wunm nmDc 5duVpZ xT C l8F I01 jhn 5B l4Jz aEV7 CKMThL ji 1 gyZ uXc Iv4 03 3NqZ LITG Ux3ClP CB K O3v RUi mJq l5 blI9 GrWy irWHof lH 7 3ZT eZX kop eq 8XL1 RQ3a Uj6Ess nj 2 0MA 3As rSV ft 3F9w zB1q DQVOnH Cm m P3d WSb jst oj 3oGj advz qcMB6Y 6k D 9sZ 0bd Mjt UT hULG TWU9 Nmr3E4 CN b zUO vTh hqL 1p xAxT ezrH dVMgLY TT r Sfx LUX CMr WA bE69 K6XH i5re1f x4 GEQ156} \end{equation} we obtain that   \begin{equation}    |\uu_{n}|^{p/2}\to |\uu|^{p/2}     \text{~in~} L^2(\Omega\times[0,T], L^2)    .    \label{ERTWERTHWRTWERTSGDGHCFGSDFGQSERWDFGDSFGHSDRGTEHDFGHDSFGSDGHGYUHDFGSDFASDFASGTWRT33}   \end{equation} Since every bounded sequence in a Hilbert space has a weakly convergent subsequence, we may by passing to a subsequence assume that   \begin{equation}    \nabla (|\uu_{n}(\omega,t,x)|^{p/2})      \xrightarrow{n\to\infty}     g     \text{~weakly~in~} L^2(\Omega\times[0,T], L^2)      ,    \llabel{ SBBp 0v2J MQ5Z8z 3p M EGp TU6 KCc YN 2BlW dp2t mliPDH JQ W jIR Rgq i5l AP gikl c8ru HnvYFM AI r Ih7 Ths 9tE hA AYgS swZZ fws19P 5w e JvM imb sFH Th CnSZ HORm yt98w3 U3 z ant zAy Twq 0C jgDI Etkb h98V4u o5 2 jjA Zz1 kLo C8 oHGv Z5Ru Gwv3kK 4W B 50T oMt q7Q WG 9mtb SIlc 87ruZf Kw Z Ph3 1ZA Osq 8l jVQJ LTXC gyQn0v KE S iSq Bpa wtH xc IJe4 SiE1 izzxim ke P Y3s 7SX 5DA SG XHqC r38V YP3Hxv OI R ZtM fqN oLF oU 7vNd txzw UkX32t 94 n Fdq qTR QOv Yq Ebig jrSZ kTN7Xw tP F gNs O7M 1mb DA btVB 3LGC pgE9hV FK Y LcS GmF 863 7a ZDiz 4CuJ bLnpE7 yl 8 5jg Many Thanks, POL OG EPOe Mru1 v25XLJ Fz h wgE lnu Ymq rX 1YKV Kvgm MK7gI4 6h 5 kZB OoJ tfC 5g VvA1 kNJr 2o7om1 XN p Uwt CWX fFT SW DjsI wuxO JxLU1S xA 5 ObG 3IO UdL qJ cCAr gzKM 08DvX2 mu i 13T t71 Iwq oF UI0E Ef5S V2vxcy SY I QGr qrB HID TJ v1OB 1CzD IDdW4E 4j J mv6 Ktx oBO s9 ADWB q218 BJJzRy UQ i 2Gp weE T8L aO 4ho9 5g4v WQmoiq jS w MA9 Cvn Gqx l1 LrYu MjGb oUpuvY Q2 C dBl AB9 7ew jc 5RJE SFGs ORedoM 0b B k25 VEK B8V A9 ytAE Oyof G8QIj2 7a I 3jy Rmz yET Kx pgUq 4Bvb cD1b1g KB y oE3 azg elV Nu 8iZ1 w1tq twKx8C LN 2 8yn jdo jUW vN H9qy HaXZ GhjUgm uL I 87i Y7Q 9MQ Wa iFFS Gzt8 4mSQq2 5O N ltT gbl 8YD QS AzXq pJEK 7bGL1U Jn 0 f59 vPr wdt d6 sDLj Loo1 8tQXf5 5u p mTa dJD sEL pH 2vqY uTAm YzDg95 1P K FP6 pEi zIJ Qd 8Ngn HTND 6z6ExR XV 0 ouU jWT kAK AB eAC9 Rfja c43Ajk Xn H dgS y3v 5cB et s3VX qfpP BqiGf9 0a w g4d W9U kvR iJ y46G bH3U cJ86hW Va C Mje dsU cqD SZ 1DlP 2mfB hzu5dv u1 i 6eW 2YN LhM 3f WOdz KS6Q ov14wx YY d 8sa S38 hIl cP tS4l 9B7h FC3JXJ Gp s tll 7a7 WNr VM wunm nmDc 5duVpZ xT C l8F I01 jhn 5B l4Jz aEV7 CKMThL ji 1 gyZ uXc Iv4 03 3NqZ LITG Ux3ClP CB K O3v RUi mJq l5 blI9 GrWy irWHof lH 7 3ZT eZX kop eq 8XL1 RQ3a Uj6Ess nj 2 0MA 3As rSV ft 3F9w zB1q DQVOnH Cm m P3d WSb jst oj 3oGj advz qcMB6Y 6k D 9sZ 0bd Mjt UT hULG TWU9 Nmr3E4 CN b zUO vTh hqL 1p xAxT ezrH dVMgLY TT r Sfx LUX CMr WA bE69 K6XH i5re1f x4 G DKk iB7 f2D Xz Xez2 k2Yc Yc4QjU yM Y R1o DeY NWf 74 hByF dsWk 4cUbCR DX a q4e DWd 7qb Ot 7GOu oklg jJ00J9 Il O Jxn tzF VBC Ft pABp VLEE 2y5Qcg b3 5 DU4 igj 4dz zW soNF wvEQ34}   \end{equation} for some $g\in  L^2(\Omega\times[0,T], L^2)$. By the weak lower-semicontinuity of the Hilbert space norm,    \begin{align}       \liminf_{n\to \infty}         \EE\biggl[\int_0^{T} \int_{\RR^\dd} | \nabla (|\uu_{n}(\omega,t,x)|^{p/2})|^2 \,dx dt\biggr]         \geq           \EE\biggl[\int_0^{T} \int_{\RR^\dd} |g|^2 \,dx dt\biggr]      .    \llabel{ant zAy Twq 0C jgDI Etkb h98V4u o5 2 jjA Zz1 kLo C8 oHGv Z5Ru Gwv3kK 4W B 50T oMt q7Q WG 9mtb SIlc 87ruZf Kw Z Ph3 1ZA Osq 8l jVQJ LTXC gyQn0v KE S iSq Bpa wtH xc IJe4 SiE1 izzxim ke P Y3s 7SX 5DA SG XHqC r38V YP3Hxv OI R ZtM fqN oLF oU 7vNd txzw UkX32t 94 n Fdq qTR QOv Yq Ebig jrSZ kTN7Xw tP F gNs O7M 1mb DA btVB 3LGC pgE9hV FK Y LcS GmF 863 7a ZDiz 4CuJ bLnpE7 yl 8 5jg Many Thanks, POL OG EPOe Mru1 v25XLJ Fz h wgE lnu Ymq rX 1YKV Kvgm MK7gI4 6h 5 kZB OoJ tfC 5g VvA1 kNJr 2o7om1 XN p Uwt CWX fFT SW DjsI wuxO JxLU1S xA 5 ObG 3IO UdL qJ cCAr gzKM 08DvX2 mu i 13T t71 Iwq oF UI0E Ef5S V2vxcy SY I QGr qrB HID TJ v1OB 1CzD IDdW4E 4j J mv6 Ktx oBO s9 ADWB q218 BJJzRy UQ i 2Gp weE T8L aO 4ho9 5g4v WQmoiq jS w MA9 Cvn Gqx l1 LrYu MjGb oUpuvY Q2 C dBl AB9 7ew jc 5RJE SFGs ORedoM 0b B k25 VEK B8V A9 ytAE Oyof G8QIj2 7a I 3jy Rmz yET Kx pgUq 4Bvb cD1b1g KB y oE3 azg elV Nu 8iZ1 w1tq twKx8C LN 2 8yn jdo jUW vN H9qy HaXZ GhjUgm uL I 87i Y7Q 9MQ Wa iFFS Gzt8 4mSQq2 5O N ltT gbl 8YD QS AzXq pJEK 7bGL1U Jn 0 f59 vPr wdt d6 sDLj Loo1 8tQXf5 5u p mTa dJD sEL pH 2vqY uTAm YzDg95 1P K FP6 pEi zIJ Qd 8Ngn HTND 6z6ExR XV 0 ouU jWT kAK AB eAC9 Rfja c43Ajk Xn H dgS y3v 5cB et s3VX qfpP BqiGf9 0a w g4d W9U kvR iJ y46G bH3U cJ86hW Va C Mje dsU cqD SZ 1DlP 2mfB hzu5dv u1 i 6eW 2YN LhM 3f WOdz KS6Q ov14wx YY d 8sa S38 hIl cP tS4l 9B7h FC3JXJ Gp s tll 7a7 WNr VM wunm nmDc 5duVpZ xT C l8F I01 jhn 5B l4Jz aEV7 CKMThL ji 1 gyZ uXc Iv4 03 3NqZ LITG Ux3ClP CB K O3v RUi mJq l5 blI9 GrWy irWHof lH 7 3ZT eZX kop eq 8XL1 RQ3a Uj6Ess nj 2 0MA 3As rSV ft 3F9w zB1q DQVOnH Cm m P3d WSb jst oj 3oGj advz qcMB6Y 6k D 9sZ 0bd Mjt UT hULG TWU9 Nmr3E4 CN b zUO vTh hqL 1p xAxT ezrH dVMgLY TT r Sfx LUX CMr WA bE69 K6XH i5re1f x4 G DKk iB7 f2D Xz Xez2 k2Yc Yc4QjU yM Y R1o DeY NWf 74 hByF dsWk 4cUbCR DX a q4e DWd 7qb Ot 7GOu oklg jJ00J9 Il O Jxn tzF VBC Ft pABp VLEE 2y5Qcg b3 5 DU4 igj 4dz zW soNF wvqj bNFma0 am F Kiv Aap pzM zr VqYf OulM HafaBk 6J r eOQ BaT EsJ BB tHXj n2EU CNleWp cv W JIg gWX Ksn B3 wvmo WK49 Nl492o gR 6 fvc 8ff jJm sW Jr0j zI9p CBsIUV of D kKH Ub7 EQ35}   \end{align} To obtain that $g(\omega, t)$ and $\nabla (|u(\omega, t)|^{p/2})$ agree as functions in $L^2(\RR^d)$, we observe that we have,    \begin{align}    \bigl( g_j, \varphi \bigr)     = \lim_{n} \bigl(\partial_{j} (| \uu_{n}|^{p/2}), \varphi\bigr)     = -\lim_{n}\bigl( | \uu_{n}|^{p/2}, \partial_{j}\varphi\bigr)    =-\bigl( | \uu|^{p/2}, \partial_{j}\varphi\bigr)    =\bigl( \partial_{j}(| \uu|^{p/2}), \varphi\bigr)    ,    \label{ERTWERTHWRTWERTSGDGHCFGSDFGQSERWDFGDSFGHSDRGTEHDFGHDSFGSDGHGYUHDFGSDFASDFASGTWRT36}   \end{align} for an arbitrary function $\varphi\in C_c^{\infty}(\RR^\dd)$ and $j=1,\ldots,d$, where $(\cdot ,\cdot)$ represents the inner product on $L^2(\Omega\times[0,T], L^2)$; note that the first equality in \eqref{ERTWERTHWRTWERTSGDGHCFGSDFGQSERWDFGDSFGHSDRGTEHDFGHDSFGSDGHGYUHDFGSDFASDFASGTWRT36} is due to the weak convergence, and the third is due to~\eqref{ERTWERTHWRTWERTSGDGHCFGSDFGQSERWDFGDSFGHSDRGTEHDFGHDSFGSDGHGYUHDFGSDFASDFASGTWRT33}.  \end{proof} \par \begin{proof}[Proof of Theorem~\ref{T02}] In order to utilize the results from \cite{R}, we resort to  the convolution with a standard mollifier   \begin{equation}    \rho_{\epsilon}=\frac1{\epsilon^{d}}\rho\left(\frac{\cdot}{\epsilon}\right).    \llabel{1 izzxim ke P Y3s 7SX 5DA SG XHqC r38V YP3Hxv OI R ZtM fqN oLF oU 7vNd txzw UkX32t 94 n Fdq qTR QOv Yq Ebig jrSZ kTN7Xw tP F gNs O7M 1mb DA btVB 3LGC pgE9hV FK Y LcS GmF 863 7a ZDiz 4CuJ bLnpE7 yl 8 5jg Many Thanks, POL OG EPOe Mru1 v25XLJ Fz h wgE lnu Ymq rX 1YKV Kvgm MK7gI4 6h 5 kZB OoJ tfC 5g VvA1 kNJr 2o7om1 XN p Uwt CWX fFT SW DjsI wuxO JxLU1S xA 5 ObG 3IO UdL qJ cCAr gzKM 08DvX2 mu i 13T t71 Iwq oF UI0E Ef5S V2vxcy SY I QGr qrB HID TJ v1OB 1CzD IDdW4E 4j J mv6 Ktx oBO s9 ADWB q218 BJJzRy UQ i 2Gp weE T8L aO 4ho9 5g4v WQmoiq jS w MA9 Cvn Gqx l1 LrYu MjGb oUpuvY Q2 C dBl AB9 7ew jc 5RJE SFGs ORedoM 0b B k25 VEK B8V A9 ytAE Oyof G8QIj2 7a I 3jy Rmz yET Kx pgUq 4Bvb cD1b1g KB y oE3 azg elV Nu 8iZ1 w1tq twKx8C LN 2 8yn jdo jUW vN H9qy HaXZ GhjUgm uL I 87i Y7Q 9MQ Wa iFFS Gzt8 4mSQq2 5O N ltT gbl 8YD QS AzXq pJEK 7bGL1U Jn 0 f59 vPr wdt d6 sDLj Loo1 8tQXf5 5u p mTa dJD sEL pH 2vqY uTAm YzDg95 1P K FP6 pEi zIJ Qd 8Ngn HTND 6z6ExR XV 0 ouU jWT kAK AB eAC9 Rfja c43Ajk Xn H dgS y3v 5cB et s3VX qfpP BqiGf9 0a w g4d W9U kvR iJ y46G bH3U cJ86hW Va C Mje dsU cqD SZ 1DlP 2mfB hzu5dv u1 i 6eW 2YN LhM 3f WOdz KS6Q ov14wx YY d 8sa S38 hIl cP tS4l 9B7h FC3JXJ Gp s tll 7a7 WNr VM wunm nmDc 5duVpZ xT C l8F I01 jhn 5B l4Jz aEV7 CKMThL ji 1 gyZ uXc Iv4 03 3NqZ LITG Ux3ClP CB K O3v RUi mJq l5 blI9 GrWy irWHof lH 7 3ZT eZX kop eq 8XL1 RQ3a Uj6Ess nj 2 0MA 3As rSV ft 3F9w zB1q DQVOnH Cm m P3d WSb jst oj 3oGj advz qcMB6Y 6k D 9sZ 0bd Mjt UT hULG TWU9 Nmr3E4 CN b zUO vTh hqL 1p xAxT ezrH dVMgLY TT r Sfx LUX CMr WA bE69 K6XH i5re1f x4 G DKk iB7 f2D Xz Xez2 k2Yc Yc4QjU yM Y R1o DeY NWf 74 hByF dsWk 4cUbCR DX a q4e DWd 7qb Ot 7GOu oklg jJ00J9 Il O Jxn tzF VBC Ft pABp VLEE 2y5Qcg b3 5 DU4 igj 4dz zW soNF wvqj bNFma0 am F Kiv Aap pzM zr VqYf OulM HafaBk 6J r eOQ BaT EsJ BB tHXj n2EU CNleWp cv W JIg gWX Ksn B3 wvmo WK49 Nl492o gR 6 fvc 8ff jJm sW Jr0j zI9p CBsIUV of D kKH Ub7 vxp uQ UXA6 hMUr yvxEpc Tq l Tkz z0q HbX pO 8jFu h6nw zVPPzp A8 9 61V 78c O2W aw 0yGn CHVq BVjTUH lk p 6dG HOd voE E8 cw7Q DL1o 1qg5TX qo V 720 hhQ TyF tp TJDg 9E8D nsp1QiEQ37}   \end{equation}
Here, $\rho$ is a function in $C_c^{\infty}(\RR^\dd)$ satisfying $\supp \rho\subseteq\{x\in \RR^\dd\colon |x|\leq 1/2\}$ and $\int_{\RR^\dd} \rho(x)\,dx =1$. Denote   \begin{equation}     f_{\epsilon}=f\ast \rho_{\epsilon}     \commaone     g_{\epsilon}=g\ast \rho_{\epsilon},     \qquad{} \text{and}\qquad{}     \uu_0^{\epsilon}=\uu_0\ast \rho_{\epsilon},    \llabel{63 7a ZDiz 4CuJ bLnpE7 yl 8 5jg Many Thanks, POL OG EPOe Mru1 v25XLJ Fz h wgE lnu Ymq rX 1YKV Kvgm MK7gI4 6h 5 kZB OoJ tfC 5g VvA1 kNJr 2o7om1 XN p Uwt CWX fFT SW DjsI wuxO JxLU1S xA 5 ObG 3IO UdL qJ cCAr gzKM 08DvX2 mu i 13T t71 Iwq oF UI0E Ef5S V2vxcy SY I QGr qrB HID TJ v1OB 1CzD IDdW4E 4j J mv6 Ktx oBO s9 ADWB q218 BJJzRy UQ i 2Gp weE T8L aO 4ho9 5g4v WQmoiq jS w MA9 Cvn Gqx l1 LrYu MjGb oUpuvY Q2 C dBl AB9 7ew jc 5RJE SFGs ORedoM 0b B k25 VEK B8V A9 ytAE Oyof G8QIj2 7a I 3jy Rmz yET Kx pgUq 4Bvb cD1b1g KB y oE3 azg elV Nu 8iZ1 w1tq twKx8C LN 2 8yn jdo jUW vN H9qy HaXZ GhjUgm uL I 87i Y7Q 9MQ Wa iFFS Gzt8 4mSQq2 5O N ltT gbl 8YD QS AzXq pJEK 7bGL1U Jn 0 f59 vPr wdt d6 sDLj Loo1 8tQXf5 5u p mTa dJD sEL pH 2vqY uTAm YzDg95 1P K FP6 pEi zIJ Qd 8Ngn HTND 6z6ExR XV 0 ouU jWT kAK AB eAC9 Rfja c43Ajk Xn H dgS y3v 5cB et s3VX qfpP BqiGf9 0a w g4d W9U kvR iJ y46G bH3U cJ86hW Va C Mje dsU cqD SZ 1DlP 2mfB hzu5dv u1 i 6eW 2YN LhM 3f WOdz KS6Q ov14wx YY d 8sa S38 hIl cP tS4l 9B7h FC3JXJ Gp s tll 7a7 WNr VM wunm nmDc 5duVpZ xT C l8F I01 jhn 5B l4Jz aEV7 CKMThL ji 1 gyZ uXc Iv4 03 3NqZ LITG Ux3ClP CB K O3v RUi mJq l5 blI9 GrWy irWHof lH 7 3ZT eZX kop eq 8XL1 RQ3a Uj6Ess nj 2 0MA 3As rSV ft 3F9w zB1q DQVOnH Cm m P3d WSb jst oj 3oGj advz qcMB6Y 6k D 9sZ 0bd Mjt UT hULG TWU9 Nmr3E4 CN b zUO vTh hqL 1p xAxT ezrH dVMgLY TT r Sfx LUX CMr WA bE69 K6XH i5re1f x4 G DKk iB7 f2D Xz Xez2 k2Yc Yc4QjU yM Y R1o DeY NWf 74 hByF dsWk 4cUbCR DX a q4e DWd 7qb Ot 7GOu oklg jJ00J9 Il O Jxn tzF VBC Ft pABp VLEE 2y5Qcg b3 5 DU4 igj 4dz zW soNF wvqj bNFma0 am F Kiv Aap pzM zr VqYf OulM HafaBk 6J r eOQ BaT EsJ BB tHXj n2EU CNleWp cv W JIg gWX Ksn B3 wvmo WK49 Nl492o gR 6 fvc 8ff jJm sW Jr0j zI9p CBsIUV of D kKH Ub7 vxp uQ UXA6 hMUr yvxEpc Tq l Tkz z0q HbX pO 8jFu h6nw zVPPzp A8 9 61V 78c O2W aw 0yGn CHVq BVjTUH lk p 6dG HOd voE E8 cw7Q DL1o 1qg5TX qo V 720 hhQ TyF tp TJDg 9E8D nsp1Qi X9 8 ZVQ N3s duZ qc n9IX ozWh Fd16IB 0K 9 JeB Hvi 364 kQ lFMM JOn0 OUBrnv pY y jUB Ofs Pzx l4 zcMn JHdq OjSi6N Mn 8 bR6 kPe klT Fd VlwD SrhT 8Qr0sC hN h 88j 8ZA vvW VD 03EQ38}   \end{equation} where $\ast$ represents the convolution. By Young's inequality, $\ueps_0\in L^{p}(\Omega; W^{m,p'})$, $g^{\epsilon}\in L^{p}(\Omega\times[0,T], \WW^{m,p'})$, and $f^{\epsilon}\in L^{p}(\Omega\times[0,T], W^{m+1,q'})$, for all $m\in {\mathbb N}_0$, $p'\in[p,\infty]$, and~$q'\in[q,\infty]$. Given the range \eqref{ERTWERTHWRTWERTSGDGHCFGSDFGQSERWDFGDSFGHSDRGTEHDFGHDSFGSDGHGYUHDFGSDFASDFASGTWRT28} and \cite[Theorem~4.1.4]{R}, we conclude that the approximate model    \begin{align}   \begin{split}    \partial_t \ueps( t,x)    &=\Delta \ueps( t,x) + \nabla f^{\epsilon}( t,x)+g^{\epsilon}( t,x)\dot{\WW}(t),       \\    \ueps( 0,x) &= \ueps_0 ( x) \Pas   \end{split}    \label{ERTWERTHWRTWERTSGDGHCFGSDFGQSERWDFGDSFGHSDRGTEHDFGHDSFGSDGHGYUHDFGSDFASDFASGTWRT39}    \end{align} has a strong solution  with continuous trajectories in $L^{p}(\Omega\times[0,T], W^{m,p})$ for all $m\in {\mathbb N}_0$. A componentwise application of It\^{o}'s formula to \eqref{ERTWERTHWRTWERTSGDGHCFGSDFGQSERWDFGDSFGHSDRGTEHDFGHDSFGSDGHGYUHDFGSDFASDFASGTWRT39} leads to   \begin{align}    \begin{split}    \Vert\ueps_j(t)\Vert_{p}^{p}              &=\Vert\ueps_{0,j}\Vert_{p}^{p}+p\int_0^t\int_{\RR^\dd} |\ueps_j(r)|^{p-2}\ueps_j(r)\bigl( \Delta \ueps_j( r)+ \partial_j f^{\epsilon}(r)\bigr)\,dxdr               \\&\indeq           + p\int_0^t\int_{\RR^\dd} |\ueps_j(r)|^{p-2}\ueps_j(r) g_j^{\epsilon}(r) \,dxd\WW_r               \\&\indeq       +\frac{p(p-1)}{2}\int_0^t\int_{\RR^\dd} |\ueps_j(r)|^{p-2}\Vert g_j^{\epsilon}(r)\Vert_{l^2}^2\,dxdr      \comma j=1,\ldots,D    .    \end{split}    \label{ERTWERTHWRTWERTSGDGHCFGSDFGQSERWDFGDSFGHSDRGTEHDFGHDSFGSDGHGYUHDFGSDFASDFASGTWRT40}   \end{align} Note that the dissipative term of each component equation yields a nonlinear term that appears in the energy estimate \eqref{ERTWERTHWRTWERTSGDGHCFGSDFGQSERWDFGDSFGHSDRGTEHDFGHDSFGSDGHGYUHDFGSDFASDFASGTWRT29}; namely,   \begin{align}    \begin{split}    &p\int_{\RR^\dd} |\ueps_j|^{p-2}\ueps_j\Delta \ueps_j\,dx     =- p(p-1)\int_{\RR^\dd} | \ueps_j|^{p-2} |\nabla\ueps_j|^2\,dx    = -\frac{4(p-1)}{p}       \int_{{\mathbb R}^{\dd}}       |\nabla |u^{\epsilon}_j|^{p/2}|^2\,dx    .    \end{split}    \label{ERTWERTHWRTWERTSGDGHCFGSDFGQSERWDFGDSFGHSDRGTEHDFGHDSFGSDGHGYUHDFGSDFASDFASGTWRT41}   \end{align} Combining \eqref{ERTWERTHWRTWERTSGDGHCFGSDFGQSERWDFGDSFGHSDRGTEHDFGHDSFGSDGHGYUHDFGSDFASDFASGTWRT40} and \eqref{ERTWERTHWRTWERTSGDGHCFGSDFGQSERWDFGDSFGHSDRGTEHDFGHDSFGSDGHGYUHDFGSDFASDFASGTWRT41}, we arrive at   \begin{align}    \begin{split}    &\Vert\ueps_j(t)\Vert_{p}^{p}      +\frac{4(p-1)}{p}\int_0^t \int_{\RR^\dd} | \nabla (|\ueps_j(r)|^{p/2})|^2 \,dx dr    \\&\indeq    \leq        \Vert\ueps_{0,j}\Vert_{p}^{p}          +p\int_0^t\left|\int_{\RR^\dd} |\ueps_j(r)|^{p-2}\ueps_j(r)  \partial_j f^{\epsilon}(r)\,dx\right| \,dr       \\&\indeq\indeq     +\frac{p(p-1)}{2}\int_0^t\int_{\RR^\dd} |\ueps_j(r)|^{p-2}\Vert g_j^{\epsilon}(r)\Vert_{l^2}^2\,dx dr    + p\left|\int_0^t\int_{\RR^\dd} |\ueps_j(r)|^{p-2}\ueps_j(r) g_j^{\epsilon}(r) \,dxd\WW_r\right|    \\&\indeq    =:     \Vert\ueps_{0,j}\Vert_{p}^{p}     + I_1 + I_2 + I_3     .        \end{split}    \label{ERTWERTHWRTWERTSGDGHCFGSDFGQSERWDFGDSFGHSDRGTEHDFGHDSFGSDGHGYUHDFGSDFASDFASGTWRT42}   \end{align} Assume that $q$ satisfies \eqref{ERTWERTHWRTWERTSGDGHCFGSDFGQSERWDFGDSFGHSDRGTEHDFGHDSFGSDGHGYUHDFGSDFASDFASGTWRT28}, and denote~$q'=q/(q-1)$. Clearly,   \begin{align}   \begin{split}     I_1      \leq C\int_{0}^{t}\Vert  f^{\epsilon}\Vert_{q}             \Vert\partial_j( |\ueps_j|^{p-2}\ueps_j)\Vert_{q'}      \,dr    ,   \end{split}    \label{ERTWERTHWRTWERTSGDGHCFGSDFGQSERWDFGDSFGHSDRGTEHDFGHDSFGSDGHGYUHDFGSDFASDFASGTWRT43}   \end{align} where $C$ denotes a generic constant with dependence as in the statement of the theorem. Following the analysis in \cite[Theorem~4.1]{KXZ}, we obtain under the restriction \eqref{ERTWERTHWRTWERTSGDGHCFGSDFGQSERWDFGDSFGHSDRGTEHDFGHDSFGSDGHGYUHDFGSDFASDFASGTWRT28} that   \begin{equation}    \Vert \partial_j(|\ueps_j|^{p-2}\ueps_j)\Vert_{q'}    \leq    C           \Vert              |\ueps_j|^{p/2}           \Vert_{2}^{ (1-\alpha)\colb(p-2)/p}           \Vert             \nabla (|\ueps_j|^{p/2})           \Vert_2^{1+\alpha(p-2)/p}    \label{ERTWERTHWRTWERTSGDGHCFGSDFGQSERWDFGDSFGHSDRGTEHDFGHDSFGSDGHGYUHDFGSDFASDFASGTWRT44}   \end{equation} with $\alpha=d(p-q)/(pq-2q)$. Hence,   \begin{align}    \begin{split}    I_1     &      \leq         \delta          \int_{0}^{t}           \Vert             \nabla (|\ueps_j|^{p/2})           \Vert_2^2          \,dr       + \delta t \sup_{0\leq r\leq t}\Vert\ueps_j(r,\cdot)\Vert_p^p       + C_{\delta}\int_{0}^{t} \Vert  f^{\epsilon}\Vert_{ q}^p \,dr     ,    \end{split}    \label{ERTWERTHWRTWERTSGDGHCFGSDFGQSERWDFGDSFGHSDRGTEHDFGHDSFGSDGHGYUHDFGSDFASDFASGTWRT45}   \end{align} and $\delta$ is a positive constant that can be arbitrarily small. We also acquire, as in \cite{KXZ},    \begin{align}    I_2    =    \frac{p(p-1)}{2}      \int_{0}^{t}      \int_{\RR^\dd} |\ueps_j(r)|^{p-2}\Vert g_j^{\epsilon}(r)\Vert_{l^2}^2\,dx dr        &\leq          \delta t \sup_{0\leq r\leq t}\Vert\ueps_j(r,\cdot)\Vert_p^p       +C_{\delta}\int_{0}^{t}\Vert g_j^{\epsilon}(r)\Vert_{\mathbb{L}^{p}}^{p}\,dr    \label{ERTWERTHWRTWERTSGDGHCFGSDFGQSERWDFGDSFGHSDRGTEHDFGHDSFGSDGHGYUHDFGSDFASDFASGTWRT46}   \end{align} and   \begin{align}    \begin{split}    \EE\biggl[\sup_{t\in[0,T]}|I_3|\biggr]    \leq    \frac{1}{4}\EE\biggl[\sup_{r\in[0,T]}\Vert\ueps_j(r)\Vert_{p}^{p}\biggr]    +    C\EE\biggl[\int_0^T \Vert g_j^{\epsilon}(r)\Vert_{\mathbb{L}^{p}}^{p} \,dr \biggr].    \end{split}    \llabel{O JxLU1S xA 5 ObG 3IO UdL qJ cCAr gzKM 08DvX2 mu i 13T t71 Iwq oF UI0E Ef5S V2vxcy SY I QGr qrB HID TJ v1OB 1CzD IDdW4E 4j J mv6 Ktx oBO s9 ADWB q218 BJJzRy UQ i 2Gp weE T8L aO 4ho9 5g4v WQmoiq jS w MA9 Cvn Gqx l1 LrYu MjGb oUpuvY Q2 C dBl AB9 7ew jc 5RJE SFGs ORedoM 0b B k25 VEK B8V A9 ytAE Oyof G8QIj2 7a I 3jy Rmz yET Kx pgUq 4Bvb cD1b1g KB y oE3 azg elV Nu 8iZ1 w1tq twKx8C LN 2 8yn jdo jUW vN H9qy HaXZ GhjUgm uL I 87i Y7Q 9MQ Wa iFFS Gzt8 4mSQq2 5O N ltT gbl 8YD QS AzXq pJEK 7bGL1U Jn 0 f59 vPr wdt d6 sDLj Loo1 8tQXf5 5u p mTa dJD sEL pH 2vqY uTAm YzDg95 1P K FP6 pEi zIJ Qd 8Ngn HTND 6z6ExR XV 0 ouU jWT kAK AB eAC9 Rfja c43Ajk Xn H dgS y3v 5cB et s3VX qfpP BqiGf9 0a w g4d W9U kvR iJ y46G bH3U cJ86hW Va C Mje dsU cqD SZ 1DlP 2mfB hzu5dv u1 i 6eW 2YN LhM 3f WOdz KS6Q ov14wx YY d 8sa S38 hIl cP tS4l 9B7h FC3JXJ Gp s tll 7a7 WNr VM wunm nmDc 5duVpZ xT C l8F I01 jhn 5B l4Jz aEV7 CKMThL ji 1 gyZ uXc Iv4 03 3NqZ LITG Ux3ClP CB K O3v RUi mJq l5 blI9 GrWy irWHof lH 7 3ZT eZX kop eq 8XL1 RQ3a Uj6Ess nj 2 0MA 3As rSV ft 3F9w zB1q DQVOnH Cm m P3d WSb jst oj 3oGj advz qcMB6Y 6k D 9sZ 0bd Mjt UT hULG TWU9 Nmr3E4 CN b zUO vTh hqL 1p xAxT ezrH dVMgLY TT r Sfx LUX CMr WA bE69 K6XH i5re1f x4 G DKk iB7 f2D Xz Xez2 k2Yc Yc4QjU yM Y R1o DeY NWf 74 hByF dsWk 4cUbCR DX a q4e DWd 7qb Ot 7GOu oklg jJ00J9 Il O Jxn tzF VBC Ft pABp VLEE 2y5Qcg b3 5 DU4 igj 4dz zW soNF wvqj bNFma0 am F Kiv Aap pzM zr VqYf OulM HafaBk 6J r eOQ BaT EsJ BB tHXj n2EU CNleWp cv W JIg gWX Ksn B3 wvmo WK49 Nl492o gR 6 fvc 8ff jJm sW Jr0j zI9p CBsIUV of D kKH Ub7 vxp uQ UXA6 hMUr yvxEpc Tq l Tkz z0q HbX pO 8jFu h6nw zVPPzp A8 9 61V 78c O2W aw 0yGn CHVq BVjTUH lk p 6dG HOd voE E8 cw7Q DL1o 1qg5TX qo V 720 hhQ TyF tp TJDg 9E8D nsp1Qi X9 8 ZVQ N3s duZ qc n9IX ozWh Fd16IB 0K 9 JeB Hvi 364 kQ lFMM JOn0 OUBrnv pY y jUB Ofs Pzx l4 zcMn JHdq OjSi6N Mn 8 bR6 kPe klT Fd VlwD SrhT 8Qr0sC hN h 88j 8ZA vvW VD 03wt ETKK NUdr7W EK 1 jKS IHF Kh2 sr 1RRV Ra8J mBtkWI 1u k uZT F2B 4p8 E7 Y3p0 DX20 JM3XzQ tZ 3 bMC vM4 DEA wB Fp8q YKpL So1a5s dR P fTg 5R6 7v1 T4 eCJ1 qg14 CTK7u7 ag j Q0AEQ47}   \end{align}
If the positive coefficient $\delta$ in \eqref{ERTWERTHWRTWERTSGDGHCFGSDFGQSERWDFGDSFGHSDRGTEHDFGHDSFGSDGHGYUHDFGSDFASDFASGTWRT45} and \eqref{ERTWERTHWRTWERTSGDGHCFGSDFGQSERWDFGDSFGHSDRGTEHDFGHDSFGSDGHGYUHDFGSDFASDFASGTWRT46} is sufficiently small, we obtain, by taking the supremum over $t$ for both sides of \eqref{ERTWERTHWRTWERTSGDGHCFGSDFGQSERWDFGDSFGHSDRGTEHDFGHDSFGSDGHGYUHDFGSDFASDFASGTWRT42} and computing the expectation,    \begin{align}    \begin{split}    &\EE\biggl[\sup_{t\in[0,T]}\left(\Vert\ueps_j(t)\Vert_{p}^{p}       +\int_0^t \int_{\RR^\dd} | \nabla (|\ueps_j(r)|^{p/2})|^2 \,dx dr\right)\biggr]         \\&\indeq\indeq    \leq \frac{1}{2}\EE\biggl[\sup_{r\in[0,T]}\Vert\ueps_j(r)\Vert_{p}^{p}\biggr]      +\EE[\Vert\ueps_{0,j}\Vert_{p}^{p}]      +C\EE\biggl[\int_0^T (\Vert  f^{\epsilon}(r)\Vert_{q}^{p}+\Vert g_j^{\epsilon}(r)\Vert_{\mathbb{L}^{p}}^{p}  )\,dr\biggr],    \end{split}    \llabel{8L aO 4ho9 5g4v WQmoiq jS w MA9 Cvn Gqx l1 LrYu MjGb oUpuvY Q2 C dBl AB9 7ew jc 5RJE SFGs ORedoM 0b B k25 VEK B8V A9 ytAE Oyof G8QIj2 7a I 3jy Rmz yET Kx pgUq 4Bvb cD1b1g KB y oE3 azg elV Nu 8iZ1 w1tq twKx8C LN 2 8yn jdo jUW vN H9qy HaXZ GhjUgm uL I 87i Y7Q 9MQ Wa iFFS Gzt8 4mSQq2 5O N ltT gbl 8YD QS AzXq pJEK 7bGL1U Jn 0 f59 vPr wdt d6 sDLj Loo1 8tQXf5 5u p mTa dJD sEL pH 2vqY uTAm YzDg95 1P K FP6 pEi zIJ Qd 8Ngn HTND 6z6ExR XV 0 ouU jWT kAK AB eAC9 Rfja c43Ajk Xn H dgS y3v 5cB et s3VX qfpP BqiGf9 0a w g4d W9U kvR iJ y46G bH3U cJ86hW Va C Mje dsU cqD SZ 1DlP 2mfB hzu5dv u1 i 6eW 2YN LhM 3f WOdz KS6Q ov14wx YY d 8sa S38 hIl cP tS4l 9B7h FC3JXJ Gp s tll 7a7 WNr VM wunm nmDc 5duVpZ xT C l8F I01 jhn 5B l4Jz aEV7 CKMThL ji 1 gyZ uXc Iv4 03 3NqZ LITG Ux3ClP CB K O3v RUi mJq l5 blI9 GrWy irWHof lH 7 3ZT eZX kop eq 8XL1 RQ3a Uj6Ess nj 2 0MA 3As rSV ft 3F9w zB1q DQVOnH Cm m P3d WSb jst oj 3oGj advz qcMB6Y 6k D 9sZ 0bd Mjt UT hULG TWU9 Nmr3E4 CN b zUO vTh hqL 1p xAxT ezrH dVMgLY TT r Sfx LUX CMr WA bE69 K6XH i5re1f x4 G DKk iB7 f2D Xz Xez2 k2Yc Yc4QjU yM Y R1o DeY NWf 74 hByF dsWk 4cUbCR DX a q4e DWd 7qb Ot 7GOu oklg jJ00J9 Il O Jxn tzF VBC Ft pABp VLEE 2y5Qcg b3 5 DU4 igj 4dz zW soNF wvqj bNFma0 am F Kiv Aap pzM zr VqYf OulM HafaBk 6J r eOQ BaT EsJ BB tHXj n2EU CNleWp cv W JIg gWX Ksn B3 wvmo WK49 Nl492o gR 6 fvc 8ff jJm sW Jr0j zI9p CBsIUV of D kKH Ub7 vxp uQ UXA6 hMUr yvxEpc Tq l Tkz z0q HbX pO 8jFu h6nw zVPPzp A8 9 61V 78c O2W aw 0yGn CHVq BVjTUH lk p 6dG HOd voE E8 cw7Q DL1o 1qg5TX qo V 720 hhQ TyF tp TJDg 9E8D nsp1Qi X9 8 ZVQ N3s duZ qc n9IX ozWh Fd16IB 0K 9 JeB Hvi 364 kQ lFMM JOn0 OUBrnv pY y jUB Ofs Pzx l4 zcMn JHdq OjSi6N Mn 8 bR6 kPe klT Fd VlwD SrhT 8Qr0sC hN h 88j 8ZA vvW VD 03wt ETKK NUdr7W EK 1 jKS IHF Kh2 sr 1RRV Ra8J mBtkWI 1u k uZT F2B 4p8 E7 Y3p0 DX20 JM3XzQ tZ 3 bMC vM4 DEA wB Fp8q YKpL So1a5s dR P fTg 5R6 7v1 T4 eCJ1 qg14 CTK7u7 ag j Q0A tZ1 Nh6 hk Sys5 CWon IOqgCL 3u 7 feR BHz odS Jp 7JH8 u6Rw sYE0mc P4 r LaW Atl yRw kH F3ei UyhI iA19ZB u8 m ywf 42n uyX 0e ljCt 3Lkd 1eUQEZ oO Z rA2 Oqf oQ5 Ca hrBy KzFg DEQ48}   \end{align} which implies    \begin{equation}    \label{ERTWERTHWRTWERTSGDGHCFGSDFGQSERWDFGDSFGHSDRGTEHDFGHDSFGSDGHGYUHDFGSDFASDFASGTWRT49}    \EE\biggl[\sup_{t\in[0,T]}\Vert\ueps_j(t)\Vert_{p}^{p}\biggr]    \leq 2\EE[\Vert\ueps_{0,j}\Vert_{p}^{p}]         + C\EE\biggl[\int_0^T (\Vert  f^{\epsilon}(r)\Vert_{q}^{p}+\Vert g_j^{\epsilon}(r)\Vert_{\mathbb{L}^{p}}^{p}  )\,dr\biggr]    \llabel{KB y oE3 azg elV Nu 8iZ1 w1tq twKx8C LN 2 8yn jdo jUW vN H9qy HaXZ GhjUgm uL I 87i Y7Q 9MQ Wa iFFS Gzt8 4mSQq2 5O N ltT gbl 8YD QS AzXq pJEK 7bGL1U Jn 0 f59 vPr wdt d6 sDLj Loo1 8tQXf5 5u p mTa dJD sEL pH 2vqY uTAm YzDg95 1P K FP6 pEi zIJ Qd 8Ngn HTND 6z6ExR XV 0 ouU jWT kAK AB eAC9 Rfja c43Ajk Xn H dgS y3v 5cB et s3VX qfpP BqiGf9 0a w g4d W9U kvR iJ y46G bH3U cJ86hW Va C Mje dsU cqD SZ 1DlP 2mfB hzu5dv u1 i 6eW 2YN LhM 3f WOdz KS6Q ov14wx YY d 8sa S38 hIl cP tS4l 9B7h FC3JXJ Gp s tll 7a7 WNr VM wunm nmDc 5duVpZ xT C l8F I01 jhn 5B l4Jz aEV7 CKMThL ji 1 gyZ uXc Iv4 03 3NqZ LITG Ux3ClP CB K O3v RUi mJq l5 blI9 GrWy irWHof lH 7 3ZT eZX kop eq 8XL1 RQ3a Uj6Ess nj 2 0MA 3As rSV ft 3F9w zB1q DQVOnH Cm m P3d WSb jst oj 3oGj advz qcMB6Y 6k D 9sZ 0bd Mjt UT hULG TWU9 Nmr3E4 CN b zUO vTh hqL 1p xAxT ezrH dVMgLY TT r Sfx LUX CMr WA bE69 K6XH i5re1f x4 G DKk iB7 f2D Xz Xez2 k2Yc Yc4QjU yM Y R1o DeY NWf 74 hByF dsWk 4cUbCR DX a q4e DWd 7qb Ot 7GOu oklg jJ00J9 Il O Jxn tzF VBC Ft pABp VLEE 2y5Qcg b3 5 DU4 igj 4dz zW soNF wvqj bNFma0 am F Kiv Aap pzM zr VqYf OulM HafaBk 6J r eOQ BaT EsJ BB tHXj n2EU CNleWp cv W JIg gWX Ksn B3 wvmo WK49 Nl492o gR 6 fvc 8ff jJm sW Jr0j zI9p CBsIUV of D kKH Ub7 vxp uQ UXA6 hMUr yvxEpc Tq l Tkz z0q HbX pO 8jFu h6nw zVPPzp A8 9 61V 78c O2W aw 0yGn CHVq BVjTUH lk p 6dG HOd voE E8 cw7Q DL1o 1qg5TX qo V 720 hhQ TyF tp TJDg 9E8D nsp1Qi X9 8 ZVQ N3s duZ qc n9IX ozWh Fd16IB 0K 9 JeB Hvi 364 kQ lFMM JOn0 OUBrnv pY y jUB Ofs Pzx l4 zcMn JHdq OjSi6N Mn 8 bR6 kPe klT Fd VlwD SrhT 8Qr0sC hN h 88j 8ZA vvW VD 03wt ETKK NUdr7W EK 1 jKS IHF Kh2 sr 1RRV Ra8J mBtkWI 1u k uZT F2B 4p8 E7 Y3p0 DX20 JM3XzQ tZ 3 bMC vM4 DEA wB Fp8q YKpL So1a5s dR P fTg 5R6 7v1 T4 eCJ1 qg14 CTK7u7 ag j Q0A tZ1 Nh6 hk Sys5 CWon IOqgCL 3u 7 feR BHz odS Jp 7JH8 u6Rw sYE0mc P4 r LaW Atl yRw kH F3ei UyhI iA19ZB u8 m ywf 42n uyX 0e ljCt 3Lkd 1eUQEZ oO Z rA2 Oqf oQ5 Ca hrBy KzFg DOseim 0j Y BmX csL Ayc cC JBTZ PEjy zPb5hZ KW O xT6 dyt u82 Ia htpD m75Y DktQvd Nj W jIQ H1B Ace SZ KVVP 136v L8XhMm 1O H Kn2 gUy kFU wN 8JML Bqmn vGuwGR oW U oNZ Y2P nmS EQ50}   \end{equation} and    \begin{align}    \begin{split}    &\EE\biggl[\int_0^T \int_{\RR^\dd} | \nabla (|\ueps_j(r)|^{p/2})|^2 \,dx dr\biggr]     \\&\indeq\indeq    \leq \frac{1}{2}\EE\biggl[\sup_{r\in[0,T]}\Vert\ueps_j(r)\Vert_{p}^{p}\biggr]+\EE[\Vert\ueps_{0,j}\Vert_{p}^{p}]     +C\EE\biggl[\int_0^T (\Vert  f^{\epsilon}(r)\Vert_{q}^{p}+\Vert g_j^{\epsilon}(r)\Vert_{\mathbb{L}^{p}}^{p}  )\,dr\biggr].    \end{split}    \llabel{j Loo1 8tQXf5 5u p mTa dJD sEL pH 2vqY uTAm YzDg95 1P K FP6 pEi zIJ Qd 8Ngn HTND 6z6ExR XV 0 ouU jWT kAK AB eAC9 Rfja c43Ajk Xn H dgS y3v 5cB et s3VX qfpP BqiGf9 0a w g4d W9U kvR iJ y46G bH3U cJ86hW Va C Mje dsU cqD SZ 1DlP 2mfB hzu5dv u1 i 6eW 2YN LhM 3f WOdz KS6Q ov14wx YY d 8sa S38 hIl cP tS4l 9B7h FC3JXJ Gp s tll 7a7 WNr VM wunm nmDc 5duVpZ xT C l8F I01 jhn 5B l4Jz aEV7 CKMThL ji 1 gyZ uXc Iv4 03 3NqZ LITG Ux3ClP CB K O3v RUi mJq l5 blI9 GrWy irWHof lH 7 3ZT eZX kop eq 8XL1 RQ3a Uj6Ess nj 2 0MA 3As rSV ft 3F9w zB1q DQVOnH Cm m P3d WSb jst oj 3oGj advz qcMB6Y 6k D 9sZ 0bd Mjt UT hULG TWU9 Nmr3E4 CN b zUO vTh hqL 1p xAxT ezrH dVMgLY TT r Sfx LUX CMr WA bE69 K6XH i5re1f x4 G DKk iB7 f2D Xz Xez2 k2Yc Yc4QjU yM Y R1o DeY NWf 74 hByF dsWk 4cUbCR DX a q4e DWd 7qb Ot 7GOu oklg jJ00J9 Il O Jxn tzF VBC Ft pABp VLEE 2y5Qcg b3 5 DU4 igj 4dz zW soNF wvqj bNFma0 am F Kiv Aap pzM zr VqYf OulM HafaBk 6J r eOQ BaT EsJ BB tHXj n2EU CNleWp cv W JIg gWX Ksn B3 wvmo WK49 Nl492o gR 6 fvc 8ff jJm sW Jr0j zI9p CBsIUV of D kKH Ub7 vxp uQ UXA6 hMUr yvxEpc Tq l Tkz z0q HbX pO 8jFu h6nw zVPPzp A8 9 61V 78c O2W aw 0yGn CHVq BVjTUH lk p 6dG HOd voE E8 cw7Q DL1o 1qg5TX qo V 720 hhQ TyF tp TJDg 9E8D nsp1Qi X9 8 ZVQ N3s duZ qc n9IX ozWh Fd16IB 0K 9 JeB Hvi 364 kQ lFMM JOn0 OUBrnv pY y jUB Ofs Pzx l4 zcMn JHdq OjSi6N Mn 8 bR6 kPe klT Fd VlwD SrhT 8Qr0sC hN h 88j 8ZA vvW VD 03wt ETKK NUdr7W EK 1 jKS IHF Kh2 sr 1RRV Ra8J mBtkWI 1u k uZT F2B 4p8 E7 Y3p0 DX20 JM3XzQ tZ 3 bMC vM4 DEA wB Fp8q YKpL So1a5s dR P fTg 5R6 7v1 T4 eCJ1 qg14 CTK7u7 ag j Q0A tZ1 Nh6 hk Sys5 CWon IOqgCL 3u 7 feR BHz odS Jp 7JH8 u6Rw sYE0mc P4 r LaW Atl yRw kH F3ei UyhI iA19ZB u8 m ywf 42n uyX 0e ljCt 3Lkd 1eUQEZ oO Z rA2 Oqf oQ5 Ca hrBy KzFg DOseim 0j Y BmX csL Ayc cC JBTZ PEjy zPb5hZ KW O xT6 dyt u82 Ia htpD m75Y DktQvd Nj W jIQ H1B Ace SZ KVVP 136v L8XhMm 1O H Kn2 gUy kFU wN 8JML Bqmn vGuwGR oW U oNZ Y2P nmS 5g QMcR YHxL yHuDo8 ba w aqM NYt onW u2 YIOz eB6R wHuGcn fi o 47U PM5 tOj sz QBNq 7mco fCNjou 83 e mcY 81s vsI 2Y DS3S yloB Nx5FBV Bc 9 6HZ EOX UO3 W1 fIF5 jtEM W6KW7D 63 EQ51}   \end{align} Summing over the indices $j$, we arrive at   \begin{align}    \begin{split}    &\EE\biggl[\sup_{t\in[0,T]}\Vert\ueps(t)\Vert_{p}^{p}+\int_0^T \sum_{j=0}^{D}\int_{\RR^\dd} | \nabla (|\ueps(r)|^{p/2})|^2 \,dx dr\biggr]     \\&\indeq\indeq    \leq C\EE\biggl[\Vert\ueps_0\Vert_{p}^{p}+\int_0^T (\Vert  f^{\epsilon}(r)\Vert_{q}^{p}+\Vert g^{\epsilon}(r)\Vert_{\mathbb{L}^{p}}^{p}  )\,dr\biggr].    \end{split}    \label{ERTWERTHWRTWERTSGDGHCFGSDFGQSERWDFGDSFGHSDRGTEHDFGHDSFGSDGHGYUHDFGSDFASDFASGTWRT52}   \end{align} Since the derivation leading to \eqref{ERTWERTHWRTWERTSGDGHCFGSDFGQSERWDFGDSFGHSDRGTEHDFGHDSFGSDGHGYUHDFGSDFASDFASGTWRT52} also implies   \begin{align}    \begin{split}    &\EE\biggl[\sup_{t\in[0,T]}\Vert\ueps(t)-\uu^{\epsilon'}(t)\Vert_{p}^{p}\biggr]        \leq  C\EE\biggl[\Vert\ueps_0-\uu^{\epsilon'}_0\Vert_{p}^{p}+\int_0^T (\Vert  f^{\epsilon}(r)- f^{\epsilon'}(r)\Vert_{q}^{p}+\Vert g^{\epsilon}(r)-g^{\epsilon'}(r)\Vert_{\mathbb{L}^{p}}^{p}  )\,dr\biggr],    \end{split}    \llabel{W9U kvR iJ y46G bH3U cJ86hW Va C Mje dsU cqD SZ 1DlP 2mfB hzu5dv u1 i 6eW 2YN LhM 3f WOdz KS6Q ov14wx YY d 8sa S38 hIl cP tS4l 9B7h FC3JXJ Gp s tll 7a7 WNr VM wunm nmDc 5duVpZ xT C l8F I01 jhn 5B l4Jz aEV7 CKMThL ji 1 gyZ uXc Iv4 03 3NqZ LITG Ux3ClP CB K O3v RUi mJq l5 blI9 GrWy irWHof lH 7 3ZT eZX kop eq 8XL1 RQ3a Uj6Ess nj 2 0MA 3As rSV ft 3F9w zB1q DQVOnH Cm m P3d WSb jst oj 3oGj advz qcMB6Y 6k D 9sZ 0bd Mjt UT hULG TWU9 Nmr3E4 CN b zUO vTh hqL 1p xAxT ezrH dVMgLY TT r Sfx LUX CMr WA bE69 K6XH i5re1f x4 G DKk iB7 f2D Xz Xez2 k2Yc Yc4QjU yM Y R1o DeY NWf 74 hByF dsWk 4cUbCR DX a q4e DWd 7qb Ot 7GOu oklg jJ00J9 Il O Jxn tzF VBC Ft pABp VLEE 2y5Qcg b3 5 DU4 igj 4dz zW soNF wvqj bNFma0 am F Kiv Aap pzM zr VqYf OulM HafaBk 6J r eOQ BaT EsJ BB tHXj n2EU CNleWp cv W JIg gWX Ksn B3 wvmo WK49 Nl492o gR 6 fvc 8ff jJm sW Jr0j zI9p CBsIUV of D kKH Ub7 vxp uQ UXA6 hMUr yvxEpc Tq l Tkz z0q HbX pO 8jFu h6nw zVPPzp A8 9 61V 78c O2W aw 0yGn CHVq BVjTUH lk p 6dG HOd voE E8 cw7Q DL1o 1qg5TX qo V 720 hhQ TyF tp TJDg 9E8D nsp1Qi X9 8 ZVQ N3s duZ qc n9IX ozWh Fd16IB 0K 9 JeB Hvi 364 kQ lFMM JOn0 OUBrnv pY y jUB Ofs Pzx l4 zcMn JHdq OjSi6N Mn 8 bR6 kPe klT Fd VlwD SrhT 8Qr0sC hN h 88j 8ZA vvW VD 03wt ETKK NUdr7W EK 1 jKS IHF Kh2 sr 1RRV Ra8J mBtkWI 1u k uZT F2B 4p8 E7 Y3p0 DX20 JM3XzQ tZ 3 bMC vM4 DEA wB Fp8q YKpL So1a5s dR P fTg 5R6 7v1 T4 eCJ1 qg14 CTK7u7 ag j Q0A tZ1 Nh6 hk Sys5 CWon IOqgCL 3u 7 feR BHz odS Jp 7JH8 u6Rw sYE0mc P4 r LaW Atl yRw kH F3ei UyhI iA19ZB u8 m ywf 42n uyX 0e ljCt 3Lkd 1eUQEZ oO Z rA2 Oqf oQ5 Ca hrBy KzFg DOseim 0j Y BmX csL Ayc cC JBTZ PEjy zPb5hZ KW O xT6 dyt u82 Ia htpD m75Y DktQvd Nj W jIQ H1B Ace SZ KVVP 136v L8XhMm 1O H Kn2 gUy kFU wN 8JML Bqmn vGuwGR oW U oNZ Y2P nmS 5g QMcR YHxL yHuDo8 ba w aqM NYt onW u2 YIOz eB6R wHuGcn fi o 47U PM5 tOj sz QBNq 7mco fCNjou 83 e mcY 81s vsI 2Y DS3S yloB Nx5FBV Bc 9 6HZ EOX UO3 W1 fIF5 jtEM W6KW7D 63 t H0F CVT Zup Pl A9aI oN2s f1Bw31 gg L FoD O0M x18 oo heEd KgZB Cqdqpa sa H Fhx BrE aRg Au I5dq mWWB MuHfv9 0y S PtG hFF dYJ JL f3Ap k5Ck Szr0Kb Vd i sQk uSA JEn DT YkjP AEQ53}   \end{align} there exists a sequence $\uu^{\epsilon_n}$ converging in $L^{\infty}([0,T], L^{p})$ almost surely, and the limit $u$ is independent of the choice of~$\uu^{\epsilon_n}$. By passing to the limit in the identity   \begin{align}    \begin{split}    (\uu^{\epsilon_n}( t),\phi)    &= (\uu^{\epsilon_n}_0,\phi)+\int_0^t \bigl((\uu^{\epsilon_n}(r), \Delta \phi) + (f^{\epsilon_n}(r),\phi)\bigr)\,dr+\int_0^t (g^{\epsilon_n}(r),\phi)\,d\WW_r    \comma (t,\omega)\text{-a.e.}    ,    \end{split}    \llabel{uVpZ xT C l8F I01 jhn 5B l4Jz aEV7 CKMThL ji 1 gyZ uXc Iv4 03 3NqZ LITG Ux3ClP CB K O3v RUi mJq l5 blI9 GrWy irWHof lH 7 3ZT eZX kop eq 8XL1 RQ3a Uj6Ess nj 2 0MA 3As rSV ft 3F9w zB1q DQVOnH Cm m P3d WSb jst oj 3oGj advz qcMB6Y 6k D 9sZ 0bd Mjt UT hULG TWU9 Nmr3E4 CN b zUO vTh hqL 1p xAxT ezrH dVMgLY TT r Sfx LUX CMr WA bE69 K6XH i5re1f x4 G DKk iB7 f2D Xz Xez2 k2Yc Yc4QjU yM Y R1o DeY NWf 74 hByF dsWk 4cUbCR DX a q4e DWd 7qb Ot 7GOu oklg jJ00J9 Il O Jxn tzF VBC Ft pABp VLEE 2y5Qcg b3 5 DU4 igj 4dz zW soNF wvqj bNFma0 am F Kiv Aap pzM zr VqYf OulM HafaBk 6J r eOQ BaT EsJ BB tHXj n2EU CNleWp cv W JIg gWX Ksn B3 wvmo WK49 Nl492o gR 6 fvc 8ff jJm sW Jr0j zI9p CBsIUV of D kKH Ub7 vxp uQ UXA6 hMUr yvxEpc Tq l Tkz z0q HbX pO 8jFu h6nw zVPPzp A8 9 61V 78c O2W aw 0yGn CHVq BVjTUH lk p 6dG HOd voE E8 cw7Q DL1o 1qg5TX qo V 720 hhQ TyF tp TJDg 9E8D nsp1Qi X9 8 ZVQ N3s duZ qc n9IX ozWh Fd16IB 0K 9 JeB Hvi 364 kQ lFMM JOn0 OUBrnv pY y jUB Ofs Pzx l4 zcMn JHdq OjSi6N Mn 8 bR6 kPe klT Fd VlwD SrhT 8Qr0sC hN h 88j 8ZA vvW VD 03wt ETKK NUdr7W EK 1 jKS IHF Kh2 sr 1RRV Ra8J mBtkWI 1u k uZT F2B 4p8 E7 Y3p0 DX20 JM3XzQ tZ 3 bMC vM4 DEA wB Fp8q YKpL So1a5s dR P fTg 5R6 7v1 T4 eCJ1 qg14 CTK7u7 ag j Q0A tZ1 Nh6 hk Sys5 CWon IOqgCL 3u 7 feR BHz odS Jp 7JH8 u6Rw sYE0mc P4 r LaW Atl yRw kH F3ei UyhI iA19ZB u8 m ywf 42n uyX 0e ljCt 3Lkd 1eUQEZ oO Z rA2 Oqf oQ5 Ca hrBy KzFg DOseim 0j Y BmX csL Ayc cC JBTZ PEjy zPb5hZ KW O xT6 dyt u82 Ia htpD m75Y DktQvd Nj W jIQ H1B Ace SZ KVVP 136v L8XhMm 1O H Kn2 gUy kFU wN 8JML Bqmn vGuwGR oW U oNZ Y2P nmS 5g QMcR YHxL yHuDo8 ba w aqM NYt onW u2 YIOz eB6R wHuGcn fi o 47U PM5 tOj sz QBNq 7mco fCNjou 83 e mcY 81s vsI 2Y DS3S yloB Nx5FBV Bc 9 6HZ EOX UO3 W1 fIF5 jtEM W6KW7D 63 t H0F CVT Zup Pl A9aI oN2s f1Bw31 gg L FoD O0M x18 oo heEd KgZB Cqdqpa sa H Fhx BrE aRg Au I5dq mWWB MuHfv9 0y S PtG hFF dYJ JL f3Ap k5Ck Szr0Kb Vd i sQk uSA JEn DT YkjP AEMu a0VCtC Ff z 9R6 Vht 8Ua cB e7op AnGa 7AbLWj Hc s nAR GMb n7a 9n paMf lftM 7jvb20 0T W xUC 4lt e92 9j oZrA IuIa o1Zqdr oC L 55L T4Q 8kN yv sIzP x4i5 9lKTq2 JB B sZb QCEEQ54}   \end{align} where $\phi$ is an arbitrary function in $C_c^{\infty}(\RR^\dd)$, we verify that $u$ is a strong $L^p$ solution to~\eqref{ERTWERTHWRTWERTSGDGHCFGSDFGQSERWDFGDSFGHSDRGTEHDFGHDSFGSDGHGYUHDFGSDFASDFASGTWRT27}. The proof of pathwise uniqueness is identical to that in \cite{KXZ}, and \eqref{ERTWERTHWRTWERTSGDGHCFGSDFGQSERWDFGDSFGHSDRGTEHDFGHDSFGSDGHGYUHDFGSDFASDFASGTWRT29} is obtained by combining \eqref{ERTWERTHWRTWERTSGDGHCFGSDFGQSERWDFGDSFGHSDRGTEHDFGHDSFGSDGHGYUHDFGSDFASDFASGTWRT52} and Lemma~\ref{L04}. \par \end{proof} \par \startnewsection{Stochastic Truncated Navier-Stokes Equation}{sec5} In this section, we establish the global existence of a unique strong solution for a truncated SNSE model.  With $N\geq1$ fixed in this section, let $\varphi$ be a smooth function from $[0,\infty)$ to $[0,1]$ such that $\varphi\equiv 1$ on $[0,2N]$ and  $\varphi\equiv 0$ on~$[4N,\infty)$, with additionally   \begin{equation}    |\varphi( t_1)-\varphi( t_2)|    \leq    C |t_1-t_2|    \comma t_1,t_2\geq 0.    \label{ERTWERTHWRTWERTSGDGHCFGSDFGQSERWDFGDSFGHSDRGTEHDFGHDSFGSDGHGYUHDFGSDFASDFASGTWRT146}   \end{equation} We consider the SNSE truncated by this function and the convolution-type projector $P_{\le k}$ from \eqref{ERTWERTHWRTWERTSGDGHCFGSDFGQSERWDFGDSFGHSDRGTEHDFGHDSFGSDGHGYUHDFGSDFASDFASGTWRT07},   \begin{align}     \begin{split}       &\partial_t\uu( t,x)       =\Delta \uu( t,x)          - \varphi(\Vert\uu(t)\Vert_{p})^2P_{\le k}\mathcal{P}\bigl(( \uu( t,x)\cdot \nabla)P_{\le k}\uu( t,x)\bigr)       \\&\indeq\indeq\indeq\indeq\indeq          +\varphi(\Vert\uu(t)\Vert_{p})^2P_{\le k}\sigma(P_{\le k}\uu( t,x))\dot{\WW}(t),      \\&   \nabla\cdot \uu( t,x) = 0          ,           \\&       \uu( 0,x)=P_{\le k} \uu_0 (x)  \Pas    \commaone x\in\RR^3  .   \end{split}   \label{ERTWERTHWRTWERTSGDGHCFGSDFGQSERWDFGDSFGHSDRGTEHDFGHDSFGSDGHGYUHDFGSDFASDFASGTWRT56}   \end{align} \par \cole \begin{Theorem} \label{T03} Let $p> 2$, $\uu_0\in L^p(\Omega; L^p)$, and $\nabla\cdot u_0=0$. For every $T>0$, there exists a unique strong solution $\uu\in L^p(\Omega; C([0,T], L^p))$ to \eqref{ERTWERTHWRTWERTSGDGHCFGSDFGQSERWDFGDSFGHSDRGTEHDFGHDSFGSDGHGYUHDFGSDFASDFASGTWRT56} such that   \begin{align}   \begin{split}          \EE\biggl[\sup_{0\leq s\leq T}\Vert\uu(s,\cdot)\Vert_p^p              +\sum_{j}\int_0^{T} \int_{\RR^3} | \nabla (|\uu_j(s,x)|^{p/2})|^2 \,dx ds             \biggr]          \leq C_k(\EE\bigl[\Vert\uu_0\Vert_p^p\bigr]+C_{T}).        \end{split}       \label{ERTWERTHWRTWERTSGDGHCFGSDFGQSERWDFGDSFGHSDRGTEHDFGHDSFGSDGHGYUHDFGSDFASDFASGTWRT57}   \end{align} \end{Theorem} \colb \par In this section, we allow all constants to depend on~$N$. (We finally vary $N$ after $\eqref{ERTWERTHWRTWERTSGDGHCFGSDFGQSERWDFGDSFGHSDRGTEHDFGHDSFGSDGHGYUHDFGSDFASDFASGTWRT149}$ below). \par In order to solve \eqref{ERTWERTHWRTWERTSGDGHCFGSDFGQSERWDFGDSFGHSDRGTEHDFGHDSFGSDGHGYUHDFGSDFASDFASGTWRT56}, we  use the iteration  \begin{align} \begin{split} &\partial_t\un  -\Delta \un  = -\phin\phinm P_{\le k}\mathcal{P}\bigl(( \unm \cdot \nabla)P_{\le k}\unm   \bigr) \\&\indeq\indeq\indeq\indeq\indeq\indeq\indeq\indeq +\phin\phinm P_{\le k}\sigma(P_{\le k}\unm )\dot{\WW}(t),  \\& \nabla\cdot \un  = 0, \\&\un ( 0)= P_{\le k}\uu_0  \Pas \commaone x\in\RR^3 \comma t\in (0,T] , \end{split} \label{ERTWERTHWRTWERTSGDGHCFGSDFGQSERWDFGDSFGHSDRGTEHDFGHDSFGSDGHGYUHDFGSDFASDFASGTWRT58} \end{align} where $\uu^{(0)}$ is the strong solution to  \begin{align} \begin{split} &\partial_t\uu^{(0)} -\Delta \uu^{(0)} =0, \\& \nabla\cdot \uu^{(0)} = 0, \\& \uu^{(0)}( 0,x)= P_{\le k}\uu_0(x)    \Pas \commaone x\in\RR^3 \commaone t\in (0,T] \end{split} \label{ERTWERTHWRTWERTSGDGHCFGSDFGQSERWDFGDSFGHSDRGTEHDFGHDSFGSDGHGYUHDFGSDFASDFASGTWRT59} \end{align} For simplicity, we abbreviate $\varphi(\Vert\un\Vert_{p})$, $\varphi(\Vert\unm\Vert_{p})$, and $\varphi(\Vert v\Vert_{p})$ as $\phin$, $\phinm$, and $\varphi_v$, respectively.  Utilizing Theorem~\ref{T02}, we conclude that $\uu^{(0)}\in L^p(\Omega; C([0,T], L^p))$ and that $\uu^{(0)}$ satisfies \begin{equation} \EE\biggl[\sup_{0\leq t\leq T}\Vert\uu^{(0)}(t,\cdot)\Vert_p^p +\sum_{j}\int_0^{T} \int_{\RR^3} | \nabla (|\uu_j^{(0)}(t,x)|^{p/2})|^2 \,dx dt \biggr] \leq C\EE[\Vert\uu_0\Vert_p^p], \label{ERTWERTHWRTWERTSGDGHCFGSDFGQSERWDFGDSFGHSDRGTEHDFGHDSFGSDGHGYUHDFGSDFASDFASGTWRT60} \end{equation} where $C$ depends on $p$, $d$, and~$T$. We next show that there exists a unique solution $\un\in L^p(\Omega; C([0,T], L^p))$ to \eqref{ERTWERTHWRTWERTSGDGHCFGSDFGQSERWDFGDSFGHSDRGTEHDFGHDSFGSDGHGYUHDFGSDFASDFASGTWRT58} for every $n$, and the sequence $\{\un\}_{n\in\NNp}$ is uniformly bounded in a manner consistent with~\eqref{ERTWERTHWRTWERTSGDGHCFGSDFGQSERWDFGDSFGHSDRGTEHDFGHDSFGSDGHGYUHDFGSDFASDFASGTWRT60}.  \par \cole \begin{Lemma} \label{L05} Let $p>2$, $\uu_0\in L^p(\Omega; L^p)$, and $\nabla\cdot u_0=0$. Then  for every $n\in\mathbb{N}$ and every $T>0$, the initial value problem \eqref{ERTWERTHWRTWERTSGDGHCFGSDFGQSERWDFGDSFGHSDRGTEHDFGHDSFGSDGHGYUHDFGSDFASDFASGTWRT58} has a unique strong solution $\un\in L^p(\Omega; C([0, T], L^p))$, and $\un$ satisfies~\eqref{ERTWERTHWRTWERTSGDGHCFGSDFGQSERWDFGDSFGHSDRGTEHDFGHDSFGSDGHGYUHDFGSDFASDFASGTWRT57}. \end{Lemma}
\colb \par \begin{proof}[Proof of Lemma~\ref{L05}]  Denote $\unm$, which is assumed to be already constructed in \eqref{ERTWERTHWRTWERTSGDGHCFGSDFGQSERWDFGDSFGHSDRGTEHDFGHDSFGSDGHGYUHDFGSDFASDFASGTWRT58}, by $v$ and consider the system   \begin{align}     \begin{split}       &\partial_t\uu( t,x)       -\Delta \uu( t,x)       =          - \varphi(\Vert\uu(t)\Vert_{p})  \varphi(\Vert\vv(t)\Vert_{p})           P_{\le k}\mathcal{P}\bigl(( \vv( t,x)\cdot \nabla)P_{\le k}\vv( t,x)\bigr)       \\&\indeq\indeq\indeq\indeq\indeq          +\varphi(\Vert\uu(t)\Vert_{p})\varphi(\Vert\vv(t)\Vert_{p})          P_{\le k}\sigma(P_{\le k}\vv( t,x))\dot{\WW}(t),      \\&   \nabla\cdot \uu( t,x) = 0          ,           \\&       \uu( 0,x)=P_{\le k} \uu_0 (x)  \Pas    \commaone x\in\RR^3  ,   \end{split}    \label{ERTWERTHWRTWERTSGDGHCFGSDFGQSERWDFGDSFGHSDRGTEHDFGHDSFGSDGHGYUHDFGSDFASDFASGTWRT163}   \end{align} which we solve by the iteration   \begin{align}   \begin{split}   &\partial_t \un   -\Delta \un     = -\phinm\varphi_v P_{\le k}\mathcal{P}\bigl((v\cdot \nabla) (P_{\le k}v)\bigr)   +\phinm\varphi_v P_{\le k}(\sigma(P_{\le k}v))\dot{\WW}(t),   \\&   \nabla\cdot \un = 0,   \\&\un( 0)= P_{\le k}\uu_0    \Pas   \commaone x\in\RR^3    .   \end{split}   \label{ERTWERTHWRTWERTSGDGHCFGSDFGQSERWDFGDSFGHSDRGTEHDFGHDSFGSDGHGYUHDFGSDFASDFASGTWRT61}   \end{align} Thus we assume that $v$ is divergence-free and satisfies \eqref{ERTWERTHWRTWERTSGDGHCFGSDFGQSERWDFGDSFGHSDRGTEHDFGHDSFGSDGHGYUHDFGSDFASDFASGTWRT57}; note that $\un$ in \eqref{ERTWERTHWRTWERTSGDGHCFGSDFGQSERWDFGDSFGHSDRGTEHDFGHDSFGSDGHGYUHDFGSDFASDFASGTWRT61} is not the same as in~\eqref{ERTWERTHWRTWERTSGDGHCFGSDFGQSERWDFGDSFGHSDRGTEHDFGHDSFGSDGHGYUHDFGSDFASDFASGTWRT58}.  Now we would like to apply Theorem~\ref{T02} to \eqref{ERTWERTHWRTWERTSGDGHCFGSDFGQSERWDFGDSFGHSDRGTEHDFGHDSFGSDGHGYUHDFGSDFASDFASGTWRT61}. Note that projectors $P_{\le k}$ and $\mathcal{P}$ commute with differentiation. Hence,    \begin{equation}   -\phinm\varphi_v P_{\le k}\mathcal{P}\bigl((v\cdot \nabla) (P_{\le k}v)\bigr)=-\phinm \phiv \partial_i P_{\le k} \mathcal{P}\bigl(v_i P_{\le k} v\bigr).    \llabel{t 3F9w zB1q DQVOnH Cm m P3d WSb jst oj 3oGj advz qcMB6Y 6k D 9sZ 0bd Mjt UT hULG TWU9 Nmr3E4 CN b zUO vTh hqL 1p xAxT ezrH dVMgLY TT r Sfx LUX CMr WA bE69 K6XH i5re1f x4 G DKk iB7 f2D Xz Xez2 k2Yc Yc4QjU yM Y R1o DeY NWf 74 hByF dsWk 4cUbCR DX a q4e DWd 7qb Ot 7GOu oklg jJ00J9 Il O Jxn tzF VBC Ft pABp VLEE 2y5Qcg b3 5 DU4 igj 4dz zW soNF wvqj bNFma0 am F Kiv Aap pzM zr VqYf OulM HafaBk 6J r eOQ BaT EsJ BB tHXj n2EU CNleWp cv W JIg gWX Ksn B3 wvmo WK49 Nl492o gR 6 fvc 8ff jJm sW Jr0j zI9p CBsIUV of D kKH Ub7 vxp uQ UXA6 hMUr yvxEpc Tq l Tkz z0q HbX pO 8jFu h6nw zVPPzp A8 9 61V 78c O2W aw 0yGn CHVq BVjTUH lk p 6dG HOd voE E8 cw7Q DL1o 1qg5TX qo V 720 hhQ TyF tp TJDg 9E8D nsp1Qi X9 8 ZVQ N3s duZ qc n9IX ozWh Fd16IB 0K 9 JeB Hvi 364 kQ lFMM JOn0 OUBrnv pY y jUB Ofs Pzx l4 zcMn JHdq OjSi6N Mn 8 bR6 kPe klT Fd VlwD SrhT 8Qr0sC hN h 88j 8ZA vvW VD 03wt ETKK NUdr7W EK 1 jKS IHF Kh2 sr 1RRV Ra8J mBtkWI 1u k uZT F2B 4p8 E7 Y3p0 DX20 JM3XzQ tZ 3 bMC vM4 DEA wB Fp8q YKpL So1a5s dR P fTg 5R6 7v1 T4 eCJ1 qg14 CTK7u7 ag j Q0A tZ1 Nh6 hk Sys5 CWon IOqgCL 3u 7 feR BHz odS Jp 7JH8 u6Rw sYE0mc P4 r LaW Atl yRw kH F3ei UyhI iA19ZB u8 m ywf 42n uyX 0e ljCt 3Lkd 1eUQEZ oO Z rA2 Oqf oQ5 Ca hrBy KzFg DOseim 0j Y BmX csL Ayc cC JBTZ PEjy zPb5hZ KW O xT6 dyt u82 Ia htpD m75Y DktQvd Nj W jIQ H1B Ace SZ KVVP 136v L8XhMm 1O H Kn2 gUy kFU wN 8JML Bqmn vGuwGR oW U oNZ Y2P nmS 5g QMcR YHxL yHuDo8 ba w aqM NYt onW u2 YIOz eB6R wHuGcn fi o 47U PM5 tOj sz QBNq 7mco fCNjou 83 e mcY 81s vsI 2Y DS3S yloB Nx5FBV Bc 9 6HZ EOX UO3 W1 fIF5 jtEM W6KW7D 63 t H0F CVT Zup Pl A9aI oN2s f1Bw31 gg L FoD O0M x18 oo heEd KgZB Cqdqpa sa H Fhx BrE aRg Au I5dq mWWB MuHfv9 0y S PtG hFF dYJ JL f3Ap k5Ck Szr0Kb Vd i sQk uSA JEn DT YkjP AEMu a0VCtC Ff z 9R6 Vht 8Ua cB e7op AnGa 7AbLWj Hc s nAR GMb n7a 9n paMf lftM 7jvb20 0T W xUC 4lt e92 9j oZrA IuIa o1Zqdr oC L 55L T4Q 8kN yv sIzP x4i5 9lKTq2 JB B sZb QCE Ctw ar VBMT H1QR 6v5srW hR r D4r wf8 ik7 KH Egee rFVT ErONml Q5 L R8v XNZ LB3 9U DzRH ZbH9 fTBhRw kA 2 n3p g4I grH xd fEFu z6RE tDqPdw N7 H TVt cE1 8hW 6y n4Gn nCE3 MEQ51EQ62}   \end{equation} By the properties \eqref{ERTWERTHWRTWERTSGDGHCFGSDFGQSERWDFGDSFGHSDRGTEHDFGHDSFGSDGHGYUHDFGSDFASDFASGTWRT08} and \eqref{ERTWERTHWRTWERTSGDGHCFGSDFGQSERWDFGDSFGHSDRGTEHDFGHDSFGSDGHGYUHDFGSDFASDFASGTWRT09} of $P_{\leq k}$, we have   \begin{align}   \begin{split}   &   \EE\biggl[\int_0^{T}\Vert    \phinm\varphi_v P_{\leq k}\mathcal{P}(v_i P_{\leq k}v)   \Vert_{q}^p \,ds   \biggr]   \leq   C_k \EE\biggl[\int_0^{T}   \phinm \phiv   \Vert    v_i P_{\leq k} v   \Vert_{r}^p    \,ds   \biggr]   \\&\indeq   \leq   C_k \EE\biggl[\int_0^{T}   \phinm \phiv   \Vert    v_i    \Vert_{p}^p    \Vert    P_{\leq k} v   \Vert_{l}^p    \,ds   \biggr]   \leq   C_k \EE\biggl[\int_0^{T}   \phinm \phiv   \Vert    P_{\leq k} v   \Vert_{l}^p    \,ds   \biggr]   \\&\indeq   \leq   C_k \EE\biggl[\int_0^{T}   \phinm \phiv   \Vert    v   \Vert_{p}^p    \,ds   \biggr]   \leq   C_{k,T}   \comma i=1,2,3,   \end{split}   \label{ERTWERTHWRTWERTSGDGHCFGSDFGQSERWDFGDSFGHSDRGTEHDFGHDSFGSDGHGYUHDFGSDFASDFASGTWRT63}   \end{align} where we used $\varphi_v \Vert v\Vert_{p}\leq C$ in the third inequality. Above, we required    \begin{equation}   r< q   \commaone   \frac{1}{p}+ \frac{1}{l}=\frac{1}{r},   \qquad{} \text{and}\qquad{}   l\geq p,   \label{ERTWERTHWRTWERTSGDGHCFGSDFGQSERWDFGDSFGHSDRGTEHDFGHDSFGSDGHGYUHDFGSDFASDFASGTWRT64}   \end{equation} which leads to $p/2\leq r<q\leq p$ under the condition~\eqref{ERTWERTHWRTWERTSGDGHCFGSDFGQSERWDFGDSFGHSDRGTEHDFGHDSFGSDGHGYUHDFGSDFASDFASGTWRT28}. Using \eqref{ERTWERTHWRTWERTSGDGHCFGSDFGQSERWDFGDSFGHSDRGTEHDFGHDSFGSDGHGYUHDFGSDFASDFASGTWRT09}, \eqref{ERTWERTHWRTWERTSGDGHCFGSDFGQSERWDFGDSFGHSDRGTEHDFGHDSFGSDGHGYUHDFGSDFASDFASGTWRT10}, and the growth assumption~\eqref{ERTWERTHWRTWERTSGDGHCFGSDFGQSERWDFGDSFGHSDRGTEHDFGHDSFGSDGHGYUHDFGSDFASDFASGTWRT03} on the noise, we conclude for the last term in \eqref{ERTWERTHWRTWERTSGDGHCFGSDFGQSERWDFGDSFGHSDRGTEHDFGHDSFGSDGHGYUHDFGSDFASDFASGTWRT61},    \begin{align}   \begin{split}   &\EE\left[   \int_0^{T}\int_{\RR^\dd}   \Vert    \phinm   \phiv P_{\le k}(\sigma(P_{\le k} v))   \Vert_{l^2( \mathcal{H},\mathcal{\RR}^\dd)}^p\,dx ds   \right]   \leq    C \EE\left[   \int_0^{T}   \phiv   (\Vert    P_{\le k} v   \Vert_{(3p/2)-}^{2p}+1)   \,ds   \right]   \\&\indeq   \leq    C_k \EE\left[   \int_0^{T}   \phiv   \bigl(\Vert    v   \Vert_{p}^{2p}+1   \bigr)   \,ds   \right]   \leq   C_{k,T}   .   \end{split}   \label{ERTWERTHWRTWERTSGDGHCFGSDFGQSERWDFGDSFGHSDRGTEHDFGHDSFGSDGHGYUHDFGSDFASDFASGTWRT65}   \end{align} Applying Theorem~\ref{T02} using \eqref{ERTWERTHWRTWERTSGDGHCFGSDFGQSERWDFGDSFGHSDRGTEHDFGHDSFGSDGHGYUHDFGSDFASDFASGTWRT63} and \eqref{ERTWERTHWRTWERTSGDGHCFGSDFGQSERWDFGDSFGHSDRGTEHDFGHDSFGSDGHGYUHDFGSDFASDFASGTWRT65} proves that \eqref{ERTWERTHWRTWERTSGDGHCFGSDFGQSERWDFGDSFGHSDRGTEHDFGHDSFGSDGHGYUHDFGSDFASDFASGTWRT163} has a unique strong solution in $L^p(\Omega; C([0, T], L^p))$ for every~$n\in\NNp$.  Thus, the iteration \eqref{ERTWERTHWRTWERTSGDGHCFGSDFGQSERWDFGDSFGHSDRGTEHDFGHDSFGSDGHGYUHDFGSDFASDFASGTWRT61} is well-defined.  \par Next, we employ the fixed-point argument on~\eqref{ERTWERTHWRTWERTSGDGHCFGSDFGQSERWDFGDSFGHSDRGTEHDFGHDSFGSDGHGYUHDFGSDFASDFASGTWRT61}. Denote  ${\zz}^{(n)}=\unp-\un$. Then   \begin{align}    \begin{split}    &\partial_t {\zz}_j^{(n)}    -\Delta {\zz}_j^{(n)}     =     \sum_{i}\partial_i f_{ij}      + g_j \dot{\WW}(t)    \comma j=1,2,3    ,   \end{split}    \label{ERTWERTHWRTWERTSGDGHCFGSDFGQSERWDFGDSFGHSDRGTEHDFGHDSFGSDGHGYUHDFGSDFASDFASGTWRT66}   \end{align} where $f_{ij} = - (\phin - \phinm) \phiv P_{\le k} ({\mathcal P}(v_i  P_{\le k}v))_j$ and $   g_j = \bigl(\phin-\phinm\bigr) \phiv P_{\le k} \sigma_j(P_{\le k} v)$. In addition, we have $\nabla\cdot {\zz}^{(n)}  = 0$ and ${\zz}^{(n)}( 0)= 0\textPas$  With the same choice of $r$ and $l$ as in \eqref{ERTWERTHWRTWERTSGDGHCFGSDFGQSERWDFGDSFGHSDRGTEHDFGHDSFGSDGHGYUHDFGSDFASDFASGTWRT63}--\eqref{ERTWERTHWRTWERTSGDGHCFGSDFGQSERWDFGDSFGHSDRGTEHDFGHDSFGSDGHGYUHDFGSDFASDFASGTWRT64}, we obtain   \begin{align}   \begin{split}   &\EE\left[\int_0^{t}\Vert                 f                \Vert_{q}^p \,ds       \right]   \leq    C        \EE\left[\int_0^{t}                \phiv^{p}                \Bigl|\Vert\un\Vert_{p}-\Vert\unm\Vert_{p}\Bigr|^{p}                \Vert                 v                \Vert_{p}^p                 \Vert                 P_{\le k}v                \Vert_{l}^p              \,ds        \right]    \\& \indeq     \leq    C    \EE\left[\int_0^{t}    \phiv^{p}    \Vert \znm\Vert_{p}^{p}    \Vert     v    \Vert_{p}^p     \Vert     P_{\le k}v    \Vert_{l}^p    \,ds    \right]      \leq      C_k        \EE\left[\int_0^{t}                \Vert \znm\Vert_{p}^{p}              \,ds        \right].   \end{split}    \llabel{ DKk iB7 f2D Xz Xez2 k2Yc Yc4QjU yM Y R1o DeY NWf 74 hByF dsWk 4cUbCR DX a q4e DWd 7qb Ot 7GOu oklg jJ00J9 Il O Jxn tzF VBC Ft pABp VLEE 2y5Qcg b3 5 DU4 igj 4dz zW soNF wvqj bNFma0 am F Kiv Aap pzM zr VqYf OulM HafaBk 6J r eOQ BaT EsJ BB tHXj n2EU CNleWp cv W JIg gWX Ksn B3 wvmo WK49 Nl492o gR 6 fvc 8ff jJm sW Jr0j zI9p CBsIUV of D kKH Ub7 vxp uQ UXA6 hMUr yvxEpc Tq l Tkz z0q HbX pO 8jFu h6nw zVPPzp A8 9 61V 78c O2W aw 0yGn CHVq BVjTUH lk p 6dG HOd voE E8 cw7Q DL1o 1qg5TX qo V 720 hhQ TyF tp TJDg 9E8D nsp1Qi X9 8 ZVQ N3s duZ qc n9IX ozWh Fd16IB 0K 9 JeB Hvi 364 kQ lFMM JOn0 OUBrnv pY y jUB Ofs Pzx l4 zcMn JHdq OjSi6N Mn 8 bR6 kPe klT Fd VlwD SrhT 8Qr0sC hN h 88j 8ZA vvW VD 03wt ETKK NUdr7W EK 1 jKS IHF Kh2 sr 1RRV Ra8J mBtkWI 1u k uZT F2B 4p8 E7 Y3p0 DX20 JM3XzQ tZ 3 bMC vM4 DEA wB Fp8q YKpL So1a5s dR P fTg 5R6 7v1 T4 eCJ1 qg14 CTK7u7 ag j Q0A tZ1 Nh6 hk Sys5 CWon IOqgCL 3u 7 feR BHz odS Jp 7JH8 u6Rw sYE0mc P4 r LaW Atl yRw kH F3ei UyhI iA19ZB u8 m ywf 42n uyX 0e ljCt 3Lkd 1eUQEZ oO Z rA2 Oqf oQ5 Ca hrBy KzFg DOseim 0j Y BmX csL Ayc cC JBTZ PEjy zPb5hZ KW O xT6 dyt u82 Ia htpD m75Y DktQvd Nj W jIQ H1B Ace SZ KVVP 136v L8XhMm 1O H Kn2 gUy kFU wN 8JML Bqmn vGuwGR oW U oNZ Y2P nmS 5g QMcR YHxL yHuDo8 ba w aqM NYt onW u2 YIOz eB6R wHuGcn fi o 47U PM5 tOj sz QBNq 7mco fCNjou 83 e mcY 81s vsI 2Y DS3S yloB Nx5FBV Bc 9 6HZ EOX UO3 W1 fIF5 jtEM W6KW7D 63 t H0F CVT Zup Pl A9aI oN2s f1Bw31 gg L FoD O0M x18 oo heEd KgZB Cqdqpa sa H Fhx BrE aRg Au I5dq mWWB MuHfv9 0y S PtG hFF dYJ JL f3Ap k5Ck Szr0Kb Vd i sQk uSA JEn DT YkjP AEMu a0VCtC Ff z 9R6 Vht 8Ua cB e7op AnGa 7AbLWj Hc s nAR GMb n7a 9n paMf lftM 7jvb20 0T W xUC 4lt e92 9j oZrA IuIa o1Zqdr oC L 55L T4Q 8kN yv sIzP x4i5 9lKTq2 JB B sZb QCE Ctw ar VBMT H1QR 6v5srW hR r D4r wf8 ik7 KH Egee rFVT ErONml Q5 L R8v XNZ LB3 9U DzRH ZbH9 fTBhRw kA 2 n3p g4I grH xd fEFu z6RE tDqPdw N7 H TVt cE1 8hW 6y n4Gn nCE3 MEQ51i Ps G Z2G Lbt CSt hu zvPF eE28 MM23ug TC d j7z 7Av TLa 1A GLiJ 5JwW CiDPyM qa 8 tAK QZ9 cfP 42 kuUz V3h6 GsGFoW m9 h cfj 51d GtW yZ zC5D aVt2 Wi5IIs gD B 0cX LM1 FtE xE REQ67}   \end{align} Similarly, for the last term in \eqref{ERTWERTHWRTWERTSGDGHCFGSDFGQSERWDFGDSFGHSDRGTEHDFGHDSFGSDGHGYUHDFGSDFASDFASGTWRT66}, we have   \begin{align}    \begin{split}       &\EE\left[         \int_0^{t}\int_{\RR^\dd}\Vert g(s,x)\Vert_{l^2(             \mathcal{H},\mathcal{\RR}^\dd)}^p\,dx ds             \right]       \leq         C_{k}         \EE\left[         \int_0^{t}\Vert \znm\Vert_{p}^p\,ds             \right]     .    \end{split}    \llabel{qj bNFma0 am F Kiv Aap pzM zr VqYf OulM HafaBk 6J r eOQ BaT EsJ BB tHXj n2EU CNleWp cv W JIg gWX Ksn B3 wvmo WK49 Nl492o gR 6 fvc 8ff jJm sW Jr0j zI9p CBsIUV of D kKH Ub7 vxp uQ UXA6 hMUr yvxEpc Tq l Tkz z0q HbX pO 8jFu h6nw zVPPzp A8 9 61V 78c O2W aw 0yGn CHVq BVjTUH lk p 6dG HOd voE E8 cw7Q DL1o 1qg5TX qo V 720 hhQ TyF tp TJDg 9E8D nsp1Qi X9 8 ZVQ N3s duZ qc n9IX ozWh Fd16IB 0K 9 JeB Hvi 364 kQ lFMM JOn0 OUBrnv pY y jUB Ofs Pzx l4 zcMn JHdq OjSi6N Mn 8 bR6 kPe klT Fd VlwD SrhT 8Qr0sC hN h 88j 8ZA vvW VD 03wt ETKK NUdr7W EK 1 jKS IHF Kh2 sr 1RRV Ra8J mBtkWI 1u k uZT F2B 4p8 E7 Y3p0 DX20 JM3XzQ tZ 3 bMC vM4 DEA wB Fp8q YKpL So1a5s dR P fTg 5R6 7v1 T4 eCJ1 qg14 CTK7u7 ag j Q0A tZ1 Nh6 hk Sys5 CWon IOqgCL 3u 7 feR BHz odS Jp 7JH8 u6Rw sYE0mc P4 r LaW Atl yRw kH F3ei UyhI iA19ZB u8 m ywf 42n uyX 0e ljCt 3Lkd 1eUQEZ oO Z rA2 Oqf oQ5 Ca hrBy KzFg DOseim 0j Y BmX csL Ayc cC JBTZ PEjy zPb5hZ KW O xT6 dyt u82 Ia htpD m75Y DktQvd Nj W jIQ H1B Ace SZ KVVP 136v L8XhMm 1O H Kn2 gUy kFU wN 8JML Bqmn vGuwGR oW U oNZ Y2P nmS 5g QMcR YHxL yHuDo8 ba w aqM NYt onW u2 YIOz eB6R wHuGcn fi o 47U PM5 tOj sz QBNq 7mco fCNjou 83 e mcY 81s vsI 2Y DS3S yloB Nx5FBV Bc 9 6HZ EOX UO3 W1 fIF5 jtEM W6KW7D 63 t H0F CVT Zup Pl A9aI oN2s f1Bw31 gg L FoD O0M x18 oo heEd KgZB Cqdqpa sa H Fhx BrE aRg Au I5dq mWWB MuHfv9 0y S PtG hFF dYJ JL f3Ap k5Ck Szr0Kb Vd i sQk uSA JEn DT YkjP AEMu a0VCtC Ff z 9R6 Vht 8Ua cB e7op AnGa 7AbLWj Hc s nAR GMb n7a 9n paMf lftM 7jvb20 0T W xUC 4lt e92 9j oZrA IuIa o1Zqdr oC L 55L T4Q 8kN yv sIzP x4i5 9lKTq2 JB B sZb QCE Ctw ar VBMT H1QR 6v5srW hR r D4r wf8 ik7 KH Egee rFVT ErONml Q5 L R8v XNZ LB3 9U DzRH ZbH9 fTBhRw kA 2 n3p g4I grH xd fEFu z6RE tDqPdw N7 H TVt cE1 8hW 6y n4Gn nCE3 MEQ51i Ps G Z2G Lbt CSt hu zvPF eE28 MM23ug TC d j7z 7Av TLa 1A GLiJ 5JwW CiDPyM qa 8 tAK QZ9 cfP 42 kuUz V3h6 GsGFoW m9 h cfj 51d GtW yZ zC5D aVt2 Wi5IIs gD B 0cX LM1 FtE xE RIZI Z0Rt QUtWcU Cm F mSj xvW pZc gl dopk 0D7a EouRku Id O ZdW FOR uqb PY 6HkW OVi7 FuVMLW nx p SaN omk rC5 uI ZK9C jpJy UIeO6k gb 7 tr2 SCY x5F 11 S6Xq OImr s7vv0u vA g rbEQ68}   \end{align} By Theorem~\ref{T02}, we obtain   \begin{equation}   \EE\biggl[\sup_{s\in[0,t]}\Vert{\zz}^{(n)}\Vert_p^p\biggr] \leq   C_k \,t\,   \EE\biggl[\sup_{s\in[0,t]} \Vert {\zz}^{(n-1)}\Vert_{p}^{p} \biggr],    \llabel{vxp uQ UXA6 hMUr yvxEpc Tq l Tkz z0q HbX pO 8jFu h6nw zVPPzp A8 9 61V 78c O2W aw 0yGn CHVq BVjTUH lk p 6dG HOd voE E8 cw7Q DL1o 1qg5TX qo V 720 hhQ TyF tp TJDg 9E8D nsp1Qi X9 8 ZVQ N3s duZ qc n9IX ozWh Fd16IB 0K 9 JeB Hvi 364 kQ lFMM JOn0 OUBrnv pY y jUB Ofs Pzx l4 zcMn JHdq OjSi6N Mn 8 bR6 kPe klT Fd VlwD SrhT 8Qr0sC hN h 88j 8ZA vvW VD 03wt ETKK NUdr7W EK 1 jKS IHF Kh2 sr 1RRV Ra8J mBtkWI 1u k uZT F2B 4p8 E7 Y3p0 DX20 JM3XzQ tZ 3 bMC vM4 DEA wB Fp8q YKpL So1a5s dR P fTg 5R6 7v1 T4 eCJ1 qg14 CTK7u7 ag j Q0A tZ1 Nh6 hk Sys5 CWon IOqgCL 3u 7 feR BHz odS Jp 7JH8 u6Rw sYE0mc P4 r LaW Atl yRw kH F3ei UyhI iA19ZB u8 m ywf 42n uyX 0e ljCt 3Lkd 1eUQEZ oO Z rA2 Oqf oQ5 Ca hrBy KzFg DOseim 0j Y BmX csL Ayc cC JBTZ PEjy zPb5hZ KW O xT6 dyt u82 Ia htpD m75Y DktQvd Nj W jIQ H1B Ace SZ KVVP 136v L8XhMm 1O H Kn2 gUy kFU wN 8JML Bqmn vGuwGR oW U oNZ Y2P nmS 5g QMcR YHxL yHuDo8 ba w aqM NYt onW u2 YIOz eB6R wHuGcn fi o 47U PM5 tOj sz QBNq 7mco fCNjou 83 e mcY 81s vsI 2Y DS3S yloB Nx5FBV Bc 9 6HZ EOX UO3 W1 fIF5 jtEM W6KW7D 63 t H0F CVT Zup Pl A9aI oN2s f1Bw31 gg L FoD O0M x18 oo heEd KgZB Cqdqpa sa H Fhx BrE aRg Au I5dq mWWB MuHfv9 0y S PtG hFF dYJ JL f3Ap k5Ck Szr0Kb Vd i sQk uSA JEn DT YkjP AEMu a0VCtC Ff z 9R6 Vht 8Ua cB e7op AnGa 7AbLWj Hc s nAR GMb n7a 9n paMf lftM 7jvb20 0T W xUC 4lt e92 9j oZrA IuIa o1Zqdr oC L 55L T4Q 8kN yv sIzP x4i5 9lKTq2 JB B sZb QCE Ctw ar VBMT H1QR 6v5srW hR r D4r wf8 ik7 KH Egee rFVT ErONml Q5 L R8v XNZ LB3 9U DzRH ZbH9 fTBhRw kA 2 n3p g4I grH xd fEFu z6RE tDqPdw N7 H TVt cE1 8hW 6y n4Gn nCE3 MEQ51i Ps G Z2G Lbt CSt hu zvPF eE28 MM23ug TC d j7z 7Av TLa 1A GLiJ 5JwW CiDPyM qa 8 tAK QZ9 cfP 42 kuUz V3h6 GsGFoW m9 h cfj 51d GtW yZ zC5D aVt2 Wi5IIs gD B 0cX LM1 FtE xE RIZI Z0Rt QUtWcU Cm F mSj xvW pZc gl dopk 0D7a EouRku Id O ZdW FOR uqb PY 6HkW OVi7 FuVMLW nx p SaN omk rC5 uI ZK9C jpJy UIeO6k gb 7 tr2 SCY x5F 11 S6Xq OImr s7vv0u vA g rb9 hGP Fnk RM j92H gczJ 660kHb BB l QSI OY7 FcX 0c uyDl LjbU 3F6vZk Gb a KaM ufj uxp n4 Mi45 7MoL NW3eIm cj 6 OOS e59 afA hg lt9S BOiF cYQipj 5u N 19N KZ5 Czc 23 1wxG x1ut EQ71}   \end{equation} which concludes the existence of a fixed-point of \eqref{ERTWERTHWRTWERTSGDGHCFGSDFGQSERWDFGDSFGHSDRGTEHDFGHDSFGSDGHGYUHDFGSDFASDFASGTWRT61} in $L^p_{\omega}L^{\infty}_t L_x^p$ provided~$t\leq 1/2C_k$. Note that the convergence to the fixed-point is exponentially rapid. With this rate of convergence, we conclude that $\varphi(\Vert\uu^{(n)}(t)\Vert_{p})\rightarrow \varphi(\Vert\uu(t)\Vert_{p})$ for a.e.-$(\omega, t)$ (see~\cite[Lemma~5.2]{KXZ}). Then following the same procedure in Theorem~\ref{T02}, we can show that $\uu$ is  indeed a solution to \eqref{ERTWERTHWRTWERTSGDGHCFGSDFGQSERWDFGDSFGHSDRGTEHDFGHDSFGSDGHGYUHDFGSDFASDFASGTWRT61}, and \eqref{ERTWERTHWRTWERTSGDGHCFGSDFGQSERWDFGDSFGHSDRGTEHDFGHDSFGSDGHGYUHDFGSDFASDFASGTWRT57} holds up to~$t$.  \par The pathwise uniqueness on a small time interval can be proven similarly. Now, let $[0,t]$ be an interval in which both existence and uniqueness hold. It is important to note that $t$ does not depend on the size of the initial data. Hence, the existence and uniqueness of a strong solution can be extended from $[0,t]$ to $[0,T]$ in finitely many steps. Clearly, \eqref{ERTWERTHWRTWERTSGDGHCFGSDFGQSERWDFGDSFGHSDRGTEHDFGHDSFGSDGHGYUHDFGSDFASDFASGTWRT57} holds up to~$T$. \end{proof} \par \begin{proof}[Proof of Theorem~\ref{T03}] Lemma~\ref{L05} shows that for every $n\in~\NNp$ there exists an iterate $\un$ solving \eqref{ERTWERTHWRTWERTSGDGHCFGSDFGQSERWDFGDSFGHSDRGTEHDFGHDSFGSDGHGYUHDFGSDFASDFASGTWRT58} in $[0,T]\times {\mathbb R}^{3}$, which in addition obeys~\eqref{ERTWERTHWRTWERTSGDGHCFGSDFGQSERWDFGDSFGHSDRGTEHDFGHDSFGSDGHGYUHDFGSDFASDFASGTWRT57}.  In order to obtain the convergence of the iterates, we consider the difference   \begin{equation}    \vn=\unp-\un    ,    \llabel{ X9 8 ZVQ N3s duZ qc n9IX ozWh Fd16IB 0K 9 JeB Hvi 364 kQ lFMM JOn0 OUBrnv pY y jUB Ofs Pzx l4 zcMn JHdq OjSi6N Mn 8 bR6 kPe klT Fd VlwD SrhT 8Qr0sC hN h 88j 8ZA vvW VD 03wt ETKK NUdr7W EK 1 jKS IHF Kh2 sr 1RRV Ra8J mBtkWI 1u k uZT F2B 4p8 E7 Y3p0 DX20 JM3XzQ tZ 3 bMC vM4 DEA wB Fp8q YKpL So1a5s dR P fTg 5R6 7v1 T4 eCJ1 qg14 CTK7u7 ag j Q0A tZ1 Nh6 hk Sys5 CWon IOqgCL 3u 7 feR BHz odS Jp 7JH8 u6Rw sYE0mc P4 r LaW Atl yRw kH F3ei UyhI iA19ZB u8 m ywf 42n uyX 0e ljCt 3Lkd 1eUQEZ oO Z rA2 Oqf oQ5 Ca hrBy KzFg DOseim 0j Y BmX csL Ayc cC JBTZ PEjy zPb5hZ KW O xT6 dyt u82 Ia htpD m75Y DktQvd Nj W jIQ H1B Ace SZ KVVP 136v L8XhMm 1O H Kn2 gUy kFU wN 8JML Bqmn vGuwGR oW U oNZ Y2P nmS 5g QMcR YHxL yHuDo8 ba w aqM NYt onW u2 YIOz eB6R wHuGcn fi o 47U PM5 tOj sz QBNq 7mco fCNjou 83 e mcY 81s vsI 2Y DS3S yloB Nx5FBV Bc 9 6HZ EOX UO3 W1 fIF5 jtEM W6KW7D 63 t H0F CVT Zup Pl A9aI oN2s f1Bw31 gg L FoD O0M x18 oo heEd KgZB Cqdqpa sa H Fhx BrE aRg Au I5dq mWWB MuHfv9 0y S PtG hFF dYJ JL f3Ap k5Ck Szr0Kb Vd i sQk uSA JEn DT YkjP AEMu a0VCtC Ff z 9R6 Vht 8Ua cB e7op AnGa 7AbLWj Hc s nAR GMb n7a 9n paMf lftM 7jvb20 0T W xUC 4lt e92 9j oZrA IuIa o1Zqdr oC L 55L T4Q 8kN yv sIzP x4i5 9lKTq2 JB B sZb QCE Ctw ar VBMT H1QR 6v5srW hR r D4r wf8 ik7 KH Egee rFVT ErONml Q5 L R8v XNZ LB3 9U DzRH ZbH9 fTBhRw kA 2 n3p g4I grH xd fEFu z6RE tDqPdw N7 H TVt cE1 8hW 6y n4Gn nCE3 MEQ51i Ps G Z2G Lbt CSt hu zvPF eE28 MM23ug TC d j7z 7Av TLa 1A GLiJ 5JwW CiDPyM qa 8 tAK QZ9 cfP 42 kuUz V3h6 GsGFoW m9 h cfj 51d GtW yZ zC5D aVt2 Wi5IIs gD B 0cX LM1 FtE xE RIZI Z0Rt QUtWcU Cm F mSj xvW pZc gl dopk 0D7a EouRku Id O ZdW FOR uqb PY 6HkW OVi7 FuVMLW nx p SaN omk rC5 uI ZK9C jpJy UIeO6k gb 7 tr2 SCY x5F 11 S6Xq OImr s7vv0u vA g rb9 hGP Fnk RM j92H gczJ 660kHb BB l QSI OY7 FcX 0c uyDl LjbU 3F6vZk Gb a KaM ufj uxp n4 Mi45 7MoL NW3eIm cj 6 OOS e59 afA hg lt9S BOiF cYQipj 5u N 19N KZ5 Czc 23 1wxG x1ut gJB4ue Mx x 5lr s8g VbZ s1 NEfI 02Rb pkfEOZ E4 e seo 9te NRU Ai nujf eJYa Ehns0Y 6X R UF1 PCf 5eE AL 9DL6 a2vm BAU5Au DD t yQN 5YL LWw PW GjMt 4hu4 FIoLCZ Lx e BVY 5lZ DCDEQ71}   \end{equation} which satisfies   \begin{align}    \begin{split}    &\partial_t\vn    -\Delta \vn +    \varphi^{(n+1)}\phin P_{\le k}\mathcal{P}\bigl(( \un \cdot \nabla)P_{\le k}\un \bigr)-\phin \phinm P_{\le k}\mathcal{P}           \bigl(          ( \unm \cdot \nabla)P_{\le k}\unm            \bigr)    \\&\indeq    =\Bigl(            \varphi^{(n+1)}\phin P_{\le k}\sigma(P_{\le k}\un )-\phin \phinm P_{\le k}\sigma(P_{\le k}\unm )     \Bigr)\dot{\WW}(t), \\    & \nabla\cdot \vn = 0,    \\&    \vn( 0)= {0}      \Pas \comma t\in (0,T].    \end{split}    \llabel{wt ETKK NUdr7W EK 1 jKS IHF Kh2 sr 1RRV Ra8J mBtkWI 1u k uZT F2B 4p8 E7 Y3p0 DX20 JM3XzQ tZ 3 bMC vM4 DEA wB Fp8q YKpL So1a5s dR P fTg 5R6 7v1 T4 eCJ1 qg14 CTK7u7 ag j Q0A tZ1 Nh6 hk Sys5 CWon IOqgCL 3u 7 feR BHz odS Jp 7JH8 u6Rw sYE0mc P4 r LaW Atl yRw kH F3ei UyhI iA19ZB u8 m ywf 42n uyX 0e ljCt 3Lkd 1eUQEZ oO Z rA2 Oqf oQ5 Ca hrBy KzFg DOseim 0j Y BmX csL Ayc cC JBTZ PEjy zPb5hZ KW O xT6 dyt u82 Ia htpD m75Y DktQvd Nj W jIQ H1B Ace SZ KVVP 136v L8XhMm 1O H Kn2 gUy kFU wN 8JML Bqmn vGuwGR oW U oNZ Y2P nmS 5g QMcR YHxL yHuDo8 ba w aqM NYt onW u2 YIOz eB6R wHuGcn fi o 47U PM5 tOj sz QBNq 7mco fCNjou 83 e mcY 81s vsI 2Y DS3S yloB Nx5FBV Bc 9 6HZ EOX UO3 W1 fIF5 jtEM W6KW7D 63 t H0F CVT Zup Pl A9aI oN2s f1Bw31 gg L FoD O0M x18 oo heEd KgZB Cqdqpa sa H Fhx BrE aRg Au I5dq mWWB MuHfv9 0y S PtG hFF dYJ JL f3Ap k5Ck Szr0Kb Vd i sQk uSA JEn DT YkjP AEMu a0VCtC Ff z 9R6 Vht 8Ua cB e7op AnGa 7AbLWj Hc s nAR GMb n7a 9n paMf lftM 7jvb20 0T W xUC 4lt e92 9j oZrA IuIa o1Zqdr oC L 55L T4Q 8kN yv sIzP x4i5 9lKTq2 JB B sZb QCE Ctw ar VBMT H1QR 6v5srW hR r D4r wf8 ik7 KH Egee rFVT ErONml Q5 L R8v XNZ LB3 9U DzRH ZbH9 fTBhRw kA 2 n3p g4I grH xd fEFu z6RE tDqPdw N7 H TVt cE1 8hW 6y n4Gn nCE3 MEQ51i Ps G Z2G Lbt CSt hu zvPF eE28 MM23ug TC d j7z 7Av TLa 1A GLiJ 5JwW CiDPyM qa 8 tAK QZ9 cfP 42 kuUz V3h6 GsGFoW m9 h cfj 51d GtW yZ zC5D aVt2 Wi5IIs gD B 0cX LM1 FtE xE RIZI Z0Rt QUtWcU Cm F mSj xvW pZc gl dopk 0D7a EouRku Id O ZdW FOR uqb PY 6HkW OVi7 FuVMLW nx p SaN omk rC5 uI ZK9C jpJy UIeO6k gb 7 tr2 SCY x5F 11 S6Xq OImr s7vv0u vA g rb9 hGP Fnk RM j92H gczJ 660kHb BB l QSI OY7 FcX 0c uyDl LjbU 3F6vZk Gb a KaM ufj uxp n4 Mi45 7MoL NW3eIm cj 6 OOS e59 afA hg lt9S BOiF cYQipj 5u N 19N KZ5 Czc 23 1wxG x1ut gJB4ue Mx x 5lr s8g VbZ s1 NEfI 02Rb pkfEOZ E4 e seo 9te NRU Ai nujf eJYa Ehns0Y 6X R UF1 PCf 5eE AL 9DL6 a2vm BAU5Au DD t yQN 5YL LWw PW GjMt 4hu4 FIoLCZ Lx e BVY 5lZ DCD 5Y yBwO IJeH VQsKob Yd q fCX 1to mCb Ej 5m1p Nx9p nLn5A3 g7 U v77 7YU gBR lN rTyj shaq BZXeAF tj y FlW jfc 57t 2f abx5 Ns4d clCMJc Tl q kfq uFD iSd DP eX6m YLQz JzUmH0 43EQ70}
  \end{align} Componentwise, we may rewrite the first equation as   \begin{align}    \begin{split}    &\partial_t{\vv}_j^{(n)}    -\Delta {\vv}_j^{(n)}    =     \sum_{i}\partial_{i} f_{ij}      + g_j \dot{\WW}(t)    \comma j=1,2,3    ,   \end{split}    \llabel{ tZ1 Nh6 hk Sys5 CWon IOqgCL 3u 7 feR BHz odS Jp 7JH8 u6Rw sYE0mc P4 r LaW Atl yRw kH F3ei UyhI iA19ZB u8 m ywf 42n uyX 0e ljCt 3Lkd 1eUQEZ oO Z rA2 Oqf oQ5 Ca hrBy KzFg DOseim 0j Y BmX csL Ayc cC JBTZ PEjy zPb5hZ KW O xT6 dyt u82 Ia htpD m75Y DktQvd Nj W jIQ H1B Ace SZ KVVP 136v L8XhMm 1O H Kn2 gUy kFU wN 8JML Bqmn vGuwGR oW U oNZ Y2P nmS 5g QMcR YHxL yHuDo8 ba w aqM NYt onW u2 YIOz eB6R wHuGcn fi o 47U PM5 tOj sz QBNq 7mco fCNjou 83 e mcY 81s vsI 2Y DS3S yloB Nx5FBV Bc 9 6HZ EOX UO3 W1 fIF5 jtEM W6KW7D 63 t H0F CVT Zup Pl A9aI oN2s f1Bw31 gg L FoD O0M x18 oo heEd KgZB Cqdqpa sa H Fhx BrE aRg Au I5dq mWWB MuHfv9 0y S PtG hFF dYJ JL f3Ap k5Ck Szr0Kb Vd i sQk uSA JEn DT YkjP AEMu a0VCtC Ff z 9R6 Vht 8Ua cB e7op AnGa 7AbLWj Hc s nAR GMb n7a 9n paMf lftM 7jvb20 0T W xUC 4lt e92 9j oZrA IuIa o1Zqdr oC L 55L T4Q 8kN yv sIzP x4i5 9lKTq2 JB B sZb QCE Ctw ar VBMT H1QR 6v5srW hR r D4r wf8 ik7 KH Egee rFVT ErONml Q5 L R8v XNZ LB3 9U DzRH ZbH9 fTBhRw kA 2 n3p g4I grH xd fEFu z6RE tDqPdw N7 H TVt cE1 8hW 6y n4Gn nCE3 MEQ51i Ps G Z2G Lbt CSt hu zvPF eE28 MM23ug TC d j7z 7Av TLa 1A GLiJ 5JwW CiDPyM qa 8 tAK QZ9 cfP 42 kuUz V3h6 GsGFoW m9 h cfj 51d GtW yZ zC5D aVt2 Wi5IIs gD B 0cX LM1 FtE xE RIZI Z0Rt QUtWcU Cm F mSj xvW pZc gl dopk 0D7a EouRku Id O ZdW FOR uqb PY 6HkW OVi7 FuVMLW nx p SaN omk rC5 uI ZK9C jpJy UIeO6k gb 7 tr2 SCY x5F 11 S6Xq OImr s7vv0u vA g rb9 hGP Fnk RM j92H gczJ 660kHb BB l QSI OY7 FcX 0c uyDl LjbU 3F6vZk Gb a KaM ufj uxp n4 Mi45 7MoL NW3eIm cj 6 OOS e59 afA hg lt9S BOiF cYQipj 5u N 19N KZ5 Czc 23 1wxG x1ut gJB4ue Mx x 5lr s8g VbZ s1 NEfI 02Rb pkfEOZ E4 e seo 9te NRU Ai nujf eJYa Ehns0Y 6X R UF1 PCf 5eE AL 9DL6 a2vm BAU5Au DD t yQN 5YL LWw PW GjMt 4hu4 FIoLCZ Lx e BVY 5lZ DCD 5Y yBwO IJeH VQsKob Yd q fCX 1to mCb Ej 5m1p Nx9p nLn5A3 g7 U v77 7YU gBR lN rTyj shaq BZXeAF tj y FlW jfc 57t 2f abx5 Ns4d clCMJc Tl q kfq uFD iSd DP eX6m YLQz JzUmH0 43 M lgF edN mXQ Pj Aoba 07MY wBaC4C nj I 4dw KCZ PO9 wx 3en8 AoqX 7JjN8K lq j Q5c bMS dhR Fs tQ8Q r2ve 2HT0uO 5W j TAi iIW n1C Wr U1BH BMvJ 3ywmAd qN D LY8 lbx XMx 0D Dvco EQ69}   \end{align} where   \begin{align}   \begin{split}    f_{ij}     &=       -      \varphi^{(n+1)}\phin P_{\le k}\bigl({\mathcal P}(\un_i P_{\le k}\un) \bigr)_j      +\phin \phinm        P_{\le k}   \bigl(              {\mathcal P}(  \unm_i P_{\le k}\unm)           \bigr)_j    \\&    =     -    \phin (\phinp-\phin) P_{\le k}({\mathcal P}(\un_i P_{\le k}\un))_j    -\phin (\phin-\phinm) P_{\le k}({\mathcal P}(\un_i P_{\le k}\un))_j    \\&\indeq    -\phin \phinm P_{\le k}({\mathcal P}(\vnm_i P_{\le k}\un))_j    -\phin \phinm P_{\le k} ({\mathcal P}(\unm_i P_{\le k}\vnm))_j    \\&    =    f_{ij}^{(1)}+   f_{ij}^{(2)}+   f_{ij}^{(3)}+   f_{ij}^{(4)}   \end{split}   \label{ERTWERTHWRTWERTSGDGHCFGSDFGQSERWDFGDSFGHSDRGTEHDFGHDSFGSDGHGYUHDFGSDFASDFASGTWRT72}   \end{align} and    \begin{align}   \begin{split}    g_j     &=      \varphi^{(n+1)}\phin P_{\le k}\sigma_j(P_{\le k}\un )-\phin \phinm P_{\le k}\sigma_j(P_{\le k}\unm )     \\&    =    \phin (\phinp-\phin) P_{\le k}\sigma_j(P_{\le k}\un )    + \phin (\phin-\phinm) P_{\le k}\sigma_j(P_{\le k}\un )    \\&\indeq    + \phin \phinm P_{\le k}\bigl(\sigma_j(P_{\le k}\un )-\sigma_j(P_{\le k}\unm )\bigr)    \\&    =    g_{j}^{(1)}+   g_{j}^{(2)}+   g_{j}^{(3)}    .   \end{split}   \label{ERTWERTHWRTWERTSGDGHCFGSDFGQSERWDFGDSFGHSDRGTEHDFGHDSFGSDGHGYUHDFGSDFASDFASGTWRT73}   \end{align} We now apply \eqref{ERTWERTHWRTWERTSGDGHCFGSDFGQSERWDFGDSFGHSDRGTEHDFGHDSFGSDGHGYUHDFGSDFASDFASGTWRT29} and estimate the forcing terms in order. For the first term in \eqref{ERTWERTHWRTWERTSGDGHCFGSDFGQSERWDFGDSFGHSDRGTEHDFGHDSFGSDGHGYUHDFGSDFASDFASGTWRT72}, we have   \begin{align}    \begin{split}       &      \EE\biggl[\int_0^{T}\Vert                 f^{(1)} \Vert_{q}^p \,ds       \biggr]        \leq        C_k \sum_{i}\EE\biggl[\int_0^{T}               (\phinp-\phin)^{p} (\phin)^{p}                \Vert \un_i P_{\le k}\un\Vert_{r}^p \,ds              \biggr]      \\&\indeq        \leq      C_k \sum_{i}\EE\biggl[\int_0^{T}      \Bigl|\Vert\unp\Vert_{p}-\Vert\un\Vert_{p}\Bigr|^{p} (\phin)^{p}      \Vert \un_i P_{\le k}\un\Vert_{r}^p \,ds      \biggr]      \\&\indeq      \leq      C_k \EE\biggl[\int_0^{T}                \Vert \vn\Vert_{p}^{p}                (\phin)^{p}                \Vert \un\Vert_{p}^{p}                \Vert P_{\le k}\un\Vert_{l}^{p}               \,ds       \biggr]      \leq      C_k \EE\biggl[\int_0^{T}                \Vert \vn\Vert_{p}^{p}               \,ds       \biggr],    \end{split}    \llabel{Oseim 0j Y BmX csL Ayc cC JBTZ PEjy zPb5hZ KW O xT6 dyt u82 Ia htpD m75Y DktQvd Nj W jIQ H1B Ace SZ KVVP 136v L8XhMm 1O H Kn2 gUy kFU wN 8JML Bqmn vGuwGR oW U oNZ Y2P nmS 5g QMcR YHxL yHuDo8 ba w aqM NYt onW u2 YIOz eB6R wHuGcn fi o 47U PM5 tOj sz QBNq 7mco fCNjou 83 e mcY 81s vsI 2Y DS3S yloB Nx5FBV Bc 9 6HZ EOX UO3 W1 fIF5 jtEM W6KW7D 63 t H0F CVT Zup Pl A9aI oN2s f1Bw31 gg L FoD O0M x18 oo heEd KgZB Cqdqpa sa H Fhx BrE aRg Au I5dq mWWB MuHfv9 0y S PtG hFF dYJ JL f3Ap k5Ck Szr0Kb Vd i sQk uSA JEn DT YkjP AEMu a0VCtC Ff z 9R6 Vht 8Ua cB e7op AnGa 7AbLWj Hc s nAR GMb n7a 9n paMf lftM 7jvb20 0T W xUC 4lt e92 9j oZrA IuIa o1Zqdr oC L 55L T4Q 8kN yv sIzP x4i5 9lKTq2 JB B sZb QCE Ctw ar VBMT H1QR 6v5srW hR r D4r wf8 ik7 KH Egee rFVT ErONml Q5 L R8v XNZ LB3 9U DzRH ZbH9 fTBhRw kA 2 n3p g4I grH xd fEFu z6RE tDqPdw N7 H TVt cE1 8hW 6y n4Gn nCE3 MEQ51i Ps G Z2G Lbt CSt hu zvPF eE28 MM23ug TC d j7z 7Av TLa 1A GLiJ 5JwW CiDPyM qa 8 tAK QZ9 cfP 42 kuUz V3h6 GsGFoW m9 h cfj 51d GtW yZ zC5D aVt2 Wi5IIs gD B 0cX LM1 FtE xE RIZI Z0Rt QUtWcU Cm F mSj xvW pZc gl dopk 0D7a EouRku Id O ZdW FOR uqb PY 6HkW OVi7 FuVMLW nx p SaN omk rC5 uI ZK9C jpJy UIeO6k gb 7 tr2 SCY x5F 11 S6Xq OImr s7vv0u vA g rb9 hGP Fnk RM j92H gczJ 660kHb BB l QSI OY7 FcX 0c uyDl LjbU 3F6vZk Gb a KaM ufj uxp n4 Mi45 7MoL NW3eIm cj 6 OOS e59 afA hg lt9S BOiF cYQipj 5u N 19N KZ5 Czc 23 1wxG x1ut gJB4ue Mx x 5lr s8g VbZ s1 NEfI 02Rb pkfEOZ E4 e seo 9te NRU Ai nujf eJYa Ehns0Y 6X R UF1 PCf 5eE AL 9DL6 a2vm BAU5Au DD t yQN 5YL LWw PW GjMt 4hu4 FIoLCZ Lx e BVY 5lZ DCD 5Y yBwO IJeH VQsKob Yd q fCX 1to mCb Ej 5m1p Nx9p nLn5A3 g7 U v77 7YU gBR lN rTyj shaq BZXeAF tj y FlW jfc 57t 2f abx5 Ns4d clCMJc Tl q kfq uFD iSd DP eX6m YLQz JzUmH0 43 M lgF edN mXQ Pj Aoba 07MY wBaC4C nj I 4dw KCZ PO9 wx 3en8 AoqX 7JjN8K lq j Q5c bMS dhR Fs tQ8Q r2ve 2HT0uO 5W j TAi iIW n1C Wr U1BH BMvJ 3ywmAd qN D LY8 lbx XMx 0D Dvco 3RL9 Qz5eqy wV Y qEN nO8 MH0 PY zeVN i3yb 2msNYY Wz G 2DC PoG 1Vb Bx e9oZ GcTU 3AZuEK bk p 6rN eTX 0DS Mc zd91 nbSV DKEkVa zI q NKU Qap NBP 5B 32Ey prwP FLvuPi wR P l1G TdEQ74}   \end{align}\colb with the same choice of the exponents $r$ and~$l$ as in~\eqref{ERTWERTHWRTWERTSGDGHCFGSDFGQSERWDFGDSFGHSDRGTEHDFGHDSFGSDGHGYUHDFGSDFASDFASGTWRT63}--\eqref{ERTWERTHWRTWERTSGDGHCFGSDFGQSERWDFGDSFGHSDRGTEHDFGHDSFGSDGHGYUHDFGSDFASDFASGTWRT64}. Similarly,   \begin{align}    \begin{split}       &    \EE\biggl[\int_0^{T}\Vert     f^{(2)} \Vert_{q}^p \,ds    \biggr]    +      \EE\biggl[\int_0^{T}\Vert                 f^{(3)} \Vert_{q}^p \,ds       \biggr]      +      \EE\biggl[\int_0^{T}\Vert                 f^{(4)} \Vert_{q}^p \,ds       \biggr]     \leq      C_k       \EE\biggl[\int_0^{T}                \Vert \vnm\Vert_{p}^{p}               \,ds          \biggr]   ,    \end{split}    \llabel{5g QMcR YHxL yHuDo8 ba w aqM NYt onW u2 YIOz eB6R wHuGcn fi o 47U PM5 tOj sz QBNq 7mco fCNjou 83 e mcY 81s vsI 2Y DS3S yloB Nx5FBV Bc 9 6HZ EOX UO3 W1 fIF5 jtEM W6KW7D 63 t H0F CVT Zup Pl A9aI oN2s f1Bw31 gg L FoD O0M x18 oo heEd KgZB Cqdqpa sa H Fhx BrE aRg Au I5dq mWWB MuHfv9 0y S PtG hFF dYJ JL f3Ap k5Ck Szr0Kb Vd i sQk uSA JEn DT YkjP AEMu a0VCtC Ff z 9R6 Vht 8Ua cB e7op AnGa 7AbLWj Hc s nAR GMb n7a 9n paMf lftM 7jvb20 0T W xUC 4lt e92 9j oZrA IuIa o1Zqdr oC L 55L T4Q 8kN yv sIzP x4i5 9lKTq2 JB B sZb QCE Ctw ar VBMT H1QR 6v5srW hR r D4r wf8 ik7 KH Egee rFVT ErONml Q5 L R8v XNZ LB3 9U DzRH ZbH9 fTBhRw kA 2 n3p g4I grH xd fEFu z6RE tDqPdw N7 H TVt cE1 8hW 6y n4Gn nCE3 MEQ51i Ps G Z2G Lbt CSt hu zvPF eE28 MM23ug TC d j7z 7Av TLa 1A GLiJ 5JwW CiDPyM qa 8 tAK QZ9 cfP 42 kuUz V3h6 GsGFoW m9 h cfj 51d GtW yZ zC5D aVt2 Wi5IIs gD B 0cX LM1 FtE xE RIZI Z0Rt QUtWcU Cm F mSj xvW pZc gl dopk 0D7a EouRku Id O ZdW FOR uqb PY 6HkW OVi7 FuVMLW nx p SaN omk rC5 uI ZK9C jpJy UIeO6k gb 7 tr2 SCY x5F 11 S6Xq OImr s7vv0u vA g rb9 hGP Fnk RM j92H gczJ 660kHb BB l QSI OY7 FcX 0c uyDl LjbU 3F6vZk Gb a KaM ufj uxp n4 Mi45 7MoL NW3eIm cj 6 OOS e59 afA hg lt9S BOiF cYQipj 5u N 19N KZ5 Czc 23 1wxG x1ut gJB4ue Mx x 5lr s8g VbZ s1 NEfI 02Rb pkfEOZ E4 e seo 9te NRU Ai nujf eJYa Ehns0Y 6X R UF1 PCf 5eE AL 9DL6 a2vm BAU5Au DD t yQN 5YL LWw PW GjMt 4hu4 FIoLCZ Lx e BVY 5lZ DCD 5Y yBwO IJeH VQsKob Yd q fCX 1to mCb Ej 5m1p Nx9p nLn5A3 g7 U v77 7YU gBR lN rTyj shaq BZXeAF tj y FlW jfc 57t 2f abx5 Ns4d clCMJc Tl q kfq uFD iSd DP eX6m YLQz JzUmH0 43 M lgF edN mXQ Pj Aoba 07MY wBaC4C nj I 4dw KCZ PO9 wx 3en8 AoqX 7JjN8K lq j Q5c bMS dhR Fs tQ8Q r2ve 2HT0uO 5W j TAi iIW n1C Wr U1BH BMvJ 3ywmAd qN D LY8 lbx XMx 0D Dvco 3RL9 Qz5eqy wV Y qEN nO8 MH0 PY zeVN i3yb 2msNYY Wz G 2DC PoG 1Vb Bx e9oZ GcTU 3AZuEK bk p 6rN eTX 0DS Mc zd91 nbSV DKEkVa zI q NKU Qap NBP 5B 32Ey prwP FLvuPi wR P l1G TdQ BZE Aw 3d90 v8P5 CPAnX4 Yo 2 q7s yr5 BW8 Hc T7tM ioha BW9U4q rb u mEQ 6Xz MKR 2B REFX k3ZO MVMYSw 9S F 5ek q0m yNK Gn H0qi vlRA 18CbEz id O iuy ZZ6 kRo oJ kLQ0 Ewmz sKllEQ75}   \end{align} which leads to   \begin{align}    \begin{split}      \EE\biggl[\int_0^{T}\Vert                 f \Vert_{q}^p \,ds       \biggr]     &     \leq      C_{k} \,T\,             \EE\biggl[\sup_{s\in[0,T]} \Vert \vnm\Vert_{p}^{p}  \biggr]       +       C_{k} \,T\,             \EE\biggl[\sup_{s\in[0,T]} \Vert \vn\Vert_{p}^{p}  \biggr]     .    \end{split}    \label{ERTWERTHWRTWERTSGDGHCFGSDFGQSERWDFGDSFGHSDRGTEHDFGHDSFGSDGHGYUHDFGSDFASDFASGTWRT76}   \end{align} For the three terms in \eqref{ERTWERTHWRTWERTSGDGHCFGSDFGQSERWDFGDSFGHSDRGTEHDFGHDSFGSDGHGYUHDFGSDFASDFASGTWRT73}, we have   \begin{align}    \begin{split}    &    \EE\biggl[         \int_0^{T}\Vert g^{(1)}\Vert_{\mathbb{L}^p}^p \,ds       \biggr]     \leq     C\EE\biggl[         \int_0^{T}  (\phin)^{p} (\phinp-\phin)^{p}         \int_{\RR^\dd}                   \Vert                P_{\le k} \sigma(P_{\le k}\un )            \Vert_{l^2(             \mathcal{H},\mathcal{\RR}^\dd)}^p\,dx ds             \biggr]     \\&\indeq        \leq      C      \EE\biggl[         \int_0^{T}                (\phin)^{p}            \Vert \vn\Vert_{p}^{p}            (            \Vert               P_{\le k}\un            \Vert_{(3p/2)-}^{2p}+1)\,ds         \biggr]    \\&\indeq      \leq      C_{k}      \EE\biggl[         \int_0^{T}           \Vert \vn\Vert_{p}^{p}           \,ds         \biggr]             \leq             C_{k} \,T\,             \EE\biggl[\sup_{s\in[0,T]} \Vert \vn\Vert_{p}^{p}  \biggr]       .    \end{split}    \label{ERTWERTHWRTWERTSGDGHCFGSDFGQSERWDFGDSFGHSDRGTEHDFGHDSFGSDGHGYUHDFGSDFASDFASGTWRT77}   \end{align} Similarly,   \begin{equation}   \EE\biggl[   \int_0^{T}\Vert g^{(2)}\Vert_{\mathbb{L}^p}^p \,ds   \biggr]   \leq   C_{k} \,T\,   \EE\biggl[\sup_{s\in[0,T]} \Vert \vnm\Vert_{p}^{p}  \biggr]    \label{ERTWERTHWRTWERTSGDGHCFGSDFGQSERWDFGDSFGHSDRGTEHDFGHDSFGSDGHGYUHDFGSDFASDFASGTWRT78}   \end{equation} and   \begin{align}    \begin{split}    &    \EE\biggl[         \int_0^{T}\Vert g^{(3)}\Vert_{\mathbb{L}^p}^p \,ds       \biggr]       \leq      C    \EE\biggl[       \int_0^{T} \phin \phinm       \Vert ( |P_{\le k}\un| + |P_{\le k}\unm|)^{1/2} |P_{\le k}\vnm|\Vert_{p}^p\,ds       \biggr]     \\&\indeq          \leq          C        \EE\biggl[\int_0^{T} \phin \phinm        (\Vert P_{\le k}\un\Vert_p^{p/2} + \Vert P_{\le k}\unm\Vert_p^{p/2}) \Vert P_{\le k}\vnm \Vert_{2p}^p\,ds        \biggr]        \\&\indeq      \leq      \delta       \EE\biggl[         \int_0^{T}           \Vert \vnm\Vert_{3p}^{p}          \,ds             \biggr]+C_{\delta} T             \EE\biggl[\sup_{s\in[0,T]} \Vert \vnm\Vert_{p}^{p}  \biggr]       ,    \end{split}\label{ERTWERTHWRTWERTSGDGHCFGSDFGQSERWDFGDSFGHSDRGTEHDFGHDSFGSDGHGYUHDFGSDFASDFASGTWRT79}    \end{align} for an arbitrarily small positive number~$\delta$. By Sobolev's embedding inequality, there exists a uniform positive constant $C$ such that   \begin{equation}   \Vert{\vn_j}\Vert^{p}_{3p}= \Vert |\vn_j|^{p/2}\Vert^{2}_{6} \le C\Vert{\nabla (|\vn_j|^{p/2})}\Vert^{2}_{2}   \comma j=1,2,3.   \label{ERTWERTHWRTWERTSGDGHCFGSDFGQSERWDFGDSFGHSDRGTEHDFGHDSFGSDGHGYUHDFGSDFASDFASGTWRT80}   \end{equation} Combining Theorem~\ref{T02} with \eqref{ERTWERTHWRTWERTSGDGHCFGSDFGQSERWDFGDSFGHSDRGTEHDFGHDSFGSDGHGYUHDFGSDFASDFASGTWRT76}--\eqref{ERTWERTHWRTWERTSGDGHCFGSDFGQSERWDFGDSFGHSDRGTEHDFGHDSFGSDGHGYUHDFGSDFASDFASGTWRT80}, we obtain   \begin{align}    \begin{split}    &    \EE\biggl[    \sup_{s\in[0,T]} \Vert \vn\Vert_{p}^{p}      +\int_0^{T}\Vert \vn\Vert_{3p}^{p} \,ds    \biggr]     \\&\indeq      \leq       \delta       \EE\biggl[      \int_0^{T}      \Vert \vnm\Vert_{3p}^{p}      \,ds      \biggr]      +       C_{k,\delta} \,T\,       \EE\biggl[\sup_{s\in[0,T]} \Vert \vnm\Vert_{p}^{p}  \biggr]        +       C_{k} \,T\,       \EE\biggl[\sup_{s\in[0,T]} \Vert \vn\Vert_{p}^{p}  \biggr]        .    \end{split}    \label{ERTWERTHWRTWERTSGDGHCFGSDFGQSERWDFGDSFGHSDRGTEHDFGHDSFGSDGHGYUHDFGSDFASDFASGTWRT81}   \end{align} We set $\delta\leq1/2$ and then $T$ sufficiently small so that $C_{k,\delta} \,T+C_{k} \,T\leq1/4$. This concludes the existence of a fixed-point $\uu$ of \eqref{ERTWERTHWRTWERTSGDGHCFGSDFGQSERWDFGDSFGHSDRGTEHDFGHDSFGSDGHGYUHDFGSDFASDFASGTWRT58} in $L^p_{\omega}L^{\infty}_t L_x^p\cap L^p_{\omega}L^{p}_t L_x^{3p}$ up to the time~$T$, and the rate of convergence is exponential. As a result, \begin{equation} \sup_{0\leq t\leq T}\Vert\un(t,\omega) -\uu(t,\omega)\Vert_{p}\rightarrow 0 \quad \mbox{  as  }n\to \infty    \llabel{t H0F CVT Zup Pl A9aI oN2s f1Bw31 gg L FoD O0M x18 oo heEd KgZB Cqdqpa sa H Fhx BrE aRg Au I5dq mWWB MuHfv9 0y S PtG hFF dYJ JL f3Ap k5Ck Szr0Kb Vd i sQk uSA JEn DT YkjP AEMu a0VCtC Ff z 9R6 Vht 8Ua cB e7op AnGa 7AbLWj Hc s nAR GMb n7a 9n paMf lftM 7jvb20 0T W xUC 4lt e92 9j oZrA IuIa o1Zqdr oC L 55L T4Q 8kN yv sIzP x4i5 9lKTq2 JB B sZb QCE Ctw ar VBMT H1QR 6v5srW hR r D4r wf8 ik7 KH Egee rFVT ErONml Q5 L R8v XNZ LB3 9U DzRH ZbH9 fTBhRw kA 2 n3p g4I grH xd fEFu z6RE tDqPdw N7 H TVt cE1 8hW 6y n4Gn nCE3 MEQ51i Ps G Z2G Lbt CSt hu zvPF eE28 MM23ug TC d j7z 7Av TLa 1A GLiJ 5JwW CiDPyM qa 8 tAK QZ9 cfP 42 kuUz V3h6 GsGFoW m9 h cfj 51d GtW yZ zC5D aVt2 Wi5IIs gD B 0cX LM1 FtE xE RIZI Z0Rt QUtWcU Cm F mSj xvW pZc gl dopk 0D7a EouRku Id O ZdW FOR uqb PY 6HkW OVi7 FuVMLW nx p SaN omk rC5 uI ZK9C jpJy UIeO6k gb 7 tr2 SCY x5F 11 S6Xq OImr s7vv0u vA g rb9 hGP Fnk RM j92H gczJ 660kHb BB l QSI OY7 FcX 0c uyDl LjbU 3F6vZk Gb a KaM ufj uxp n4 Mi45 7MoL NW3eIm cj 6 OOS e59 afA hg lt9S BOiF cYQipj 5u N 19N KZ5 Czc 23 1wxG x1ut gJB4ue Mx x 5lr s8g VbZ s1 NEfI 02Rb pkfEOZ E4 e seo 9te NRU Ai nujf eJYa Ehns0Y 6X R UF1 PCf 5eE AL 9DL6 a2vm BAU5Au DD t yQN 5YL LWw PW GjMt 4hu4 FIoLCZ Lx e BVY 5lZ DCD 5Y yBwO IJeH VQsKob Yd q fCX 1to mCb Ej 5m1p Nx9p nLn5A3 g7 U v77 7YU gBR lN rTyj shaq BZXeAF tj y FlW jfc 57t 2f abx5 Ns4d clCMJc Tl q kfq uFD iSd DP eX6m YLQz JzUmH0 43 M lgF edN mXQ Pj Aoba 07MY wBaC4C nj I 4dw KCZ PO9 wx 3en8 AoqX 7JjN8K lq j Q5c bMS dhR Fs tQ8Q r2ve 2HT0uO 5W j TAi iIW n1C Wr U1BH BMvJ 3ywmAd qN D LY8 lbx XMx 0D Dvco 3RL9 Qz5eqy wV Y qEN nO8 MH0 PY zeVN i3yb 2msNYY Wz G 2DC PoG 1Vb Bx e9oZ GcTU 3AZuEK bk p 6rN eTX 0DS Mc zd91 nbSV DKEkVa zI q NKU Qap NBP 5B 32Ey prwP FLvuPi wR P l1G TdQ BZE Aw 3d90 v8P5 CPAnX4 Yo 2 q7s yr5 BW8 Hc T7tM ioha BW9U4q rb u mEQ 6Xz MKR 2B REFX k3ZO MVMYSw 9S F 5ek q0m yNK Gn H0qi vlRA 18CbEz id O iuy ZZ6 kRo oJ kLQ0 Ewmz sKlld6 Kr K JmR xls 12K G2 bv8v LxfJ wrIcU6 Hx p q6p Fy7 Oim mo dXYt Kt0V VH22OC Aj f deT BAP vPl oK QzLE OQlq dpzxJ6 JI z Ujn TqY sQ4 BD QPW6 784x NUfsk0 aM 7 8qz MuL 9Mr Ac EQ82} \end{equation} $\PP$-almost surely (see~\cite[Lemma~5.2]{KXZ}).  \par Now we show that $u$ is a strong solution to~\eqref{ERTWERTHWRTWERTSGDGHCFGSDFGQSERWDFGDSFGHSDRGTEHDFGHDSFGSDGHGYUHDFGSDFASDFASGTWRT56}. By Lemma~\ref{L05},  \begin{align}   \begin{split}    (\un ( s),\phi)    &= (P_{\le k}\uu_0,\phi)+\int_0^s (\un (r),\Delta\phi)\,dr
   \\&\indeq     +\sum_{j}\int_0^s \bigl(\phin \phinm P_{\le k}\mathcal{P}\bigl( \unm_j P_{\le k}\unm \bigr),\partial_{j}\phi\bigr)\,dr    \\&\indeq     +\int_0^s (\phin \phinm P_{\le k}\sigma(P_{\le k}\unm ),\phi)\,d\WW_r    \comma  (s,\omega)\text{-a.e.},   \end{split}   \label{ERTWERTHWRTWERTSGDGHCFGSDFGQSERWDFGDSFGHSDRGTEHDFGHDSFGSDGHGYUHDFGSDFASDFASGTWRT83}   \end{align} for all $\phi\in C_c^{\infty}(\RR^3)$. Using the Dominated Convergence Theorem, we conclude   \begin{align}   \begin{split}    &\int_0^s (\un ,\Delta\phi)\,dr    +    \int_0^s (\phin \phinm P_{\le k}\mathcal{P}( \unm_j P_{\le k}\unm ),\partial_{j}\phi)\,dr    \\&\indeq     \rightarrow\int_0^s \bigl((\uu,\Delta\phi)+((\varphi^{(u)})^2P_{\le k}\mathcal{P}(\uu_j P_{\le k}\uu),\partial_{j}\phi)\bigr)\,dr     \comma  (s,\omega)\text{-a.e.}   \end{split}    \llabel{EMu a0VCtC Ff z 9R6 Vht 8Ua cB e7op AnGa 7AbLWj Hc s nAR GMb n7a 9n paMf lftM 7jvb20 0T W xUC 4lt e92 9j oZrA IuIa o1Zqdr oC L 55L T4Q 8kN yv sIzP x4i5 9lKTq2 JB B sZb QCE Ctw ar VBMT H1QR 6v5srW hR r D4r wf8 ik7 KH Egee rFVT ErONml Q5 L R8v XNZ LB3 9U DzRH ZbH9 fTBhRw kA 2 n3p g4I grH xd fEFu z6RE tDqPdw N7 H TVt cE1 8hW 6y n4Gn nCE3 MEQ51i Ps G Z2G Lbt CSt hu zvPF eE28 MM23ug TC d j7z 7Av TLa 1A GLiJ 5JwW CiDPyM qa 8 tAK QZ9 cfP 42 kuUz V3h6 GsGFoW m9 h cfj 51d GtW yZ zC5D aVt2 Wi5IIs gD B 0cX LM1 FtE xE RIZI Z0Rt QUtWcU Cm F mSj xvW pZc gl dopk 0D7a EouRku Id O ZdW FOR uqb PY 6HkW OVi7 FuVMLW nx p SaN omk rC5 uI ZK9C jpJy UIeO6k gb 7 tr2 SCY x5F 11 S6Xq OImr s7vv0u vA g rb9 hGP Fnk RM j92H gczJ 660kHb BB l QSI OY7 FcX 0c uyDl LjbU 3F6vZk Gb a KaM ufj uxp n4 Mi45 7MoL NW3eIm cj 6 OOS e59 afA hg lt9S BOiF cYQipj 5u N 19N KZ5 Czc 23 1wxG x1ut gJB4ue Mx x 5lr s8g VbZ s1 NEfI 02Rb pkfEOZ E4 e seo 9te NRU Ai nujf eJYa Ehns0Y 6X R UF1 PCf 5eE AL 9DL6 a2vm BAU5Au DD t yQN 5YL LWw PW GjMt 4hu4 FIoLCZ Lx e BVY 5lZ DCD 5Y yBwO IJeH VQsKob Yd q fCX 1to mCb Ej 5m1p Nx9p nLn5A3 g7 U v77 7YU gBR lN rTyj shaq BZXeAF tj y FlW jfc 57t 2f abx5 Ns4d clCMJc Tl q kfq uFD iSd DP eX6m YLQz JzUmH0 43 M lgF edN mXQ Pj Aoba 07MY wBaC4C nj I 4dw KCZ PO9 wx 3en8 AoqX 7JjN8K lq j Q5c bMS dhR Fs tQ8Q r2ve 2HT0uO 5W j TAi iIW n1C Wr U1BH BMvJ 3ywmAd qN D LY8 lbx XMx 0D Dvco 3RL9 Qz5eqy wV Y qEN nO8 MH0 PY zeVN i3yb 2msNYY Wz G 2DC PoG 1Vb Bx e9oZ GcTU 3AZuEK bk p 6rN eTX 0DS Mc zd91 nbSV DKEkVa zI q NKU Qap NBP 5B 32Ey prwP FLvuPi wR P l1G TdQ BZE Aw 3d90 v8P5 CPAnX4 Yo 2 q7s yr5 BW8 Hc T7tM ioha BW9U4q rb u mEQ 6Xz MKR 2B REFX k3ZO MVMYSw 9S F 5ek q0m yNK Gn H0qi vlRA 18CbEz id O iuy ZZ6 kRo oJ kLQ0 Ewmz sKlld6 Kr K JmR xls 12K G2 bv8v LxfJ wrIcU6 Hx p q6p Fy7 Oim mo dXYt Kt0V VH22OC Aj f deT BAP vPl oK QzLE OQlq dpzxJ6 JI z Ujn TqY sQ4 BD QPW6 784x NUfsk0 aM 7 8qz MuL 9Mr Ac uVVK Y55n M7WqnB 2R C pGZ vHh WUN g9 3F2e RT8U umC62V H3 Z dJX LMS cca 1m xoOO 6oOL OVzfpO BO X 5Ev KuL z5s EW 8a9y otqk cKbDJN Us l pYM JpJ jOW Uy 2U4Y VKH6 kVC1Vx 1u v yEQ84}   \end{align} as~$n\rightarrow\infty$. Also, by the BDG inequality,   \begin{align}   \begin{split}    &\EE\biggl[\sup_{s\in[0,T]}\biggl|\int_0^s     \Bigl(\phin \phinm P_{\le k}\sigma(P_{\le k}\unm )-(\varphi^{(u)})^2 P_{\le k}\sigma(P_{\le k}\uu),\phi    \Bigr)    \,d\WW_r\biggr|\biggr]    \\&\indeq    \leq     C\EE\biggl[\biggl(\int_0^{T}     (\phin-\varphi^{(u)})^2(\phinm)^2    \bigl\Vert    \bigl(    P_{\le k}\sigma(P_{\le k}\unm ) ,\phi    \bigr)    \bigr\Vert_{ l^2}^2\, dr\biggr)^{1/2}\biggr]    \\&\indeq\indeq    +     C\EE\biggl[\biggl(\int_0^{T} (\varphi^{(u)})^2(\phinm)^2    \bigl\Vert    \bigl(    P_{\le k}\sigma(P_{\le k}\unm )    -    P_{\le k}\sigma(P_{\le k}\uu)    ,\phi    \bigr)    \bigr\Vert_{ l^2}^2\, dr\biggr)^{1/2}\biggr]    \\&\indeq\indeq    +    C\EE\biggl[\biggl(\int_0^{T} (\phinm-\varphi^{(u)})^2(\varphi^{(u)})^2    \bigl\Vert    \bigl(    P_{\le k}\sigma(P_{\le k}\uu) ,\phi    \bigr)    \bigr\Vert_{ l^2}^2\, dr\biggr)^{1/2}\biggr]    \\&\indeq    =     I_1+I_2+I_3,   \end{split}    \label{ERTWERTHWRTWERTSGDGHCFGSDFGQSERWDFGDSFGHSDRGTEHDFGHDSFGSDGHGYUHDFGSDFASDFASGTWRT85}   \end{align}   where $\varphi^{(u)}$ is an abbreviation for $\varphi(\Vert u\Vert_p)$. We estimate the terms in this splitting using Minkowski's inequality and assumptions on~$\sigma$. First,   \begin{align}   \begin{split}   &I_1\leq C_p  \EE\biggl[  \biggl(  \int_0^{T} (\phin-\varphi)^2(\phinm)^2  (\Vert \unm \Vert_{(3p/2)-}^4+1)\, dr  \biggr)^{1/2}   \biggr]  \\&\indeq  \leq C_p  \EE\biggl[ \sup_{r\in[0,T]}\Vert\un -\uu\Vert_{p}  \biggl(  \int_0^{T}   (\Vert \unm \Vert_{3p}^{2-}+1)  \, dr  \biggr)^{1/2}\biggr]  \\&\indeq  \leq C_{p}T^{((p-2)/2p)+}  \biggl(\EE\Big[  \sup_{r\in[0,T]}\Vert\un -\uu\Vert_{p}^p  \Big]  \biggr)^{1/p}   \biggl(\EE\biggl[ \int_0^{T}   (\Vert \unm \Vert_{3p}^p+1)\, dr\biggr]\biggr)^{(1/p)-}.   \end{split}   \label{ERTWERTHWRTWERTSGDGHCFGSDFGQSERWDFGDSFGHSDRGTEHDFGHDSFGSDGHGYUHDFGSDFASDFASGTWRT86}   \end{align} Here, $((p-2)/2p)+$ refers to a power that is greater than and sufficiently close to $(p-2)/2p$, and $(1/p)-$ a power that is less than and sufficiently close to~$1/p$. The closeness is not arbitrary, but rather determined by the preceding steps.  By~\eqref{ERTWERTHWRTWERTSGDGHCFGSDFGQSERWDFGDSFGHSDRGTEHDFGHDSFGSDGHGYUHDFGSDFASDFASGTWRT04},   \begin{align}   \begin{split}   &I_2   \leq C_p   \EE\biggl[        \biggl(          \int_0^{T}                      (\varphi\phinm)^2                  \Bigl\Vert                    (|P_{\le k}\unm|+|P_{\le k}u|                   )^{1/2}           |P_{\le k}\unm-P_{\le k}u|                  \Bigr\Vert_{ p}^2\, dr        \biggr)^{1/2}      \biggr]   \\&\indeq   \leq C_p   \EE\biggl[      \biggl(\int_0^{T} (\varphi\phinm)^2       \Bigl(\Vert P_{\le k}\unm \Vert_p       +       \Vert P_{\le k}u\Vert_p       \Bigr)       \Vert P_{\le k}\unm-P_{\le k}u\Vert_{2p}^2       \, dr      \biggr)^{1/2}      \biggr]   \\&\indeq    \leq    C_{p}T^{(2p-3)/4p}   \biggl(\EE\biggl[   \sup_{r\in[0,T]}\Vert\unm -\uu\Vert_{p}^p            \biggr]   \biggr)^{1/4p}   \biggl(\EE\biggl[     \int_0^{T}      (\Vert \unm \Vert_{3p}^p+\Vert u \Vert_{3p}^p)\, dr             \biggr]   \biggr)^{3/4p}.   \end{split}   \label{ERTWERTHWRTWERTSGDGHCFGSDFGQSERWDFGDSFGHSDRGTEHDFGHDSFGSDGHGYUHDFGSDFASDFASGTWRT87}   \end{align} Similarly to $I_1$,    \begin{align}   \begin{split}   &I_3   \leq C_p   \EE\biggl[\biggl(\int_0^{T} (\varphi)^2(\phinm-\varphi)^2   \Vert   (   P_{\le k}\sigma(P_{\le k}\uu )   ,\phi   )   \Vert_{ l^2}^2\, dr\biggr)^{1/2}\biggr]   \\&\indeq   \leq    C_{p}T^{((p-2)/2p)+}   \biggl(\EE\Big[   \sup_{r\in[0,T]}\Vert\unm -\uu\Vert_{p}^p   \Big]   \biggr)^{1/p}   \biggl(\EE\biggl[   \int_0^{T}    (\Vert u \Vert_{3p}^p+1)\, dr\biggr]\biggr)^{(1/p)-}.   \end{split}   \label{ERTWERTHWRTWERTSGDGHCFGSDFGQSERWDFGDSFGHSDRGTEHDFGHDSFGSDGHGYUHDFGSDFASDFASGTWRT88}   \end{align} Since $\{\un\}_{n\in \NNp}$ is uniformly bounded in $L^p_{\omega}L^{p}_t L_x^{3p}$, then based on the estimates above, the right side of \eqref{ERTWERTHWRTWERTSGDGHCFGSDFGQSERWDFGDSFGHSDRGTEHDFGHDSFGSDGHGYUHDFGSDFASDFASGTWRT85} goes to zero exponentially fast as $n\rightarrow \infty$, and then,   \begin{equation}    \int_0^s (\phin \phinm P_{\le k}\sigma(P_{\le k}\unm ),\phi)\,d\WW_r         \xrightarrow{n \to \infty}        \int_0^s (\varphi^2 P_{\le k}\sigma(P_{\le k} \uu),\phi)\,d\WW_r    \comma (s,\omega)\text{-a.e.}    \llabel{ Ctw ar VBMT H1QR 6v5srW hR r D4r wf8 ik7 KH Egee rFVT ErONml Q5 L R8v XNZ LB3 9U DzRH ZbH9 fTBhRw kA 2 n3p g4I grH xd fEFu z6RE tDqPdw N7 H TVt cE1 8hW 6y n4Gn nCE3 MEQ51i Ps G Z2G Lbt CSt hu zvPF eE28 MM23ug TC d j7z 7Av TLa 1A GLiJ 5JwW CiDPyM qa 8 tAK QZ9 cfP 42 kuUz V3h6 GsGFoW m9 h cfj 51d GtW yZ zC5D aVt2 Wi5IIs gD B 0cX LM1 FtE xE RIZI Z0Rt QUtWcU Cm F mSj xvW pZc gl dopk 0D7a EouRku Id O ZdW FOR uqb PY 6HkW OVi7 FuVMLW nx p SaN omk rC5 uI ZK9C jpJy UIeO6k gb 7 tr2 SCY x5F 11 S6Xq OImr s7vv0u vA g rb9 hGP Fnk RM j92H gczJ 660kHb BB l QSI OY7 FcX 0c uyDl LjbU 3F6vZk Gb a KaM ufj uxp n4 Mi45 7MoL NW3eIm cj 6 OOS e59 afA hg lt9S BOiF cYQipj 5u N 19N KZ5 Czc 23 1wxG x1ut gJB4ue Mx x 5lr s8g VbZ s1 NEfI 02Rb pkfEOZ E4 e seo 9te NRU Ai nujf eJYa Ehns0Y 6X R UF1 PCf 5eE AL 9DL6 a2vm BAU5Au DD t yQN 5YL LWw PW GjMt 4hu4 FIoLCZ Lx e BVY 5lZ DCD 5Y yBwO IJeH VQsKob Yd q fCX 1to mCb Ej 5m1p Nx9p nLn5A3 g7 U v77 7YU gBR lN rTyj shaq BZXeAF tj y FlW jfc 57t 2f abx5 Ns4d clCMJc Tl q kfq uFD iSd DP eX6m YLQz JzUmH0 43 M lgF edN mXQ Pj Aoba 07MY wBaC4C nj I 4dw KCZ PO9 wx 3en8 AoqX 7JjN8K lq j Q5c bMS dhR Fs tQ8Q r2ve 2HT0uO 5W j TAi iIW n1C Wr U1BH BMvJ 3ywmAd qN D LY8 lbx XMx 0D Dvco 3RL9 Qz5eqy wV Y qEN nO8 MH0 PY zeVN i3yb 2msNYY Wz G 2DC PoG 1Vb Bx e9oZ GcTU 3AZuEK bk p 6rN eTX 0DS Mc zd91 nbSV DKEkVa zI q NKU Qap NBP 5B 32Ey prwP FLvuPi wR P l1G TdQ BZE Aw 3d90 v8P5 CPAnX4 Yo 2 q7s yr5 BW8 Hc T7tM ioha BW9U4q rb u mEQ 6Xz MKR 2B REFX k3ZO MVMYSw 9S F 5ek q0m yNK Gn H0qi vlRA 18CbEz id O iuy ZZ6 kRo oJ kLQ0 Ewmz sKlld6 Kr K JmR xls 12K G2 bv8v LxfJ wrIcU6 Hx p q6p Fy7 Oim mo dXYt Kt0V VH22OC Aj f deT BAP vPl oK QzLE OQlq dpzxJ6 JI z Ujn TqY sQ4 BD QPW6 784x NUfsk0 aM 7 8qz MuL 9Mr Ac uVVK Y55n M7WqnB 2R C pGZ vHh WUN g9 3F2e RT8U umC62V H3 Z dJX LMS cca 1m xoOO 6oOL OVzfpO BO X 5Ev KuL z5s EW 8a9y otqk cKbDJN Us l pYM JpJ jOW Uy 2U4Y VKH6 kVC1Vx 1u v ykO yDs zo5 bz d36q WH1k J7Jtkg V1 J xqr Fnq mcU yZ JTp9 oFIc FAk0IT A9 3 SrL axO 9oU Z3 jG6f BRL1 iZ7ZE6 zj 8 G3M Hu8 6Ay jt 3flY cmTk jiTSYv CF t JLq cJP tN7 E3 POqG OKe0EQ89}   \end{equation} Letting $n\rightarrow \infty$ in \eqref{ERTWERTHWRTWERTSGDGHCFGSDFGQSERWDFGDSFGHSDRGTEHDFGHDSFGSDGHGYUHDFGSDFASDFASGTWRT83}, we obtain that $\uu$ solves~\eqref{ERTWERTHWRTWERTSGDGHCFGSDFGQSERWDFGDSFGHSDRGTEHDFGHDSFGSDGHGYUHDFGSDFASDFASGTWRT56}.  Also, the inequality \eqref{ERTWERTHWRTWERTSGDGHCFGSDFGQSERWDFGDSFGHSDRGTEHDFGHDSFGSDGHGYUHDFGSDFASDFASGTWRT57} follows by using Lemmas~\ref{L04} and~\ref{L05}. Thus the existence of a strong solution is established. \par Next, we proceed to prove the pathwise uniqueness of solutions.  Suppose that \eqref{ERTWERTHWRTWERTSGDGHCFGSDFGQSERWDFGDSFGHSDRGTEHDFGHDSFGSDGHGYUHDFGSDFASDFASGTWRT56} has two strong solutions $u,v \in L^p(\Omega; C([0,T], L^p))$. Then $w=u-v$ satisfies    \begin{align}   \begin{split}    &\partial_t w     -\Delta w     = -     \varphi_u^2 P_{\le k}\mathcal{P}\bigl(( u \cdot \nabla)P_{\le k}u   \bigr)     +     \varphi_v^2P_{\le k}\mathcal{P}\bigl(( v \cdot \nabla)P_{\le k}v   \bigr)     \\&\indeq\indeq\indeq\indeq\indeq\indeq    +\bigl(    \varphi_u^2P_{\le k}\sigma(P_{\le k}u)-\varphi_v^2P_{\le k}\sigma(P_{\le k}v)    \bigr)\dot{\WW}(t),     \\&    \nabla\cdot w  = 0,    \\&    w(0) = 0     \Pas   \end{split} \llabel{i Ps G Z2G Lbt CSt hu zvPF eE28 MM23ug TC d j7z 7Av TLa 1A GLiJ 5JwW CiDPyM qa 8 tAK QZ9 cfP 42 kuUz V3h6 GsGFoW m9 h cfj 51d GtW yZ zC5D aVt2 Wi5IIs gD B 0cX LM1 FtE xE RIZI Z0Rt QUtWcU Cm F mSj xvW pZc gl dopk 0D7a EouRku Id O ZdW FOR uqb PY 6HkW OVi7 FuVMLW nx p SaN omk rC5 uI ZK9C jpJy UIeO6k gb 7 tr2 SCY x5F 11 S6Xq OImr s7vv0u vA g rb9 hGP Fnk RM j92H gczJ 660kHb BB l QSI OY7 FcX 0c uyDl LjbU 3F6vZk Gb a KaM ufj uxp n4 Mi45 7MoL NW3eIm cj 6 OOS e59 afA hg lt9S BOiF cYQipj 5u N 19N KZ5 Czc 23 1wxG x1ut gJB4ue Mx x 5lr s8g VbZ s1 NEfI 02Rb pkfEOZ E4 e seo 9te NRU Ai nujf eJYa Ehns0Y 6X R UF1 PCf 5eE AL 9DL6 a2vm BAU5Au DD t yQN 5YL LWw PW GjMt 4hu4 FIoLCZ Lx e BVY 5lZ DCD 5Y yBwO IJeH VQsKob Yd q fCX 1to mCb Ej 5m1p Nx9p nLn5A3 g7 U v77 7YU gBR lN rTyj shaq BZXeAF tj y FlW jfc 57t 2f abx5 Ns4d clCMJc Tl q kfq uFD iSd DP eX6m YLQz JzUmH0 43 M lgF edN mXQ Pj Aoba 07MY wBaC4C nj I 4dw KCZ PO9 wx 3en8 AoqX 7JjN8K lq j Q5c bMS dhR Fs tQ8Q r2ve 2HT0uO 5W j TAi iIW n1C Wr U1BH BMvJ 3ywmAd qN D LY8 lbx XMx 0D Dvco 3RL9 Qz5eqy wV Y qEN nO8 MH0 PY zeVN i3yb 2msNYY Wz G 2DC PoG 1Vb Bx e9oZ GcTU 3AZuEK bk p 6rN eTX 0DS Mc zd91 nbSV DKEkVa zI q NKU Qap NBP 5B 32Ey prwP FLvuPi wR P l1G TdQ BZE Aw 3d90 v8P5 CPAnX4 Yo 2 q7s yr5 BW8 Hc T7tM ioha BW9U4q rb u mEQ 6Xz MKR 2B REFX k3ZO MVMYSw 9S F 5ek q0m yNK Gn H0qi vlRA 18CbEz id O iuy ZZ6 kRo oJ kLQ0 Ewmz sKlld6 Kr K JmR xls 12K G2 bv8v LxfJ wrIcU6 Hx p q6p Fy7 Oim mo dXYt Kt0V VH22OC Aj f deT BAP vPl oK QzLE OQlq dpzxJ6 JI z Ujn TqY sQ4 BD QPW6 784x NUfsk0 aM 7 8qz MuL 9Mr Ac uVVK Y55n M7WqnB 2R C pGZ vHh WUN g9 3F2e RT8U umC62V H3 Z dJX LMS cca 1m xoOO 6oOL OVzfpO BO X 5Ev KuL z5s EW 8a9y otqk cKbDJN Us l pYM JpJ jOW Uy 2U4Y VKH6 kVC1Vx 1u v ykO yDs zo5 bz d36q WH1k J7Jtkg V1 J xqr Fnq mcU yZ JTp9 oFIc FAk0IT A9 3 SrL axO 9oU Z3 jG6f BRL1 iZ7ZE6 zj 8 G3M Hu8 6Ay jt 3flY cmTk jiTSYv CF t JLq cJP tN7 E3 POqG OKe0 3K3WV0 ep W XDQ C97 YSb AD ZUNp 81GF fCPbj3 iq E t0E NXy pLv fo Iz6z oFoF 9lkIun Xj Y yYL 52U bRB jx kQUS U9mm XtzIHO Cz 1 KH4 9ez 6Pz qW F223 C0Iz 3CsvuT R9 s VtQ CcM 1eEQ90} \end{align} on $(0,T]\times {\mathbb R}^{3}$. As above, we write the first equation componentwise as   \begin{align}   \begin{split}   &\partial_t {w}_j
  -\Delta {w}_j   =   \sum_{i}\partial_{i} {f}_{ij}    + {g}_j \dot{\WW}(t)   \comma j=1,2,3   ,   \end{split}   \llabel{IZI Z0Rt QUtWcU Cm F mSj xvW pZc gl dopk 0D7a EouRku Id O ZdW FOR uqb PY 6HkW OVi7 FuVMLW nx p SaN omk rC5 uI ZK9C jpJy UIeO6k gb 7 tr2 SCY x5F 11 S6Xq OImr s7vv0u vA g rb9 hGP Fnk RM j92H gczJ 660kHb BB l QSI OY7 FcX 0c uyDl LjbU 3F6vZk Gb a KaM ufj uxp n4 Mi45 7MoL NW3eIm cj 6 OOS e59 afA hg lt9S BOiF cYQipj 5u N 19N KZ5 Czc 23 1wxG x1ut gJB4ue Mx x 5lr s8g VbZ s1 NEfI 02Rb pkfEOZ E4 e seo 9te NRU Ai nujf eJYa Ehns0Y 6X R UF1 PCf 5eE AL 9DL6 a2vm BAU5Au DD t yQN 5YL LWw PW GjMt 4hu4 FIoLCZ Lx e BVY 5lZ DCD 5Y yBwO IJeH VQsKob Yd q fCX 1to mCb Ej 5m1p Nx9p nLn5A3 g7 U v77 7YU gBR lN rTyj shaq BZXeAF tj y FlW jfc 57t 2f abx5 Ns4d clCMJc Tl q kfq uFD iSd DP eX6m YLQz JzUmH0 43 M lgF edN mXQ Pj Aoba 07MY wBaC4C nj I 4dw KCZ PO9 wx 3en8 AoqX 7JjN8K lq j Q5c bMS dhR Fs tQ8Q r2ve 2HT0uO 5W j TAi iIW n1C Wr U1BH BMvJ 3ywmAd qN D LY8 lbx XMx 0D Dvco 3RL9 Qz5eqy wV Y qEN nO8 MH0 PY zeVN i3yb 2msNYY Wz G 2DC PoG 1Vb Bx e9oZ GcTU 3AZuEK bk p 6rN eTX 0DS Mc zd91 nbSV DKEkVa zI q NKU Qap NBP 5B 32Ey prwP FLvuPi wR P l1G TdQ BZE Aw 3d90 v8P5 CPAnX4 Yo 2 q7s yr5 BW8 Hc T7tM ioha BW9U4q rb u mEQ 6Xz MKR 2B REFX k3ZO MVMYSw 9S F 5ek q0m yNK Gn H0qi vlRA 18CbEz id O iuy ZZ6 kRo oJ kLQ0 Ewmz sKlld6 Kr K JmR xls 12K G2 bv8v LxfJ wrIcU6 Hx p q6p Fy7 Oim mo dXYt Kt0V VH22OC Aj f deT BAP vPl oK QzLE OQlq dpzxJ6 JI z Ujn TqY sQ4 BD QPW6 784x NUfsk0 aM 7 8qz MuL 9Mr Ac uVVK Y55n M7WqnB 2R C pGZ vHh WUN g9 3F2e RT8U umC62V H3 Z dJX LMS cca 1m xoOO 6oOL OVzfpO BO X 5Ev KuL z5s EW 8a9y otqk cKbDJN Us l pYM JpJ jOW Uy 2U4Y VKH6 kVC1Vx 1u v ykO yDs zo5 bz d36q WH1k J7Jtkg V1 J xqr Fnq mcU yZ JTp9 oFIc FAk0IT A9 3 SrL axO 9oU Z3 jG6f BRL1 iZ7ZE6 zj 8 G3M Hu8 6Ay jt 3flY cmTk jiTSYv CF t JLq cJP tN7 E3 POqG OKe0 3K3WV0 ep W XDQ C97 YSb AD ZUNp 81GF fCPbj3 iq E t0E NXy pLv fo Iz6z oFoF 9lkIun Xj Y yYL 52U bRB jx kQUS U9mm XtzIHO Cz 1 KH4 9ez 6Pz qW F223 C0Iz 3CsvuT R9 s VtQ CcM 1eo pD Py2l EEzL U0USJt Jb 9 zgy Gyf iQ4 fo Cx26 k4jL E0ula6 aS I rZQ HER 5HV CE BL55 WCtB 2LCmve TD z Vcp 7UR gI7 Qu FbFw 9VTx JwGrzs VW M 9sM JeJ Nd2 VG GFsi WuqC 3YxXoJ GEQ91}   \end{align}           where    \begin{align}   \begin{split}   {f}_{ij}   &=   -   \varphi_u^2 P_{\le k}\bigl({\mathcal P}(u_i P_{\le k}u) \bigr)_j   +\varphi_v^2   P_{\le k}\bigl(   {\mathcal P}(  v_i P_{\le k}v)   \bigr)_j   \\&   =   -    \varphi_u (\varphi_u-\varphi_v) P_{\le k}({\mathcal P}(u_i P_{\le k}u))_j   -\varphi_u \varphi_v P_{\le k}({\mathcal P}(w_i P_{\le k}u))_j   \\&\indeq -\varphi_u \varphi_v P_{\le k}({\mathcal P}(v_i P_{\le k}w))_j - \varphi_v (\varphi_u-\varphi_v)P_{\le k} ({\mathcal P}(v_i P_{\le k} v))_j            \end{split} \llabel{9 hGP Fnk RM j92H gczJ 660kHb BB l QSI OY7 FcX 0c uyDl LjbU 3F6vZk Gb a KaM ufj uxp n4 Mi45 7MoL NW3eIm cj 6 OOS e59 afA hg lt9S BOiF cYQipj 5u N 19N KZ5 Czc 23 1wxG x1ut gJB4ue Mx x 5lr s8g VbZ s1 NEfI 02Rb pkfEOZ E4 e seo 9te NRU Ai nujf eJYa Ehns0Y 6X R UF1 PCf 5eE AL 9DL6 a2vm BAU5Au DD t yQN 5YL LWw PW GjMt 4hu4 FIoLCZ Lx e BVY 5lZ DCD 5Y yBwO IJeH VQsKob Yd q fCX 1to mCb Ej 5m1p Nx9p nLn5A3 g7 U v77 7YU gBR lN rTyj shaq BZXeAF tj y FlW jfc 57t 2f abx5 Ns4d clCMJc Tl q kfq uFD iSd DP eX6m YLQz JzUmH0 43 M lgF edN mXQ Pj Aoba 07MY wBaC4C nj I 4dw KCZ PO9 wx 3en8 AoqX 7JjN8K lq j Q5c bMS dhR Fs tQ8Q r2ve 2HT0uO 5W j TAi iIW n1C Wr U1BH BMvJ 3ywmAd qN D LY8 lbx XMx 0D Dvco 3RL9 Qz5eqy wV Y qEN nO8 MH0 PY zeVN i3yb 2msNYY Wz G 2DC PoG 1Vb Bx e9oZ GcTU 3AZuEK bk p 6rN eTX 0DS Mc zd91 nbSV DKEkVa zI q NKU Qap NBP 5B 32Ey prwP FLvuPi wR P l1G TdQ BZE Aw 3d90 v8P5 CPAnX4 Yo 2 q7s yr5 BW8 Hc T7tM ioha BW9U4q rb u mEQ 6Xz MKR 2B REFX k3ZO MVMYSw 9S F 5ek q0m yNK Gn H0qi vlRA 18CbEz id O iuy ZZ6 kRo oJ kLQ0 Ewmz sKlld6 Kr K JmR xls 12K G2 bv8v LxfJ wrIcU6 Hx p q6p Fy7 Oim mo dXYt Kt0V VH22OC Aj f deT BAP vPl oK QzLE OQlq dpzxJ6 JI z Ujn TqY sQ4 BD QPW6 784x NUfsk0 aM 7 8qz MuL 9Mr Ac uVVK Y55n M7WqnB 2R C pGZ vHh WUN g9 3F2e RT8U umC62V H3 Z dJX LMS cca 1m xoOO 6oOL OVzfpO BO X 5Ev KuL z5s EW 8a9y otqk cKbDJN Us l pYM JpJ jOW Uy 2U4Y VKH6 kVC1Vx 1u v ykO yDs zo5 bz d36q WH1k J7Jtkg V1 J xqr Fnq mcU yZ JTp9 oFIc FAk0IT A9 3 SrL axO 9oU Z3 jG6f BRL1 iZ7ZE6 zj 8 G3M Hu8 6Ay jt 3flY cmTk jiTSYv CF t JLq cJP tN7 E3 POqG OKe0 3K3WV0 ep W XDQ C97 YSb AD ZUNp 81GF fCPbj3 iq E t0E NXy pLv fo Iz6z oFoF 9lkIun Xj Y yYL 52U bRB jx kQUS U9mm XtzIHO Cz 1 KH4 9ez 6Pz qW F223 C0Iz 3CsvuT R9 s VtQ CcM 1eo pD Py2l EEzL U0USJt Jb 9 zgy Gyf iQ4 fo Cx26 k4jL E0ula6 aS I rZQ HER 5HV CE BL55 WCtB 2LCmve TD z Vcp 7UR gI7 Qu FbFw 9VTx JwGrzs VW M 9sM JeJ Nd2 VG GFsi WuqC 3YxXoJ GK w Io7 1fg sGm 0P YFBz X8eX 7pf9GJ b1 o XUs 1q0 6KP Ls MucN ytQb L0Z0Qq m1 l SPj 9MT etk L6 KfsC 6Zob Yhc2qu Xy 9 GPm ZYj 1Go ei feJ3 pRAf n6Ypy6 jN s 4Y5 nSE pqN 4m RmamEQ92} \end{align} and \begin{align} \begin{split} {g}_j &= \varphi_u^2P_{\le k}\sigma_j(P_{\le k}u )-P_{\le k}\varphi_v^2\sigma_j(P_{\le k}v) \\& =  \varphi_u (\varphi_u-\varphi_v) P_{\le k}\sigma_j(P_{\le k}u ) +\varphi_v (\varphi_u-\varphi_v) P_{\le k}\sigma_j(P_{\le k}v ) \\&\indeq +\varphi_u \varphi_v P_{\le k}\bigl(\sigma_j(P_{\le k}u )-\sigma_j(P_{\le k}v )\bigr) . \end{split} \llabel{gJB4ue Mx x 5lr s8g VbZ s1 NEfI 02Rb pkfEOZ E4 e seo 9te NRU Ai nujf eJYa Ehns0Y 6X R UF1 PCf 5eE AL 9DL6 a2vm BAU5Au DD t yQN 5YL LWw PW GjMt 4hu4 FIoLCZ Lx e BVY 5lZ DCD 5Y yBwO IJeH VQsKob Yd q fCX 1to mCb Ej 5m1p Nx9p nLn5A3 g7 U v77 7YU gBR lN rTyj shaq BZXeAF tj y FlW jfc 57t 2f abx5 Ns4d clCMJc Tl q kfq uFD iSd DP eX6m YLQz JzUmH0 43 M lgF edN mXQ Pj Aoba 07MY wBaC4C nj I 4dw KCZ PO9 wx 3en8 AoqX 7JjN8K lq j Q5c bMS dhR Fs tQ8Q r2ve 2HT0uO 5W j TAi iIW n1C Wr U1BH BMvJ 3ywmAd qN D LY8 lbx XMx 0D Dvco 3RL9 Qz5eqy wV Y qEN nO8 MH0 PY zeVN i3yb 2msNYY Wz G 2DC PoG 1Vb Bx e9oZ GcTU 3AZuEK bk p 6rN eTX 0DS Mc zd91 nbSV DKEkVa zI q NKU Qap NBP 5B 32Ey prwP FLvuPi wR P l1G TdQ BZE Aw 3d90 v8P5 CPAnX4 Yo 2 q7s yr5 BW8 Hc T7tM ioha BW9U4q rb u mEQ 6Xz MKR 2B REFX k3ZO MVMYSw 9S F 5ek q0m yNK Gn H0qi vlRA 18CbEz id O iuy ZZ6 kRo oJ kLQ0 Ewmz sKlld6 Kr K JmR xls 12K G2 bv8v LxfJ wrIcU6 Hx p q6p Fy7 Oim mo dXYt Kt0V VH22OC Aj f deT BAP vPl oK QzLE OQlq dpzxJ6 JI z Ujn TqY sQ4 BD QPW6 784x NUfsk0 aM 7 8qz MuL 9Mr Ac uVVK Y55n M7WqnB 2R C pGZ vHh WUN g9 3F2e RT8U umC62V H3 Z dJX LMS cca 1m xoOO 6oOL OVzfpO BO X 5Ev KuL z5s EW 8a9y otqk cKbDJN Us l pYM JpJ jOW Uy 2U4Y VKH6 kVC1Vx 1u v ykO yDs zo5 bz d36q WH1k J7Jtkg V1 J xqr Fnq mcU yZ JTp9 oFIc FAk0IT A9 3 SrL axO 9oU Z3 jG6f BRL1 iZ7ZE6 zj 8 G3M Hu8 6Ay jt 3flY cmTk jiTSYv CF t JLq cJP tN7 E3 POqG OKe0 3K3WV0 ep W XDQ C97 YSb AD ZUNp 81GF fCPbj3 iq E t0E NXy pLv fo Iz6z oFoF 9lkIun Xj Y yYL 52U bRB jx kQUS U9mm XtzIHO Cz 1 KH4 9ez 6Pz qW F223 C0Iz 3CsvuT R9 s VtQ CcM 1eo pD Py2l EEzL U0USJt Jb 9 zgy Gyf iQ4 fo Cx26 k4jL E0ula6 aS I rZQ HER 5HV CE BL55 WCtB 2LCmve TD z Vcp 7UR gI7 Qu FbFw 9VTx JwGrzs VW M 9sM JeJ Nd2 VG GFsi WuqC 3YxXoJ GK w Io7 1fg sGm 0P YFBz X8eX 7pf9GJ b1 o XUs 1q0 6KP Ls MucN ytQb L0Z0Qq m1 l SPj 9MT etk L6 KfsC 6Zob Yhc2qu Xy 9 GPm ZYj 1Go ei feJ3 pRAf n6Ypy6 jN s 4Y5 nSE pqN 4m Rmam AGfY HhSaBr Ls D THC SEl UyR Mh 66XU 7hNz pZVC5V nV 7 VjL 7kv WKf 7P 5hj6 t1vu gkLGdN X8 b gOX HWm 6W4 YE mxFG 4WaN EbGKsv 0p 4 OG0 Nrd uTe Za xNXq V4Bp mOdXIq 9a b PeD PEQ93} \end{align} Analogously to \eqref{ERTWERTHWRTWERTSGDGHCFGSDFGQSERWDFGDSFGHSDRGTEHDFGHDSFGSDGHGYUHDFGSDFASDFASGTWRT81}, we can show that  \begin{align} \begin{split}  \EE\biggl[\sup_{s\in[0,T]} \Vert w\Vert_{p}^{p}  \biggr]  &\leq C \EE\biggl[\int_0^{T}\Vert  {f} \Vert_{q}^p \,ds \biggr] +C\EE \biggl[ \int_0^{T}\Vert {g}(s,x)\Vert_{\mathbb{L}^p}^p\,ds \biggr] \leq C_{K} \,T\, \EE\biggl[\sup_{s\in[0,T]} \Vert w\Vert_{p}^{p}  \biggr]  \end{split} \llabel{ 5Y yBwO IJeH VQsKob Yd q fCX 1to mCb Ej 5m1p Nx9p nLn5A3 g7 U v77 7YU gBR lN rTyj shaq BZXeAF tj y FlW jfc 57t 2f abx5 Ns4d clCMJc Tl q kfq uFD iSd DP eX6m YLQz JzUmH0 43 M lgF edN mXQ Pj Aoba 07MY wBaC4C nj I 4dw KCZ PO9 wx 3en8 AoqX 7JjN8K lq j Q5c bMS dhR Fs tQ8Q r2ve 2HT0uO 5W j TAi iIW n1C Wr U1BH BMvJ 3ywmAd qN D LY8 lbx XMx 0D Dvco 3RL9 Qz5eqy wV Y qEN nO8 MH0 PY zeVN i3yb 2msNYY Wz G 2DC PoG 1Vb Bx e9oZ GcTU 3AZuEK bk p 6rN eTX 0DS Mc zd91 nbSV DKEkVa zI q NKU Qap NBP 5B 32Ey prwP FLvuPi wR P l1G TdQ BZE Aw 3d90 v8P5 CPAnX4 Yo 2 q7s yr5 BW8 Hc T7tM ioha BW9U4q rb u mEQ 6Xz MKR 2B REFX k3ZO MVMYSw 9S F 5ek q0m yNK Gn H0qi vlRA 18CbEz id O iuy ZZ6 kRo oJ kLQ0 Ewmz sKlld6 Kr K JmR xls 12K G2 bv8v LxfJ wrIcU6 Hx p q6p Fy7 Oim mo dXYt Kt0V VH22OC Aj f deT BAP vPl oK QzLE OQlq dpzxJ6 JI z Ujn TqY sQ4 BD QPW6 784x NUfsk0 aM 7 8qz MuL 9Mr Ac uVVK Y55n M7WqnB 2R C pGZ vHh WUN g9 3F2e RT8U umC62V H3 Z dJX LMS cca 1m xoOO 6oOL OVzfpO BO X 5Ev KuL z5s EW 8a9y otqk cKbDJN Us l pYM JpJ jOW Uy 2U4Y VKH6 kVC1Vx 1u v ykO yDs zo5 bz d36q WH1k J7Jtkg V1 J xqr Fnq mcU yZ JTp9 oFIc FAk0IT A9 3 SrL axO 9oU Z3 jG6f BRL1 iZ7ZE6 zj 8 G3M Hu8 6Ay jt 3flY cmTk jiTSYv CF t JLq cJP tN7 E3 POqG OKe0 3K3WV0 ep W XDQ C97 YSb AD ZUNp 81GF fCPbj3 iq E t0E NXy pLv fo Iz6z oFoF 9lkIun Xj Y yYL 52U bRB jx kQUS U9mm XtzIHO Cz 1 KH4 9ez 6Pz qW F223 C0Iz 3CsvuT R9 s VtQ CcM 1eo pD Py2l EEzL U0USJt Jb 9 zgy Gyf iQ4 fo Cx26 k4jL E0ula6 aS I rZQ HER 5HV CE BL55 WCtB 2LCmve TD z Vcp 7UR gI7 Qu FbFw 9VTx JwGrzs VW M 9sM JeJ Nd2 VG GFsi WuqC 3YxXoJ GK w Io7 1fg sGm 0P YFBz X8eX 7pf9GJ b1 o XUs 1q0 6KP Ls MucN ytQb L0Z0Qq m1 l SPj 9MT etk L6 KfsC 6Zob Yhc2qu Xy 9 GPm ZYj 1Go ei feJ3 pRAf n6Ypy6 jN s 4Y5 nSE pqN 4m Rmam AGfY HhSaBr Ls D THC SEl UyR Mh 66XU 7hNz pZVC5V nV 7 VjL 7kv WKf 7P 5hj6 t1vu gkLGdN X8 b gOX HWm 6W4 YE mxFG 4WaN EbGKsv 0p 4 OG0 Nrd uTe Za xNXq V4Bp mOdXIq 9a b PeD PbU Z4N Xt ohbY egCf xBNttE wc D YSD 637 jJ2 ms 6Ta1 J2xZ PtKnPw AX A tJA Rc8 n5d 93 TZi7 q6Wo nEDLwW Sz e Sue YFX 8cM hm Y6is 15pX aOYBbV fS C haL kBR Ks6 UO qG4j DVab fbdEQ94} \end{align} and prove the pathwise uniqueness by assuming that $T$ is sufficiently small.  Thus, we have obtained a unique strong solution of \eqref{ERTWERTHWRTWERTSGDGHCFGSDFGQSERWDFGDSFGHSDRGTEHDFGHDSFGSDGHGYUHDFGSDFASDFASGTWRT56} in $L^p(\Omega; C([0,T], L^p))$.  \par Now, we relax the smallness assumption of~$T$. Denote by $[0,t^{*}]$ the small interval in which both existence and uniqueness hold, and note that $t^{*}$ does not depend on the initial data. Let $n^{\ast}$ be a positive integer such that $T/n^{\ast}\leq t^{\ast}$. Set $t_i= i T/n^{\ast}$ for $i\in\{ 0,1,\ldots, n^{\ast}-1\}$.  Applying the existence and pathwise uniqueness results and the estimate \eqref{ERTWERTHWRTWERTSGDGHCFGSDFGQSERWDFGDSFGHSDRGTEHDFGHDSFGSDGHGYUHDFGSDFASDFASGTWRT57} consecutively on $[t_i, t_{i+1}]$,  we obtain a unique strong solution to \eqref{ERTWERTHWRTWERTSGDGHCFGSDFGQSERWDFGDSFGHSDRGTEHDFGHDSFGSDGHGYUHDFGSDFASDFASGTWRT56} in $[0,T]$ and \eqref{ERTWERTHWRTWERTSGDGHCFGSDFGQSERWDFGDSFGHSDRGTEHDFGHDSFGSDGHGYUHDFGSDFASDFASGTWRT57} holds on the whole time interval.    \end{proof} \par We also derive an $L^2$-energy estimate for the next section. \cole \begin{Lemma}   \label{L0}   Let $p> 2$, $\uu_0\in L^p(\Omega; L^p)\cap L^2(\Omega; L^2)$, and $\nabla\cdot u_0=0$. For every $T>0$, there   exists a unique strong solution $\uu\in L^p(\Omega; C([0,T], L^p))\cap L^2(\Omega; C([0,T], L^2))$ to \eqref{ERTWERTHWRTWERTSGDGHCFGSDFGQSERWDFGDSFGHSDRGTEHDFGHDSFGSDGHGYUHDFGSDFASDFASGTWRT56}   such that   \begin{align}   \begin{split}   \EE\biggl[\sup_{0\leq s\leq T}\Vert\uu(s,\cdot)\Vert_2^2   +\int_0^{T} \Vert \nabla u\Vert_2^2 \,ds   \biggr]   \leq C\bigl(\EE\bigl[\Vert\uu_0\Vert_2^2\bigr]+1\bigr),    \end{split}   \label{ERTWERTHWRTWERTSGDGHCFGSDFGQSERWDFGDSFGHSDRGTEHDFGHDSFGSDGHGYUHDFGSDFASDFASGTWRT97}   \end{align} where $C$ does not depend on $k$ (see~\eqref{ERTWERTHWRTWERTSGDGHCFGSDFGQSERWDFGDSFGHSDRGTEHDFGHDSFGSDGHGYUHDFGSDFASDFASGTWRT56}). \end{Lemma} \colb \begin{proof}[Proof of Lemma~\ref{L0}] The existence of a unique strong solution $u$ in $L^p(\Omega; C([0,T], L^p))$ has been established in Theorem~\ref{T03}. Here, we focus on the $L^2$-estimate~\eqref{ERTWERTHWRTWERTSGDGHCFGSDFGQSERWDFGDSFGHSDRGTEHDFGHDSFGSDGHGYUHDFGSDFASDFASGTWRT97}. By It\^{o}'s formula, for $t\in (0,T]$, \begin{align} \begin{split} &\sum_j\Vert u_j(t)\Vert_{2}^{2} +2\int_0^t \sum_j\Vert \nabla u_j\Vert_2^2 \,dr \leq  \sum_j\Vert u_{0,j}\Vert_{2}^{2} +2\int_0^t\varphi_u^2\left|\int_{\RR^\dd} u_j(r)   \partial_i P_{\le k} \mathcal{P}\bigl(u_i P_{\le k} u\bigr)_j \,dx\right| \,dr \\&\indeq +\int_0^t  \varphi_u^4\int_{\RR^\dd} \sum_j\Vert P_{\le k} \sigma_j(P_{\le k} u)(r)\Vert_{l^2}^2\,dx dr + 2\left|\int_0^t \varphi_u^2\int_{\RR^\dd}\sum_j u_j(r)P_{\le k} \sigma_j(P_{\le k} u)(r) \,dxd\WW_r\right|. \end{split}    \llabel{ M lgF edN mXQ Pj Aoba 07MY wBaC4C nj I 4dw KCZ PO9 wx 3en8 AoqX 7JjN8K lq j Q5c bMS dhR Fs tQ8Q r2ve 2HT0uO 5W j TAi iIW n1C Wr U1BH BMvJ 3ywmAd qN D LY8 lbx XMx 0D Dvco 3RL9 Qz5eqy wV Y qEN nO8 MH0 PY zeVN i3yb 2msNYY Wz G 2DC PoG 1Vb Bx e9oZ GcTU 3AZuEK bk p 6rN eTX 0DS Mc zd91 nbSV DKEkVa zI q NKU Qap NBP 5B 32Ey prwP FLvuPi wR P l1G TdQ BZE Aw 3d90 v8P5 CPAnX4 Yo 2 q7s yr5 BW8 Hc T7tM ioha BW9U4q rb u mEQ 6Xz MKR 2B REFX k3ZO MVMYSw 9S F 5ek q0m yNK Gn H0qi vlRA 18CbEz id O iuy ZZ6 kRo oJ kLQ0 Ewmz sKlld6 Kr K JmR xls 12K G2 bv8v LxfJ wrIcU6 Hx p q6p Fy7 Oim mo dXYt Kt0V VH22OC Aj f deT BAP vPl oK QzLE OQlq dpzxJ6 JI z Ujn TqY sQ4 BD QPW6 784x NUfsk0 aM 7 8qz MuL 9Mr Ac uVVK Y55n M7WqnB 2R C pGZ vHh WUN g9 3F2e RT8U umC62V H3 Z dJX LMS cca 1m xoOO 6oOL OVzfpO BO X 5Ev KuL z5s EW 8a9y otqk cKbDJN Us l pYM JpJ jOW Uy 2U4Y VKH6 kVC1Vx 1u v ykO yDs zo5 bz d36q WH1k J7Jtkg V1 J xqr Fnq mcU yZ JTp9 oFIc FAk0IT A9 3 SrL axO 9oU Z3 jG6f BRL1 iZ7ZE6 zj 8 G3M Hu8 6Ay jt 3flY cmTk jiTSYv CF t JLq cJP tN7 E3 POqG OKe0 3K3WV0 ep W XDQ C97 YSb AD ZUNp 81GF fCPbj3 iq E t0E NXy pLv fo Iz6z oFoF 9lkIun Xj Y yYL 52U bRB jx kQUS U9mm XtzIHO Cz 1 KH4 9ez 6Pz qW F223 C0Iz 3CsvuT R9 s VtQ CcM 1eo pD Py2l EEzL U0USJt Jb 9 zgy Gyf iQ4 fo Cx26 k4jL E0ula6 aS I rZQ HER 5HV CE BL55 WCtB 2LCmve TD z Vcp 7UR gI7 Qu FbFw 9VTx JwGrzs VW M 9sM JeJ Nd2 VG GFsi WuqC 3YxXoJ GK w Io7 1fg sGm 0P YFBz X8eX 7pf9GJ b1 o XUs 1q0 6KP Ls MucN ytQb L0Z0Qq m1 l SPj 9MT etk L6 KfsC 6Zob Yhc2qu Xy 9 GPm ZYj 1Go ei feJ3 pRAf n6Ypy6 jN s 4Y5 nSE pqN 4m Rmam AGfY HhSaBr Ls D THC SEl UyR Mh 66XU 7hNz pZVC5V nV 7 VjL 7kv WKf 7P 5hj6 t1vu gkLGdN X8 b gOX HWm 6W4 YE mxFG 4WaN EbGKsv 0p 4 OG0 Nrd uTe Za xNXq V4Bp mOdXIq 9a b PeD PbU Z4N Xt ohbY egCf xBNttE wc D YSD 637 jJ2 ms 6Ta1 J2xZ PtKnPw AX A tJA Rc8 n5d 93 TZi7 q6Wo nEDLwW Sz e Sue YFX 8cM hm Y6is 15pX aOYBbV fS C haL kBR Ks6 UO qG4j DVab fbdtny fi D BFI 7uh B39 FJ 6mYr CUUT f2X38J 43 K yZg 87i gFR 5R z1t3 jH9x lOg1h7 P7 W w8w jMJ qH3 l5 J5wU 8eH0 OogRCv L7 f JJg 1ug RfM XI GSuE Efbh 3hdNY3 x1 9 7jR qeP cdu sbEQ157} \end{align} Utilizing the cancellation law for the nonlinear convective term and the assumptions on the noise coefficients, we arrive at  \begin{align} \begin{split} &\EE\biggl[\sup_{0\leq s\leq t\wedge \eta_M}\Vert\uu(s,\cdot)\Vert_2^2 +\int_0^{t\wedge \eta_M} \Vert \nabla u\Vert_2^2 \,ds \biggr] \leq  C\EE\bigl[\Vert\uu_0\Vert_2^2\bigr] +C\EE\biggl[ \int_0^{t\wedge \eta_M} (\Vert\sigma(P_{\le k} u)\Vert_{\mathbb{L}^2}^2 +\Vert  u\Vert_{2}^2)\,ds \biggr] \\&\indeq \indeq \leq C(\EE\bigl[\Vert\uu_0\Vert_2^2\bigr]+1) +C\int_0^{t}\EE\biggl[  \sup_{0\leq r\leq s\wedge \eta_M}\Vert\uu(r,\cdot)\Vert_2^2  +\int_0^{s\wedge \eta_M} \Vert \nabla u\Vert_2^2 \,dr \biggr] \,ds \comma t\in (0,T], \end{split}    \llabel{3RL9 Qz5eqy wV Y qEN nO8 MH0 PY zeVN i3yb 2msNYY Wz G 2DC PoG 1Vb Bx e9oZ GcTU 3AZuEK bk p 6rN eTX 0DS Mc zd91 nbSV DKEkVa zI q NKU Qap NBP 5B 32Ey prwP FLvuPi wR P l1G TdQ BZE Aw 3d90 v8P5 CPAnX4 Yo 2 q7s yr5 BW8 Hc T7tM ioha BW9U4q rb u mEQ 6Xz MKR 2B REFX k3ZO MVMYSw 9S F 5ek q0m yNK Gn H0qi vlRA 18CbEz id O iuy ZZ6 kRo oJ kLQ0 Ewmz sKlld6 Kr K JmR xls 12K G2 bv8v LxfJ wrIcU6 Hx p q6p Fy7 Oim mo dXYt Kt0V VH22OC Aj f deT BAP vPl oK QzLE OQlq dpzxJ6 JI z Ujn TqY sQ4 BD QPW6 784x NUfsk0 aM 7 8qz MuL 9Mr Ac uVVK Y55n M7WqnB 2R C pGZ vHh WUN g9 3F2e RT8U umC62V H3 Z dJX LMS cca 1m xoOO 6oOL OVzfpO BO X 5Ev KuL z5s EW 8a9y otqk cKbDJN Us l pYM JpJ jOW Uy 2U4Y VKH6 kVC1Vx 1u v ykO yDs zo5 bz d36q WH1k J7Jtkg V1 J xqr Fnq mcU yZ JTp9 oFIc FAk0IT A9 3 SrL axO 9oU Z3 jG6f BRL1 iZ7ZE6 zj 8 G3M Hu8 6Ay jt 3flY cmTk jiTSYv CF t JLq cJP tN7 E3 POqG OKe0 3K3WV0 ep W XDQ C97 YSb AD ZUNp 81GF fCPbj3 iq E t0E NXy pLv fo Iz6z oFoF 9lkIun Xj Y yYL 52U bRB jx kQUS U9mm XtzIHO Cz 1 KH4 9ez 6Pz qW F223 C0Iz 3CsvuT R9 s VtQ CcM 1eo pD Py2l EEzL U0USJt Jb 9 zgy Gyf iQ4 fo Cx26 k4jL E0ula6 aS I rZQ HER 5HV CE BL55 WCtB 2LCmve TD z Vcp 7UR gI7 Qu FbFw 9VTx JwGrzs VW M 9sM JeJ Nd2 VG GFsi WuqC 3YxXoJ GK w Io7 1fg sGm 0P YFBz X8eX 7pf9GJ b1 o XUs 1q0 6KP Ls MucN ytQb L0Z0Qq m1 l SPj 9MT etk L6 KfsC 6Zob Yhc2qu Xy 9 GPm ZYj 1Go ei feJ3 pRAf n6Ypy6 jN s 4Y5 nSE pqN 4m Rmam AGfY HhSaBr Ls D THC SEl UyR Mh 66XU 7hNz pZVC5V nV 7 VjL 7kv WKf 7P 5hj6 t1vu gkLGdN X8 b gOX HWm 6W4 YE mxFG 4WaN EbGKsv 0p 4 OG0 Nrd uTe Za xNXq V4Bp mOdXIq 9a b PeD PbU Z4N Xt ohbY egCf xBNttE wc D YSD 637 jJ2 ms 6Ta1 J2xZ PtKnPw AX A tJA Rc8 n5d 93 TZi7 q6Wo nEDLwW Sz e Sue YFX 8cM hm Y6is 15pX aOYBbV fS C haL kBR Ks6 UO qG4j DVab fbdtny fi D BFI 7uh B39 FJ 6mYr CUUT f2X38J 43 K yZg 87i gFR 5R z1t3 jH9x lOg1h7 P7 W w8w jMJ qH3 l5 J5wU 8eH0 OogRCv L7 f JJg 1ug RfM XI GSuE Efbh 3hdNY3 x1 9 7jR qeP cdu sb fkuJ hEpw MvNBZV zL u qxJ 9b1 BTf Yk RJLj Oo1a EPIXvZ Aj v Xne fhK GsJ Ga wqjt U7r6 MPoydE H2 6 203 mGi JhF nT NCDB YlnP oKO6Pu XU 3 uu9 mSg 41v ma kk0E WUpS UtGBtD e6 d EQ158} \end{align} where $\eta_M:=\inf\{t>0: \Vert u\Vert _{2}\geq M \}$ was introduced to ensure the finiteness of the right side of the inequality. Then, using Gr\"{o}nwall's lemma, \begin{align*} \begin{split} \EE\biggl[\sup_{0\leq s\leq T\wedge \eta_M}\Vert\uu(s,\cdot)\Vert_2^2 +\int_0^{T\wedge \eta_M} \Vert \nabla u\Vert_2^2 \,ds \biggr] \leq C(\EE\bigl[\Vert\uu_0\Vert_2^2\bigr]+1).  \end{split} \end{align*} Note that $C$ does not depend on $M$. Thus, we may remove $\eta_M$ and obtain \eqref{ERTWERTHWRTWERTSGDGHCFGSDFGQSERWDFGDSFGHSDRGTEHDFGHDSFGSDGHGYUHDFGSDFASDFASGTWRT97} by sending $M\to\infty$. \end{proof} \par \startnewsection{Local existence of a strong solution}{sec03} To construct a local solution for the stochastic Navier-Stokes equations, we use the truncated system \begin{align} \begin{split} \label{ERTWERTHWRTWERTSGDGHCFGSDFGQSERWDFGDSFGHSDRGTEHDFGHDSFGSDGHGYUHDFGSDFASDFASGTWRT95} &\partial_t u^{(n)}  -  \Delta u^{(n)} + (\phin)^2 \pkn\mathcal{P}((u^{(n)}\cdot\nabla) \pkn u^{(n)}) = (\phin)^2 \pkn\sigma(\pkn u^{(n)})\dot{\WW}(t) \\ & \nabla\cdot \un = 0, \\& \un( 0)= \pkn \mathcal{P}\left(\varphi\left(\frac{\cdot}{n}\right) u_0 \right)  \Pas , \end{split} \end{align} where $\phin$ is as in after \eqref{ERTWERTHWRTWERTSGDGHCFGSDFGQSERWDFGDSFGHSDRGTEHDFGHDSFGSDGHGYUHDFGSDFASDFASGTWRT59}, $n\in\mathbb{N}$,  the index $k(n)$ is an increasing and unbounded integer-valued function of $n$, the initial datum $u_0$ is divergence-free,  and $\varphi\in C_c^{\infty}(\RR^d)$ is a cut-off function as in the beginning of Section~\ref{sec5}. In this section, we fix $d=3$, although other space dimensions can also be addressed using the same method. We shall specify the dependence of $k(n)$ on $n$ further below in this section. Clearly, $\varphi\big(\fractext{\cdot}{n}\big) u_0\in L^2(\RR^d)\cap L^p(\RR^d)$ if $u_0\in L^p(\RR^d)$, where $p>2$, and also $\Vert \varphi(\fractext{\cdot}{n}) u_0 \Vert_p\leq \Vert u_0 \Vert_p$. Moreover,    \begin{equation}    \left\Vert \mathcal{P}\left(\varphi\left(\frac{\cdot}{n}\right) u_0 \right)  - u_0 \right\Vert_p 
     =\left\Vert \mathcal{P}\left(\varphi\left(\frac{\cdot}{n}\right) u_0 \right)  - \mathcal{P}u_0 \right\Vert_p      \leq C      \left\Vert \left(\varphi\left(\frac{\cdot}{n}\right) -1 \right)   u_0 \right\Vert_p    \to 0 \text{ as }n\to \infty.    \llabel{Q BZE Aw 3d90 v8P5 CPAnX4 Yo 2 q7s yr5 BW8 Hc T7tM ioha BW9U4q rb u mEQ 6Xz MKR 2B REFX k3ZO MVMYSw 9S F 5ek q0m yNK Gn H0qi vlRA 18CbEz id O iuy ZZ6 kRo oJ kLQ0 Ewmz sKlld6 Kr K JmR xls 12K G2 bv8v LxfJ wrIcU6 Hx p q6p Fy7 Oim mo dXYt Kt0V VH22OC Aj f deT BAP vPl oK QzLE OQlq dpzxJ6 JI z Ujn TqY sQ4 BD QPW6 784x NUfsk0 aM 7 8qz MuL 9Mr Ac uVVK Y55n M7WqnB 2R C pGZ vHh WUN g9 3F2e RT8U umC62V H3 Z dJX LMS cca 1m xoOO 6oOL OVzfpO BO X 5Ev KuL z5s EW 8a9y otqk cKbDJN Us l pYM JpJ jOW Uy 2U4Y VKH6 kVC1Vx 1u v ykO yDs zo5 bz d36q WH1k J7Jtkg V1 J xqr Fnq mcU yZ JTp9 oFIc FAk0IT A9 3 SrL axO 9oU Z3 jG6f BRL1 iZ7ZE6 zj 8 G3M Hu8 6Ay jt 3flY cmTk jiTSYv CF t JLq cJP tN7 E3 POqG OKe0 3K3WV0 ep W XDQ C97 YSb AD ZUNp 81GF fCPbj3 iq E t0E NXy pLv fo Iz6z oFoF 9lkIun Xj Y yYL 52U bRB jx kQUS U9mm XtzIHO Cz 1 KH4 9ez 6Pz qW F223 C0Iz 3CsvuT R9 s VtQ CcM 1eo pD Py2l EEzL U0USJt Jb 9 zgy Gyf iQ4 fo Cx26 k4jL E0ula6 aS I rZQ HER 5HV CE BL55 WCtB 2LCmve TD z Vcp 7UR gI7 Qu FbFw 9VTx JwGrzs VW M 9sM JeJ Nd2 VG GFsi WuqC 3YxXoJ GK w Io7 1fg sGm 0P YFBz X8eX 7pf9GJ b1 o XUs 1q0 6KP Ls MucN ytQb L0Z0Qq m1 l SPj 9MT etk L6 KfsC 6Zob Yhc2qu Xy 9 GPm ZYj 1Go ei feJ3 pRAf n6Ypy6 jN s 4Y5 nSE pqN 4m Rmam AGfY HhSaBr Ls D THC SEl UyR Mh 66XU 7hNz pZVC5V nV 7 VjL 7kv WKf 7P 5hj6 t1vu gkLGdN X8 b gOX HWm 6W4 YE mxFG 4WaN EbGKsv 0p 4 OG0 Nrd uTe Za xNXq V4Bp mOdXIq 9a b PeD PbU Z4N Xt ohbY egCf xBNttE wc D YSD 637 jJ2 ms 6Ta1 J2xZ PtKnPw AX A tJA Rc8 n5d 93 TZi7 q6Wo nEDLwW Sz e Sue YFX 8cM hm Y6is 15pX aOYBbV fS C haL kBR Ks6 UO qG4j DVab fbdtny fi D BFI 7uh B39 FJ 6mYr CUUT f2X38J 43 K yZg 87i gFR 5R z1t3 jH9x lOg1h7 P7 W w8w jMJ qH3 l5 J5wU 8eH0 OogRCv L7 f JJg 1ug RfM XI GSuE Efbh 3hdNY3 x1 9 7jR qeP cdu sb fkuJ hEpw MvNBZV zL u qxJ 9b1 BTf Yk RJLj Oo1a EPIXvZ Aj v Xne fhK GsJ Ga wqjt U7r6 MPoydE H2 6 203 mGi JhF nT NCDB YlnP oKO6Pu XU 3 uu9 mSg 41v ma kk0E WUpS UtGBtD e6 d Kdx ZNT FuT i1 fMcM hq7P Ovf0hg Hl 8 fqv I3R K39 fn 9MaC Zgow 6e1iXj KC 5 lHO lpG pkK Xd Dxtz 0HxE fSMjXY L8 F vh7 dmJ kE8 QA KDo1 FqML HOZ2iL 9i I m3L Kva YiN K9 sb48 NxwEQ116}   \end{equation} By Theorem~\ref{T03}, the system \eqref{ERTWERTHWRTWERTSGDGHCFGSDFGQSERWDFGDSFGHSDRGTEHDFGHDSFGSDGHGYUHDFGSDFASDFASGTWRT95} has a unique strong solution $\un$ in $L^p(\Omega; C([0,T], L^p))$ for every~$n\in\mathbb{N}$ and every $T>0$. However, the energy estimate \eqref{ERTWERTHWRTWERTSGDGHCFGSDFGQSERWDFGDSFGHSDRGTEHDFGHDSFGSDGHGYUHDFGSDFASDFASGTWRT57} resulting from the proof depends on~$n$, which is the main difficulty when one passes with $n$ to infinity. Below, we show that the estimate is uniform up to a properly defined stopping time if $u_0$ has a deterministic upper bound. \par Let $K\geq1$ and assume that $\Vert u_0\Vert_{p}\leq K$. For the next theorem, we introduce a constant $M_0\geq 1$ such that    \begin{equation}    \sum_j         \left\Vert \pkn             \left(\varphi               \left(\frac{\cdot}{n}               \right) u_{0,j}             \right)         \right\Vert_{p}^p         +         \left\Vert            \pkn \left(\varphi\left(\frac{\cdot}{n}\right) u_0 \right)         \right\Vert_{p}^p     \leq     \frac14    M_0K^p    ,    \llabel{d6 Kr K JmR xls 12K G2 bv8v LxfJ wrIcU6 Hx p q6p Fy7 Oim mo dXYt Kt0V VH22OC Aj f deT BAP vPl oK QzLE OQlq dpzxJ6 JI z Ujn TqY sQ4 BD QPW6 784x NUfsk0 aM 7 8qz MuL 9Mr Ac uVVK Y55n M7WqnB 2R C pGZ vHh WUN g9 3F2e RT8U umC62V H3 Z dJX LMS cca 1m xoOO 6oOL OVzfpO BO X 5Ev KuL z5s EW 8a9y otqk cKbDJN Us l pYM JpJ jOW Uy 2U4Y VKH6 kVC1Vx 1u v ykO yDs zo5 bz d36q WH1k J7Jtkg V1 J xqr Fnq mcU yZ JTp9 oFIc FAk0IT A9 3 SrL axO 9oU Z3 jG6f BRL1 iZ7ZE6 zj 8 G3M Hu8 6Ay jt 3flY cmTk jiTSYv CF t JLq cJP tN7 E3 POqG OKe0 3K3WV0 ep W XDQ C97 YSb AD ZUNp 81GF fCPbj3 iq E t0E NXy pLv fo Iz6z oFoF 9lkIun Xj Y yYL 52U bRB jx kQUS U9mm XtzIHO Cz 1 KH4 9ez 6Pz qW F223 C0Iz 3CsvuT R9 s VtQ CcM 1eo pD Py2l EEzL U0USJt Jb 9 zgy Gyf iQ4 fo Cx26 k4jL E0ula6 aS I rZQ HER 5HV CE BL55 WCtB 2LCmve TD z Vcp 7UR gI7 Qu FbFw 9VTx JwGrzs VW M 9sM JeJ Nd2 VG GFsi WuqC 3YxXoJ GK w Io7 1fg sGm 0P YFBz X8eX 7pf9GJ b1 o XUs 1q0 6KP Ls MucN ytQb L0Z0Qq m1 l SPj 9MT etk L6 KfsC 6Zob Yhc2qu Xy 9 GPm ZYj 1Go ei feJ3 pRAf n6Ypy6 jN s 4Y5 nSE pqN 4m Rmam AGfY HhSaBr Ls D THC SEl UyR Mh 66XU 7hNz pZVC5V nV 7 VjL 7kv WKf 7P 5hj6 t1vu gkLGdN X8 b gOX HWm 6W4 YE mxFG 4WaN EbGKsv 0p 4 OG0 Nrd uTe Za xNXq V4Bp mOdXIq 9a b PeD PbU Z4N Xt ohbY egCf xBNttE wc D YSD 637 jJ2 ms 6Ta1 J2xZ PtKnPw AX A tJA Rc8 n5d 93 TZi7 q6Wo nEDLwW Sz e Sue YFX 8cM hm Y6is 15pX aOYBbV fS C haL kBR Ks6 UO qG4j DVab fbdtny fi D BFI 7uh B39 FJ 6mYr CUUT f2X38J 43 K yZg 87i gFR 5R z1t3 jH9x lOg1h7 P7 W w8w jMJ qH3 l5 J5wU 8eH0 OogRCv L7 f JJg 1ug RfM XI GSuE Efbh 3hdNY3 x1 9 7jR qeP cdu sb fkuJ hEpw MvNBZV zL u qxJ 9b1 BTf Yk RJLj Oo1a EPIXvZ Aj v Xne fhK GsJ Ga wqjt U7r6 MPoydE H2 6 203 mGi JhF nT NCDB YlnP oKO6Pu XU 3 uu9 mSg 41v ma kk0E WUpS UtGBtD e6 d Kdx ZNT FuT i1 fMcM hq7P Ovf0hg Hl 8 fqv I3R K39 fn 9MaC Zgow 6e1iXj KC 5 lHO lpG pkK Xd Dxtz 0HxE fSMjXY L8 F vh7 dmJ kE8 QA KDo1 FqML HOZ2iL 9i I m3L Kva YiN K9 sb48 NxwY NR0nx2 t5 b WCk x2a 31k a8 fUIa RGzr 7oigRX 5s m 9PQ 7Sr 5St ZE Ymp8 VIWS hdzgDI 9v R F5J 81x 33n Ne fjBT VvGP vGsxQh Al G Fbe 1bQ i6J ap OJJa ceGq 1vvb8r F2 F 3M6 8eD lEQ117}   \end{equation} for all~$n\in\NNp$, which is possible due to~\eqref{ERTWERTHWRTWERTSGDGHCFGSDFGQSERWDFGDSFGHSDRGTEHDFGHDSFGSDGHGYUHDFGSDFASDFASGTWRT08}. Then, for every constant $M\geq M_0$, we define a stopping time   \begin{equation}     \tau^n_M := \inf\biggl\{t>0: \sup_{[0, t]}\Vert \un\Vert _{p}^p + \int_{0}^{t} \Vert \un(s)\Vert_{3p}^p \,ds\geq MK^p \biggr\}    .    \label{ERTWERTHWRTWERTSGDGHCFGSDFGQSERWDFGDSFGHSDRGTEHDFGHDSFGSDGHGYUHDFGSDFASDFASGTWRT96} \end{equation} This definition ensures the (strict) positivity of~$\tau^n_M$, while the distinction between $M$ and $M_0$ is needed in Lemma~\ref{L07} below. Let $T>0$ be a arbitrarily prescribed deterministic time. It is shown below that the $L^p$-type energy corresponding to $\un$ up to $\tau^n_M$ is uniform with respect to~$n$. \par \cole \begin{Theorem} \label{T04} Let~$p>d=3$ and~$K\geq1$, and suppose that $\Vert u_{0}\Vert_{p} \leq K$. Then, the unique strong solution $\un$ of the  system~\eqref{ERTWERTHWRTWERTSGDGHCFGSDFGQSERWDFGDSFGHSDRGTEHDFGHDSFGSDGHGYUHDFGSDFASDFASGTWRT95} satisfies   \begin{equation}   \label{ERTWERTHWRTWERTSGDGHCFGSDFGQSERWDFGDSFGHSDRGTEHDFGHDSFGSDGHGYUHDFGSDFASDFASGTWRT99}   \EE\biggl[\sup_{[0, T\wedge \tau^n_M]}\Vert \un \Vert_{p}^p + \int_{0}^{T\wedge \tau^n_M} \int_{\RR^\dd}\sum_j |\nabla |u_j^{(n)}|^{p/2}|^2 \,dx ds\biggr] \leq C K^p,    \end{equation} where $C$ depends on $p$, $d$, $M$, and $T$, but is independent of~$n$ and~$K$. Moreover,    \begin{equation}   \label{ERTWERTHWRTWERTSGDGHCFGSDFGQSERWDFGDSFGHSDRGTEHDFGHDSFGSDGHGYUHDFGSDFASDFASGTWRT98}   \lim_{t\rightarrow0} \sup_n \PP\biggl(\sup_{[0, t\wedge \tau^n_M]}\Vert \un \Vert_{p}^p + \int_{0}^{t\wedge \tau^n_M} \int_{\RR^\dd}\sum_j |\nabla (|u_j^{(n)}|^{p/2})|^2 \,dx ds\geq M_0K^p \biggr) = 0.   \end{equation} \end{Theorem} \colb \par \begin{proof}[Proof of Theorem~\ref{T04}] To derive a uniform energy bound, we apply Theorem~\ref{T02}. First, we have   \begin{align}   \begin{split}   &   \EE\biggl[\int_0^{t\wedge \tau^n_M}\Vert    (\phin)^2 \pkn\mathcal{P}(\un_i \pkn\un)   \Vert_{q}^p \,ds   \biggr]   \leq C   \EE\biggl[\int_0^{t\wedge \tau^n_M}\Vert    (\phin)^2 \mathcal{P}(\un_i \pkn \un)   \Vert_{q}^p \,ds   \biggr]   \\&\indeq   \leq   C\EE\biggl[\int_0^{t\wedge \tau^n_M}   (\phin)^2   \Vert    \un_i    \Vert_{p}^p    \Vert    \pkn \un   \Vert_{l}^p    \,ds   \biggr]   \leq   C\EE\biggl[\int_0^{t\wedge \tau^n_M}   (\phin)^2   \Vert    \un_i    \Vert_{p}^p    \Vert     \un   \Vert_{l}^p    \,ds   \biggr]   \\&\indeq   \leq   C \EE\biggl[\int_0^{t\wedge \tau^n_M}   (\phin)^2   (\Vert   \un   \Vert_{3p}^p +   \Vert    \un   \Vert_{p}^p )   \,ds   \biggr]   \leq   C_{T, M}K^p    \comma  t\in (0,T]\comma i=1,2,3    ,   \end{split}   \label{ERTWERTHWRTWERTSGDGHCFGSDFGQSERWDFGDSFGHSDRGTEHDFGHDSFGSDGHGYUHDFGSDFASDFASGTWRT100}   \end{align} where we chose   \begin{equation}     \frac{3p}{p+1}< q< \frac{3p}{4}     \commaone     \frac{1}{p}+ \frac{1}{l}=\frac{1}{q},     \qquad{}\text{and}\qquad{}      p< l< 3p,     \label{ERTWERTHWRTWERTSGDGHCFGSDFGQSERWDFGDSFGHSDRGTEHDFGHDSFGSDGHGYUHDFGSDFASDFASGTWRT101}   \end{equation} which is possible if~$p>d=3$. For the stochastic term in \eqref{ERTWERTHWRTWERTSGDGHCFGSDFGQSERWDFGDSFGHSDRGTEHDFGHDSFGSDGHGYUHDFGSDFASDFASGTWRT95}, we have \begin{align} \begin{split} &\EE\biggl[ \int_0^{t\wedge \tau^n_M}\int_{\RR^\dd}   (\phin)^2 \Vert  \pkn\sigma( P_{\leq k(n)}\un) \Vert_{l^2( \mathcal{H},\RR^\dd)}^p\,dx\,ds \biggr] \leq C \EE\biggl[ \int_0^{t\wedge \tau^n_M}   (\phin)^2(\Vert  \un \Vert_{(3p/2)-}^{2p}+1) \,ds \biggr] \\&\indeq\quad \leq C \EE\biggl[\int_0^{t\wedge \tau^n_M}   (\phin)^2   (\Vert    \un   \Vert_{p}^{p+}   \Vert    \un   \Vert_{3p}^{p-}   +1) \,ds\biggr] \leq C_{T,M}K^p \comma t\in (0,T] , \end{split} \llabel{uVVK Y55n M7WqnB 2R C pGZ vHh WUN g9 3F2e RT8U umC62V H3 Z dJX LMS cca 1m xoOO 6oOL OVzfpO BO X 5Ev KuL z5s EW 8a9y otqk cKbDJN Us l pYM JpJ jOW Uy 2U4Y VKH6 kVC1Vx 1u v ykO yDs zo5 bz d36q WH1k J7Jtkg V1 J xqr Fnq mcU yZ JTp9 oFIc FAk0IT A9 3 SrL axO 9oU Z3 jG6f BRL1 iZ7ZE6 zj 8 G3M Hu8 6Ay jt 3flY cmTk jiTSYv CF t JLq cJP tN7 E3 POqG OKe0 3K3WV0 ep W XDQ C97 YSb AD ZUNp 81GF fCPbj3 iq E t0E NXy pLv fo Iz6z oFoF 9lkIun Xj Y yYL 52U bRB jx kQUS U9mm XtzIHO Cz 1 KH4 9ez 6Pz qW F223 C0Iz 3CsvuT R9 s VtQ CcM 1eo pD Py2l EEzL U0USJt Jb 9 zgy Gyf iQ4 fo Cx26 k4jL E0ula6 aS I rZQ HER 5HV CE BL55 WCtB 2LCmve TD z Vcp 7UR gI7 Qu FbFw 9VTx JwGrzs VW M 9sM JeJ Nd2 VG GFsi WuqC 3YxXoJ GK w Io7 1fg sGm 0P YFBz X8eX 7pf9GJ b1 o XUs 1q0 6KP Ls MucN ytQb L0Z0Qq m1 l SPj 9MT etk L6 KfsC 6Zob Yhc2qu Xy 9 GPm ZYj 1Go ei feJ3 pRAf n6Ypy6 jN s 4Y5 nSE pqN 4m Rmam AGfY HhSaBr Ls D THC SEl UyR Mh 66XU 7hNz pZVC5V nV 7 VjL 7kv WKf 7P 5hj6 t1vu gkLGdN X8 b gOX HWm 6W4 YE mxFG 4WaN EbGKsv 0p 4 OG0 Nrd uTe Za xNXq V4Bp mOdXIq 9a b PeD PbU Z4N Xt ohbY egCf xBNttE wc D YSD 637 jJ2 ms 6Ta1 J2xZ PtKnPw AX A tJA Rc8 n5d 93 TZi7 q6Wo nEDLwW Sz e Sue YFX 8cM hm Y6is 15pX aOYBbV fS C haL kBR Ks6 UO qG4j DVab fbdtny fi D BFI 7uh B39 FJ 6mYr CUUT f2X38J 43 K yZg 87i gFR 5R z1t3 jH9x lOg1h7 P7 W w8w jMJ qH3 l5 J5wU 8eH0 OogRCv L7 f JJg 1ug RfM XI GSuE Efbh 3hdNY3 x1 9 7jR qeP cdu sb fkuJ hEpw MvNBZV zL u qxJ 9b1 BTf Yk RJLj Oo1a EPIXvZ Aj v Xne fhK GsJ Ga wqjt U7r6 MPoydE H2 6 203 mGi JhF nT NCDB YlnP oKO6Pu XU 3 uu9 mSg 41v ma kk0E WUpS UtGBtD e6 d Kdx ZNT FuT i1 fMcM hq7P Ovf0hg Hl 8 fqv I3R K39 fn 9MaC Zgow 6e1iXj KC 5 lHO lpG pkK Xd Dxtz 0HxE fSMjXY L8 F vh7 dmJ kE8 QA KDo1 FqML HOZ2iL 9i I m3L Kva YiN K9 sb48 NxwY NR0nx2 t5 b WCk x2a 31k a8 fUIa RGzr 7oigRX 5s m 9PQ 7Sr 5St ZE Ymp8 VIWS hdzgDI 9v R F5J 81x 33n Ne fjBT VvGP vGsxQh Al G Fbe 1bQ i6J ap OJJa ceGq 1vvb8r F2 F 3M6 8eD lzG tX tVm5 y14v mwIXa2 OG Y hxU sXJ 0qg l5 ZGAt HPZd oDWrSb BS u NKi 6KW gr3 9s 9tc7 WM4A ws1PzI 5c C O7Z 8y9 lMT LA dwhz Mxz9 hjlWHj bJ 5 CqM jht y9l Mn 4rc7 6Amk KJimvH EQ102} \end{align} which after applying Theorem~\ref{T02} concludes the proof of~\eqref{ERTWERTHWRTWERTSGDGHCFGSDFGQSERWDFGDSFGHSDRGTEHDFGHDSFGSDGHGYUHDFGSDFASDFASGTWRT99}. To show \eqref{ERTWERTHWRTWERTSGDGHCFGSDFGQSERWDFGDSFGHSDRGTEHDFGHDSFGSDGHGYUHDFGSDFASDFASGTWRT98}, we apply It\^o's formula to~\eqref{ERTWERTHWRTWERTSGDGHCFGSDFGQSERWDFGDSFGHSDRGTEHDFGHDSFGSDGHGYUHDFGSDFASDFASGTWRT95} componentwise on $[0, T\wedge \tau^n_M]$, obtaining   \begin{align}   \begin{split}   \Vert\un_j(t)\Vert_{p}^{p}   &=   \left\Vert      \pkn \left(            \varphi              \left(\frac{\cdot}{n}              \right) u_{0,j}           \right)    \right\Vert_{p}^{p}   -\frac{4(p-1)}{p} \int_0^{t\wedge \tau^n_M}\int_{\RR^\dd} |\nabla (|\un_j|^{p/2})|^2\,dx ds   \\&\indeq +p\int_0^{t\wedge \tau^n_M}(\phin)^2 \int_{\RR^\dd} |\un_j|^{p-2}\un_j\pkn\bigl(\mathcal{P}\bigl(( \un\cdot \nabla)\pkn\un\bigr)\bigr)_j\,dx ds \\&\indeq + p\int_0^{t\wedge \tau^n_M}(\phin)^2 \int_{\RR^\dd} |\un_j|^{p-2}\un_j \pkn\sigma_j(\pkn\un) \,dxd\WW_s \\&\indeq +\frac{p(p-1)}{2}\int_0^{t\wedge \tau^n_M}(\phin)^4 \int_{\RR^\dd} |\un_j|^{p-2}\Vert \pkn\sigma_j(\pkn\un)\Vert_{l^2}^2\,dx ds \comma j=1,2,3 . \end{split}  \label{ERTWERTHWRTWERTSGDGHCFGSDFGQSERWDFGDSFGHSDRGTEHDFGHDSFGSDGHGYUHDFGSDFASDFASGTWRT103} \end{align} For the convective term, we recall \eqref{ERTWERTHWRTWERTSGDGHCFGSDFGQSERWDFGDSFGHSDRGTEHDFGHDSFGSDGHGYUHDFGSDFASDFASGTWRT43}--\eqref{ERTWERTHWRTWERTSGDGHCFGSDFGQSERWDFGDSFGHSDRGTEHDFGHDSFGSDGHGYUHDFGSDFASDFASGTWRT44} and obtain   \begin{align}   \begin{split}   &\int_0^{t\wedge\tau_M^n}(\phin)^2   \left|   \int_{\RR^\dd}  |\un_j|^{p-2}\un_j\pkn\bigl(\mathcal{P}\bigl(( \un\cdot \nabla)\pkn\un\bigr)\bigr)_j\,dx   \right|   \,ds   \\&\indeq   \leq    \frac{1}{p^2}   \int_{0}^{t\wedge\tau_M^n}   \Vert   \nabla (|\un_j|^{p/2})   \Vert_2^2   \,ds   + C   \int_{0}^{t\wedge\tau_M^n}    \biggl(   \Vert\un_j\Vert_p^p   + \sum_i   \Vert  (\phin)^2 \pkn\mathcal{P}(\un_i \pkn \un)\Vert_{ q}^p    \biggr)   \,ds     \\&\indeq   \leq    \frac{1}{p^2}   \int_{0}^{t\wedge\tau_M^n}   \Vert   \nabla (|\un_j|^{p/2})   \Vert_2^2   \,ds   + C   \int_{0}^{t\wedge\tau_M^n}    (\Vert\un\Vert_p^p   +    \Vert\un\Vert_{3p}^{\theta p})   \,ds   \comma j=1,2,3   ,   \end{split}    \llabel{kO yDs zo5 bz d36q WH1k J7Jtkg V1 J xqr Fnq mcU yZ JTp9 oFIc FAk0IT A9 3 SrL axO 9oU Z3 jG6f BRL1 iZ7ZE6 zj 8 G3M Hu8 6Ay jt 3flY cmTk jiTSYv CF t JLq cJP tN7 E3 POqG OKe0 3K3WV0 ep W XDQ C97 YSb AD ZUNp 81GF fCPbj3 iq E t0E NXy pLv fo Iz6z oFoF 9lkIun Xj Y yYL 52U bRB jx kQUS U9mm XtzIHO Cz 1 KH4 9ez 6Pz qW F223 C0Iz 3CsvuT R9 s VtQ CcM 1eo pD Py2l EEzL U0USJt Jb 9 zgy Gyf iQ4 fo Cx26 k4jL E0ula6 aS I rZQ HER 5HV CE BL55 WCtB 2LCmve TD z Vcp 7UR gI7 Qu FbFw 9VTx JwGrzs VW M 9sM JeJ Nd2 VG GFsi WuqC 3YxXoJ GK w Io7 1fg sGm 0P YFBz X8eX 7pf9GJ b1 o XUs 1q0 6KP Ls MucN ytQb L0Z0Qq m1 l SPj 9MT etk L6 KfsC 6Zob Yhc2qu Xy 9 GPm ZYj 1Go ei feJ3 pRAf n6Ypy6 jN s 4Y5 nSE pqN 4m Rmam AGfY HhSaBr Ls D THC SEl UyR Mh 66XU 7hNz pZVC5V nV 7 VjL 7kv WKf 7P 5hj6 t1vu gkLGdN X8 b gOX HWm 6W4 YE mxFG 4WaN EbGKsv 0p 4 OG0 Nrd uTe Za xNXq V4Bp mOdXIq 9a b PeD PbU Z4N Xt ohbY egCf xBNttE wc D YSD 637 jJ2 ms 6Ta1 J2xZ PtKnPw AX A tJA Rc8 n5d 93 TZi7 q6Wo nEDLwW Sz e Sue YFX 8cM hm Y6is 15pX aOYBbV fS C haL kBR Ks6 UO qG4j DVab fbdtny fi D BFI 7uh B39 FJ 6mYr CUUT f2X38J 43 K yZg 87i gFR 5R z1t3 jH9x lOg1h7 P7 W w8w jMJ qH3 l5 J5wU 8eH0 OogRCv L7 f JJg 1ug RfM XI GSuE Efbh 3hdNY3 x1 9 7jR qeP cdu sb fkuJ hEpw MvNBZV zL u qxJ 9b1 BTf Yk RJLj Oo1a EPIXvZ Aj v Xne fhK GsJ Ga wqjt U7r6 MPoydE H2 6 203 mGi JhF nT NCDB YlnP oKO6Pu XU 3 uu9 mSg 41v ma kk0E WUpS UtGBtD e6 d Kdx ZNT FuT i1 fMcM hq7P Ovf0hg Hl 8 fqv I3R K39 fn 9MaC Zgow 6e1iXj KC 5 lHO lpG pkK Xd Dxtz 0HxE fSMjXY L8 F vh7 dmJ kE8 QA KDo1 FqML HOZ2iL 9i I m3L Kva YiN K9 sb48 NxwY NR0nx2 t5 b WCk x2a 31k a8 fUIa RGzr 7oigRX 5s m 9PQ 7Sr 5St ZE Ymp8 VIWS hdzgDI 9v R F5J 81x 33n Ne fjBT VvGP vGsxQh Al G Fbe 1bQ i6J ap OJJa ceGq 1vvb8r F2 F 3M6 8eD lzG tX tVm5 y14v mwIXa2 OG Y hxU sXJ 0qg l5 ZGAt HPZd oDWrSb BS u NKi 6KW gr3 9s 9tc7 WM4A ws1PzI 5c C O7Z 8y9 lMT LA dwhz Mxz9 hjlWHj bJ 5 CqM jht y9l Mn 4rc7 6Amk KJimvH 9r O tbc tCK rsi B0 4cFV Dl1g cvfWh6 5n x y9Z S4W Pyo QB yr3v fBkj TZKtEZ 7r U fdM icd yCV qn D036 HJWM tYfL9f yX x O7m IcF E1O uL QsAQ NfWv 6kV8Im 7Q 6 GsX NCV 0YP oC jnWEQ104}   \end{align} where $C$ is independent of $n$ and~$M$, and $\theta\in(0,1)$ (cf.~\eqref{ERTWERTHWRTWERTSGDGHCFGSDFGQSERWDFGDSFGHSDRGTEHDFGHDSFGSDGHGYUHDFGSDFASDFASGTWRT100}--\eqref{ERTWERTHWRTWERTSGDGHCFGSDFGQSERWDFGDSFGHSDRGTEHDFGHDSFGSDGHGYUHDFGSDFASDFASGTWRT101}). Note that $\int_{0}^{t\wedge\tau_M^n} \Vert \nabla (|\un_j|^{p/2}) \Vert_2^2 \,ds$ is finite almost surely due to~\eqref{ERTWERTHWRTWERTSGDGHCFGSDFGQSERWDFGDSFGHSDRGTEHDFGHDSFGSDGHGYUHDFGSDFASDFASGTWRT99}. For the quadratic variation, we have    \begin{align}   \begin{split}   &\int_0^{t\wedge\tau_M^n}(\phin)^4\int_{\RR^\dd} |\un_j|^{p-2}\Vert \pkn\sigma_j(\pkn\un)\Vert_{l^2}^2\,dx ds    \\&\indeq   \leq    C\int_0^{t\wedge\tau_M^n}\phin   (\Vert\un_j\Vert_p^{p}+   \Vert\un\Vert_{(3p/2)-}^{2p}+1)\,ds   \\&\indeq   \leq   C\int_0^{t\wedge\tau_M^n}\phin   (\Vert\un_j\Vert_p^{p}+   \Vert\un\Vert_{p}^{p+}\Vert\un\Vert_{3p}^{p-}+1)\,ds     \\&\indeq   \leq C   \int_{0}^{t\wedge\tau_M^n}    (\Vert\un\Vert_p^p   +    \Vert\un\Vert_{3p}^{p-}+1)   \,ds   \comma j=1,2,3   .   \end{split}    \label{ERTWERTHWRTWERTSGDGHCFGSDFGQSERWDFGDSFGHSDRGTEHDFGHDSFGSDGHGYUHDFGSDFASDFASGTWRT105}   \end{align}
From \eqref{ERTWERTHWRTWERTSGDGHCFGSDFGQSERWDFGDSFGHSDRGTEHDFGHDSFGSDGHGYUHDFGSDFASDFASGTWRT103} and the definition of $\tau_M^n$, it follows that   \begin{align}   \begin{split}    &\sum_{j}\Vert\un_j\Vert_{p}^{p}    +     \sum_{j} \int_0^{t\wedge\tau_M^n}    \int_{\RR^\dd} |\nabla (|\un_j|^{p/2})|^2    \,dx ds     \\&\indeq   \leq     \sum_{j}\biggl\Vert \pkn \left(\varphi\left(\frac{\cdot}{n}\right) u_{0,j} \right) \biggr\Vert_{p}^{p}      +Ct(MK^p+1)      +      C\int_0^{t\wedge\tau_M^n}    (\Vert\un\Vert_{3p}^{p-}+\Vert\un\Vert_{3p}^{\theta p})\,ds    \\ & \indeq\indeq   + C\biggl| \sum_{j}   \int_0^{t\wedge\tau_M^n}\int_{\RR^\dd} |\un_j|^{p-2}\un_j \pkn\sigma_j(\pkn\un)\,dxd\WW_s      \biggr|   ,   \end{split}    \llabel{ 3K3WV0 ep W XDQ C97 YSb AD ZUNp 81GF fCPbj3 iq E t0E NXy pLv fo Iz6z oFoF 9lkIun Xj Y yYL 52U bRB jx kQUS U9mm XtzIHO Cz 1 KH4 9ez 6Pz qW F223 C0Iz 3CsvuT R9 s VtQ CcM 1eo pD Py2l EEzL U0USJt Jb 9 zgy Gyf iQ4 fo Cx26 k4jL E0ula6 aS I rZQ HER 5HV CE BL55 WCtB 2LCmve TD z Vcp 7UR gI7 Qu FbFw 9VTx JwGrzs VW M 9sM JeJ Nd2 VG GFsi WuqC 3YxXoJ GK w Io7 1fg sGm 0P YFBz X8eX 7pf9GJ b1 o XUs 1q0 6KP Ls MucN ytQb L0Z0Qq m1 l SPj 9MT etk L6 KfsC 6Zob Yhc2qu Xy 9 GPm ZYj 1Go ei feJ3 pRAf n6Ypy6 jN s 4Y5 nSE pqN 4m Rmam AGfY HhSaBr Ls D THC SEl UyR Mh 66XU 7hNz pZVC5V nV 7 VjL 7kv WKf 7P 5hj6 t1vu gkLGdN X8 b gOX HWm 6W4 YE mxFG 4WaN EbGKsv 0p 4 OG0 Nrd uTe Za xNXq V4Bp mOdXIq 9a b PeD PbU Z4N Xt ohbY egCf xBNttE wc D YSD 637 jJ2 ms 6Ta1 J2xZ PtKnPw AX A tJA Rc8 n5d 93 TZi7 q6Wo nEDLwW Sz e Sue YFX 8cM hm Y6is 15pX aOYBbV fS C haL kBR Ks6 UO qG4j DVab fbdtny fi D BFI 7uh B39 FJ 6mYr CUUT f2X38J 43 K yZg 87i gFR 5R z1t3 jH9x lOg1h7 P7 W w8w jMJ qH3 l5 J5wU 8eH0 OogRCv L7 f JJg 1ug RfM XI GSuE Efbh 3hdNY3 x1 9 7jR qeP cdu sb fkuJ hEpw MvNBZV zL u qxJ 9b1 BTf Yk RJLj Oo1a EPIXvZ Aj v Xne fhK GsJ Ga wqjt U7r6 MPoydE H2 6 203 mGi JhF nT NCDB YlnP oKO6Pu XU 3 uu9 mSg 41v ma kk0E WUpS UtGBtD e6 d Kdx ZNT FuT i1 fMcM hq7P Ovf0hg Hl 8 fqv I3R K39 fn 9MaC Zgow 6e1iXj KC 5 lHO lpG pkK Xd Dxtz 0HxE fSMjXY L8 F vh7 dmJ kE8 QA KDo1 FqML HOZ2iL 9i I m3L Kva YiN K9 sb48 NxwY NR0nx2 t5 b WCk x2a 31k a8 fUIa RGzr 7oigRX 5s m 9PQ 7Sr 5St ZE Ymp8 VIWS hdzgDI 9v R F5J 81x 33n Ne fjBT VvGP vGsxQh Al G Fbe 1bQ i6J ap OJJa ceGq 1vvb8r F2 F 3M6 8eD lzG tX tVm5 y14v mwIXa2 OG Y hxU sXJ 0qg l5 ZGAt HPZd oDWrSb BS u NKi 6KW gr3 9s 9tc7 WM4A ws1PzI 5c C O7Z 8y9 lMT LA dwhz Mxz9 hjlWHj bJ 5 CqM jht y9l Mn 4rc7 6Amk KJimvH 9r O tbc tCK rsi B0 4cFV Dl1g cvfWh6 5n x y9Z S4W Pyo QB yr3v fBkj TZKtEZ 7r U fdM icd yCV qn D036 HJWM tYfL9f yX x O7m IcF E1O uL QsAQ NfWv 6kV8Im 7Q 6 GsX NCV 0YP oC jnWn 6L25 qUMTe7 1v a hnH DAo XAb Tc zhPc fjrj W5M5G0 nz N M5T nlJ WOP Lh M6U2 ZFxw pg4Nej P8 U Q09 JX9 n7S kE WixE Rwgy Fvttzp 4A s v5F Tnn MzL Vh FUn5 6tFY CxZ1Bz Q3 E TfD EQ106}   \end{align} which implies    \begin{align}   \begin{split}   &\sup_{[0, t\wedge \tau_M^n]}\Vert\un\Vert_{p}^{p}   +    \sum_{j} \int_0^{t\wedge\tau_M^n}   \int_{\RR^\dd} |\nabla (|\un_j|^{p/2})|^2   \,dx ds    \\&\indeq   \leq    2\sum_{j}\left\Vert                \pkn \left(\varphi\left(\frac{\cdot}{n}\right) u_{0,j} \right)              \right\Vert_{p}^{p}   +Ct(MK^p+1)   +   C\int_0^{t\wedge\tau_M^n}   (\Vert\un\Vert_{3p}^{p-}+\Vert\un\Vert_{3p}^{\theta p})\,ds   \\ & \indeq\indeq     + C\sup_{[0, t\wedge \tau_M^n]}\biggl| \sum_{j}   \int_0^{s}(\phin)^2\int_{\RR^\dd} |\un_j|^{p-2}\un_j \pkn\sigma_j(\pkn \un)\,dxd\WW_r   \biggr|   .   \end{split}    \llabel{o pD Py2l EEzL U0USJt Jb 9 zgy Gyf iQ4 fo Cx26 k4jL E0ula6 aS I rZQ HER 5HV CE BL55 WCtB 2LCmve TD z Vcp 7UR gI7 Qu FbFw 9VTx JwGrzs VW M 9sM JeJ Nd2 VG GFsi WuqC 3YxXoJ GK w Io7 1fg sGm 0P YFBz X8eX 7pf9GJ b1 o XUs 1q0 6KP Ls MucN ytQb L0Z0Qq m1 l SPj 9MT etk L6 KfsC 6Zob Yhc2qu Xy 9 GPm ZYj 1Go ei feJ3 pRAf n6Ypy6 jN s 4Y5 nSE pqN 4m Rmam AGfY HhSaBr Ls D THC SEl UyR Mh 66XU 7hNz pZVC5V nV 7 VjL 7kv WKf 7P 5hj6 t1vu gkLGdN X8 b gOX HWm 6W4 YE mxFG 4WaN EbGKsv 0p 4 OG0 Nrd uTe Za xNXq V4Bp mOdXIq 9a b PeD PbU Z4N Xt ohbY egCf xBNttE wc D YSD 637 jJ2 ms 6Ta1 J2xZ PtKnPw AX A tJA Rc8 n5d 93 TZi7 q6Wo nEDLwW Sz e Sue YFX 8cM hm Y6is 15pX aOYBbV fS C haL kBR Ks6 UO qG4j DVab fbdtny fi D BFI 7uh B39 FJ 6mYr CUUT f2X38J 43 K yZg 87i gFR 5R z1t3 jH9x lOg1h7 P7 W w8w jMJ qH3 l5 J5wU 8eH0 OogRCv L7 f JJg 1ug RfM XI GSuE Efbh 3hdNY3 x1 9 7jR qeP cdu sb fkuJ hEpw MvNBZV zL u qxJ 9b1 BTf Yk RJLj Oo1a EPIXvZ Aj v Xne fhK GsJ Ga wqjt U7r6 MPoydE H2 6 203 mGi JhF nT NCDB YlnP oKO6Pu XU 3 uu9 mSg 41v ma kk0E WUpS UtGBtD e6 d Kdx ZNT FuT i1 fMcM hq7P Ovf0hg Hl 8 fqv I3R K39 fn 9MaC Zgow 6e1iXj KC 5 lHO lpG pkK Xd Dxtz 0HxE fSMjXY L8 F vh7 dmJ kE8 QA KDo1 FqML HOZ2iL 9i I m3L Kva YiN K9 sb48 NxwY NR0nx2 t5 b WCk x2a 31k a8 fUIa RGzr 7oigRX 5s m 9PQ 7Sr 5St ZE Ymp8 VIWS hdzgDI 9v R F5J 81x 33n Ne fjBT VvGP vGsxQh Al G Fbe 1bQ i6J ap OJJa ceGq 1vvb8r F2 F 3M6 8eD lzG tX tVm5 y14v mwIXa2 OG Y hxU sXJ 0qg l5 ZGAt HPZd oDWrSb BS u NKi 6KW gr3 9s 9tc7 WM4A ws1PzI 5c C O7Z 8y9 lMT LA dwhz Mxz9 hjlWHj bJ 5 CqM jht y9l Mn 4rc7 6Amk KJimvH 9r O tbc tCK rsi B0 4cFV Dl1g cvfWh6 5n x y9Z S4W Pyo QB yr3v fBkj TZKtEZ 7r U fdM icd yCV qn D036 HJWM tYfL9f yX x O7m IcF E1O uL QsAQ NfWv 6kV8Im 7Q 6 GsX NCV 0YP oC jnWn 6L25 qUMTe7 1v a hnH DAo XAb Tc zhPc fjrj W5M5G0 nz N M5T nlJ WOP Lh M6U2 ZFxw pg4Nej P8 U Q09 JX9 n7S kE WixE Rwgy Fvttzp 4A s v5F Tnn MzL Vh FUn5 6tFY CxZ1Bz Q3 E TfD lCa d7V fo MwPm ngrD HPfZV0 aY k Ojr ZUw 799 et oYuB MIC4 ovEY8D OL N URV Q5l ti1 iS NZAd wWr6 Q8oPFf ae 5 lAR 9gD RSi HO eJOW wxLv 20GoMt 2H z 7Yc aly PZx eR uFM0 7gaV 9UEQ107}   \end{align}  Since $\sum_{j}\Vert \pkn (\varphi(\fractext{\cdot}{n}) u_{0,j} ) \Vert_{p}^{p}\leq M_0K^p/4$, by subadditivity of the probability measure,   \begin{align}   \begin{split}   &\PP\biggl(   \sup_{[0, t\wedge \tau_M^n]} \Vert u^{(n)}\Vert_p^p   +\sum_{j} \int_0^{t\wedge \tau_M^n}    \int_{\RR^3} | \nabla (|u^{(n)}_j|^{p/2})|^2 \,dx ds   \geq  M_0K^p   \biggr)   \\&\indeq   \leq    \PP\biggl(  C t(MK^p+1)   \geq \frac{M_0K^p}{6}   \biggr)+\PP\biggl(   C\int_0^{t\wedge \tau_M^n} (\Vert\un\Vert_{3p}^{p-}+\Vert\un\Vert_{3p}^{\theta p})\,ds   \geq \frac{M_0K^p}{6}   \biggr)   \\&\indeq\indeq   +   \PP\biggl(   C\sum_j\sup_{[0,t\wedge \tau_M^n]}   \biggl| \int_0^s(\phin)^2\int_{\RR^\dd} |\un_j|^{p-2}\un_j \pkn\sigma_j(\pkn \un)\,dxd\WW_r\biggr|   \geq \frac{M_0K^p}{6}   \biggr).   \end{split}    \llabel{K w Io7 1fg sGm 0P YFBz X8eX 7pf9GJ b1 o XUs 1q0 6KP Ls MucN ytQb L0Z0Qq m1 l SPj 9MT etk L6 KfsC 6Zob Yhc2qu Xy 9 GPm ZYj 1Go ei feJ3 pRAf n6Ypy6 jN s 4Y5 nSE pqN 4m Rmam AGfY HhSaBr Ls D THC SEl UyR Mh 66XU 7hNz pZVC5V nV 7 VjL 7kv WKf 7P 5hj6 t1vu gkLGdN X8 b gOX HWm 6W4 YE mxFG 4WaN EbGKsv 0p 4 OG0 Nrd uTe Za xNXq V4Bp mOdXIq 9a b PeD PbU Z4N Xt ohbY egCf xBNttE wc D YSD 637 jJ2 ms 6Ta1 J2xZ PtKnPw AX A tJA Rc8 n5d 93 TZi7 q6Wo nEDLwW Sz e Sue YFX 8cM hm Y6is 15pX aOYBbV fS C haL kBR Ks6 UO qG4j DVab fbdtny fi D BFI 7uh B39 FJ 6mYr CUUT f2X38J 43 K yZg 87i gFR 5R z1t3 jH9x lOg1h7 P7 W w8w jMJ qH3 l5 J5wU 8eH0 OogRCv L7 f JJg 1ug RfM XI GSuE Efbh 3hdNY3 x1 9 7jR qeP cdu sb fkuJ hEpw MvNBZV zL u qxJ 9b1 BTf Yk RJLj Oo1a EPIXvZ Aj v Xne fhK GsJ Ga wqjt U7r6 MPoydE H2 6 203 mGi JhF nT NCDB YlnP oKO6Pu XU 3 uu9 mSg 41v ma kk0E WUpS UtGBtD e6 d Kdx ZNT FuT i1 fMcM hq7P Ovf0hg Hl 8 fqv I3R K39 fn 9MaC Zgow 6e1iXj KC 5 lHO lpG pkK Xd Dxtz 0HxE fSMjXY L8 F vh7 dmJ kE8 QA KDo1 FqML HOZ2iL 9i I m3L Kva YiN K9 sb48 NxwY NR0nx2 t5 b WCk x2a 31k a8 fUIa RGzr 7oigRX 5s m 9PQ 7Sr 5St ZE Ymp8 VIWS hdzgDI 9v R F5J 81x 33n Ne fjBT VvGP vGsxQh Al G Fbe 1bQ i6J ap OJJa ceGq 1vvb8r F2 F 3M6 8eD lzG tX tVm5 y14v mwIXa2 OG Y hxU sXJ 0qg l5 ZGAt HPZd oDWrSb BS u NKi 6KW gr3 9s 9tc7 WM4A ws1PzI 5c C O7Z 8y9 lMT LA dwhz Mxz9 hjlWHj bJ 5 CqM jht y9l Mn 4rc7 6Amk KJimvH 9r O tbc tCK rsi B0 4cFV Dl1g cvfWh6 5n x y9Z S4W Pyo QB yr3v fBkj TZKtEZ 7r U fdM icd yCV qn D036 HJWM tYfL9f yX x O7m IcF E1O uL QsAQ NfWv 6kV8Im 7Q 6 GsX NCV 0YP oC jnWn 6L25 qUMTe7 1v a hnH DAo XAb Tc zhPc fjrj W5M5G0 nz N M5T nlJ WOP Lh M6U2 ZFxw pg4Nej P8 U Q09 JX9 n7S kE WixE Rwgy Fvttzp 4A s v5F Tnn MzL Vh FUn5 6tFY CxZ1Bz Q3 E TfD lCa d7V fo MwPm ngrD HPfZV0 aY k Ojr ZUw 799 et oYuB MIC4 ovEY8D OL N URV Q5l ti1 iS NZAd wWr6 Q8oPFf ae 5 lAR 9gD RSi HO eJOW wxLv 20GoMt 2H z 7Yc aly PZx eR uFM0 7gaV 9UIz7S 43 k 5Tr ZiD Mt7 pE NCYi uHL7 gac7Gq yN 6 Z1u x56 YZh 2d yJVx 9MeU OMWBQf l0 E mIc 5Zr yfy 3i rahC y9Pi MJ7ofo Op d enn sLi xZx Jt CjC9 M71v O0fxiR 51 m FIB QRo 1oW IEQ108}   \end{align} Note that $\lim_{t\to 0}\PP(Ct(M^p+1) \geq M_0K^p/6)=0$. Then by H\"{o}lder's inequality and the definition of $\tau_M^n$,    \begin{equation}   \lim_{t\rightarrow0} \sup_n\PP\biggl(C\int_0^{t\wedge \tau_M^n} (\Vert\un\Vert_{3p}^{p-}+\Vert\un\Vert_{3p}^{\theta p})\,dr            \geq \frac{M_0K^p}{6}\biggr) = 0.    \llabel{ AGfY HhSaBr Ls D THC SEl UyR Mh 66XU 7hNz pZVC5V nV 7 VjL 7kv WKf 7P 5hj6 t1vu gkLGdN X8 b gOX HWm 6W4 YE mxFG 4WaN EbGKsv 0p 4 OG0 Nrd uTe Za xNXq V4Bp mOdXIq 9a b PeD PbU Z4N Xt ohbY egCf xBNttE wc D YSD 637 jJ2 ms 6Ta1 J2xZ PtKnPw AX A tJA Rc8 n5d 93 TZi7 q6Wo nEDLwW Sz e Sue YFX 8cM hm Y6is 15pX aOYBbV fS C haL kBR Ks6 UO qG4j DVab fbdtny fi D BFI 7uh B39 FJ 6mYr CUUT f2X38J 43 K yZg 87i gFR 5R z1t3 jH9x lOg1h7 P7 W w8w jMJ qH3 l5 J5wU 8eH0 OogRCv L7 f JJg 1ug RfM XI GSuE Efbh 3hdNY3 x1 9 7jR qeP cdu sb fkuJ hEpw MvNBZV zL u qxJ 9b1 BTf Yk RJLj Oo1a EPIXvZ Aj v Xne fhK GsJ Ga wqjt U7r6 MPoydE H2 6 203 mGi JhF nT NCDB YlnP oKO6Pu XU 3 uu9 mSg 41v ma kk0E WUpS UtGBtD e6 d Kdx ZNT FuT i1 fMcM hq7P Ovf0hg Hl 8 fqv I3R K39 fn 9MaC Zgow 6e1iXj KC 5 lHO lpG pkK Xd Dxtz 0HxE fSMjXY L8 F vh7 dmJ kE8 QA KDo1 FqML HOZ2iL 9i I m3L Kva YiN K9 sb48 NxwY NR0nx2 t5 b WCk x2a 31k a8 fUIa RGzr 7oigRX 5s m 9PQ 7Sr 5St ZE Ymp8 VIWS hdzgDI 9v R F5J 81x 33n Ne fjBT VvGP vGsxQh Al G Fbe 1bQ i6J ap OJJa ceGq 1vvb8r F2 F 3M6 8eD lzG tX tVm5 y14v mwIXa2 OG Y hxU sXJ 0qg l5 ZGAt HPZd oDWrSb BS u NKi 6KW gr3 9s 9tc7 WM4A ws1PzI 5c C O7Z 8y9 lMT LA dwhz Mxz9 hjlWHj bJ 5 CqM jht y9l Mn 4rc7 6Amk KJimvH 9r O tbc tCK rsi B0 4cFV Dl1g cvfWh6 5n x y9Z S4W Pyo QB yr3v fBkj TZKtEZ 7r U fdM icd yCV qn D036 HJWM tYfL9f yX x O7m IcF E1O uL QsAQ NfWv 6kV8Im 7Q 6 GsX NCV 0YP oC jnWn 6L25 qUMTe7 1v a hnH DAo XAb Tc zhPc fjrj W5M5G0 nz N M5T nlJ WOP Lh M6U2 ZFxw pg4Nej P8 U Q09 JX9 n7S kE WixE Rwgy Fvttzp 4A s v5F Tnn MzL Vh FUn5 6tFY CxZ1Bz Q3 E TfD lCa d7V fo MwPm ngrD HPfZV0 aY k Ojr ZUw 799 et oYuB MIC4 ovEY8D OL N URV Q5l ti1 iS NZAd wWr6 Q8oPFf ae 5 lAR 9gD RSi HO eJOW wxLv 20GoMt 2H z 7Yc aly PZx eR uFM0 7gaV 9UIz7S 43 k 5Tr ZiD Mt7 pE NCYi uHL7 gac7Gq yN 6 Z1u x56 YZh 2d yJVx 9MeU OMWBQf l0 E mIc 5Zr yfy 3i rahC y9Pi MJ7ofo Op d enn sLi xZx Jt CjC9 M71v O0fxiR 51 m FIB QRo 1oW Iq 3gDP stD2 ntfoX7 YU o S5k GuV IGM cf HZe3 7ZoG A1dDmk XO 2 KYR LpJ jII om M6Nu u8O0 jO5Nab Ub R nZn 15k hG9 4S 21V4 Ip45 7ooaiP u2 j hIz osW FDu O5 HdGr djvv tTLBjo vL LEQ109}   \end{equation} To get a similar bound for the terms involving $\sigma_j$, we appeal to the BDG inequality, obtaining   \begin{align}   \begin{split}   &\EE\biggl[   \sup_{0\leq s\leq t\wedge \tau_M^n}\biggl| \int_0^s(\phin)^2\int_{\RR^\dd} |\un_j|^{p-2}\un_j \pkn\sigma_j(\pkn\un)\,dxd\WW_r\biggr|   \biggr]   \\&\indeq   \leq    C\EE\biggl[\biggl(   \int_0^{t\wedge \tau_M^n}\biggl(\int_{\RR^\dd} \phin|\un_j|^{p-1}\Vert \pkn\sigma_j(\pkn\un)\Vert_{l^2}\,dx\biggr)^2\,ds   \biggr)^{1/2}\biggr]   \\&\indeq   \leq Ct^{1/2-1/p}\EE\biggl[ \sup_{0\leq s \leq t\wedge \tau_M^n}\Vert\un(s)\Vert_{p}^{p-1}   \biggl(   \int_0^{t\wedge \tau_M^n} \phin\int_{\RR^\dd}\Vert \pkn\sigma_j(\pkn\un)\Vert_{l^2}^p\,dx\,ds   \biggr)^{1/p}   \biggr]   \\&\indeq   \leq    Ct^{1/2-1/p}\EE\biggl[\sup_{0\leq s \leq t\wedge \tau_M^n}\Vert\un(s)\Vert_{p}^{p-1}   \biggl(   \int_0^{t\wedge \tau_M^n}\phin (\Vert \un\Vert_{(3p/2)-}^{2p}+1)\,ds   \biggr)^{1/p}   \biggr]   \\&\indeq    \leq Ct^{1/2-1/p}\EE\biggl[ \sup_{0\leq s \leq t\wedge \tau_M^n}\Vert\un(s)\Vert_{p}^{p-1}   \biggl(   \int_0^{t\wedge \tau_M^n} (\phin\Vert\un\Vert_{p}^{p+}\Vert\un\Vert_{3p}^{p-}+1)\,ds   \biggr)^{1/p}      \biggr]   \\&\indeq   \leq Ct^{(1/2-1/p)+}M^{(p-1)/p}K^{p-1}(MK^{p}+t)^{(1/p)-} \comma t\in (0,T]\commaone    j=1,2,3
  .   \end{split}    \llabel{bU Z4N Xt ohbY egCf xBNttE wc D YSD 637 jJ2 ms 6Ta1 J2xZ PtKnPw AX A tJA Rc8 n5d 93 TZi7 q6Wo nEDLwW Sz e Sue YFX 8cM hm Y6is 15pX aOYBbV fS C haL kBR Ks6 UO qG4j DVab fbdtny fi D BFI 7uh B39 FJ 6mYr CUUT f2X38J 43 K yZg 87i gFR 5R z1t3 jH9x lOg1h7 P7 W w8w jMJ qH3 l5 J5wU 8eH0 OogRCv L7 f JJg 1ug RfM XI GSuE Efbh 3hdNY3 x1 9 7jR qeP cdu sb fkuJ hEpw MvNBZV zL u qxJ 9b1 BTf Yk RJLj Oo1a EPIXvZ Aj v Xne fhK GsJ Ga wqjt U7r6 MPoydE H2 6 203 mGi JhF nT NCDB YlnP oKO6Pu XU 3 uu9 mSg 41v ma kk0E WUpS UtGBtD e6 d Kdx ZNT FuT i1 fMcM hq7P Ovf0hg Hl 8 fqv I3R K39 fn 9MaC Zgow 6e1iXj KC 5 lHO lpG pkK Xd Dxtz 0HxE fSMjXY L8 F vh7 dmJ kE8 QA KDo1 FqML HOZ2iL 9i I m3L Kva YiN K9 sb48 NxwY NR0nx2 t5 b WCk x2a 31k a8 fUIa RGzr 7oigRX 5s m 9PQ 7Sr 5St ZE Ymp8 VIWS hdzgDI 9v R F5J 81x 33n Ne fjBT VvGP vGsxQh Al G Fbe 1bQ i6J ap OJJa ceGq 1vvb8r F2 F 3M6 8eD lzG tX tVm5 y14v mwIXa2 OG Y hxU sXJ 0qg l5 ZGAt HPZd oDWrSb BS u NKi 6KW gr3 9s 9tc7 WM4A ws1PzI 5c C O7Z 8y9 lMT LA dwhz Mxz9 hjlWHj bJ 5 CqM jht y9l Mn 4rc7 6Amk KJimvH 9r O tbc tCK rsi B0 4cFV Dl1g cvfWh6 5n x y9Z S4W Pyo QB yr3v fBkj TZKtEZ 7r U fdM icd yCV qn D036 HJWM tYfL9f yX x O7m IcF E1O uL QsAQ NfWv 6kV8Im 7Q 6 GsX NCV 0YP oC jnWn 6L25 qUMTe7 1v a hnH DAo XAb Tc zhPc fjrj W5M5G0 nz N M5T nlJ WOP Lh M6U2 ZFxw pg4Nej P8 U Q09 JX9 n7S kE WixE Rwgy Fvttzp 4A s v5F Tnn MzL Vh FUn5 6tFY CxZ1Bz Q3 E TfD lCa d7V fo MwPm ngrD HPfZV0 aY k Ojr ZUw 799 et oYuB MIC4 ovEY8D OL N URV Q5l ti1 iS NZAd wWr6 Q8oPFf ae 5 lAR 9gD RSi HO eJOW wxLv 20GoMt 2H z 7Yc aly PZx eR uFM0 7gaV 9UIz7S 43 k 5Tr ZiD Mt7 pE NCYi uHL7 gac7Gq yN 6 Z1u x56 YZh 2d yJVx 9MeU OMWBQf l0 E mIc 5Zr yfy 3i rahC y9Pi MJ7ofo Op d enn sLi xZx Jt CjC9 M71v O0fxiR 51 m FIB QRo 1oW Iq 3gDP stD2 ntfoX7 YU o S5k GuV IGM cf HZe3 7ZoG A1dDmk XO 2 KYR LpJ jII om M6Nu u8O0 jO5Nab Ub R nZn 15k hG9 4S 21V4 Ip45 7ooaiP u2 j hIz osW FDu O5 HdGr djvv tTLBjo vL L iCo 6L5 Lwa Pm vD6Z pal6 9Ljn11 re T 2CP mvj rL3 xH mDYK uv5T npC1fM oU R RTo Loi lk0 FE ghak m5M9 cOIPdQ lG D LnX erC ykJ C1 0FHh vvnY aTGuqU rf T QPv wEq iHO vO hD6A nXEQ110}   \end{align} By Chebyshev's inequality, we have   \begin{align}   \begin{split}   \lim_{t\rightarrow 0} \sup_n \PP&\biggl(   \sum_j\sup_{0\leq s\leq t\wedge \tau^n_M}   \biggl| \int_0^s\int_{\RR^\dd} |\un_j|^{p-2}\un_j \pkn\sigma_j(\pkn\un)\,dxd\WW_r\biggr|    \geq \frac{M_0K^p}{6}   \biggr)   \\&  \leq  \lim_{t\rightarrow 0} \sup_n C_{M_0, M,K,p}t^{(1/2-1/p)+} = 0,   \end{split}   \llabel{tny fi D BFI 7uh B39 FJ 6mYr CUUT f2X38J 43 K yZg 87i gFR 5R z1t3 jH9x lOg1h7 P7 W w8w jMJ qH3 l5 J5wU 8eH0 OogRCv L7 f JJg 1ug RfM XI GSuE Efbh 3hdNY3 x1 9 7jR qeP cdu sb fkuJ hEpw MvNBZV zL u qxJ 9b1 BTf Yk RJLj Oo1a EPIXvZ Aj v Xne fhK GsJ Ga wqjt U7r6 MPoydE H2 6 203 mGi JhF nT NCDB YlnP oKO6Pu XU 3 uu9 mSg 41v ma kk0E WUpS UtGBtD e6 d Kdx ZNT FuT i1 fMcM hq7P Ovf0hg Hl 8 fqv I3R K39 fn 9MaC Zgow 6e1iXj KC 5 lHO lpG pkK Xd Dxtz 0HxE fSMjXY L8 F vh7 dmJ kE8 QA KDo1 FqML HOZ2iL 9i I m3L Kva YiN K9 sb48 NxwY NR0nx2 t5 b WCk x2a 31k a8 fUIa RGzr 7oigRX 5s m 9PQ 7Sr 5St ZE Ymp8 VIWS hdzgDI 9v R F5J 81x 33n Ne fjBT VvGP vGsxQh Al G Fbe 1bQ i6J ap OJJa ceGq 1vvb8r F2 F 3M6 8eD lzG tX tVm5 y14v mwIXa2 OG Y hxU sXJ 0qg l5 ZGAt HPZd oDWrSb BS u NKi 6KW gr3 9s 9tc7 WM4A ws1PzI 5c C O7Z 8y9 lMT LA dwhz Mxz9 hjlWHj bJ 5 CqM jht y9l Mn 4rc7 6Amk KJimvH 9r O tbc tCK rsi B0 4cFV Dl1g cvfWh6 5n x y9Z S4W Pyo QB yr3v fBkj TZKtEZ 7r U fdM icd yCV qn D036 HJWM tYfL9f yX x O7m IcF E1O uL QsAQ NfWv 6kV8Im 7Q 6 GsX NCV 0YP oC jnWn 6L25 qUMTe7 1v a hnH DAo XAb Tc zhPc fjrj W5M5G0 nz N M5T nlJ WOP Lh M6U2 ZFxw pg4Nej P8 U Q09 JX9 n7S kE WixE Rwgy Fvttzp 4A s v5F Tnn MzL Vh FUn5 6tFY CxZ1Bz Q3 E TfD lCa d7V fo MwPm ngrD HPfZV0 aY k Ojr ZUw 799 et oYuB MIC4 ovEY8D OL N URV Q5l ti1 iS NZAd wWr6 Q8oPFf ae 5 lAR 9gD RSi HO eJOW wxLv 20GoMt 2H z 7Yc aly PZx eR uFM0 7gaV 9UIz7S 43 k 5Tr ZiD Mt7 pE NCYi uHL7 gac7Gq yN 6 Z1u x56 YZh 2d yJVx 9MeU OMWBQf l0 E mIc 5Zr yfy 3i rahC y9Pi MJ7ofo Op d enn sLi xZx Jt CjC9 M71v O0fxiR 51 m FIB QRo 1oW Iq 3gDP stD2 ntfoX7 YU o S5k GuV IGM cf HZe3 7ZoG A1dDmk XO 2 KYR LpJ jII om M6Nu u8O0 jO5Nab Ub R nZn 15k hG9 4S 21V4 Ip45 7ooaiP u2 j hIz osW FDu O5 HdGr djvv tTLBjo vL L iCo 6L5 Lwa Pm vD6Z pal6 9Ljn11 re T 2CP mvj rL3 xH mDYK uv5T npC1fM oU R RTo Loi lk0 FE ghak m5M9 cOIPdQ lG D LnX erC ykJ C1 0FHh vvnY aTGuqU rf T QPv wEq iHO vO hD6A nXuv GlzVAv pz d Ok3 6ym yUo Fb AcAA BItO es52Vq d0 Y c7U 2gB t0W fF VQZh rJHr lBLdCx 8I o dWp AlD S8C HB rNLz xWp6 ypjuwW mg X toy 1vP bra uH yMNb kUrZ D6Ee2f zI D tkZ Eti EQ111}   \end{align} concluding the proof. \end{proof} \par Now we proceed to verify the Cauchy condition for the sequence~$\{\un\}$.  \par \cole \begin{Lemma} \label{L06}Let~$p>d=3$ and~$K\geq1$, and let $k(n+1)$ be the smallest integer that is greater than $(n+1) \EE[\Vert \varphi(\fractext{\cdot}{n+1}) u_0 \Vert_2^2]$ and $k(n)$ in~\eqref{ERTWERTHWRTWERTSGDGHCFGSDFGQSERWDFGDSFGHSDRGTEHDFGHDSFGSDGHGYUHDFGSDFASDFASGTWRT95}. Suppose that $\Vert u_{0}\Vert_{p} \leq K$. Then, there exists $t>0$ such that   \begin{align}   \begin{split}   \lim_{m\to\infty} \sup_{n>m}\EE\biggl[   \sup_{0\leq s\leq \tau_{n, m}}\Vert u^{(n)}(s)-u^{(m)}(s)\Vert_p^p   +\int_0^{\tau_{n, m}}    \Vert u^{(n)}(s)-u^{(m)}(s)\Vert_{3p}^p \,ds   \biggr]   =0,   \end{split}    \label{ERTWERTHWRTWERTSGDGHCFGSDFGQSERWDFGDSFGHSDRGTEHDFGHDSFGSDGHGYUHDFGSDFASDFASGTWRT112}   \end{align} where $\tau_{n, m} =\tau^n_M\wedge \tau^n_M\wedge t$, and the choice of~$t$ may depend on $K$ and $M$ (see~\eqref{ERTWERTHWRTWERTSGDGHCFGSDFGQSERWDFGDSFGHSDRGTEHDFGHDSFGSDGHGYUHDFGSDFASDFASGTWRT96}). \end{Lemma} \colb \par \begin{proof}[Proof of Lemma~\ref{L06}]  By~\eqref{ERTWERTHWRTWERTSGDGHCFGSDFGQSERWDFGDSFGHSDRGTEHDFGHDSFGSDGHGYUHDFGSDFASDFASGTWRT96}, $\tau_{n,m}$ is a positive stopping time for every~$t>0$. Let $n>m$, and denote $\wnm=\un-\um$. Clearly, $\wnm$ satisfies    \begin{align}   \begin{split}   &\partial_t \wnm   =\Delta \wnm   -   \Bigl(     (\phin)^2 \pkn \mathcal{P}((u^{(n)}\cdot\nabla) \pkn u^{(n)})-(\varphi^{(m)})^2\pkm\mathcal{P}((u^{(m)}\cdot\nabla) \pkm u^{(m)})   \Bigr)   \\&\indeq\indeq\indeq\indeq\quad   +\Bigl(   (\phin)^2 \pkn\sigma(\pkn\un)-(\varphi^{(m)})^2\pkm\sigma(\pkm\um)   \Bigr)\,\dot{\WW}(t),   \\ & \nabla\cdot \wnm = 0,   \\& \wnm(0) = \pkn \mathcal{P}      \left(\varphi\left(\frac{\cdot}{n}\right) u_0 \right)-\pkm \mathcal{P}     \left(\varphi\left(\frac{\cdot}{m}\right) u_0 \right).   \end{split}   \llabel{ fkuJ hEpw MvNBZV zL u qxJ 9b1 BTf Yk RJLj Oo1a EPIXvZ Aj v Xne fhK GsJ Ga wqjt U7r6 MPoydE H2 6 203 mGi JhF nT NCDB YlnP oKO6Pu XU 3 uu9 mSg 41v ma kk0E WUpS UtGBtD e6 d Kdx ZNT FuT i1 fMcM hq7P Ovf0hg Hl 8 fqv I3R K39 fn 9MaC Zgow 6e1iXj KC 5 lHO lpG pkK Xd Dxtz 0HxE fSMjXY L8 F vh7 dmJ kE8 QA KDo1 FqML HOZ2iL 9i I m3L Kva YiN K9 sb48 NxwY NR0nx2 t5 b WCk x2a 31k a8 fUIa RGzr 7oigRX 5s m 9PQ 7Sr 5St ZE Ymp8 VIWS hdzgDI 9v R F5J 81x 33n Ne fjBT VvGP vGsxQh Al G Fbe 1bQ i6J ap OJJa ceGq 1vvb8r F2 F 3M6 8eD lzG tX tVm5 y14v mwIXa2 OG Y hxU sXJ 0qg l5 ZGAt HPZd oDWrSb BS u NKi 6KW gr3 9s 9tc7 WM4A ws1PzI 5c C O7Z 8y9 lMT LA dwhz Mxz9 hjlWHj bJ 5 CqM jht y9l Mn 4rc7 6Amk KJimvH 9r O tbc tCK rsi B0 4cFV Dl1g cvfWh6 5n x y9Z S4W Pyo QB yr3v fBkj TZKtEZ 7r U fdM icd yCV qn D036 HJWM tYfL9f yX x O7m IcF E1O uL QsAQ NfWv 6kV8Im 7Q 6 GsX NCV 0YP oC jnWn 6L25 qUMTe7 1v a hnH DAo XAb Tc zhPc fjrj W5M5G0 nz N M5T nlJ WOP Lh M6U2 ZFxw pg4Nej P8 U Q09 JX9 n7S kE WixE Rwgy Fvttzp 4A s v5F Tnn MzL Vh FUn5 6tFY CxZ1Bz Q3 E TfD lCa d7V fo MwPm ngrD HPfZV0 aY k Ojr ZUw 799 et oYuB MIC4 ovEY8D OL N URV Q5l ti1 iS NZAd wWr6 Q8oPFf ae 5 lAR 9gD RSi HO eJOW wxLv 20GoMt 2H z 7Yc aly PZx eR uFM0 7gaV 9UIz7S 43 k 5Tr ZiD Mt7 pE NCYi uHL7 gac7Gq yN 6 Z1u x56 YZh 2d yJVx 9MeU OMWBQf l0 E mIc 5Zr yfy 3i rahC y9Pi MJ7ofo Op d enn sLi xZx Jt CjC9 M71v O0fxiR 51 m FIB QRo 1oW Iq 3gDP stD2 ntfoX7 YU o S5k GuV IGM cf HZe3 7ZoG A1dDmk XO 2 KYR LpJ jII om M6Nu u8O0 jO5Nab Ub R nZn 15k hG9 4S 21V4 Ip45 7ooaiP u2 j hIz osW FDu O5 HdGr djvv tTLBjo vL L iCo 6L5 Lwa Pm vD6Z pal6 9Ljn11 re T 2CP mvj rL3 xH mDYK uv5T npC1fM oU R RTo Loi lk0 FE ghak m5M9 cOIPdQ lG D LnX erC ykJ C1 0FHh vvnY aTGuqU rf T QPv wEq iHO vO hD6A nXuv GlzVAv pz d Ok3 6ym yUo Fb AcAA BItO es52Vq d0 Y c7U 2gB t0W fF VQZh rJHr lBLdCx 8I o dWp AlD S8C HB rNLz xWp6 ypjuwW mg X toy 1vP bra uH yMNb kUrZ D6Ee2f zI D tkZ Eti Lmg re 1woD juLB BSdasY Vc F Uhy ViC xB1 5y Ltql qoUh gL3bZN YV k orz wa3 650 qW hF22 epiX cAjA4Z V4 b cXx uB3 NQN p0 GxW2 Vs1z jtqe2p LE B iS3 0E0 NKH gY N50v XaK6 pNpwdBEQ113}   \end{align} Denote $\pnm=\pkn-\pkm$. Under the divergence-free condition, the nonlinear term can be regarded as the sum of partial derivatives of    \begin{align}   \begin{split}   & -(\phin)^2 \pkn \mathcal{P}(u^{(n)}_i \pkn u^{(n)})+(\varphi^{(m)})^2\pkm\mathcal{P}(u^{(m)}_i \pkm u^{(m)})   \\ &\quad    =   -\phin(\phin-\varphi^{(m)})\pkn\mathcal{P}(u^{(n)}_i \pkn u^{(n)})   -   \phin\varphi^{(m)}\pkn\mathcal{P}(u^{(n)}_i \pkn u^{(n,m)} )   \\ & \quad   \quad   -   \phin\varphi^{(m)}\pkn\mathcal{P}(u^{(n)}_i \pnm u^{(m)} )   -   \phin\varphi^{(m)}\pkn\mathcal{P}(u^{(n, m)}_i  \pkm u^{(m)})    \\ & \quad   \quad   -   \phin\varphi^{(m)}\pnm\mathcal{P}(u^{(m)}_i \pkm u^{( m)})   -    (\phin-\varphi^{(m)})\varphi^{(m)}\pkm\mathcal{P}(\um_i\pkm\um)   \\ & \quad     =: \sum_{k=1}^{6} f_{i}^{(k)}   \comma i=1,2,3   .   \end{split}    \label{ERTWERTHWRTWERTSGDGHCFGSDFGQSERWDFGDSFGHSDRGTEHDFGHDSFGSDGHGYUHDFGSDFASDFASGTWRT114}   \end{align}  First, consider $f^{(3)}$ and apply Theorem \ref{T02} with $1/q=(4/3p)_{+}$. (For example, we may choose $1/q=(1-1/4p)(4/3p)+(1/4p)(p+1)/3p$). This leads to   \begin{align}  \begin{split}  &\EE   \left[  \int_{0}^{\tnm}\Vert f_i^{(3)}\Vert_{q}^p\,ds   \right]  \leq  C\EE \left[\int_{0}^{\tnm}\phin\varphi^{(m)}\Vert \mathcal{P}(u^{( n)}_i  \pnm u^{(m)}) \Vert_{q}^p\,ds  \right]  \\ &  \indeq  \leq   C\EE \left[\int_{0}^{\tnm}\phin\varphi^{(m)}\Vert u^{( n)} \Vert_{p}^p  \Vert | \pnm u^{(m)}|^{\delta} \Vert_{2/\delta}^p  \Vert | \pnm u^{(m)}|^{1-\delta} \Vert_{l}^p  \,ds  \right]  \\ &  \indeq   \leq  C\EE \left[\int_{0}^{\tnm}\phin\varphi^{(m)}  \Vert \pnm u^{(m)} \Vert_{2}^{\delta p}  \Vert \pnm u^{(m)}\Vert_{(1-\delta )l}^{(1-\delta)p}  \,ds  \right]  \\ &  \indeq  \leq  C\EE \left[\int_{0}^{\tnm}\varphi^{(m)}  \Vert \pnm u^{(m)} \Vert_{2}^{\delta p}  \Vert \pnm u^{(m)}\Vert_{p}^{(3p-(1-\delta)l)p/2l}  \Vert \pnm u^{(m)}\Vert_{3p}^{(3l(1-\delta)-3p)p/2l}  \,ds  \right]\comma i=1,2,3.  \end{split}    \llabel{Kdx ZNT FuT i1 fMcM hq7P Ovf0hg Hl 8 fqv I3R K39 fn 9MaC Zgow 6e1iXj KC 5 lHO lpG pkK Xd Dxtz 0HxE fSMjXY L8 F vh7 dmJ kE8 QA KDo1 FqML HOZ2iL 9i I m3L Kva YiN K9 sb48 NxwY NR0nx2 t5 b WCk x2a 31k a8 fUIa RGzr 7oigRX 5s m 9PQ 7Sr 5St ZE Ymp8 VIWS hdzgDI 9v R F5J 81x 33n Ne fjBT VvGP vGsxQh Al G Fbe 1bQ i6J ap OJJa ceGq 1vvb8r F2 F 3M6 8eD lzG tX tVm5 y14v mwIXa2 OG Y hxU sXJ 0qg l5 ZGAt HPZd oDWrSb BS u NKi 6KW gr3 9s 9tc7 WM4A ws1PzI 5c C O7Z 8y9 lMT LA dwhz Mxz9 hjlWHj bJ 5 CqM jht y9l Mn 4rc7 6Amk KJimvH 9r O tbc tCK rsi B0 4cFV Dl1g cvfWh6 5n x y9Z S4W Pyo QB yr3v fBkj TZKtEZ 7r U fdM icd yCV qn D036 HJWM tYfL9f yX x O7m IcF E1O uL QsAQ NfWv 6kV8Im 7Q 6 GsX NCV 0YP oC jnWn 6L25 qUMTe7 1v a hnH DAo XAb Tc zhPc fjrj W5M5G0 nz N M5T nlJ WOP Lh M6U2 ZFxw pg4Nej P8 U Q09 JX9 n7S kE WixE Rwgy Fvttzp 4A s v5F Tnn MzL Vh FUn5 6tFY CxZ1Bz Q3 E TfD lCa d7V fo MwPm ngrD HPfZV0 aY k Ojr ZUw 799 et oYuB MIC4 ovEY8D OL N URV Q5l ti1 iS NZAd wWr6 Q8oPFf ae 5 lAR 9gD RSi HO eJOW wxLv 20GoMt 2H z 7Yc aly PZx eR uFM0 7gaV 9UIz7S 43 k 5Tr ZiD Mt7 pE NCYi uHL7 gac7Gq yN 6 Z1u x56 YZh 2d yJVx 9MeU OMWBQf l0 E mIc 5Zr yfy 3i rahC y9Pi MJ7ofo Op d enn sLi xZx Jt CjC9 M71v O0fxiR 51 m FIB QRo 1oW Iq 3gDP stD2 ntfoX7 YU o S5k GuV IGM cf HZe3 7ZoG A1dDmk XO 2 KYR LpJ jII om M6Nu u8O0 jO5Nab Ub R nZn 15k hG9 4S 21V4 Ip45 7ooaiP u2 j hIz osW FDu O5 HdGr djvv tTLBjo vL L iCo 6L5 Lwa Pm vD6Z pal6 9Ljn11 re T 2CP mvj rL3 xH mDYK uv5T npC1fM oU R RTo Loi lk0 FE ghak m5M9 cOIPdQ lG D LnX erC ykJ C1 0FHh vvnY aTGuqU rf T QPv wEq iHO vO hD6A nXuv GlzVAv pz d Ok3 6ym yUo Fb AcAA BItO es52Vq d0 Y c7U 2gB t0W fF VQZh rJHr lBLdCx 8I o dWp AlD S8C HB rNLz xWp6 ypjuwW mg X toy 1vP bra uH yMNb kUrZ D6Ee2f zI D tkZ Eti Lmg re 1woD juLB BSdasY Vc F Uhy ViC xB1 5y Ltql qoUh gL3bZN YV k orz wa3 650 qW hF22 epiX cAjA4Z V4 b cXx uB3 NQN p0 GxW2 Vs1z jtqe2p LE B iS3 0E0 NKH gY N50v XaK6 pNpwdB X2 Y v7V 0Ud dTc Pi dRNN CLG4 7Fc3PL Bx K 3Be x1X zyX cj 0Z6a Jk0H KuQnwd Dh P Q1Q rwA 05v 9c 3pnz ttzt x2IirW CZ B oS5 xlO KCi D3 WFh4 dvCL QANAQJ Gg y vOD NTD FKj Mc 0REQ159}  \end{align}  where $0<\delta\ll 1$ so that $1/q=1/p+\delta/2+1/l$ and $p<(1-\delta)l<(1+\delta(p-3))l<3p$. (A possible choice corresponding to $1/q=(1-1/4p)(4/3p)+(1/4p)(p+1)/3p$ is $\delta=(p-3)/12p^2$). Then by the continuity properties of $\pnm$, we have  \begin{align}  \begin{split}  &\EE   \left[  \int_{0}^{\tnm}\Vert f^{(3)}\Vert_{q}^p\,ds   \right]  \leq  C\EE \left[\int_{0}^{\tnm}  \Vert \pnm u^{(m)} \Vert_{2}^{\delta p}  \Vert \pnm u^{(m)}\Vert_{3p}^{(3l(1-\delta)-3p)p/2l}  \,ds  \right]  \\ &  \indeq   \leq  \frac{C}{ (k(m)\wedge k(n))^{\delta p}}  \EE \left[\int_{0}^{\tnm}
 \Vert \nabla u^{(m)} \Vert_{2}^{\delta p}  \Vert \pnm u^{(m)}\Vert_{3p}^{(3l(1-\delta)-3p)p/2l}  \,ds  \right]  \\ &  \indeq  \leq \frac{C}{ (k(m)\wedge k(n))^{\delta p}}  \EE \left[\int_{0}^{\tnm}   \bigl(      \Vert u^{(m)}\Vert_{3p}^{p}     +\Vert \nabla u^{(m)} \Vert_{2}^{2l\delta p/(3l\delta-l+3p)}   \bigr)  \,ds  \right]  \\ &  \indeq  \leq \frac{C_T}{ (k(m)\wedge k(n))^{\delta p}}  \EE \left[\int_{0}^{\tnm}  \bigl(  \Vert  u^{(m)}\Vert_{3p}^{p}  +\Vert \nabla u^{(m)} \Vert_{2}^{2}  \bigr)  \,ds  \right].  \end{split} \label{ERTWERTHWRTWERTSGDGHCFGSDFGQSERWDFGDSFGHSDRGTEHDFGHDSFGSDGHGYUHDFGSDFASDFASGTWRT115}  \end{align}    Next, we estimate $f^{(5)}$ with the same~$q$. Let $\kappa=(p-2)/(p+2)$ and $\bar{l}\in (q, 3p/4)$, so that $1/(1+\kappa)=1/2+1/p$ and $1/q=\theta/(1+\kappa)+(1-\theta)/\bar{l}$ for some~$\theta \in (0, 1)$. Since $\pkn$ and $\pkm$ commute with the Helmholtz-Hodge projector $\mathcal{P}$ and the differentiation, we get   \begin{align}   \begin{split}   &\EE    \left[   \int_{0}^{\tnm}\Vert f_i^{(5)}\Vert_{q}^p\,ds    \right]   \\&\indeq   \leq   C\EE \left[\int_{0}^{\tnm}\varphi^{(m)}   \Vert      \pnm(u^{(m)}_i \pkm u^{(m)} )   \Vert_{1+\kappa}^{\theta p}   \Vert      \pnm(u^{(m)}_i \pkm u^{(m)} )   \Vert_{\bar{l}}^{(1-\theta) p}   \,ds\right]   \\ & \indeq   \leq    \frac{C}{ (k(m)\wedge k(n))^{\theta p}}   \EE \left[\int_{0}^{\tnm}\varphi^{(m)}   \Vert      \nabla(u^{(m)}_i \pkm u^{(m)} )   \Vert_{1+\kappa}^{\theta p}   \Vert      u^{(m)}_i \pkm u^{(m)}    \Vert_{\bar{l}}^{(1-\theta) p}   \,ds\right]   \\ & \indeq   \leq    \frac{C}{ (k(m)\wedge k(n))^{\theta p}}   \EE \left[\int_{0}^{\tnm}\varphi^{(m)}   \Vert      \nabla u^{(m)}   \Vert_{2}^{\theta p}   \Vert      u^{(m)}    \Vert_{p}^{\theta p}   \Vert      u^{(m)}    \Vert_{2\bar{l}}^{2(1-\theta) p}   \,ds\right]\\   &\quad \leq   \frac{C}{ (k(m)\wedge k(n))^{\theta p}}   \EE \left[\int_{0}^{\tnm}\varphi^{(m)}   \Vert      \nabla u^{(m)}   \Vert_{2}^{\theta p}   \Vert      u^{(m)}    \Vert_{p}^{\theta p+2\bar{\delta}(1-\theta) p}   \Vert      u^{(m)}    \Vert_{3p}^{2(1-\bar{\delta})(1-\theta) p}   \,ds\right]   \comma i=1,2,3   ,   \end{split}    \llabel{Y NR0nx2 t5 b WCk x2a 31k a8 fUIa RGzr 7oigRX 5s m 9PQ 7Sr 5St ZE Ymp8 VIWS hdzgDI 9v R F5J 81x 33n Ne fjBT VvGP vGsxQh Al G Fbe 1bQ i6J ap OJJa ceGq 1vvb8r F2 F 3M6 8eD lzG tX tVm5 y14v mwIXa2 OG Y hxU sXJ 0qg l5 ZGAt HPZd oDWrSb BS u NKi 6KW gr3 9s 9tc7 WM4A ws1PzI 5c C O7Z 8y9 lMT LA dwhz Mxz9 hjlWHj bJ 5 CqM jht y9l Mn 4rc7 6Amk KJimvH 9r O tbc tCK rsi B0 4cFV Dl1g cvfWh6 5n x y9Z S4W Pyo QB yr3v fBkj TZKtEZ 7r U fdM icd yCV qn D036 HJWM tYfL9f yX x O7m IcF E1O uL QsAQ NfWv 6kV8Im 7Q 6 GsX NCV 0YP oC jnWn 6L25 qUMTe7 1v a hnH DAo XAb Tc zhPc fjrj W5M5G0 nz N M5T nlJ WOP Lh M6U2 ZFxw pg4Nej P8 U Q09 JX9 n7S kE WixE Rwgy Fvttzp 4A s v5F Tnn MzL Vh FUn5 6tFY CxZ1Bz Q3 E TfD lCa d7V fo MwPm ngrD HPfZV0 aY k Ojr ZUw 799 et oYuB MIC4 ovEY8D OL N URV Q5l ti1 iS NZAd wWr6 Q8oPFf ae 5 lAR 9gD RSi HO eJOW wxLv 20GoMt 2H z 7Yc aly PZx eR uFM0 7gaV 9UIz7S 43 k 5Tr ZiD Mt7 pE NCYi uHL7 gac7Gq yN 6 Z1u x56 YZh 2d yJVx 9MeU OMWBQf l0 E mIc 5Zr yfy 3i rahC y9Pi MJ7ofo Op d enn sLi xZx Jt CjC9 M71v O0fxiR 51 m FIB QRo 1oW Iq 3gDP stD2 ntfoX7 YU o S5k GuV IGM cf HZe3 7ZoG A1dDmk XO 2 KYR LpJ jII om M6Nu u8O0 jO5Nab Ub R nZn 15k hG9 4S 21V4 Ip45 7ooaiP u2 j hIz osW FDu O5 HdGr djvv tTLBjo vL L iCo 6L5 Lwa Pm vD6Z pal6 9Ljn11 re T 2CP mvj rL3 xH mDYK uv5T npC1fM oU R RTo Loi lk0 FE ghak m5M9 cOIPdQ lG D LnX erC ykJ C1 0FHh vvnY aTGuqU rf T QPv wEq iHO vO hD6A nXuv GlzVAv pz d Ok3 6ym yUo Fb AcAA BItO es52Vq d0 Y c7U 2gB t0W fF VQZh rJHr lBLdCx 8I o dWp AlD S8C HB rNLz xWp6 ypjuwW mg X toy 1vP bra uH yMNb kUrZ D6Ee2f zI D tkZ Eti Lmg re 1woD juLB BSdasY Vc F Uhy ViC xB1 5y Ltql qoUh gL3bZN YV k orz wa3 650 qW hF22 epiX cAjA4Z V4 b cXx uB3 NQN p0 GxW2 Vs1z jtqe2p LE B iS3 0E0 NKH gY N50v XaK6 pNpwdB X2 Y v7V 0Ud dTc Pi dRNN CLG4 7Fc3PL Bx K 3Be x1X zyX cj 0Z6a Jk0H KuQnwd Dh P Q1Q rwA 05v 9c 3pnz ttzt x2IirW CZ B oS5 xlO KCi D3 WFh4 dvCL QANAQJ Gg y vOD NTD FKj Mc 0RJP m4HU SQkLnT Q4 Y 6CC MvN jAR Zb lir7 RFsI NzHiJl cg f xSC Hts ZOG 1V uOzk 5G1C LtmRYI eD 3 5BB uxZ JdY LO CwS9 lokS NasDLj 5h 8 yni u7h u3c di zYh1 PdwE l3m8Xt yX Q RCAEQ160}   \end{align} where $\bar{\delta}\in(0,1)$ solves $1/(2\bar{l})=\bar{\delta}/p+(1-\bar{\delta})/(3p)$. Using the properties of the cut-off function $\varphi$, we obtain \begin{align} \begin{split} &\EE  \left[ \int_{0}^{\tnm}\Vert f^{(5)}\Vert_{q}^p\,ds  \right] \leq \frac{C}{ (k(m)\wedge k(n))^{\theta p}} \EE \left[\int_{0}^{\tnm} \Vert    \nabla u^{(m)} \Vert_{2}^{\theta p} \Vert    u^{(m)}  \Vert_{3p}^{2(1-\bar{\delta})(1-\theta) p} \,ds\right] \\ & \indeq \leq  \frac{C}{ (k(m)\wedge k(n))^{\theta p}} \EE  \left[\int_{0}^{\tnm} \bigl( \Vert    \nabla u^{(m)} \Vert_{2}^{2} + \Vert    u^{(m)}  \Vert_{3p}^{4(1-\bar{\delta})(1-\theta) p/(2-\theta p)} \bigr) \,ds\right] \\ & \indeq \leq  \frac{C_T}{ (k(m)\wedge k(n))^{\theta p}} \EE \left[ \int_{0}^{\tnm} \bigl( \Vert    \nabla u^{(m)} \Vert_{2}^{2} + \Vert    u^{(m)}  \Vert_{3p}^{ p} \bigr) \,ds\right]. \end{split}   \label{ERTWERTHWRTWERTSGDGHCFGSDFGQSERWDFGDSFGHSDRGTEHDFGHDSFGSDGHGYUHDFGSDFASDFASGTWRT118} \end{align} Above, we required $4(1-\bar{\delta})(1-\theta) p/(2-\theta p)<p$, which is equivalent to  \begin{equation}\label{ERTWERTHWRTWERTSGDGHCFGSDFGQSERWDFGDSFGHSDRGTEHDFGHDSFGSDGHGYUHDFGSDFASDFASGTWRT119} \bar{\delta}>\frac{2-4\theta+\theta p}{4-4\theta}=1-\frac{p}{4}+\frac{p-2}{4(1-\theta)}. \end{equation} Note that the right side of this inequality is an increasing and continuous function of $\theta$ when $\theta\in [0,1)$. In particular, it is equal to $1/2$ when~$\theta=0$. On the other hand, $\bar{l}=q$ if~$\theta=0$. In this situation, $\bar{\delta}>1/2$, and thus we can achieve \eqref{ERTWERTHWRTWERTSGDGHCFGSDFGQSERWDFGDSFGHSDRGTEHDFGHDSFGSDGHGYUHDFGSDFASDFASGTWRT119} by allowing $\theta$ sufficiently close to zero, i.e., by setting $\bar{l}$ close enough to~$q$.  With the same $q$, there exist an arbitrarily small constant $\eps>0$ and a constant $C_{\eps}$ such that    \begin{align}   \begin{split}   &\EE    \left[   \int_{0}^{\tnm}\Vert  f^{(2)}+f^{(4)}\Vert_{q}^p\,ds    \right]   \\&\indeq
  \le     \EE    \left[   \int_{0}^{\tnm} \phin\varphi^{(m)}   (\Vert \un\Vert_{p}^p + \Vert \um\Vert_{p}^p)   (C_{\eps}\Vert \wnm\Vert_{p}^p + \eps\Vert \wnm\Vert_{3p}^p)   \,ds   \right]   \\ & \indeq   \leq   C_{\eps}t\EE    \left[   \sup_{s\in [0,\tau_{n, m}]}\Vert \wnm(s)\Vert_p^p   \right]    + \eps \EE    \left[   \int_0^{\tau_{n, m}}    \Vert \wnm\Vert_{3p}^p\,ds   \right].   \end{split}    \llabel{zG tX tVm5 y14v mwIXa2 OG Y hxU sXJ 0qg l5 ZGAt HPZd oDWrSb BS u NKi 6KW gr3 9s 9tc7 WM4A ws1PzI 5c C O7Z 8y9 lMT LA dwhz Mxz9 hjlWHj bJ 5 CqM jht y9l Mn 4rc7 6Amk KJimvH 9r O tbc tCK rsi B0 4cFV Dl1g cvfWh6 5n x y9Z S4W Pyo QB yr3v fBkj TZKtEZ 7r U fdM icd yCV qn D036 HJWM tYfL9f yX x O7m IcF E1O uL QsAQ NfWv 6kV8Im 7Q 6 GsX NCV 0YP oC jnWn 6L25 qUMTe7 1v a hnH DAo XAb Tc zhPc fjrj W5M5G0 nz N M5T nlJ WOP Lh M6U2 ZFxw pg4Nej P8 U Q09 JX9 n7S kE WixE Rwgy Fvttzp 4A s v5F Tnn MzL Vh FUn5 6tFY CxZ1Bz Q3 E TfD lCa d7V fo MwPm ngrD HPfZV0 aY k Ojr ZUw 799 et oYuB MIC4 ovEY8D OL N URV Q5l ti1 iS NZAd wWr6 Q8oPFf ae 5 lAR 9gD RSi HO eJOW wxLv 20GoMt 2H z 7Yc aly PZx eR uFM0 7gaV 9UIz7S 43 k 5Tr ZiD Mt7 pE NCYi uHL7 gac7Gq yN 6 Z1u x56 YZh 2d yJVx 9MeU OMWBQf l0 E mIc 5Zr yfy 3i rahC y9Pi MJ7ofo Op d enn sLi xZx Jt CjC9 M71v O0fxiR 51 m FIB QRo 1oW Iq 3gDP stD2 ntfoX7 YU o S5k GuV IGM cf HZe3 7ZoG A1dDmk XO 2 KYR LpJ jII om M6Nu u8O0 jO5Nab Ub R nZn 15k hG9 4S 21V4 Ip45 7ooaiP u2 j hIz osW FDu O5 HdGr djvv tTLBjo vL L iCo 6L5 Lwa Pm vD6Z pal6 9Ljn11 re T 2CP mvj rL3 xH mDYK uv5T npC1fM oU R RTo Loi lk0 FE ghak m5M9 cOIPdQ lG D LnX erC ykJ C1 0FHh vvnY aTGuqU rf T QPv wEq iHO vO hD6A nXuv GlzVAv pz d Ok3 6ym yUo Fb AcAA BItO es52Vq d0 Y c7U 2gB t0W fF VQZh rJHr lBLdCx 8I o dWp AlD S8C HB rNLz xWp6 ypjuwW mg X toy 1vP bra uH yMNb kUrZ D6Ee2f zI D tkZ Eti Lmg re 1woD juLB BSdasY Vc F Uhy ViC xB1 5y Ltql qoUh gL3bZN YV k orz wa3 650 qW hF22 epiX cAjA4Z V4 b cXx uB3 NQN p0 GxW2 Vs1z jtqe2p LE B iS3 0E0 NKH gY N50v XaK6 pNpwdB X2 Y v7V 0Ud dTc Pi dRNN CLG4 7Fc3PL Bx K 3Be x1X zyX cj 0Z6a Jk0H KuQnwd Dh P Q1Q rwA 05v 9c 3pnz ttzt x2IirW CZ B oS5 xlO KCi D3 WFh4 dvCL QANAQJ Gg y vOD NTD FKj Mc 0RJP m4HU SQkLnT Q4 Y 6CC MvN jAR Zb lir7 RFsI NzHiJl cg f xSC Hts ZOG 1V uOzk 5G1C LtmRYI eD 3 5BB uxZ JdY LO CwS9 lokS NasDLj 5h 8 yni u7h u3c di zYh1 PdwE l3m8Xt yX Q RCA bwe aLi N8 qA9N 6DRE wy6gZe xs A 4fG EKH KQP PP KMbk sY1j M4h3Jj gS U One p1w RqN GA grL4 c18W v4kchD gR x 7Gj jIB zcK QV f7gA TrZx Oy6FF7 y9 3 iuu AQt 9TK Rx S5GO TFGx 4EQ120}   \end{align} We also conclude, with $\eps$ representing an arbitrarily small constant, that   \begin{align}   \begin{split}   & \EE    \left[   \int_{0}^{\tnm}\Vert f^{(1)} + f^{(6)}\Vert_{q}^p\,ds    \right]   \leq     \EE    \left[   \int_{0}^{\tnm}  \phin(\phin-\varphi^{(m)})^p   \Vert \un\Vert_{p}^p    (C_{\eps}\Vert \un\Vert_{p}^p + \eps\Vert \un\Vert_{3p}^p)   \,ds   \right]   \\ & \indeq\indeq   +\EE    \left[   \int_{0}^{\tnm}  \varphi^{(m)}(\phin-\varphi^{(m)})^p   \Vert \um\Vert_{p}^p    (C_{\eps}\Vert \um\Vert_{p}^p + \eps\Vert \um\Vert_{3p}^p)   \,ds   \right]   \\ & \indeq   \leq    (\eps+ C_{\eps} t)C_MK^p\EE    \left[   \sup_{s\in [0,\tau_{n, m}]}\Vert \wnm(s)\Vert_p^p   \right] .   \end{split}    \llabel{9r O tbc tCK rsi B0 4cFV Dl1g cvfWh6 5n x y9Z S4W Pyo QB yr3v fBkj TZKtEZ 7r U fdM icd yCV qn D036 HJWM tYfL9f yX x O7m IcF E1O uL QsAQ NfWv 6kV8Im 7Q 6 GsX NCV 0YP oC jnWn 6L25 qUMTe7 1v a hnH DAo XAb Tc zhPc fjrj W5M5G0 nz N M5T nlJ WOP Lh M6U2 ZFxw pg4Nej P8 U Q09 JX9 n7S kE WixE Rwgy Fvttzp 4A s v5F Tnn MzL Vh FUn5 6tFY CxZ1Bz Q3 E TfD lCa d7V fo MwPm ngrD HPfZV0 aY k Ojr ZUw 799 et oYuB MIC4 ovEY8D OL N URV Q5l ti1 iS NZAd wWr6 Q8oPFf ae 5 lAR 9gD RSi HO eJOW wxLv 20GoMt 2H z 7Yc aly PZx eR uFM0 7gaV 9UIz7S 43 k 5Tr ZiD Mt7 pE NCYi uHL7 gac7Gq yN 6 Z1u x56 YZh 2d yJVx 9MeU OMWBQf l0 E mIc 5Zr yfy 3i rahC y9Pi MJ7ofo Op d enn sLi xZx Jt CjC9 M71v O0fxiR 51 m FIB QRo 1oW Iq 3gDP stD2 ntfoX7 YU o S5k GuV IGM cf HZe3 7ZoG A1dDmk XO 2 KYR LpJ jII om M6Nu u8O0 jO5Nab Ub R nZn 15k hG9 4S 21V4 Ip45 7ooaiP u2 j hIz osW FDu O5 HdGr djvv tTLBjo vL L iCo 6L5 Lwa Pm vD6Z pal6 9Ljn11 re T 2CP mvj rL3 xH mDYK uv5T npC1fM oU R RTo Loi lk0 FE ghak m5M9 cOIPdQ lG D LnX erC ykJ C1 0FHh vvnY aTGuqU rf T QPv wEq iHO vO hD6A nXuv GlzVAv pz d Ok3 6ym yUo Fb AcAA BItO es52Vq d0 Y c7U 2gB t0W fF VQZh rJHr lBLdCx 8I o dWp AlD S8C HB rNLz xWp6 ypjuwW mg X toy 1vP bra uH yMNb kUrZ D6Ee2f zI D tkZ Eti Lmg re 1woD juLB BSdasY Vc F Uhy ViC xB1 5y Ltql qoUh gL3bZN YV k orz wa3 650 qW hF22 epiX cAjA4Z V4 b cXx uB3 NQN p0 GxW2 Vs1z jtqe2p LE B iS3 0E0 NKH gY N50v XaK6 pNpwdB X2 Y v7V 0Ud dTc Pi dRNN CLG4 7Fc3PL Bx K 3Be x1X zyX cj 0Z6a Jk0H KuQnwd Dh P Q1Q rwA 05v 9c 3pnz ttzt x2IirW CZ B oS5 xlO KCi D3 WFh4 dvCL QANAQJ Gg y vOD NTD FKj Mc 0RJP m4HU SQkLnT Q4 Y 6CC MvN jAR Zb lir7 RFsI NzHiJl cg f xSC Hts ZOG 1V uOzk 5G1C LtmRYI eD 3 5BB uxZ JdY LO CwS9 lokS NasDLj 5h 8 yni u7h u3c di zYh1 PdwE l3m8Xt yX Q RCA bwe aLi N8 qA9N 6DRE wy6gZe xs A 4fG EKH KQP PP KMbk sY1j M4h3Jj gS U One p1w RqN GA grL4 c18W v4kchD gR x 7Gj jIB zcK QV f7gA TrZx Oy6FF7 y9 3 iuu AQt 9TK Rx S5GO TFGx 4Xx1U3 R4 s 7U1 mpa bpD Hg kicx aCjk hnobr0 p4 c ody xTC kVj 8t W4iP 2OhT RF6kU2 k2 o oZJ Fsq Y4B FS NI3u W2fj OMFf7x Jv e ilb UVT ArC Tv qWLi vbRp g2wpAJ On l RUE PKh j9h EQ121}   \end{align} \par On the other hand, the stochastic coefficient can be written as   \begin{align}   \begin{split}   & (\phin)^2\pkn\sigma(\pkn\un)-(\varphi^{(m)})^2\pkm\sigma(\pkm\um)   \\&\indeq   =   \phin(\phin-\varphi^{(m)})\pkn\sigma(\pkn\un)   +  \phin\varphi^{(m)}\pkn\bigl(\sigma(\pkn\un)-\sigma(\pkn\um)\bigr)    \\&\indeq\indeq   +   \phin\varphi^{(m)}\pkn\bigl(\sigma(\pkn\um)-\sigma(\pkm\um) \bigr)   \\&\indeq\indeq     +    \phin\varphi^{(m)}\pnm\sigma(\pkm\um)    +    (\phin-\varphi^{(m)})\varphi^{(m)}\pkm\sigma(\pkm\um)    \\ &    = \sum_{k=1}^{5} g^{(k)}   .    \end{split}    \llabel{n 6L25 qUMTe7 1v a hnH DAo XAb Tc zhPc fjrj W5M5G0 nz N M5T nlJ WOP Lh M6U2 ZFxw pg4Nej P8 U Q09 JX9 n7S kE WixE Rwgy Fvttzp 4A s v5F Tnn MzL Vh FUn5 6tFY CxZ1Bz Q3 E TfD lCa d7V fo MwPm ngrD HPfZV0 aY k Ojr ZUw 799 et oYuB MIC4 ovEY8D OL N URV Q5l ti1 iS NZAd wWr6 Q8oPFf ae 5 lAR 9gD RSi HO eJOW wxLv 20GoMt 2H z 7Yc aly PZx eR uFM0 7gaV 9UIz7S 43 k 5Tr ZiD Mt7 pE NCYi uHL7 gac7Gq yN 6 Z1u x56 YZh 2d yJVx 9MeU OMWBQf l0 E mIc 5Zr yfy 3i rahC y9Pi MJ7ofo Op d enn sLi xZx Jt CjC9 M71v O0fxiR 51 m FIB QRo 1oW Iq 3gDP stD2 ntfoX7 YU o S5k GuV IGM cf HZe3 7ZoG A1dDmk XO 2 KYR LpJ jII om M6Nu u8O0 jO5Nab Ub R nZn 15k hG9 4S 21V4 Ip45 7ooaiP u2 j hIz osW FDu O5 HdGr djvv tTLBjo vL L iCo 6L5 Lwa Pm vD6Z pal6 9Ljn11 re T 2CP mvj rL3 xH mDYK uv5T npC1fM oU R RTo Loi lk0 FE ghak m5M9 cOIPdQ lG D LnX erC ykJ C1 0FHh vvnY aTGuqU rf T QPv wEq iHO vO hD6A nXuv GlzVAv pz d Ok3 6ym yUo Fb AcAA BItO es52Vq d0 Y c7U 2gB t0W fF VQZh rJHr lBLdCx 8I o dWp AlD S8C HB rNLz xWp6 ypjuwW mg X toy 1vP bra uH yMNb kUrZ D6Ee2f zI D tkZ Eti Lmg re 1woD juLB BSdasY Vc F Uhy ViC xB1 5y Ltql qoUh gL3bZN YV k orz wa3 650 qW hF22 epiX cAjA4Z V4 b cXx uB3 NQN p0 GxW2 Vs1z jtqe2p LE B iS3 0E0 NKH gY N50v XaK6 pNpwdB X2 Y v7V 0Ud dTc Pi dRNN CLG4 7Fc3PL Bx K 3Be x1X zyX cj 0Z6a Jk0H KuQnwd Dh P Q1Q rwA 05v 9c 3pnz ttzt x2IirW CZ B oS5 xlO KCi D3 WFh4 dvCL QANAQJ Gg y vOD NTD FKj Mc 0RJP m4HU SQkLnT Q4 Y 6CC MvN jAR Zb lir7 RFsI NzHiJl cg f xSC Hts ZOG 1V uOzk 5G1C LtmRYI eD 3 5BB uxZ JdY LO CwS9 lokS NasDLj 5h 8 yni u7h u3c di zYh1 PdwE l3m8Xt yX Q RCA bwe aLi N8 qA9N 6DRE wy6gZe xs A 4fG EKH KQP PP KMbk sY1j M4h3Jj gS U One p1w RqN GA grL4 c18W v4kchD gR x 7Gj jIB zcK QV f7gA TrZx Oy6FF7 y9 3 iuu AQt 9TK Rx S5GO TFGx 4Xx1U3 R4 s 7U1 mpa bpD Hg kicx aCjk hnobr0 p4 c ody xTC kVj 8t W4iP 2OhT RF6kU2 k2 o oZJ Fsq Y4B FS NI3u W2fj OMFf7x Jv e ilb UVT ArC Tv qWLi vbRp g2wpAJ On l RUE PKh j9h dG M0Mi gcqQ wkyunB Jr T LDc Pgn OSC HO sSgQ sR35 MB7Bgk Pk 6 nJh 01P Cxd Ds w514 O648 VD8iJ5 4F W 6rs 6Sy qGz MK fXop oe4e o52UNB 4Q 8 f8N Uz8 u2n GO AXHW gKtG AtGGJs bm EQ122}   \end{align}   Using continuity of $\pkm$ and $\pkn$ and the Lipschitz condition of $\varphi$ in \eqref{ERTWERTHWRTWERTSGDGHCFGSDFGQSERWDFGDSFGHSDRGTEHDFGHDSFGSDGHGYUHDFGSDFASDFASGTWRT146}, we obtain   \begin{align}   \begin{split} &\EE\left[ \int_0^{\tnm} \Vert   g^{(1)}+ g^{(5)} \Vert_{\mathbb{L}^p}^p\,ds \right] \leq  C \EE \left[ \int_0^{\tnm} \Vert  \wnm \Vert_{p}^{p}(\varphi^{(n)}\Vert \un \Vert_{(3p/2)-}^{2p} +\varphi^{(m)}\Vert \um \Vert_{(3p/2)-}^{2p}+ 1) \,ds \right] \\& \quad \leq  C \EE \left[ \int_0^{\tnm} \Vert  \wnm \Vert_{p}^{p} (\Vert \un \Vert_{3p}^{p-} +\Vert \um \Vert_{3p}^{p-}+ 1) \,ds \right] \\& \quad  \leq  C\EE \left[ \sup_{s\in [0,\tau^n_M\wedge \tau^n_M\wedge t]}\Vert \wnm(s)\Vert_p^p \int_0^{\tau^n_M\wedge \tau^n_M\wedge t} (\Vert \un \Vert_{3p}^{p-} +\Vert \um \Vert_{3p}^{p-}+1) \,ds \right]  \\ & \quad \leq  C_MK^pt^{\alpha}\EE  \left[ \sup_{s\in [0,\tau_{n, m}]}\Vert \wnm(s)\Vert_p^p \right] \end{split}\label{ERTWERTHWRTWERTSGDGHCFGSDFGQSERWDFGDSFGHSDRGTEHDFGHDSFGSDGHGYUHDFGSDFASDFASGTWRT123} \end{align} for some~$\alpha\in (0,1)$. Next, by \eqref{ERTWERTHWRTWERTSGDGHCFGSDFGQSERWDFGDSFGHSDRGTEHDFGHDSFGSDGHGYUHDFGSDFASDFASGTWRT04}, \begin{align}   \begin{split} &\EE\left[ \int_0^{\tnm} \Vert  g^{(2)} \Vert_{\mathbb{L}^p}^p\,ds \right] \leq  C \EE \left[ \int_0^{\tnm}\varphi^{(n)}\varphi^{(m)} \Vert (  |\pkn\un| + |\pkn\um|)^{1/2} |\pkn\wnm| \Vert_{p}^p \,ds \right]
\\&\quad  \leq  C \EE \left[ \int_0^{\tnm} \Bigl( \varphi^{(n)} \Vert  \pkn\un \Vert_{p}^{p/2} \Vert   \pkn\wnm \Vert_{2p}^p + \varphi^{(m)} \Vert  \pkn\um \Vert_{p}^{p/2} \Vert   \pkn\wnm \Vert_{2p}^p \Bigr) \,ds \right] \\&\quad  \leq   C_{\epsilon}t\EE  \biggl[ \sup_{s\in [0,\tau_{n, m}]}\Vert \wnm(s)\Vert_p^p \biggr]  + \epsilon\EE  \biggl[ \int_0^{\tnm} \Vert \wnm(s)\Vert_{3p}^p \,ds \biggr], \end{split}    \label{ERTWERTHWRTWERTSGDGHCFGSDFGQSERWDFGDSFGHSDRGTEHDFGHDSFGSDGHGYUHDFGSDFASDFASGTWRT124} \end{align} where $\epsilon>0$ can be arbitrarily small. It follows from the same assumption \eqref{ERTWERTHWRTWERTSGDGHCFGSDFGQSERWDFGDSFGHSDRGTEHDFGHDSFGSDGHGYUHDFGSDFASDFASGTWRT04} that \begin{align}   \begin{split} &\EE\left[ \int_0^{\tnm} \Vert  g^{(3)} \Vert_{\mathbb{L}^p}^p\,ds \right] \leq  C \EE \left[ \int_0^{\tnm}\varphi^{(m)} \Vert (  |\pkn\um| + |\pkm\um|)^{1/2} |\pnm\um| \Vert_{p}^p \,ds \right] \\&\quad  \leq  C \EE \left[ \int_0^{\tnm}\varphi^{(m)} \Bigl( \Vert  \pkn\um \Vert_{p}^{p/2} + \Vert  \pkm\um \Vert_{p}^{p/2} \Bigr) \Vert   P_{n,m}\um \Vert_{2p}^p \,ds \right] \\&\quad  \leq  C \EE \left[ \int_0^{\tnm}\varphi^{(m)} \Vert  |P_{n,m}\um|^{\tilde{\delta}} \Vert_{2/\tilde{\delta}}^{p} \Vert   |P_{n,m}\um|^{1-\tilde{\delta}} \Vert_{\tilde{l}}^p \,ds \right] \\&\quad  \leq  C \EE \left[ \int_0^{\tnm}\varphi^{(m)} \Vert  P_{n,m}\um \Vert_{2}^{\tilde{\delta}p} \Vert   P_{n,m}\um \Vert_{(1-\tilde{\delta})\tilde{l}}^{(1-\tilde{\delta})p} \,ds \right], \end{split}    \llabel{lCa d7V fo MwPm ngrD HPfZV0 aY k Ojr ZUw 799 et oYuB MIC4 ovEY8D OL N URV Q5l ti1 iS NZAd wWr6 Q8oPFf ae 5 lAR 9gD RSi HO eJOW wxLv 20GoMt 2H z 7Yc aly PZx eR uFM0 7gaV 9UIz7S 43 k 5Tr ZiD Mt7 pE NCYi uHL7 gac7Gq yN 6 Z1u x56 YZh 2d yJVx 9MeU OMWBQf l0 E mIc 5Zr yfy 3i rahC y9Pi MJ7ofo Op d enn sLi xZx Jt CjC9 M71v O0fxiR 51 m FIB QRo 1oW Iq 3gDP stD2 ntfoX7 YU o S5k GuV IGM cf HZe3 7ZoG A1dDmk XO 2 KYR LpJ jII om M6Nu u8O0 jO5Nab Ub R nZn 15k hG9 4S 21V4 Ip45 7ooaiP u2 j hIz osW FDu O5 HdGr djvv tTLBjo vL L iCo 6L5 Lwa Pm vD6Z pal6 9Ljn11 re T 2CP mvj rL3 xH mDYK uv5T npC1fM oU R RTo Loi lk0 FE ghak m5M9 cOIPdQ lG D LnX erC ykJ C1 0FHh vvnY aTGuqU rf T QPv wEq iHO vO hD6A nXuv GlzVAv pz d Ok3 6ym yUo Fb AcAA BItO es52Vq d0 Y c7U 2gB t0W fF VQZh rJHr lBLdCx 8I o dWp AlD S8C HB rNLz xWp6 ypjuwW mg X toy 1vP bra uH yMNb kUrZ D6Ee2f zI D tkZ Eti Lmg re 1woD juLB BSdasY Vc F Uhy ViC xB1 5y Ltql qoUh gL3bZN YV k orz wa3 650 qW hF22 epiX cAjA4Z V4 b cXx uB3 NQN p0 GxW2 Vs1z jtqe2p LE B iS3 0E0 NKH gY N50v XaK6 pNpwdB X2 Y v7V 0Ud dTc Pi dRNN CLG4 7Fc3PL Bx K 3Be x1X zyX cj 0Z6a Jk0H KuQnwd Dh P Q1Q rwA 05v 9c 3pnz ttzt x2IirW CZ B oS5 xlO KCi D3 WFh4 dvCL QANAQJ Gg y vOD NTD FKj Mc 0RJP m4HU SQkLnT Q4 Y 6CC MvN jAR Zb lir7 RFsI NzHiJl cg f xSC Hts ZOG 1V uOzk 5G1C LtmRYI eD 3 5BB uxZ JdY LO CwS9 lokS NasDLj 5h 8 yni u7h u3c di zYh1 PdwE l3m8Xt yX Q RCA bwe aLi N8 qA9N 6DRE wy6gZe xs A 4fG EKH KQP PP KMbk sY1j M4h3Jj gS U One p1w RqN GA grL4 c18W v4kchD gR x 7Gj jIB zcK QV f7gA TrZx Oy6FF7 y9 3 iuu AQt 9TK Rx S5GO TFGx 4Xx1U3 R4 s 7U1 mpa bpD Hg kicx aCjk hnobr0 p4 c ody xTC kVj 8t W4iP 2OhT RF6kU2 k2 o oZJ Fsq Y4B FS NI3u W2fj OMFf7x Jv e ilb UVT ArC Tv qWLi vbRp g2wpAJ On l RUE PKh j9h dG M0Mi gcqQ wkyunB Jr T LDc Pgn OSC HO sSgQ sR35 MB7Bgk Pk 6 nJh 01P Cxd Ds w514 O648 VD8iJ5 4F W 6rs 6Sy qGz MK fXop oe4e o52UNB 4Q 8 f8N Uz8 u2n GO AXHW gKtG AtGGJs bm z 2qj vSv GBu 5e 4JgL Aqrm gMmS08 ZF s xQm 28M 3z4 Ho 1xxj j8Uk bMbm8M 0c L PL5 TS2 kIQ jZ Kb9Q Ux2U i5Aflw 1S L DGI uWU dCP jy wVVM 2ct8 cmgOBS 7d Q ViX R8F bta 1m tEFj TEQ125} \end{align} where $\tilde{\delta}$ is a constant in~$(0,1/(3p-2))$. Suppose that $1/(1-\tilde{\delta})\tilde{l}=\tilde{\theta}/p+(1-\tilde{\theta})/(3p)$ for some $\tilde{\theta}\in (0,1)$. Then, \begin{align}   \begin{split} &\EE\left[ \int_0^{\tnm} \Vert  g^{(3)} \Vert_{\mathbb{L}^p}^p\,ds \right] \leq  \frac{C}{ (k(m)\wedge k(n))^{\tilde{\delta} p}} \EE \left[ \int_0^{\tnm}\varphi^{(m)} \Vert  \nabla\um \Vert_{2}^{\tilde{\delta}p} \Vert   \um \Vert_{p}^{(1-\tilde{\delta})\tilde{\theta} p} \Vert   \um \Vert_{3p}^{(1-\tilde{\delta})(1-\tilde{\theta})p} \,ds \right] \\&\qquad  \leq  \frac{C}{ (k(m)\wedge k(n))^{\tilde{\delta} p}} \EE \left[ \int_0^{\tnm} \Bigl( \Vert  \nabla\um \Vert_{2}^{2} + \Vert    \um \Vert_{3p}^{2(1-\tilde{\delta})(1-\tilde{\theta})p/(2-\tilde{\delta}p)} \Bigr) \,ds \right] \\&\qquad  \leq  \frac{C_T}{ (k(m)\wedge k(n))^{\tilde{\delta} p}} \EE \left[ \int_0^{\tnm} ( \Vert  \nabla\um \Vert_{2}^{2} + \Vert    \um \Vert_{3p}^{p} ) \,ds \right] . \end{split}    \label{ERTWERTHWRTWERTSGDGHCFGSDFGQSERWDFGDSFGHSDRGTEHDFGHDSFGSDGHGYUHDFGSDFASDFASGTWRT126} \end{align} Above, we required $2(1-\tilde{\delta})(1-\tilde{\theta})p/(2-\tilde{\delta}p)<p$, which is equivalent to $\tilde{\theta}>1-(2-\tilde{\delta}p)/2(1-\tilde{\delta})$ and holds when $\tilde{\delta}$ is sufficiently small. Lastly, we use the growth assumption on $\sigma$ and conclude \begin{align}   \begin{split} &\EE\left[ \int_0^{\tnm} \Vert  g^{(4)} \Vert_{\mathbb{L}^p}^p\,ds
\right] \leq  \frac{C}{ (k(m)\wedge k(n))^{p}} \EE \left[ \int_0^{\tnm} (\varphi^{(m)}\Vert \um \Vert_{3p/2}^{2p} + 1) \,ds \right] \leq  \frac{C_MK^p}{ (k(m)\wedge k(n))^{p}}. \end{split}    \label{ERTWERTHWRTWERTSGDGHCFGSDFGQSERWDFGDSFGHSDRGTEHDFGHDSFGSDGHGYUHDFGSDFASDFASGTWRT127} \end{align} \par Combining the estimates above, choosing an appropriate $\epsilon$, and making $t$ sufficiently small, we arrive at   \begin{align}   \begin{split}   & \EE\biggl[   \sup_{0\leq s\leq \tau_{n, m}}\Vert u^{(n)}(s)-u^{(m)}(s)\Vert_p^p   +\int_0^{\tau_{n, m}}    \sum_{j}    \int_{\RR^d} | \nabla (|u^{(n)}_j-u^{(m)}_j|^{p/2})|^2 \,dx\,ds   \biggr]   \\&\indeq   \leq    \frac{C_T}{ (k(m)\wedge k(n))^{\beta p}}\EE   \left[   \int_0^{\tnm}   (   \Vert    \nabla\um   \Vert_{2}^{2}   +   \Vert      \um   \Vert_{3p}^{p}   )   \,ds   \right]+   \frac{C_MK^p}{ (k(m)\wedge k(n))^{p}}   \\ &\qquad +   \EE\left[   \left\Vert    \pkn \mathcal{P} \left(\varphi\left(\frac{\cdot}{n}\right) u_0 \right)      -\pkm \mathcal{P} \left(\varphi      \left(\frac{\cdot}{m}\right) u_0 \right)   \right\Vert_p^p   \right]   \end{split}   \llabel{Iz7S 43 k 5Tr ZiD Mt7 pE NCYi uHL7 gac7Gq yN 6 Z1u x56 YZh 2d yJVx 9MeU OMWBQf l0 E mIc 5Zr yfy 3i rahC y9Pi MJ7ofo Op d enn sLi xZx Jt CjC9 M71v O0fxiR 51 m FIB QRo 1oW Iq 3gDP stD2 ntfoX7 YU o S5k GuV IGM cf HZe3 7ZoG A1dDmk XO 2 KYR LpJ jII om M6Nu u8O0 jO5Nab Ub R nZn 15k hG9 4S 21V4 Ip45 7ooaiP u2 j hIz osW FDu O5 HdGr djvv tTLBjo vL L iCo 6L5 Lwa Pm vD6Z pal6 9Ljn11 re T 2CP mvj rL3 xH mDYK uv5T npC1fM oU R RTo Loi lk0 FE ghak m5M9 cOIPdQ lG D LnX erC ykJ C1 0FHh vvnY aTGuqU rf T QPv wEq iHO vO hD6A nXuv GlzVAv pz d Ok3 6ym yUo Fb AcAA BItO es52Vq d0 Y c7U 2gB t0W fF VQZh rJHr lBLdCx 8I o dWp AlD S8C HB rNLz xWp6 ypjuwW mg X toy 1vP bra uH yMNb kUrZ D6Ee2f zI D tkZ Eti Lmg re 1woD juLB BSdasY Vc F Uhy ViC xB1 5y Ltql qoUh gL3bZN YV k orz wa3 650 qW hF22 epiX cAjA4Z V4 b cXx uB3 NQN p0 GxW2 Vs1z jtqe2p LE B iS3 0E0 NKH gY N50v XaK6 pNpwdB X2 Y v7V 0Ud dTc Pi dRNN CLG4 7Fc3PL Bx K 3Be x1X zyX cj 0Z6a Jk0H KuQnwd Dh P Q1Q rwA 05v 9c 3pnz ttzt x2IirW CZ B oS5 xlO KCi D3 WFh4 dvCL QANAQJ Gg y vOD NTD FKj Mc 0RJP m4HU SQkLnT Q4 Y 6CC MvN jAR Zb lir7 RFsI NzHiJl cg f xSC Hts ZOG 1V uOzk 5G1C LtmRYI eD 3 5BB uxZ JdY LO CwS9 lokS NasDLj 5h 8 yni u7h u3c di zYh1 PdwE l3m8Xt yX Q RCA bwe aLi N8 qA9N 6DRE wy6gZe xs A 4fG EKH KQP PP KMbk sY1j M4h3Jj gS U One p1w RqN GA grL4 c18W v4kchD gR x 7Gj jIB zcK QV f7gA TrZx Oy6FF7 y9 3 iuu AQt 9TK Rx S5GO TFGx 4Xx1U3 R4 s 7U1 mpa bpD Hg kicx aCjk hnobr0 p4 c ody xTC kVj 8t W4iP 2OhT RF6kU2 k2 o oZJ Fsq Y4B FS NI3u W2fj OMFf7x Jv e ilb UVT ArC Tv qWLi vbRp g2wpAJ On l RUE PKh j9h dG M0Mi gcqQ wkyunB Jr T LDc Pgn OSC HO sSgQ sR35 MB7Bgk Pk 6 nJh 01P Cxd Ds w514 O648 VD8iJ5 4F W 6rs 6Sy qGz MK fXop oe4e o52UNB 4Q 8 f8N Uz8 u2n GO AXHW gKtG AtGGJs bm z 2qj vSv GBu 5e 4JgL Aqrm gMmS08 ZF s xQm 28M 3z4 Ho 1xxj j8Uk bMbm8M 0c L PL5 TS2 kIQ jZ Kb9Q Ux2U i5Aflw 1S L DGI uWU dCP jy wVVM 2ct8 cmgOBS 7d Q ViX R8F bta 1m tEFj TO0k owcK2d 6M Z iW8 PrK PI1 sX WJNB cREV Y4H5QQ GH b plP bwd Txp OI 5OQZ AKyi ix7Qey YI 9 1Ea 16r KXK L2 ifQX QPdP NL6EJi Hc K rBs 2qG tQb aq edOj Lixj GiNWr1 Pb Y SZe SxxEQ128}   \end{align} for some~$\beta\in(0,1)$. With~\eqref{ERTWERTHWRTWERTSGDGHCFGSDFGQSERWDFGDSFGHSDRGTEHDFGHDSFGSDGHGYUHDFGSDFASDFASGTWRT80},~\eqref{ERTWERTHWRTWERTSGDGHCFGSDFGQSERWDFGDSFGHSDRGTEHDFGHDSFGSDGHGYUHDFGSDFASDFASGTWRT14}, and Lemmas~\ref{L02} and~\ref{L0}, we conclude~\eqref{ERTWERTHWRTWERTSGDGHCFGSDFGQSERWDFGDSFGHSDRGTEHDFGHDSFGSDGHGYUHDFGSDFASDFASGTWRT112}. \end{proof} \par The following lemma states a pointwise Cauchy condition in the probability space. Moreover, it asserts the existence of a positive stopping time up to which the Cauchy condition uniformly holds for a subsequence of approximate solutions. \par \cole  \begin{Lemma} \label{L07} Let~$p>d=3$ and~$K\geq1$. Suppose that $\Vert u_{0}\Vert_{p} \leq K$. Then, there exist a stopping time $\tau$ with $\PP(\tau>0)=1$, a subsequence $\{u^{(n_k)}\}$, and an adapted process $u\in L^p(\Omega, C([0,\tau], L^p))\cap L^{p}(\Omega, L^{p}([0,\tau], L^{3p}))$ such that    \begin{align}   \begin{split}   \lim_{k\to\infty}   \left(   \sup_{0\leq s\leq \tau}\Vert u^{(n_k)}(s)-u(s)\Vert_p^p   +\int_0^{\tau}    \Vert u^{(n_k)}(s)-u(s)\Vert_{3p}^p\,ds   \right)   =0\quad{}    \Pas,   \end{split}   \label{ERTWERTHWRTWERTSGDGHCFGSDFGQSERWDFGDSFGHSDRGTEHDFGHDSFGSDGHGYUHDFGSDFASDFASGTWRT129}   \end{align} and $\tau$ depends on~$K$. Moreover,    \begin{equation}   \label{ERTWERTHWRTWERTSGDGHCFGSDFGQSERWDFGDSFGHSDRGTEHDFGHDSFGSDGHGYUHDFGSDFASDFASGTWRT130}   \EE\biggl[\sup_{[0, \tau]}\Vert u \Vert_{p}^p + \int_{0}^{\tau} \int_{\RR^\dd}\sum_j |\nabla |u_j|^{p/2}|^2 \,dx ds\biggr] \leq C K^p,    \end{equation} where $C$ is independent of~$K$. \end{Lemma} \colb \par \begin{proof}[Proof of Lemma~\ref{L07}]  By Theorem~\ref{T04} and \eqref{ERTWERTHWRTWERTSGDGHCFGSDFGQSERWDFGDSFGHSDRGTEHDFGHDSFGSDGHGYUHDFGSDFASDFASGTWRT80}, there exist constants $M$ and $M_1$ such that $M_1>M_0$, $M>2^{p+1} M_1$, and     \begin{align}   \begin{split}   &\lim_{t\to 0} \sup_{n} \PP\biggl(   \sup_{s\in [0,\tau^n_{M} \wedge t] }\Vert u^{(n)}(s)\Vert_p^p   +\int_0^{\tau^n_{M}\wedge t}    \Vert u^{(n)}(s)\Vert_{3p}^{p} \, ds\geq \frac{M_1}{2^p}K^p   \biggr)   \\&\quad\le   \lim_{t\rightarrow0} \sup_n \PP\paren{\sup_{s\in [0,\tau^n_{M} \wedge t] }\Vert u^{(n)}(s)\Vert_p^p + \int_{0}^{\tau^n_{M}\wedge  t} \int_{\RR^\dd}\sum_j |\nabla |u_j^{(n)}|^{p/2}|^2 \,dx ds \geq     M_0K^p }     =0   .   \end{split}\llabel{q 3gDP stD2 ntfoX7 YU o S5k GuV IGM cf HZe3 7ZoG A1dDmk XO 2 KYR LpJ jII om M6Nu u8O0 jO5Nab Ub R nZn 15k hG9 4S 21V4 Ip45 7ooaiP u2 j hIz osW FDu O5 HdGr djvv tTLBjo vL L iCo 6L5 Lwa Pm vD6Z pal6 9Ljn11 re T 2CP mvj rL3 xH mDYK uv5T npC1fM oU R RTo Loi lk0 FE ghak m5M9 cOIPdQ lG D LnX erC ykJ C1 0FHh vvnY aTGuqU rf T QPv wEq iHO vO hD6A nXuv GlzVAv pz d Ok3 6ym yUo Fb AcAA BItO es52Vq d0 Y c7U 2gB t0W fF VQZh rJHr lBLdCx 8I o dWp AlD S8C HB rNLz xWp6 ypjuwW mg X toy 1vP bra uH yMNb kUrZ D6Ee2f zI D tkZ Eti Lmg re 1woD juLB BSdasY Vc F Uhy ViC xB1 5y Ltql qoUh gL3bZN YV k orz wa3 650 qW hF22 epiX cAjA4Z V4 b cXx uB3 NQN p0 GxW2 Vs1z jtqe2p LE B iS3 0E0 NKH gY N50v XaK6 pNpwdB X2 Y v7V 0Ud dTc Pi dRNN CLG4 7Fc3PL Bx K 3Be x1X zyX cj 0Z6a Jk0H KuQnwd Dh P Q1Q rwA 05v 9c 3pnz ttzt x2IirW CZ B oS5 xlO KCi D3 WFh4 dvCL QANAQJ Gg y vOD NTD FKj Mc 0RJP m4HU SQkLnT Q4 Y 6CC MvN jAR Zb lir7 RFsI NzHiJl cg f xSC Hts ZOG 1V uOzk 5G1C LtmRYI eD 3 5BB uxZ JdY LO CwS9 lokS NasDLj 5h 8 yni u7h u3c di zYh1 PdwE l3m8Xt yX Q RCA bwe aLi N8 qA9N 6DRE wy6gZe xs A 4fG EKH KQP PP KMbk sY1j M4h3Jj gS U One p1w RqN GA grL4 c18W v4kchD gR x 7Gj jIB zcK QV f7gA TrZx Oy6FF7 y9 3 iuu AQt 9TK Rx S5GO TFGx 4Xx1U3 R4 s 7U1 mpa bpD Hg kicx aCjk hnobr0 p4 c ody xTC kVj 8t W4iP 2OhT RF6kU2 k2 o oZJ Fsq Y4B FS NI3u W2fj OMFf7x Jv e ilb UVT ArC Tv qWLi vbRp g2wpAJ On l RUE PKh j9h dG M0Mi gcqQ wkyunB Jr T LDc Pgn OSC HO sSgQ sR35 MB7Bgk Pk 6 nJh 01P Cxd Ds w514 O648 VD8iJ5 4F W 6rs 6Sy qGz MK fXop oe4e o52UNB 4Q 8 f8N Uz8 u2n GO AXHW gKtG AtGGJs bm z 2qj vSv GBu 5e 4JgL Aqrm gMmS08 ZF s xQm 28M 3z4 Ho 1xxj j8Uk bMbm8M 0c L PL5 TS2 kIQ jZ Kb9Q Ux2U i5Aflw 1S L DGI uWU dCP jy wVVM 2ct8 cmgOBS 7d Q ViX R8F bta 1m tEFj TO0k owcK2d 6M Z iW8 PrK PI1 sX WJNB cREV Y4H5QQ GH b plP bwd Txp OI 5OQZ AKyi ix7Qey YI 9 1Ea 16r KXK L2 ifQX QPdP NL6EJi Hc K rBs 2qG tQb aq edOj Lixj GiNWr1 Pb Y SZe Sxx Fin aK 9Eki CHV2 a13f7G 3G 3 oDK K0i bKV y4 53E2 nFQS 8Hnqg0 E3 2 ADd dEV nmJ 7H Bc1t 2K2i hCzZuy 9k p sHn 8Ko uAR kv sHKP y8Yo dOOqBi hF 1 Z3C vUF hmj gB muZq 7ggW Lg5dQEQ131}   \end{align} Let $\bar{t}$ be the deterministic time suggested by Lemma~\ref{L06}, and denote $\tau_{n, m} =\tau^n_M\wedge \tau^n_M\wedge \bar{t}$. From \eqref{ERTWERTHWRTWERTSGDGHCFGSDFGQSERWDFGDSFGHSDRGTEHDFGHDSFGSDGHGYUHDFGSDFASDFASGTWRT112}, we infer the existence of a subsequence $\{u^{(n_k)}\}$ for which    \begin{align}   \begin{split}   \EE\biggl[   \sup_{s\in [0,\tau_{n_{k+1}, n_k}] }\Vert u^{(n_{k+1})}(s)-u^{(n_k)}(s)\Vert_p^p   +\int_0^{\tau_{n_{k+1}, n_k}}    \Vert u^{(n_{k+1})}(s)-u^{(n_k)}(s)\Vert_{3p}^p\,ds   \biggr]   \leq\frac{1}{8^{kp}}    .   \end{split}\label{ERTWERTHWRTWERTSGDGHCFGSDFGQSERWDFGDSFGHSDRGTEHDFGHDSFGSDGHGYUHDFGSDFASDFASGTWRT133}   \end{align} Consider the stopping time   \begin{align}   \begin{split}    \eta_k = \inf\biggl\{t>0: \sup_{[0, t]}\Vert u^{(n_k)}\Vert _{p} + \biggl(\int_{0}^{t} \Vert u^{(n_k)}\Vert _{3p}^{p}\, ds\biggr)^{1/p}\geq M_1^{1/p}K+\frac{1}{2^k}\biggr\},   \end{split}   \llabel{ iCo 6L5 Lwa Pm vD6Z pal6 9Ljn11 re T 2CP mvj rL3 xH mDYK uv5T npC1fM oU R RTo Loi lk0 FE ghak m5M9 cOIPdQ lG D LnX erC ykJ C1 0FHh vvnY aTGuqU rf T QPv wEq iHO vO hD6A nXuv GlzVAv pz d Ok3 6ym yUo Fb AcAA BItO es52Vq d0 Y c7U 2gB t0W fF VQZh rJHr lBLdCx 8I o dWp AlD S8C HB rNLz xWp6 ypjuwW mg X toy 1vP bra uH yMNb kUrZ D6Ee2f zI D tkZ Eti Lmg re 1woD juLB BSdasY Vc F Uhy ViC xB1 5y Ltql qoUh gL3bZN YV k orz wa3 650 qW hF22 epiX cAjA4Z V4 b cXx uB3 NQN p0 GxW2 Vs1z jtqe2p LE B iS3 0E0 NKH gY N50v XaK6 pNpwdB X2 Y v7V 0Ud dTc Pi dRNN CLG4 7Fc3PL Bx K 3Be x1X zyX cj 0Z6a Jk0H KuQnwd Dh P Q1Q rwA 05v 9c 3pnz ttzt x2IirW CZ B oS5 xlO KCi D3 WFh4 dvCL QANAQJ Gg y vOD NTD FKj Mc 0RJP m4HU SQkLnT Q4 Y 6CC MvN jAR Zb lir7 RFsI NzHiJl cg f xSC Hts ZOG 1V uOzk 5G1C LtmRYI eD 3 5BB uxZ JdY LO CwS9 lokS NasDLj 5h 8 yni u7h u3c di zYh1 PdwE l3m8Xt yX Q RCA bwe aLi N8 qA9N 6DRE wy6gZe xs A 4fG EKH KQP PP KMbk sY1j M4h3Jj gS U One p1w RqN GA grL4 c18W v4kchD gR x 7Gj jIB zcK QV f7gA TrZx Oy6FF7 y9 3 iuu AQt 9TK Rx S5GO TFGx 4Xx1U3 R4 s 7U1 mpa bpD Hg kicx aCjk hnobr0 p4 c ody xTC kVj 8t W4iP 2OhT RF6kU2 k2 o oZJ Fsq Y4B FS NI3u W2fj OMFf7x Jv e ilb UVT ArC Tv qWLi vbRp g2wpAJ On l RUE PKh j9h dG M0Mi gcqQ wkyunB Jr T LDc Pgn OSC HO sSgQ sR35 MB7Bgk Pk 6 nJh 01P Cxd Ds w514 O648 VD8iJ5 4F W 6rs 6Sy qGz MK fXop oe4e o52UNB 4Q 8 f8N Uz8 u2n GO AXHW gKtG AtGGJs bm z 2qj vSv GBu 5e 4JgL Aqrm gMmS08 ZF s xQm 28M 3z4 Ho 1xxj j8Uk bMbm8M 0c L PL5 TS2 kIQ jZ Kb9Q Ux2U i5Aflw 1S L DGI uWU dCP jy wVVM 2ct8 cmgOBS 7d Q ViX R8F bta 1m tEFj TO0k owcK2d 6M Z iW8 PrK PI1 sX WJNB cREV Y4H5QQ GH b plP bwd Txp OI 5OQZ AKyi ix7Qey YI 9 1Ea 16r KXK L2 ifQX QPdP NL6EJi Hc K rBs 2qG tQb aq edOj Lixj GiNWr1 Pb Y SZe Sxx Fin aK 9Eki CHV2 a13f7G 3G 3 oDK K0i bKV y4 53E2 nFQS 8Hnqg0 E3 2 ADd dEV nmJ 7H Bc1t 2K2i hCzZuy 9k p sHn 8Ko uAR kv sHKP y8Yo dOOqBi hF 1 Z3C vUF hmj gB muZq 7ggW Lg5dQB 1k p Fxk k35 GFo dk 00YD 13qI qqbLwy QC c yZR wHA fp7 9o imtC c5CV 8cEuwU w7 k 8Q7 nCq WkM gY rtVR IySM tZUGCH XV 9 mr9 GHZ ol0 VE eIjQ vwgw 17pDhX JS F UcY bqU gnG V8 IEQ134}   \end{align} and the probability events    \begin{equation}    \Omega_k = \biggl\{\omega: \sup_{s\in [0,\tau_{n_{k+1}, n_k}] }\Vert u^{(n_{k+1})}(s)-u^{(n_k)}(s)\Vert_p     +\biggl(\int_0^{\tau_{n_{k+1}, n_k}}         \Vert u^{(n_{k+1})}(s)-u^{(n_k)}(s)\Vert_{3p}^p\,ds      \biggr)^{1/p}     \geq \frac{1}{4^k}               \biggr\} . \label{ERTWERTHWRTWERTSGDGHCFGSDFGQSERWDFGDSFGHSDRGTEHDFGHDSFGSDGHGYUHDFGSDFASDFASGTWRT135} \end{equation} Our choices of $M_1$ and $M$ ensure that $0<\eta_k\leq \tau^{n_{k}}_{M}$ for every~$k\in \NNp$. Moreover, $\eta_k\wedge \bar{t}\geq \eta_{k+1}\wedge \bar{t}$ in the set $\bigcap_{k\geq N} \Omega_k^c$ for every $N\in\NNp$ and~$k\geq~N$. Hence, $\lim_{k\to\infty} \eta_k\wedge \bar{t}$ exists in $\bigcap_{k\geq N} \Omega_k^c$. In fact, this limit is well-defined almost everywhere in the probability space $\Omega$ because by \eqref{ERTWERTHWRTWERTSGDGHCFGSDFGQSERWDFGDSFGHSDRGTEHDFGHDSFGSDGHGYUHDFGSDFASDFASGTWRT133} and Chebyshev's inequality,   \begin{align}   \begin{split}    &\lim_{N\to\infty} \PP\biggl(    \bigcap_{k\geq N} \Omega_k^c    \biggr)    =    \PP\biggl(    \bigcup_N \bigcap_{k\geq N} \Omega_k^c    \biggr)    =    1-\PP\biggl(    \bigcap_N \bigcup_{k\geq N} \Omega_k    \biggr)    \geq \lim_{N\to\infty}    \biggl(1- \sum_{k\geq N}    \PP(\Omega_k)    \biggr)    \geq 1.    \end{split}    \llabel{uv GlzVAv pz d Ok3 6ym yUo Fb AcAA BItO es52Vq d0 Y c7U 2gB t0W fF VQZh rJHr lBLdCx 8I o dWp AlD S8C HB rNLz xWp6 ypjuwW mg X toy 1vP bra uH yMNb kUrZ D6Ee2f zI D tkZ Eti Lmg re 1woD juLB BSdasY Vc F Uhy ViC xB1 5y Ltql qoUh gL3bZN YV k orz wa3 650 qW hF22 epiX cAjA4Z V4 b cXx uB3 NQN p0 GxW2 Vs1z jtqe2p LE B iS3 0E0 NKH gY N50v XaK6 pNpwdB X2 Y v7V 0Ud dTc Pi dRNN CLG4 7Fc3PL Bx K 3Be x1X zyX cj 0Z6a Jk0H KuQnwd Dh P Q1Q rwA 05v 9c 3pnz ttzt x2IirW CZ B oS5 xlO KCi D3 WFh4 dvCL QANAQJ Gg y vOD NTD FKj Mc 0RJP m4HU SQkLnT Q4 Y 6CC MvN jAR Zb lir7 RFsI NzHiJl cg f xSC Hts ZOG 1V uOzk 5G1C LtmRYI eD 3 5BB uxZ JdY LO CwS9 lokS NasDLj 5h 8 yni u7h u3c di zYh1 PdwE l3m8Xt yX Q RCA bwe aLi N8 qA9N 6DRE wy6gZe xs A 4fG EKH KQP PP KMbk sY1j M4h3Jj gS U One p1w RqN GA grL4 c18W v4kchD gR x 7Gj jIB zcK QV f7gA TrZx Oy6FF7 y9 3 iuu AQt 9TK Rx S5GO TFGx 4Xx1U3 R4 s 7U1 mpa bpD Hg kicx aCjk hnobr0 p4 c ody xTC kVj 8t W4iP 2OhT RF6kU2 k2 o oZJ Fsq Y4B FS NI3u W2fj OMFf7x Jv e ilb UVT ArC Tv qWLi vbRp g2wpAJ On l RUE PKh j9h dG M0Mi gcqQ wkyunB Jr T LDc Pgn OSC HO sSgQ sR35 MB7Bgk Pk 6 nJh 01P Cxd Ds w514 O648 VD8iJ5 4F W 6rs 6Sy qGz MK fXop oe4e o52UNB 4Q 8 f8N Uz8 u2n GO AXHW gKtG AtGGJs bm z 2qj vSv GBu 5e 4JgL Aqrm gMmS08 ZF s xQm 28M 3z4 Ho 1xxj j8Uk bMbm8M 0c L PL5 TS2 kIQ jZ Kb9Q Ux2U i5Aflw 1S L DGI uWU dCP jy wVVM 2ct8 cmgOBS 7d Q ViX R8F bta 1m tEFj TO0k owcK2d 6M Z iW8 PrK PI1 sX WJNB cREV Y4H5QQ GH b plP bwd Txp OI 5OQZ AKyi ix7Qey YI 9 1Ea 16r KXK L2 ifQX QPdP NL6EJi Hc K rBs 2qG tQb aq edOj Lixj GiNWr1 Pb Y SZe Sxx Fin aK 9Eki CHV2 a13f7G 3G 3 oDK K0i bKV y4 53E2 nFQS 8Hnqg0 E3 2 ADd dEV nmJ 7H Bc1t 2K2i hCzZuy 9k p sHn 8Ko uAR kv sHKP y8Yo dOOqBi hF 1 Z3C vUF hmj gB muZq 7ggW Lg5dQB 1k p Fxk k35 GFo dk 00YD 13qI qqbLwy QC c yZR wHA fp7 9o imtC c5CV 8cEuwU w7 k 8Q7 nCq WkM gY rtVR IySM tZUGCH XV 9 mr9 GHZ ol0 VE eIjQ vwgw 17pDhX JS F UcY bqU gnG V8 IFWb S1GX az0ZTt 81 w 7En IhF F72 v2 PkWO Xlkr w6IPu5 67 9 vcW 1f6 z99 lM 2LI1 Y6Na axfl18 gT 0 gDp tVl CN4 jf GSbC ro5D v78Cxa uk Y iUI WWy YDR w8 z7Kj Px7C hC7zJv b1 b 0rEQ136} \end{align} We then show that $\tau:=\lim_{k\to\infty} \eta_k\wedge \bar{t}$ is non-degenerate. For $\epsilon\in (0, \bar{t})$,  \begin{align} \begin{split} &\PP ( \tau<\epsilon ) =  \PP\biggl(      \bigcup_N \bigcap_{k\geq N} \big\{\eta_k\wedge \bar{t}<\epsilon\big\}    \biggr) = \lim_{N\to\infty} \PP\biggl( \bigcap_{k\geq N} \Omega_k^c\cap\{\eta_N<\epsilon\}  \biggr) = \lim_{N\to\infty} \PP( \eta_N<\epsilon ) , \end{split}    \llabel{Lmg re 1woD juLB BSdasY Vc F Uhy ViC xB1 5y Ltql qoUh gL3bZN YV k orz wa3 650 qW hF22 epiX cAjA4Z V4 b cXx uB3 NQN p0 GxW2 Vs1z jtqe2p LE B iS3 0E0 NKH gY N50v XaK6 pNpwdB X2 Y v7V 0Ud dTc Pi dRNN CLG4 7Fc3PL Bx K 3Be x1X zyX cj 0Z6a Jk0H KuQnwd Dh P Q1Q rwA 05v 9c 3pnz ttzt x2IirW CZ B oS5 xlO KCi D3 WFh4 dvCL QANAQJ Gg y vOD NTD FKj Mc 0RJP m4HU SQkLnT Q4 Y 6CC MvN jAR Zb lir7 RFsI NzHiJl cg f xSC Hts ZOG 1V uOzk 5G1C LtmRYI eD 3 5BB uxZ JdY LO CwS9 lokS NasDLj 5h 8 yni u7h u3c di zYh1 PdwE l3m8Xt yX Q RCA bwe aLi N8 qA9N 6DRE wy6gZe xs A 4fG EKH KQP PP KMbk sY1j M4h3Jj gS U One p1w RqN GA grL4 c18W v4kchD gR x 7Gj jIB zcK QV f7gA TrZx Oy6FF7 y9 3 iuu AQt 9TK Rx S5GO TFGx 4Xx1U3 R4 s 7U1 mpa bpD Hg kicx aCjk hnobr0 p4 c ody xTC kVj 8t W4iP 2OhT RF6kU2 k2 o oZJ Fsq Y4B FS NI3u W2fj OMFf7x Jv e ilb UVT ArC Tv qWLi vbRp g2wpAJ On l RUE PKh j9h dG M0Mi gcqQ wkyunB Jr T LDc Pgn OSC HO sSgQ sR35 MB7Bgk Pk 6 nJh 01P Cxd Ds w514 O648 VD8iJ5 4F W 6rs 6Sy qGz MK fXop oe4e o52UNB 4Q 8 f8N Uz8 u2n GO AXHW gKtG AtGGJs bm z 2qj vSv GBu 5e 4JgL Aqrm gMmS08 ZF s xQm 28M 3z4 Ho 1xxj j8Uk bMbm8M 0c L PL5 TS2 kIQ jZ Kb9Q Ux2U i5Aflw 1S L DGI uWU dCP jy wVVM 2ct8 cmgOBS 7d Q ViX R8F bta 1m tEFj TO0k owcK2d 6M Z iW8 PrK PI1 sX WJNB cREV Y4H5QQ GH b plP bwd Txp OI 5OQZ AKyi ix7Qey YI 9 1Ea 16r KXK L2 ifQX QPdP NL6EJi Hc K rBs 2qG tQb aq edOj Lixj GiNWr1 Pb Y SZe Sxx Fin aK 9Eki CHV2 a13f7G 3G 3 oDK K0i bKV y4 53E2 nFQS 8Hnqg0 E3 2 ADd dEV nmJ 7H Bc1t 2K2i hCzZuy 9k p sHn 8Ko uAR kv sHKP y8Yo dOOqBi hF 1 Z3C vUF hmj gB muZq 7ggW Lg5dQB 1k p Fxk k35 GFo dk 00YD 13qI qqbLwy QC c yZR wHA fp7 9o imtC c5CV 8cEuwU w7 k 8Q7 nCq WkM gY rtVR IySM tZUGCH XV 9 mr9 GHZ ol0 VE eIjQ vwgw 17pDhX JS F UcY bqU gnG V8 IFWb S1GX az0ZTt 81 w 7En IhF F72 v2 PkWO Xlkr w6IPu5 67 9 vcW 1f6 z99 lM 2LI1 Y6Na axfl18 gT 0 gDp tVl CN4 jf GSbC ro5D v78Cxa uk Y iUI WWy YDR w8 z7Kj Px7C hC7zJv b1 b 0rF d7n Mxk 09 1wHv y4u5 vLLsJ8 Nm A kWt xuf 4P5 Nw P23b 06sF NQ6xgD hu R GbK 7j2 O4g y4 p4BL top3 h2kfyI 9w O 4Aa EWb 36Y yH YiI1 S3CO J7aN1r 0s Q OrC AC4 vL7 yr CGkI RlNu EQ137} \end{align} and thus, \begin{align} \begin{split} &\PP ( \tau=0 ) = \lim_{j\to\infty} \PP\left( \tau<\frac{1}{j} \right) = \lim_{j\to\infty}\lim_{N\to\infty}
\PP\left( \eta_N<\frac{1}{j} \right) \\ &\quad \leq  \lim_{j\to\infty}\lim_{k\to\infty} \PP\biggl( \sup_{[0, 1/j]}\Vert u^{(n_k)}\Vert _{p} + \biggl(\int_{0}^{1/j} \Vert u^{(n_k)}\Vert _{3p}^{p}\, ds\biggr)^{1/p}> M_1^{1/p}K+\frac{1}{2^k} \biggr) \\ &\quad \leq  \lim_{t\to 0}\sup_k \PP\left( \sup_{[0, t]}\Vert u^{(n_k)}\Vert _{p} + \left(\int_{0}^{t} \Vert u^{(n_k)}\Vert _{3p}^{p}\, ds\right)^{1/p} >M_1^{1/p}K+\frac{1}{2^k} \right)=0. \end{split}    \llabel{ X2 Y v7V 0Ud dTc Pi dRNN CLG4 7Fc3PL Bx K 3Be x1X zyX cj 0Z6a Jk0H KuQnwd Dh P Q1Q rwA 05v 9c 3pnz ttzt x2IirW CZ B oS5 xlO KCi D3 WFh4 dvCL QANAQJ Gg y vOD NTD FKj Mc 0RJP m4HU SQkLnT Q4 Y 6CC MvN jAR Zb lir7 RFsI NzHiJl cg f xSC Hts ZOG 1V uOzk 5G1C LtmRYI eD 3 5BB uxZ JdY LO CwS9 lokS NasDLj 5h 8 yni u7h u3c di zYh1 PdwE l3m8Xt yX Q RCA bwe aLi N8 qA9N 6DRE wy6gZe xs A 4fG EKH KQP PP KMbk sY1j M4h3Jj gS U One p1w RqN GA grL4 c18W v4kchD gR x 7Gj jIB zcK QV f7gA TrZx Oy6FF7 y9 3 iuu AQt 9TK Rx S5GO TFGx 4Xx1U3 R4 s 7U1 mpa bpD Hg kicx aCjk hnobr0 p4 c ody xTC kVj 8t W4iP 2OhT RF6kU2 k2 o oZJ Fsq Y4B FS NI3u W2fj OMFf7x Jv e ilb UVT ArC Tv qWLi vbRp g2wpAJ On l RUE PKh j9h dG M0Mi gcqQ wkyunB Jr T LDc Pgn OSC HO sSgQ sR35 MB7Bgk Pk 6 nJh 01P Cxd Ds w514 O648 VD8iJ5 4F W 6rs 6Sy qGz MK fXop oe4e o52UNB 4Q 8 f8N Uz8 u2n GO AXHW gKtG AtGGJs bm z 2qj vSv GBu 5e 4JgL Aqrm gMmS08 ZF s xQm 28M 3z4 Ho 1xxj j8Uk bMbm8M 0c L PL5 TS2 kIQ jZ Kb9Q Ux2U i5Aflw 1S L DGI uWU dCP jy wVVM 2ct8 cmgOBS 7d Q ViX R8F bta 1m tEFj TO0k owcK2d 6M Z iW8 PrK PI1 sX WJNB cREV Y4H5QQ GH b plP bwd Txp OI 5OQZ AKyi ix7Qey YI 9 1Ea 16r KXK L2 ifQX QPdP NL6EJi Hc K rBs 2qG tQb aq edOj Lixj GiNWr1 Pb Y SZe Sxx Fin aK 9Eki CHV2 a13f7G 3G 3 oDK K0i bKV y4 53E2 nFQS 8Hnqg0 E3 2 ADd dEV nmJ 7H Bc1t 2K2i hCzZuy 9k p sHn 8Ko uAR kv sHKP y8Yo dOOqBi hF 1 Z3C vUF hmj gB muZq 7ggW Lg5dQB 1k p Fxk k35 GFo dk 00YD 13qI qqbLwy QC c yZR wHA fp7 9o imtC c5CV 8cEuwU w7 k 8Q7 nCq WkM gY rtVR IySM tZUGCH XV 9 mr9 GHZ ol0 VE eIjQ vwgw 17pDhX JS F UcY bqU gnG V8 IFWb S1GX az0ZTt 81 w 7En IhF F72 v2 PkWO Xlkr w6IPu5 67 9 vcW 1f6 z99 lM 2LI1 Y6Na axfl18 gT 0 gDp tVl CN4 jf GSbC ro5D v78Cxa uk Y iUI WWy YDR w8 z7Kj Px7C hC7zJv b1 b 0rF d7n Mxk 09 1wHv y4u5 vLLsJ8 Nm A kWt xuf 4P5 Nw P23b 06sF NQ6xgD hu R GbK 7j2 O4g y4 p4BL top3 h2kfyI 9w O 4Aa EWb 36Y yH YiI1 S3CO J7aN1r 0s Q OrC AC4 vL7 yr CGkI RlNu GbOuuk 1a w LDK 2zl Ka4 0h yJnD V4iF xsqO00 1r q CeO AO2 es7 DR aCpU G54F 2i97xS Qr c bPZ 6K8 Kud n9 e6SY o396 Fr8LUx yX O jdF sMr l54 Eh T8vr xxF2 phKPbs zr l pMA ubE RMGEQ138} \end{align} To prove that $\{u^{(n_k)}\}$ has an $\omega$-pointwise limit in $L^p(\Omega, C([0,\tau], L^p))\cap L^p(\Omega, L^{p}([0,\tau], L^{3p}))$, we resort to \cite[Lemma~5.2]{KXZ} and~\eqref{ERTWERTHWRTWERTSGDGHCFGSDFGQSERWDFGDSFGHSDRGTEHDFGHDSFGSDGHGYUHDFGSDFASDFASGTWRT133}. For all $N\in \NN$, \begin{equation} u^{(n_k)}\chi_{\cap_{j\geq N} \Omega_j^c} \xrightarrow{k \to \infty} u_N \mbox{ in }C([0,\tau], L^p)\cap L^{p}([0,\tau], L^{3p}) \Pas ,    \llabel{JP m4HU SQkLnT Q4 Y 6CC MvN jAR Zb lir7 RFsI NzHiJl cg f xSC Hts ZOG 1V uOzk 5G1C LtmRYI eD 3 5BB uxZ JdY LO CwS9 lokS NasDLj 5h 8 yni u7h u3c di zYh1 PdwE l3m8Xt yX Q RCA bwe aLi N8 qA9N 6DRE wy6gZe xs A 4fG EKH KQP PP KMbk sY1j M4h3Jj gS U One p1w RqN GA grL4 c18W v4kchD gR x 7Gj jIB zcK QV f7gA TrZx Oy6FF7 y9 3 iuu AQt 9TK Rx S5GO TFGx 4Xx1U3 R4 s 7U1 mpa bpD Hg kicx aCjk hnobr0 p4 c ody xTC kVj 8t W4iP 2OhT RF6kU2 k2 o oZJ Fsq Y4B FS NI3u W2fj OMFf7x Jv e ilb UVT ArC Tv qWLi vbRp g2wpAJ On l RUE PKh j9h dG M0Mi gcqQ wkyunB Jr T LDc Pgn OSC HO sSgQ sR35 MB7Bgk Pk 6 nJh 01P Cxd Ds w514 O648 VD8iJ5 4F W 6rs 6Sy qGz MK fXop oe4e o52UNB 4Q 8 f8N Uz8 u2n GO AXHW gKtG AtGGJs bm z 2qj vSv GBu 5e 4JgL Aqrm gMmS08 ZF s xQm 28M 3z4 Ho 1xxj j8Uk bMbm8M 0c L PL5 TS2 kIQ jZ Kb9Q Ux2U i5Aflw 1S L DGI uWU dCP jy wVVM 2ct8 cmgOBS 7d Q ViX R8F bta 1m tEFj TO0k owcK2d 6M Z iW8 PrK PI1 sX WJNB cREV Y4H5QQ GH b plP bwd Txp OI 5OQZ AKyi ix7Qey YI 9 1Ea 16r KXK L2 ifQX QPdP NL6EJi Hc K rBs 2qG tQb aq edOj Lixj GiNWr1 Pb Y SZe Sxx Fin aK 9Eki CHV2 a13f7G 3G 3 oDK K0i bKV y4 53E2 nFQS 8Hnqg0 E3 2 ADd dEV nmJ 7H Bc1t 2K2i hCzZuy 9k p sHn 8Ko uAR kv sHKP y8Yo dOOqBi hF 1 Z3C vUF hmj gB muZq 7ggW Lg5dQB 1k p Fxk k35 GFo dk 00YD 13qI qqbLwy QC c yZR wHA fp7 9o imtC c5CV 8cEuwU w7 k 8Q7 nCq WkM gY rtVR IySM tZUGCH XV 9 mr9 GHZ ol0 VE eIjQ vwgw 17pDhX JS F UcY bqU gnG V8 IFWb S1GX az0ZTt 81 w 7En IhF F72 v2 PkWO Xlkr w6IPu5 67 9 vcW 1f6 z99 lM 2LI1 Y6Na axfl18 gT 0 gDp tVl CN4 jf GSbC ro5D v78Cxa uk Y iUI WWy YDR w8 z7Kj Px7C hC7zJv b1 b 0rF d7n Mxk 09 1wHv y4u5 vLLsJ8 Nm A kWt xuf 4P5 Nw P23b 06sF NQ6xgD hu R GbK 7j2 O4g y4 p4BL top3 h2kfyI 9w O 4Aa EWb 36Y yH YiI1 S3CO J7aN1r 0s Q OrC AC4 vL7 yr CGkI RlNu GbOuuk 1a w LDK 2zl Ka4 0h yJnD V4iF xsqO00 1r q CeO AO2 es7 DR aCpU G54F 2i97xS Qr c bPZ 6K8 Kud n9 e6SY o396 Fr8LUx yX O jdF sMr l54 Eh T8vr xxF2 phKPbs zr l pMA ubE RMG QA aCBu 2Lqw Gasprf IZ O iKV Vbu Vae 6a bauf y9Kc Fk6cBl Z5 r KUj htW E1C nt 9Rmd whJR ySGVSO VT v 9FY 4uz yAH Sp 6yT9 s6R6 oOi3aq Zl L 7bI vWZ 18c Fa iwpt C1nd Fyp4oK xDEQ139} \end{equation} for some adapted process $u_{N}$, and $\{u^{(n_k)}\}$ converges to $u=\lim_{N\to\infty} u_{N}$ in the same topology almost surely. By \eqref{ERTWERTHWRTWERTSGDGHCFGSDFGQSERWDFGDSFGHSDRGTEHDFGHDSFGSDGHGYUHDFGSDFASDFASGTWRT99}, $\{u^{(n_k)}\chi_{\Omega_N}\}$ are uniformly bounded in $L^p(\Omega, L^{\infty}([0,\tau], L^p))\cap L^p(\Omega, L^{p}([0,\tau], L^{3p}))$. Therefore, $u_{N}$ lives in $L^p(\Omega, L^{\infty}([0,\tau], L^p))\cap L^p(\Omega, L^{p}([0,\tau], L^{3p}))$, and so does~$u$. The stopping time $\tau$ depends on $M$, $M_1$, and $K$. Since both $M$ and $M_1$ are determined by $K$, $\tau$ is essentially determined by~$K$. Recall that $\tau$ is the lower bound of $\{\tau^{n_{k}}_{M}\wedge \bar{t}\}_{k\geq N}$ in $\bigcap_{k\geq N} \Omega_k^c$. Hence, for all $k\geq N$,  \begin{equation}   \EE\biggl[ \sup_{[0, \tau]}\Vert   \indic_{\cap_{i\geq N} \Omega_i^c} u^{(n_k)} \Vert_{p}^p  + \int_{0}^{\tau} \int_{\RR^\dd}\sum_j  |\nabla | \indic_{\cap_{i\geq N} \Omega_i^c} u^{(n_k)}_j|^{p/2}|^2  \,dx ds  \biggr]   \leq C K^p,     \llabel{ bwe aLi N8 qA9N 6DRE wy6gZe xs A 4fG EKH KQP PP KMbk sY1j M4h3Jj gS U One p1w RqN GA grL4 c18W v4kchD gR x 7Gj jIB zcK QV f7gA TrZx Oy6FF7 y9 3 iuu AQt 9TK Rx S5GO TFGx 4Xx1U3 R4 s 7U1 mpa bpD Hg kicx aCjk hnobr0 p4 c ody xTC kVj 8t W4iP 2OhT RF6kU2 k2 o oZJ Fsq Y4B FS NI3u W2fj OMFf7x Jv e ilb UVT ArC Tv qWLi vbRp g2wpAJ On l RUE PKh j9h dG M0Mi gcqQ wkyunB Jr T LDc Pgn OSC HO sSgQ sR35 MB7Bgk Pk 6 nJh 01P Cxd Ds w514 O648 VD8iJ5 4F W 6rs 6Sy qGz MK fXop oe4e o52UNB 4Q 8 f8N Uz8 u2n GO AXHW gKtG AtGGJs bm z 2qj vSv GBu 5e 4JgL Aqrm gMmS08 ZF s xQm 28M 3z4 Ho 1xxj j8Uk bMbm8M 0c L PL5 TS2 kIQ jZ Kb9Q Ux2U i5Aflw 1S L DGI uWU dCP jy wVVM 2ct8 cmgOBS 7d Q ViX R8F bta 1m tEFj TO0k owcK2d 6M Z iW8 PrK PI1 sX WJNB cREV Y4H5QQ GH b plP bwd Txp OI 5OQZ AKyi ix7Qey YI 9 1Ea 16r KXK L2 ifQX QPdP NL6EJi Hc K rBs 2qG tQb aq edOj Lixj GiNWr1 Pb Y SZe Sxx Fin aK 9Eki CHV2 a13f7G 3G 3 oDK K0i bKV y4 53E2 nFQS 8Hnqg0 E3 2 ADd dEV nmJ 7H Bc1t 2K2i hCzZuy 9k p sHn 8Ko uAR kv sHKP y8Yo dOOqBi hF 1 Z3C vUF hmj gB muZq 7ggW Lg5dQB 1k p Fxk k35 GFo dk 00YD 13qI qqbLwy QC c yZR wHA fp7 9o imtC c5CV 8cEuwU w7 k 8Q7 nCq WkM gY rtVR IySM tZUGCH XV 9 mr9 GHZ ol0 VE eIjQ vwgw 17pDhX JS F UcY bqU gnG V8 IFWb S1GX az0ZTt 81 w 7En IhF F72 v2 PkWO Xlkr w6IPu5 67 9 vcW 1f6 z99 lM 2LI1 Y6Na axfl18 gT 0 gDp tVl CN4 jf GSbC ro5D v78Cxa uk Y iUI WWy YDR w8 z7Kj Px7C hC7zJv b1 b 0rF d7n Mxk 09 1wHv y4u5 vLLsJ8 Nm A kWt xuf 4P5 Nw P23b 06sF NQ6xgD hu R GbK 7j2 O4g y4 p4BL top3 h2kfyI 9w O 4Aa EWb 36Y yH YiI1 S3CO J7aN1r 0s Q OrC AC4 vL7 yr CGkI RlNu GbOuuk 1a w LDK 2zl Ka4 0h yJnD V4iF xsqO00 1r q CeO AO2 es7 DR aCpU G54F 2i97xS Qr c bPZ 6K8 Kud n9 e6SY o396 Fr8LUx yX O jdF sMr l54 Eh T8vr xxF2 phKPbs zr l pMA ubE RMG QA aCBu 2Lqw Gasprf IZ O iKV Vbu Vae 6a bauf y9Kc Fk6cBl Z5 r KUj htW E1C nt 9Rmd whJR ySGVSO VT v 9FY 4uz yAH Sp 6yT9 s6R6 oOi3aq Zl L 7bI vWZ 18c Fa iwpt C1nd Fyp4oK xD f Qz2 813 6a8 zX wsGl Ysh9 Gp3Tal nr R UKt tBK eFr 45 43qU 2hh3 WbYw09 g2 W LIX zvQ zMk j5 f0xL seH9 dscinG wu P JLP 1gE N5W qY sSoW Peqj MimTyb Hj j cbn 0NO 5hz P9 W40r EQ154} \end{equation} where $C$ is a positive constant independent of~$K$. Finally, we send $k\to \infty$ first using Lemma~\ref{L04}, \eqref{ERTWERTHWRTWERTSGDGHCFGSDFGQSERWDFGDSFGHSDRGTEHDFGHDSFGSDGHGYUHDFGSDFASDFASGTWRT99}, and \eqref{ERTWERTHWRTWERTSGDGHCFGSDFGQSERWDFGDSFGHSDRGTEHDFGHDSFGSDGHGYUHDFGSDFASDFASGTWRT112}, and then send~$N\to \infty$, arriving at~\eqref{ERTWERTHWRTWERTSGDGHCFGSDFGQSERWDFGDSFGHSDRGTEHDFGHDSFGSDGHGYUHDFGSDFASDFASGTWRT130}.  \end{proof} \par Next, we prove the pathwise uniqueness of strong solutions to the limit model \begin{align} \begin{split} &\partial_t\uu( t,x) =\Delta \uu( t,x) - \varphi(\Vert\uu(t)\Vert_{p})^2\mathcal{P}\bigl(( \uu( t,x)\cdot \nabla)\uu( t,x)\bigr) \\&\indeq\indeq\indeq\indeq\indeq +\varphi(\Vert\uu(t)\Vert_{p})^2\sigma(\uu( t,x))\dot{\WW}(t), \\&   \nabla\cdot \uu( t,x) = 0 , \\& \uu( 0,x)= \uu_0 (x)  \Pas \commaone x\in\RR^d. \end{split} \label{ERTWERTHWRTWERTSGDGHCFGSDFGQSERWDFGDSFGHSDRGTEHDFGHDSFGSDGHGYUHDFGSDFASDFASGTWRT140} \end{align} \par \cole \begin{Theorem}   \label{L08} Let $p>d=3$ and~$T>0$. Suppose that $\nabla\cdot \uu_0=0$ and $u_0\in L^p(\Omega, L^p)$. If a pair of local strong solutions $(u^{(1)}, \tau\wedge T)$ and $(u^{(2)},\tau\wedge T)$ of \eqref{ERTWERTHWRTWERTSGDGHCFGSDFGQSERWDFGDSFGHSDRGTEHDFGHDSFGSDGHGYUHDFGSDFASDFASGTWRT140} obey   \begin{equation}\label{eng}   \EE\biggl[\sup_{0\leq s\leq \tau}\Vert u(s)\Vert_p^p   +\int_0^{\tau} \sum_j\int_{\RR^d}| \nabla (|u_j(s,x)| ^{p/2})|^2 \,dx ds\biggr]\leq C\EE[\Vert u_0\Vert^p_p+1]         ,   \end{equation} for some positive constant $C$, then we have   $   \PP(u^{(1)}(t)=u^{(2)}(t), ~\forall t\in [0,\tau\wedge T])=1.   $ \end{Theorem} \par \colb \begin{proof}[Proof of Theorem~\ref{L08}]  Let $M>0$, and introduce the stopping times   \begin{align}   \begin{split}   \eta_i^M=\inf\biggl\{ t>0:    \sup_{0\leq s\leq t}\Vert u^{(i)}(s)\Vert_p^p   +\int_0^{t} \sum_j\int_{\RR^3}| \nabla (|u^{(i)}_j(s,x)| ^{p/2})|^2 \,dx ds   \geq M^p   \biggr\}   .   \end{split}    \llabel{Xx1U3 R4 s 7U1 mpa bpD Hg kicx aCjk hnobr0 p4 c ody xTC kVj 8t W4iP 2OhT RF6kU2 k2 o oZJ Fsq Y4B FS NI3u W2fj OMFf7x Jv e ilb UVT ArC Tv qWLi vbRp g2wpAJ On l RUE PKh j9h dG M0Mi gcqQ wkyunB Jr T LDc Pgn OSC HO sSgQ sR35 MB7Bgk Pk 6 nJh 01P Cxd Ds w514 O648 VD8iJ5 4F W 6rs 6Sy qGz MK fXop oe4e o52UNB 4Q 8 f8N Uz8 u2n GO AXHW gKtG AtGGJs bm z 2qj vSv GBu 5e 4JgL Aqrm gMmS08 ZF s xQm 28M 3z4 Ho 1xxj j8Uk bMbm8M 0c L PL5 TS2 kIQ jZ Kb9Q Ux2U i5Aflw 1S L DGI uWU dCP jy wVVM 2ct8 cmgOBS 7d Q ViX R8F bta 1m tEFj TO0k owcK2d 6M Z iW8 PrK PI1 sX WJNB cREV Y4H5QQ GH b plP bwd Txp OI 5OQZ AKyi ix7Qey YI 9 1Ea 16r KXK L2 ifQX QPdP NL6EJi Hc K rBs 2qG tQb aq edOj Lixj GiNWr1 Pb Y SZe Sxx Fin aK 9Eki CHV2 a13f7G 3G 3 oDK K0i bKV y4 53E2 nFQS 8Hnqg0 E3 2 ADd dEV nmJ 7H Bc1t 2K2i hCzZuy 9k p sHn 8Ko uAR kv sHKP y8Yo dOOqBi hF 1 Z3C vUF hmj gB muZq 7ggW Lg5dQB 1k p Fxk k35 GFo dk 00YD 13qI qqbLwy QC c yZR wHA fp7 9o imtC c5CV 8cEuwU w7 k 8Q7 nCq WkM gY rtVR IySM tZUGCH XV 9 mr9 GHZ ol0 VE eIjQ vwgw 17pDhX JS F UcY bqU gnG V8 IFWb S1GX az0ZTt 81 w 7En IhF F72 v2 PkWO Xlkr w6IPu5 67 9 vcW 1f6 z99 lM 2LI1 Y6Na axfl18 gT 0 gDp tVl CN4 jf GSbC ro5D v78Cxa uk Y iUI WWy YDR w8 z7Kj Px7C hC7zJv b1 b 0rF d7n Mxk 09 1wHv y4u5 vLLsJ8 Nm A kWt xuf 4P5 Nw P23b 06sF NQ6xgD hu R GbK 7j2 O4g y4 p4BL top3 h2kfyI 9w O 4Aa EWb 36Y yH YiI1 S3CO J7aN1r 0s Q OrC AC4 vL7 yr CGkI RlNu GbOuuk 1a w LDK 2zl Ka4 0h yJnD V4iF xsqO00 1r q CeO AO2 es7 DR aCpU G54F 2i97xS Qr c bPZ 6K8 Kud n9 e6SY o396 Fr8LUx yX O jdF sMr l54 Eh T8vr xxF2 phKPbs zr l pMA ubE RMG QA aCBu 2Lqw Gasprf IZ O iKV Vbu Vae 6a bauf y9Kc Fk6cBl Z5 r KUj htW E1C nt 9Rmd whJR ySGVSO VT v 9FY 4uz yAH Sp 6yT9 s6R6 oOi3aq Zl L 7bI vWZ 18c Fa iwpt C1nd Fyp4oK xD f Qz2 813 6a8 zX wsGl Ysh9 Gp3Tal nr R UKt tBK eFr 45 43qU 2hh3 WbYw09 g2 W LIX zvQ zMk j5 f0xL seH9 dscinG wu P JLP 1gE N5W qY sSoW Peqj MimTyb Hj j cbn 0NO 5hz P9 W40r 2w77 TAoz70 N1 a u09 boc DSx Gc 3tvK LXaC 1dKgw9 H3 o 2kE oul In9 TS PyL2 HXO7 tSZse0 1Z 9 Hds lDq 0tm SO AVqt A1FQ zEMKSb ak z nw8 39w nH1 Dp CjGI k5X3 B6S6UI 7H I gAa f9EQ161}   \end{align}   If $\Vert u_0(\omega)\Vert_p< M$, then $\eta_i^M(\omega)>0$; otherwise $\eta_i^M(\omega)=0$ for~$i=1,2$. Define $\eta^M=\eta_1^M\wedge \eta_2^M\wedge \tau$. Due to \eqref{eng}, $\lim_{M\to\infty}\PP(\eta^M= \tau)=1$. Let $S\in (0, T)$, and denote $w=u^{(1)}-u^{(2)}$, $\varphi_1=\varphi(\Vert u^{(1)}(t)\Vert_{p})$, and $\varphi_2=\varphi(\Vert u^{(2)}(t)\Vert_{p})$. On $[0,\eta^M\wedge S]$, the difference $w$ satisfies   \begin{align}   \begin{split}   &\partial_t w   =\Delta w    -   \Bigl((\varphi_1)^2\mathcal{P}((u^{(1)}\cdot\nabla)  u^{(1)})-(\varphi_2)^2\mathcal{P}((u^{(2)}\cdot\nabla)  u^{(2)})\Bigr)   \\&\indeq\indeq\indeq\indeq\quad   +\Bigl(   (\varphi_1)^2\sigma(u^{(1)})-(\varphi_2)^2\sigma(u^{(2)})    \Bigr)\,\dot{\WW}(t),    \\   &\nabla\cdot w = 0,   \\   &w( 0)= 0\Pas   \end{split}    \llabel{dG M0Mi gcqQ wkyunB Jr T LDc Pgn OSC HO sSgQ sR35 MB7Bgk Pk 6 nJh 01P Cxd Ds w514 O648 VD8iJ5 4F W 6rs 6Sy qGz MK fXop oe4e o52UNB 4Q 8 f8N Uz8 u2n GO AXHW gKtG AtGGJs bm z 2qj vSv GBu 5e 4JgL Aqrm gMmS08 ZF s xQm 28M 3z4 Ho 1xxj j8Uk bMbm8M 0c L PL5 TS2 kIQ jZ Kb9Q Ux2U i5Aflw 1S L DGI uWU dCP jy wVVM 2ct8 cmgOBS 7d Q ViX R8F bta 1m tEFj TO0k owcK2d 6M Z iW8 PrK PI1 sX WJNB cREV Y4H5QQ GH b plP bwd Txp OI 5OQZ AKyi ix7Qey YI 9 1Ea 16r KXK L2 ifQX QPdP NL6EJi Hc K rBs 2qG tQb aq edOj Lixj GiNWr1 Pb Y SZe Sxx Fin aK 9Eki CHV2 a13f7G 3G 3 oDK K0i bKV y4 53E2 nFQS 8Hnqg0 E3 2 ADd dEV nmJ 7H Bc1t 2K2i hCzZuy 9k p sHn 8Ko uAR kv sHKP y8Yo dOOqBi hF 1 Z3C vUF hmj gB muZq 7ggW Lg5dQB 1k p Fxk k35 GFo dk 00YD 13qI qqbLwy QC c yZR wHA fp7 9o imtC c5CV 8cEuwU w7 k 8Q7 nCq WkM gY rtVR IySM tZUGCH XV 9 mr9 GHZ ol0 VE eIjQ vwgw 17pDhX JS F UcY bqU gnG V8 IFWb S1GX az0ZTt 81 w 7En IhF F72 v2 PkWO Xlkr w6IPu5 67 9 vcW 1f6 z99 lM 2LI1 Y6Na axfl18 gT 0 gDp tVl CN4 jf GSbC ro5D v78Cxa uk Y iUI WWy YDR w8 z7Kj Px7C hC7zJv b1 b 0rF d7n Mxk 09 1wHv y4u5 vLLsJ8 Nm A kWt xuf 4P5 Nw P23b 06sF NQ6xgD hu R GbK 7j2 O4g y4 p4BL top3 h2kfyI 9w O 4Aa EWb 36Y yH YiI1 S3CO J7aN1r 0s Q OrC AC4 vL7 yr CGkI RlNu GbOuuk 1a w LDK 2zl Ka4 0h yJnD V4iF xsqO00 1r q CeO AO2 es7 DR aCpU G54F 2i97xS Qr c bPZ 6K8 Kud n9 e6SY o396 Fr8LUx yX O jdF sMr l54 Eh T8vr xxF2 phKPbs zr l pMA ubE RMG QA aCBu 2Lqw Gasprf IZ O iKV Vbu Vae 6a bauf y9Kc Fk6cBl Z5 r KUj htW E1C nt 9Rmd whJR ySGVSO VT v 9FY 4uz yAH Sp 6yT9 s6R6 oOi3aq Zl L 7bI vWZ 18c Fa iwpt C1nd Fyp4oK xD f Qz2 813 6a8 zX wsGl Ysh9 Gp3Tal nr R UKt tBK eFr 45 43qU 2hh3 WbYw09 g2 W LIX zvQ zMk j5 f0xL seH9 dscinG wu P JLP 1gE N5W qY sSoW Peqj MimTyb Hj j cbn 0NO 5hz P9 W40r 2w77 TAoz70 N1 a u09 boc DSx Gc 3tvK LXaC 1dKgw9 H3 o 2kE oul In9 TS PyL2 HXO7 tSZse0 1Z 9 Hds lDq 0tm SO AVqt A1FQ zEMKSb ak z nw8 39w nH1 Dp CjGI k5X3 B6S6UI 7H I gAa f9E V33 Bk kuo3 FyEi 8Ty2AB PY z SWj Pj5 tYZ ET Yzg6 Ix5t ATPMdl Gk e 67X b7F ktE sz yFyc mVhG JZ29aP gz k Yj4 cEr HCd P7 XFHU O9zo y4AZai SR O pIn 0tp 7kZ zU VHQt m3ip 3xEdEQ161}   \end{align} As in Lemma~\ref{L06}, we apply It\^{o}'s formula. (We refer to the estimates for $f^{(2)}, f^{(4)}, f^{(1)}, f^{(6)}$ and $g^{(1)}, g^{(2)}, g^{(5)}$ in Lemma~\ref{L06}. Since the limit model~\eqref{ERTWERTHWRTWERTSGDGHCFGSDFGQSERWDFGDSFGHSDRGTEHDFGHDSFGSDGHGYUHDFGSDFASDFASGTWRT140} does not contain the convolution operator, the terms resulting from the difference of two convolution projects are not present here.) We arrive at \begin{align*} \begin{split} &\EE\biggl[ \sup_{0\leq s\leq \eta^M\wedge S}\Vert w(s)\Vert_p^p +\int_0^{\eta^M\wedge S} \sum_j\int_{\RR^d}| \nabla (|w_j(s,x)| ^{p/2})|^2 \,dx ds \biggr] \\ &\indeq \leq  \EE\biggl[ \int_0^{\eta^{M}\wedge S} (\varepsilon\Vert  w  \Vert_{3p}^p+C_{\varepsilon,M,T} \Vert  w  \Vert_{p}^p ) \,dt \biggr]+ \EE  \left[ \int_{0}^{\eta^{M}\wedge S}   \Vert  w  \Vert_{p}^p (C_{\varepsilon,M,T}+ \eps\varphi_1\Vert u^{(1)}\Vert_{3p}^p +\eps\varphi_2\Vert u^{(2)}\Vert_{3p}^p) \,ds \right] \\ &\indeq \quad+ \EE \left[ \int_0^{\eta^{M}\wedge S} \Vert  w \Vert_{p}^{p}(\varphi_1\Vert u^{(2)} \Vert_{(3p/2)-}^{2p} +\varphi_2\Vert u^{(2)} \Vert_{(3p/2)-}^{2p}+ 1) \,ds \right] . \end{split} \end{align*} Utilizing the truncations $\varphi_1,\varphi_2$ and the stopping time $\eta^M$, and applying Young's inequality to the last term, we obtain   \begin{align} \begin{split} &\EE\biggl[ \sup_{0\leq s\leq \eta^M\wedge S}\Vert w(s)\Vert_p^p +\int_0^{\eta^M\wedge S} \sum_j\int_{\RR^d}| \nabla (|w_j(s,x)| ^{p/2})|^2 \,dx ds \biggr] \\ &\indeq \leq  \EE\biggl[ \int_0^{\eta^{M}\wedge S} (\varepsilon\Vert  w  \Vert_{3p}^p+C_{\varepsilon,M,T} \Vert  w  \Vert_{p}^p ) \,dt \biggr] + \varepsilon \EE\biggl[ \sup_{0\leq s\leq \eta^M\wedge S} \Vert  w  \Vert_{p}^p  \int_0^{\eta^{M}\wedge S} ( \Vert  u^{(1)} \Vert_{3p}^p + \Vert  u^{(2)} \Vert_{3p}^p  ) \,dt\biggr] , \end{split} \llabel{z 2qj vSv GBu 5e 4JgL Aqrm gMmS08 ZF s xQm 28M 3z4 Ho 1xxj j8Uk bMbm8M 0c L PL5 TS2 kIQ jZ Kb9Q Ux2U i5Aflw 1S L DGI uWU dCP jy wVVM 2ct8 cmgOBS 7d Q ViX R8F bta 1m tEFj TO0k owcK2d 6M Z iW8 PrK PI1 sX WJNB cREV Y4H5QQ GH b plP bwd Txp OI 5OQZ AKyi ix7Qey YI 9 1Ea 16r KXK L2 ifQX QPdP NL6EJi Hc K rBs 2qG tQb aq edOj Lixj GiNWr1 Pb Y SZe Sxx Fin aK 9Eki CHV2 a13f7G 3G 3 oDK K0i bKV y4 53E2 nFQS 8Hnqg0 E3 2 ADd dEV nmJ 7H Bc1t 2K2i hCzZuy 9k p sHn 8Ko uAR kv sHKP y8Yo dOOqBi hF 1 Z3C vUF hmj gB muZq 7ggW Lg5dQB 1k p Fxk k35 GFo dk 00YD 13qI qqbLwy QC c yZR wHA fp7 9o imtC c5CV 8cEuwU w7 k 8Q7 nCq WkM gY rtVR IySM tZUGCH XV 9 mr9 GHZ ol0 VE eIjQ vwgw 17pDhX JS F UcY bqU gnG V8 IFWb S1GX az0ZTt 81 w 7En IhF F72 v2 PkWO Xlkr w6IPu5 67 9 vcW 1f6 z99 lM 2LI1 Y6Na axfl18 gT 0 gDp tVl CN4 jf GSbC ro5D v78Cxa uk Y iUI WWy YDR w8 z7Kj Px7C hC7zJv b1 b 0rF d7n Mxk 09 1wHv y4u5 vLLsJ8 Nm A kWt xuf 4P5 Nw P23b 06sF NQ6xgD hu R GbK 7j2 O4g y4 p4BL top3 h2kfyI 9w O 4Aa EWb 36Y yH YiI1 S3CO J7aN1r 0s Q OrC AC4 vL7 yr CGkI RlNu GbOuuk 1a w LDK 2zl Ka4 0h yJnD V4iF xsqO00 1r q CeO AO2 es7 DR aCpU G54F 2i97xS Qr c bPZ 6K8 Kud n9 e6SY o396 Fr8LUx yX O jdF sMr l54 Eh T8vr xxF2 phKPbs zr l pMA ubE RMG QA aCBu 2Lqw Gasprf IZ O iKV Vbu Vae 6a bauf y9Kc Fk6cBl Z5 r KUj htW E1C nt 9Rmd whJR ySGVSO VT v 9FY 4uz yAH Sp 6yT9 s6R6 oOi3aq Zl L 7bI vWZ 18c Fa iwpt C1nd Fyp4oK xD f Qz2 813 6a8 zX wsGl Ysh9 Gp3Tal nr R UKt tBK eFr 45 43qU 2hh3 WbYw09 g2 W LIX zvQ zMk j5 f0xL seH9 dscinG wu P JLP 1gE N5W qY sSoW Peqj MimTyb Hj j cbn 0NO 5hz P9 W40r 2w77 TAoz70 N1 a u09 boc DSx Gc 3tvK LXaC 1dKgw9 H3 o 2kE oul In9 TS PyL2 HXO7 tSZse0 1Z 9 Hds lDq 0tm SO AVqt A1FQ zEMKSb ak z nw8 39w nH1 Dp CjGI k5X3 B6S6UI 7H I gAa f9E V33 Bk kuo3 FyEi 8Ty2AB PY z SWj Pj5 tYZ ET Yzg6 Ix5t ATPMdl Gk e 67X b7F ktE sz yFyc mVhG JZ29aP gz k Yj4 cEr HCd P7 XFHU O9zo y4AZai SR O pIn 0tp 7kZ zU VHQt m3ip 3xEd41 By 7 2ux IiY 8BC Lb OYGo LDwp juza6i Pa k Zdh aD3 xSX yj pdOw oqQq Jl6RFg lO t X67 nm7 s1l ZJ mGUr dIdX Q7jps7 rc d ACY ZMs BKA Nx tkqf Nhkt sbBf2O BN Z 5pf oqS Xtd 3c EQ161} \end{align} which can be further simplified to    \begin{align}   \begin{split}   \EE\biggl[   \sup_{0\leq s\leq \eta^M\wedge S}\Vert w(s)\Vert_p^p   \biggr]   \leq    C_{M,T}\int_0^{S} \EE\biggl[   \sup_{0\leq s\leq \eta^M\wedge t}   \Vert    w    \Vert_{p}^p \biggr]   \,dt.   \end{split}    \llabel{O0k owcK2d 6M Z iW8 PrK PI1 sX WJNB cREV Y4H5QQ GH b plP bwd Txp OI 5OQZ AKyi ix7Qey YI 9 1Ea 16r KXK L2 ifQX QPdP NL6EJi Hc K rBs 2qG tQb aq edOj Lixj GiNWr1 Pb Y SZe Sxx Fin aK 9Eki CHV2 a13f7G 3G 3 oDK K0i bKV y4 53E2 nFQS 8Hnqg0 E3 2 ADd dEV nmJ 7H Bc1t 2K2i hCzZuy 9k p sHn 8Ko uAR kv sHKP y8Yo dOOqBi hF 1 Z3C vUF hmj gB muZq 7ggW Lg5dQB 1k p Fxk k35 GFo dk 00YD 13qI qqbLwy QC c yZR wHA fp7 9o imtC c5CV 8cEuwU w7 k 8Q7 nCq WkM gY rtVR IySM tZUGCH XV 9 mr9 GHZ ol0 VE eIjQ vwgw 17pDhX JS F UcY bqU gnG V8 IFWb S1GX az0ZTt 81 w 7En IhF F72 v2 PkWO Xlkr w6IPu5 67 9 vcW 1f6 z99 lM 2LI1 Y6Na axfl18 gT 0 gDp tVl CN4 jf GSbC ro5D v78Cxa uk Y iUI WWy YDR w8 z7Kj Px7C hC7zJv b1 b 0rF d7n Mxk 09 1wHv y4u5 vLLsJ8 Nm A kWt xuf 4P5 Nw P23b 06sF NQ6xgD hu R GbK 7j2 O4g y4 p4BL top3 h2kfyI 9w O 4Aa EWb 36Y yH YiI1 S3CO J7aN1r 0s Q OrC AC4 vL7 yr CGkI RlNu GbOuuk 1a w LDK 2zl Ka4 0h yJnD V4iF xsqO00 1r q CeO AO2 es7 DR aCpU G54F 2i97xS Qr c bPZ 6K8 Kud n9 e6SY o396 Fr8LUx yX O jdF sMr l54 Eh T8vr xxF2 phKPbs zr l pMA ubE RMG QA aCBu 2Lqw Gasprf IZ O iKV Vbu Vae 6a bauf y9Kc Fk6cBl Z5 r KUj htW E1C nt 9Rmd whJR ySGVSO VT v 9FY 4uz yAH Sp 6yT9 s6R6 oOi3aq Zl L 7bI vWZ 18c Fa iwpt C1nd Fyp4oK xD f Qz2 813 6a8 zX wsGl Ysh9 Gp3Tal nr R UKt tBK eFr 45 43qU 2hh3 WbYw09 g2 W LIX zvQ zMk j5 f0xL seH9 dscinG wu P JLP 1gE N5W qY sSoW Peqj MimTyb Hj j cbn 0NO 5hz P9 W40r 2w77 TAoz70 N1 a u09 boc DSx Gc 3tvK LXaC 1dKgw9 H3 o 2kE oul In9 TS PyL2 HXO7 tSZse0 1Z 9 Hds lDq 0tm SO AVqt A1FQ zEMKSb ak z nw8 39w nH1 Dp CjGI k5X3 B6S6UI 7H I gAa f9E V33 Bk kuo3 FyEi 8Ty2AB PY z SWj Pj5 tYZ ET Yzg6 Ix5t ATPMdl Gk e 67X b7F ktE sz yFyc mVhG JZ29aP gz k Yj4 cEr HCd P7 XFHU O9zo y4AZai SR O pIn 0tp 7kZ zU VHQt m3ip 3xEd41 By 7 2ux IiY 8BC Lb OYGo LDwp juza6i Pa k Zdh aD3 xSX yj pdOw oqQq Jl6RFg lO t X67 nm7 s1l ZJ mGUr dIdX Q7jps7 rc d ACY ZMs BKA Nx tkqf Nhkt sbBf2O BN Z 5pf oqS Xtd 3c HFLN tLgR oHrnNl wR n ylZ NWV NfH vO B1nU Ayjt xTWW4o Cq P Rtu Vua nMk Lv qbxp Ni0x YnOkcd FB d rw1 Nu7 cKy bL jCF7 P4dx j0Sbz9 fa V CWk VFo s9t 2a QIPK ORuE jEMtbS Hs Y eEQ161}   \end{align}   By Gr\"{o}nwall's lemma, we conclude that $w\equiv 0$ a.s.~on $[0, \eta^M\wedge S]$. Note that $S$ is independent of~$M$. Sending $M$ to infinity and $S$ to $T$, we obtain the pathwise uniqueness on $[0,\tau\wedge T]$.  \end{proof} \par \begin{proof}[Proof of Theorem~\ref{T01}]  We first assume $\Vert u_0\Vert_p\leq K$~a.s.\ for some positive value $K$ and apply Lemma~\ref{L07}. Let $\tau_K$ be the stopping time suggested by the theorem. Recall that $\tau^{n_{N}}_{M}\geq \tau_K$ in $\cap_{k\geq N} \Omega_k^c$ if $\Omega_k$ is defined as~\eqref{ERTWERTHWRTWERTSGDGHCFGSDFGQSERWDFGDSFGHSDRGTEHDFGHDSFGSDGHGYUHDFGSDFASDFASGTWRT135}. For simplicity, we denote the convergent subsequence by~$\{u^{(n)}\}$ and the corresponding stopping time by $\tau^n$ (cf.~\eqref{ERTWERTHWRTWERTSGDGHCFGSDFGQSERWDFGDSFGHSDRGTEHDFGHDSFGSDGHGYUHDFGSDFASDFASGTWRT96}). By Theorem~\ref{T03}, all $\{u^{(n)}\}$ are strong solutions of \eqref{ERTWERTHWRTWERTSGDGHCFGSDFGQSERWDFGDSFGHSDRGTEHDFGHDSFGSDGHGYUHDFGSDFASDFASGTWRT95} up to an arbitrary deterministic time $T$, and thus   \begin{align}   \begin{split}   &\bigl(\indic_{[0,\tau_K\wedge T]}(t)u^{(n)} (t),\phi\bigr)   = (u_0^{(n)},\phi)+\int_0^t \indic_{[0,\tau_K\wedge T]}(s)   \left(u^{(n)} ,\Delta\phi\right)   \,ds   \\&   \indeq\indeq+   \sum_{j}\int_0^t \indic_{[0,\tau_K\wedge T]}(s)(\phin)^2\left(\pkn \mathcal{P}\left( u^{(n)}_j \pkn u^{(n)} \right),\partial_{j}\phi\right)\,ds   \\&   \indeq\indeq   +\int_0^t \indic_{[0,\tau_K\wedge T]}(s)(\phin)^2   \left(\pkn\sigma(\pkn u^{(n)} ),\phi\right)   \,d\WW_s   \comma  (t,\omega)\text{-a.e.}         ,   \end{split}\label{ERTWERTHWRTWERTSGDGHCFGSDFGQSERWDFGDSFGHSDRGTEHDFGHDSFGSDGHGYUHDFGSDFASDFASGTWRT141}   \end{align} for all $\phi\in C_c^{\infty}(\RR^3)$ and~$t>0$. Utilizing \eqref{ERTWERTHWRTWERTSGDGHCFGSDFGQSERWDFGDSFGHSDRGTEHDFGHDSFGSDGHGYUHDFGSDFASDFASGTWRT129}, we may pass to the limit in \eqref{ERTWERTHWRTWERTSGDGHCFGSDFGQSERWDFGDSFGHSDRGTEHDFGHDSFGSDGHGYUHDFGSDFASDFASGTWRT141} and conclude   \begin{align}   \begin{split}   &\bigl(\indic_{[0,\tau_K\wedge T]}(t)u^{(n)} (t),\phi\bigr)   -   \int_0^t \indic_{[0,\tau_K\wedge T]}(s)\left(u^{(n)}(s) ,\Delta\phi\right)\,ds   -   (u_0^{(n)},\phi)   \\&\indeq   \to   \bigl(\indic_{[0,\tau_K\wedge T]}(t)u (t),\phi\bigr)   -   \int_0^t\indic_{[0,\tau_K\wedge T]}(s) (\uu(s),\Delta\phi)\,ds   -   (u_0,\phi)   \comma  (t,\omega)\text{-a.e.}   \end{split}
  \llabel{ Fin aK 9Eki CHV2 a13f7G 3G 3 oDK K0i bKV y4 53E2 nFQS 8Hnqg0 E3 2 ADd dEV nmJ 7H Bc1t 2K2i hCzZuy 9k p sHn 8Ko uAR kv sHKP y8Yo dOOqBi hF 1 Z3C vUF hmj gB muZq 7ggW Lg5dQB 1k p Fxk k35 GFo dk 00YD 13qI qqbLwy QC c yZR wHA fp7 9o imtC c5CV 8cEuwU w7 k 8Q7 nCq WkM gY rtVR IySM tZUGCH XV 9 mr9 GHZ ol0 VE eIjQ vwgw 17pDhX JS F UcY bqU gnG V8 IFWb S1GX az0ZTt 81 w 7En IhF F72 v2 PkWO Xlkr w6IPu5 67 9 vcW 1f6 z99 lM 2LI1 Y6Na axfl18 gT 0 gDp tVl CN4 jf GSbC ro5D v78Cxa uk Y iUI WWy YDR w8 z7Kj Px7C hC7zJv b1 b 0rF d7n Mxk 09 1wHv y4u5 vLLsJ8 Nm A kWt xuf 4P5 Nw P23b 06sF NQ6xgD hu R GbK 7j2 O4g y4 p4BL top3 h2kfyI 9w O 4Aa EWb 36Y yH YiI1 S3CO J7aN1r 0s Q OrC AC4 vL7 yr CGkI RlNu GbOuuk 1a w LDK 2zl Ka4 0h yJnD V4iF xsqO00 1r q CeO AO2 es7 DR aCpU G54F 2i97xS Qr c bPZ 6K8 Kud n9 e6SY o396 Fr8LUx yX O jdF sMr l54 Eh T8vr xxF2 phKPbs zr l pMA ubE RMG QA aCBu 2Lqw Gasprf IZ O iKV Vbu Vae 6a bauf y9Kc Fk6cBl Z5 r KUj htW E1C nt 9Rmd whJR ySGVSO VT v 9FY 4uz yAH Sp 6yT9 s6R6 oOi3aq Zl L 7bI vWZ 18c Fa iwpt C1nd Fyp4oK xD f Qz2 813 6a8 zX wsGl Ysh9 Gp3Tal nr R UKt tBK eFr 45 43qU 2hh3 WbYw09 g2 W LIX zvQ zMk j5 f0xL seH9 dscinG wu P JLP 1gE N5W qY sSoW Peqj MimTyb Hj j cbn 0NO 5hz P9 W40r 2w77 TAoz70 N1 a u09 boc DSx Gc 3tvK LXaC 1dKgw9 H3 o 2kE oul In9 TS PyL2 HXO7 tSZse0 1Z 9 Hds lDq 0tm SO AVqt A1FQ zEMKSb ak z nw8 39w nH1 Dp CjGI k5X3 B6S6UI 7H I gAa f9E V33 Bk kuo3 FyEi 8Ty2AB PY z SWj Pj5 tYZ ET Yzg6 Ix5t ATPMdl Gk e 67X b7F ktE sz yFyc mVhG JZ29aP gz k Yj4 cEr HCd P7 XFHU O9zo y4AZai SR O pIn 0tp 7kZ zU VHQt m3ip 3xEd41 By 7 2ux IiY 8BC Lb OYGo LDwp juza6i Pa k Zdh aD3 xSX yj pdOw oqQq Jl6RFg lO t X67 nm7 s1l ZJ mGUr dIdX Q7jps7 rc d ACY ZMs BKA Nx tkqf Nhkt sbBf2O BN Z 5pf oqS Xtd 3c HFLN tLgR oHrnNl wR n ylZ NWV NfH vO B1nU Ayjt xTWW4o Cq P Rtu Vua nMk Lv qbxp Ni0x YnOkcd FB d rw1 Nu7 cKy bL jCF7 P4dx j0Sbz9 fa V CWk VFo s9t 2a QIPK ORuE jEMtbS Hs Y eG5 Z7u MWW Aw RnR8 FwFC zXVVxn FU f yKL Nk4 eOI ly n3Cl I5HP 8XP6S4 KF f Il6 2Vl bXg ca uth8 61pU WUx2aQ TW g rZw cAx 52T kq oZXV g0QG rBrrpe iw u WyJ td9 ooD 8t UzAd LSnIEQ142}   \end{align}   as~$n\to\infty$. We can also prove convergence for the nonlinear term by a similar splitting as~\eqref{ERTWERTHWRTWERTSGDGHCFGSDFGQSERWDFGDSFGHSDRGTEHDFGHDSFGSDGHGYUHDFGSDFASDFASGTWRT114}. But the proof is simpler than that following~\eqref{ERTWERTHWRTWERTSGDGHCFGSDFGQSERWDFGDSFGHSDRGTEHDFGHDSFGSDGHGYUHDFGSDFASDFASGTWRT114} because we have more flexibility in choosing exponents. As for the noise term, we apply the BDG inequality, obtaining   \begin{align}   \begin{split}   &\EE\biggl[\indic_{\cap_{i\geq N} \Omega_i^c}\sup_{t\in[0,T]}\biggl|\int_0^t   \indic_{[0,\tau_K\wedge T]}(s)   ((\phin)^2\pkn\sigma(\pkn u^{(n)} )-(\varphi^{(u)})^2\sigma(\uu),\phi)\,d\WW_s\biggr|   \biggr]   \\  & \quad\leq   C\EE\biggl[   \biggl(   \int_0^{\tau_K\wedge \tau^n\wedge T} (\phin)^2(\phin-\varphi^{(u)})^2   \bigl\Vert   \bigl(   \pkn(\sigma(\pkn u^{(n)} )),\phi   \bigr)   \bigr\Vert_{ l^2}^2\, ds   \biggr)^{1/2}\biggr]   \\  & \qquad +   C\EE\biggl[   \biggl(   \int_0^{\tau_K\wedge \tau^n\wedge T} (\phin)^2(\varphi^{(u)})^2   \Bigl\Vert   \bigl(   \pkn(\sigma(\pkn u^{(n)} )-\pkn\sigma(\pkn u)),\phi   \bigr)   \Bigr\Vert_{ l^2}^2\, ds   \biggr)^{1/2}\biggr]   \\  & \qquad +   C\EE\biggl[   \biggl(   \int_0^{\tau_K\wedge \tau^n\wedge T} (\phin)^2(\varphi^{(u)})^2   \bigl\Vert   \bigl(   \pkn(\sigma(\pkn u )-\sigma( u )),\phi   \bigr)   \bigr\Vert_{ l^2}^2\, ds   \biggr)^{1/2}\biggr]     \\  & \qquad +   C\EE\biggl[   \biggl(   \int_0^{\tau_K\wedge \tau^n\wedge T} (\phin)^2(\varphi^{(u)})^2   \bigl\Vert   \bigl(   (\pkn \sigma( u )-\sigma( u )),\phi   \bigr)   \bigr\Vert_{ l^2}^2\, ds   \biggr)^{1/2}\biggr]   \\  & \qquad   +   C\EE\biggl[\biggl(   \int_0^{\tau_K\wedge \tau^n\wedge T} (\phin-\varphi^{(u)})^2(\varphi^{(u)})^2\bigl\Vert\bigl(   \sigma(\uu),\phi   \bigr)\bigr\Vert_{ l^2}^2\, ds   \biggr)^{1/2}\biggr]=\sum_{i=1}^5 G_i   \end{split}    \llabel{B 1k p Fxk k35 GFo dk 00YD 13qI qqbLwy QC c yZR wHA fp7 9o imtC c5CV 8cEuwU w7 k 8Q7 nCq WkM gY rtVR IySM tZUGCH XV 9 mr9 GHZ ol0 VE eIjQ vwgw 17pDhX JS F UcY bqU gnG V8 IFWb S1GX az0ZTt 81 w 7En IhF F72 v2 PkWO Xlkr w6IPu5 67 9 vcW 1f6 z99 lM 2LI1 Y6Na axfl18 gT 0 gDp tVl CN4 jf GSbC ro5D v78Cxa uk Y iUI WWy YDR w8 z7Kj Px7C hC7zJv b1 b 0rF d7n Mxk 09 1wHv y4u5 vLLsJ8 Nm A kWt xuf 4P5 Nw P23b 06sF NQ6xgD hu R GbK 7j2 O4g y4 p4BL top3 h2kfyI 9w O 4Aa EWb 36Y yH YiI1 S3CO J7aN1r 0s Q OrC AC4 vL7 yr CGkI RlNu GbOuuk 1a w LDK 2zl Ka4 0h yJnD V4iF xsqO00 1r q CeO AO2 es7 DR aCpU G54F 2i97xS Qr c bPZ 6K8 Kud n9 e6SY o396 Fr8LUx yX O jdF sMr l54 Eh T8vr xxF2 phKPbs zr l pMA ubE RMG QA aCBu 2Lqw Gasprf IZ O iKV Vbu Vae 6a bauf y9Kc Fk6cBl Z5 r KUj htW E1C nt 9Rmd whJR ySGVSO VT v 9FY 4uz yAH Sp 6yT9 s6R6 oOi3aq Zl L 7bI vWZ 18c Fa iwpt C1nd Fyp4oK xD f Qz2 813 6a8 zX wsGl Ysh9 Gp3Tal nr R UKt tBK eFr 45 43qU 2hh3 WbYw09 g2 W LIX zvQ zMk j5 f0xL seH9 dscinG wu P JLP 1gE N5W qY sSoW Peqj MimTyb Hj j cbn 0NO 5hz P9 W40r 2w77 TAoz70 N1 a u09 boc DSx Gc 3tvK LXaC 1dKgw9 H3 o 2kE oul In9 TS PyL2 HXO7 tSZse0 1Z 9 Hds lDq 0tm SO AVqt A1FQ zEMKSb ak z nw8 39w nH1 Dp CjGI k5X3 B6S6UI 7H I gAa f9E V33 Bk kuo3 FyEi 8Ty2AB PY z SWj Pj5 tYZ ET Yzg6 Ix5t ATPMdl Gk e 67X b7F ktE sz yFyc mVhG JZ29aP gz k Yj4 cEr HCd P7 XFHU O9zo y4AZai SR O pIn 0tp 7kZ zU VHQt m3ip 3xEd41 By 7 2ux IiY 8BC Lb OYGo LDwp juza6i Pa k Zdh aD3 xSX yj pdOw oqQq Jl6RFg lO t X67 nm7 s1l ZJ mGUr dIdX Q7jps7 rc d ACY ZMs BKA Nx tkqf Nhkt sbBf2O BN Z 5pf oqS Xtd 3c HFLN tLgR oHrnNl wR n ylZ NWV NfH vO B1nU Ayjt xTWW4o Cq P Rtu Vua nMk Lv qbxp Ni0x YnOkcd FB d rw1 Nu7 cKy bL jCF7 P4dx j0Sbz9 fa V CWk VFo s9t 2a QIPK ORuE jEMtbS Hs Y eG5 Z7u MWW Aw RnR8 FwFC zXVVxn FU f yKL Nk4 eOI ly n3Cl I5HP 8XP6S4 KF f Il6 2Vl bXg ca uth8 61pU WUx2aQ TW g rZw cAx 52T kq oZXV g0QG rBrrpe iw u WyJ td9 ooD 8t UzAd LSnI tarmhP AW B mnm nsb xLI qX 4RQS TyoF DIikpe IL h WZZ 8ic JGa 91 HxRb 97kn Whp9sA Vz P o85 60p RN2 PS MGMM FK5X W52OnW Iy o Yng xWn o86 8S Kbbu 1Iq1 SyPkHJ VC v seV GWr hUEQ144}   \end{align} for~$n\geq N$. We analyze these terms by employing Minkowski's inequality and the assumptions on $\sigma$, like what we did in \eqref{ERTWERTHWRTWERTSGDGHCFGSDFGQSERWDFGDSFGHSDRGTEHDFGHDSFGSDGHGYUHDFGSDFASDFASGTWRT86}--\eqref{ERTWERTHWRTWERTSGDGHCFGSDFGQSERWDFGDSFGHSDRGTEHDFGHDSFGSDGHGYUHDFGSDFASDFASGTWRT88}. First, note that \begin{align} \begin{split} G_1+G_5 &\leq C_{p}T^{((p-2)/2p)+} \biggl(\EE\Big[ \sup_{r\in[0,\tau_K\wedge \tau^n\wedge T]}\Vert\un -\uu\Vert_{p}^p \Big] \biggr)^{1/p} \biggl(\EE\biggl[ \int_0^{\tau_K\wedge \tau^n\wedge T}  \bigl(\Vert \un \Vert_{3p}^p+\Vert u \Vert_{3p}^p+1\bigr)\, dr\biggr]\biggr)^{(1/p)-} \\& \leq C_{p,T,K}\biggl(\EE\Big[ \sup_{r\in[0,\tau_K\wedge \tau^n\wedge T]}\Vert\un -\uu\Vert_{p}^p \Big] \biggr)^{1/p}. \end{split} \llabel{FWb S1GX az0ZTt 81 w 7En IhF F72 v2 PkWO Xlkr w6IPu5 67 9 vcW 1f6 z99 lM 2LI1 Y6Na axfl18 gT 0 gDp tVl CN4 jf GSbC ro5D v78Cxa uk Y iUI WWy YDR w8 z7Kj Px7C hC7zJv b1 b 0rF d7n Mxk 09 1wHv y4u5 vLLsJ8 Nm A kWt xuf 4P5 Nw P23b 06sF NQ6xgD hu R GbK 7j2 O4g y4 p4BL top3 h2kfyI 9w O 4Aa EWb 36Y yH YiI1 S3CO J7aN1r 0s Q OrC AC4 vL7 yr CGkI RlNu GbOuuk 1a w LDK 2zl Ka4 0h yJnD V4iF xsqO00 1r q CeO AO2 es7 DR aCpU G54F 2i97xS Qr c bPZ 6K8 Kud n9 e6SY o396 Fr8LUx yX O jdF sMr l54 Eh T8vr xxF2 phKPbs zr l pMA ubE RMG QA aCBu 2Lqw Gasprf IZ O iKV Vbu Vae 6a bauf y9Kc Fk6cBl Z5 r KUj htW E1C nt 9Rmd whJR ySGVSO VT v 9FY 4uz yAH Sp 6yT9 s6R6 oOi3aq Zl L 7bI vWZ 18c Fa iwpt C1nd Fyp4oK xD f Qz2 813 6a8 zX wsGl Ysh9 Gp3Tal nr R UKt tBK eFr 45 43qU 2hh3 WbYw09 g2 W LIX zvQ zMk j5 f0xL seH9 dscinG wu P JLP 1gE N5W qY sSoW Peqj MimTyb Hj j cbn 0NO 5hz P9 W40r 2w77 TAoz70 N1 a u09 boc DSx Gc 3tvK LXaC 1dKgw9 H3 o 2kE oul In9 TS PyL2 HXO7 tSZse0 1Z 9 Hds lDq 0tm SO AVqt A1FQ zEMKSb ak z nw8 39w nH1 Dp CjGI k5X3 B6S6UI 7H I gAa f9E V33 Bk kuo3 FyEi 8Ty2AB PY z SWj Pj5 tYZ ET Yzg6 Ix5t ATPMdl Gk e 67X b7F ktE sz yFyc mVhG JZ29aP gz k Yj4 cEr HCd P7 XFHU O9zo y4AZai SR O pIn 0tp 7kZ zU VHQt m3ip 3xEd41 By 7 2ux IiY 8BC Lb OYGo LDwp juza6i Pa k Zdh aD3 xSX yj pdOw oqQq Jl6RFg lO t X67 nm7 s1l ZJ mGUr dIdX Q7jps7 rc d ACY ZMs BKA Nx tkqf Nhkt sbBf2O BN Z 5pf oqS Xtd 3c HFLN tLgR oHrnNl wR n ylZ NWV NfH vO B1nU Ayjt xTWW4o Cq P Rtu Vua nMk Lv qbxp Ni0x YnOkcd FB d rw1 Nu7 cKy bL jCF7 P4dx j0Sbz9 fa V CWk VFo s9t 2a QIPK ORuE jEMtbS Hs Y eG5 Z7u MWW Aw RnR8 FwFC zXVVxn FU f yKL Nk4 eOI ly n3Cl I5HP 8XP6S4 KF f Il6 2Vl bXg ca uth8 61pU WUx2aQ TW g rZw cAx 52T kq oZXV g0QG rBrrpe iw u WyJ td9 ooD 8t UzAd LSnI tarmhP AW B mnm nsb xLI qX 4RQS TyoF DIikpe IL h WZZ 8ic JGa 91 HxRb 97kn Whp9sA Vz P o85 60p RN2 PS MGMM FK5X W52OnW Iy o Yng xWn o86 8S Kbbu 1Iq1 SyPkHJ VC v seV GWr hUd ew Xw6C SY1b e3hD9P Kh a 1y0 SRw yxi AG zdCM VMmi JaemmP 8x r bJX bKL DYE 1F pXUK ADtF 9ewhNe fd 2 XRu tTl 1HY JV p5cA hM1J fK7UIc pk d TbE ndM 6FW HA 72Pg LHzX lUo39o WEQ145} \end{align} Next,    \begin{align} \begin{split} &G_2+G_3\leq  C_p \EE\biggl[\biggl(\int_0^{\tau_K\wedge \tau^n\wedge T} (\varphi^{(n)}\varphi^{(u)})^2 (\Vert \pkn \un \Vert_p + \Vert \pkn u\Vert_p) \Vert\pkn\un- \pkn u\Vert_{2p}^2 \, dr\biggr)^{1/2}\biggr] \\&\indeq\quad + C_p \EE\biggl[\biggl(\int_0^{\tau_K\wedge \tau^n\wedge T} (\varphi^{(n)}\varphi^{(u)})^2 (\Vert \pkn u \Vert_p + \Vert u\Vert_p) \Vert \pkn u-u\Vert_{2p}^2 \, dr\biggr)^{1/2}\biggr] \\&\indeq \leq  C_{p}T^{(2p-3)/4p} \biggl(\EE\Big[ \sup_{r\in[0,\tau_K\wedge \tau^n\wedge T]}\Vert\un -\uu\Vert_{p}^p \Big] \biggr)^{1/4p} \biggl(\EE\biggl[ \int_0^{\tau_K\wedge \tau^n\wedge T}  \bigl(\Vert \un \Vert_{3p}^p+\Vert u \Vert_{3p}^p\bigr)\, dr\biggr]\biggr)^{3/4p} \\&\indeq\quad + C_pT^{(2p-3)/4p} \biggl(\EE\Bigl[ \sup_{r\in[0,\tau_K\wedge \tau^n\wedge T]}\Vert \pkn \uu -\uu\Vert_{p}^p \Bigr] \biggr)^{1/4p} \biggl(\EE\biggl[ \int_0^{\tau_K\wedge \tau^n\wedge T}  \Vert u \Vert_{3p}^p \, dr\biggr]\biggr)^{3/4p}. \end{split} \llabel{F d7n Mxk 09 1wHv y4u5 vLLsJ8 Nm A kWt xuf 4P5 Nw P23b 06sF NQ6xgD hu R GbK 7j2 O4g y4 p4BL top3 h2kfyI 9w O 4Aa EWb 36Y yH YiI1 S3CO J7aN1r 0s Q OrC AC4 vL7 yr CGkI RlNu GbOuuk 1a w LDK 2zl Ka4 0h yJnD V4iF xsqO00 1r q CeO AO2 es7 DR aCpU G54F 2i97xS Qr c bPZ 6K8 Kud n9 e6SY o396 Fr8LUx yX O jdF sMr l54 Eh T8vr xxF2 phKPbs zr l pMA ubE RMG QA aCBu 2Lqw Gasprf IZ O iKV Vbu Vae 6a bauf y9Kc Fk6cBl Z5 r KUj htW E1C nt 9Rmd whJR ySGVSO VT v 9FY 4uz yAH Sp 6yT9 s6R6 oOi3aq Zl L 7bI vWZ 18c Fa iwpt C1nd Fyp4oK xD f Qz2 813 6a8 zX wsGl Ysh9 Gp3Tal nr R UKt tBK eFr 45 43qU 2hh3 WbYw09 g2 W LIX zvQ zMk j5 f0xL seH9 dscinG wu P JLP 1gE N5W qY sSoW Peqj MimTyb Hj j cbn 0NO 5hz P9 W40r 2w77 TAoz70 N1 a u09 boc DSx Gc 3tvK LXaC 1dKgw9 H3 o 2kE oul In9 TS PyL2 HXO7 tSZse0 1Z 9 Hds lDq 0tm SO AVqt A1FQ zEMKSb ak z nw8 39w nH1 Dp CjGI k5X3 B6S6UI 7H I gAa f9E V33 Bk kuo3 FyEi 8Ty2AB PY z SWj Pj5 tYZ ET Yzg6 Ix5t ATPMdl Gk e 67X b7F ktE sz yFyc mVhG JZ29aP gz k Yj4 cEr HCd P7 XFHU O9zo y4AZai SR O pIn 0tp 7kZ zU VHQt m3ip 3xEd41 By 7 2ux IiY 8BC Lb OYGo LDwp juza6i Pa k Zdh aD3 xSX yj pdOw oqQq Jl6RFg lO t X67 nm7 s1l ZJ mGUr dIdX Q7jps7 rc d ACY ZMs BKA Nx tkqf Nhkt sbBf2O BN Z 5pf oqS Xtd 3c HFLN tLgR oHrnNl wR n ylZ NWV NfH vO B1nU Ayjt xTWW4o Cq P Rtu Vua nMk Lv qbxp Ni0x YnOkcd FB d rw1 Nu7 cKy bL jCF7 P4dx j0Sbz9 fa V CWk VFo s9t 2a QIPK ORuE jEMtbS Hs Y eG5 Z7u MWW Aw RnR8 FwFC zXVVxn FU f yKL Nk4 eOI ly n3Cl I5HP 8XP6S4 KF f Il6 2Vl bXg ca uth8 61pU WUx2aQ TW g rZw cAx 52T kq oZXV g0QG rBrrpe iw u WyJ td9 ooD 8t UzAd LSnI tarmhP AW B mnm nsb xLI qX 4RQS TyoF DIikpe IL h WZZ 8ic JGa 91 HxRb 97kn Whp9sA Vz P o85 60p RN2 PS MGMM FK5X W52OnW Iy o Yng xWn o86 8S Kbbu 1Iq1 SyPkHJ VC v seV GWr hUd ew Xw6C SY1b e3hD9P Kh a 1y0 SRw yxi AG zdCM VMmi JaemmP 8x r bJX bKL DYE 1F pXUK ADtF 9ewhNe fd 2 XRu tTl 1HY JV p5cA hM1J fK7UIc pk d TbE ndM 6FW HA 72Pg LHzX lUo39o W9 0 BuD eJS lnV Rv z8VD V48t Id4Dtg FO O a47 LEH 8Qw nR GNBM 0RRU LluASz jx x wGI BHm Vyy Ld kGww 5eEg HFvsFU nz l 0vg OaQ DCV Ez 64r8 UvVH TtDykr Eu F aS3 5p5 yn6 QZ UcX3EQ146} \end{align} Lastly,    \begin{align} \begin{split} &G_4\leq C_p \EE\biggl[\biggl(\int_0^{\tau_K\wedge \tau^n\wedge T} (\varphi^{(n)}\varphi^{(u)})^2   \Vert \pkn\sigma(u)-\sigma(u)\Vert_{\mathbb{L}^q} ^2 \, dr\biggr)^{1/2}\biggr]. \end{split}    \llabel{GbOuuk 1a w LDK 2zl Ka4 0h yJnD V4iF xsqO00 1r q CeO AO2 es7 DR aCpU G54F 2i97xS Qr c bPZ 6K8 Kud n9 e6SY o396 Fr8LUx yX O jdF sMr l54 Eh T8vr xxF2 phKPbs zr l pMA ubE RMG QA aCBu 2Lqw Gasprf IZ O iKV Vbu Vae 6a bauf y9Kc Fk6cBl Z5 r KUj htW E1C nt 9Rmd whJR ySGVSO VT v 9FY 4uz yAH Sp 6yT9 s6R6 oOi3aq Zl L 7bI vWZ 18c Fa iwpt C1nd Fyp4oK xD f Qz2 813 6a8 zX wsGl Ysh9 Gp3Tal nr R UKt tBK eFr 45 43qU 2hh3 WbYw09 g2 W LIX zvQ zMk j5 f0xL seH9 dscinG wu P JLP 1gE N5W qY sSoW Peqj MimTyb Hj j cbn 0NO 5hz P9 W40r 2w77 TAoz70 N1 a u09 boc DSx Gc 3tvK LXaC 1dKgw9 H3 o 2kE oul In9 TS PyL2 HXO7 tSZse0 1Z 9 Hds lDq 0tm SO AVqt A1FQ zEMKSb ak z nw8 39w nH1 Dp CjGI k5X3 B6S6UI 7H I gAa f9E V33 Bk kuo3 FyEi 8Ty2AB PY z SWj Pj5 tYZ ET Yzg6 Ix5t ATPMdl Gk e 67X b7F ktE sz yFyc mVhG JZ29aP gz k Yj4 cEr HCd P7 XFHU O9zo y4AZai SR O pIn 0tp 7kZ zU VHQt m3ip 3xEd41 By 7 2ux IiY 8BC Lb OYGo LDwp juza6i Pa k Zdh aD3 xSX yj pdOw oqQq Jl6RFg lO t X67 nm7 s1l ZJ mGUr dIdX Q7jps7 rc d ACY ZMs BKA Nx tkqf Nhkt sbBf2O BN Z 5pf oqS Xtd 3c HFLN tLgR oHrnNl wR n ylZ NWV NfH vO B1nU Ayjt xTWW4o Cq P Rtu Vua nMk Lv qbxp Ni0x YnOkcd FB d rw1 Nu7 cKy bL jCF7 P4dx j0Sbz9 fa V CWk VFo s9t 2a QIPK ORuE jEMtbS Hs Y eG5 Z7u MWW Aw RnR8 FwFC zXVVxn FU f yKL Nk4 eOI ly n3Cl I5HP 8XP6S4 KF f Il6 2Vl bXg ca uth8 61pU WUx2aQ TW g rZw cAx 52T kq oZXV g0QG rBrrpe iw u WyJ td9 ooD 8t UzAd LSnI tarmhP AW B mnm nsb xLI qX 4RQS TyoF DIikpe IL h WZZ 8ic JGa 91 HxRb 97kn Whp9sA Vz P o85 60p RN2 PS MGMM FK5X W52OnW Iy o Yng xWn o86 8S Kbbu 1Iq1 SyPkHJ VC v seV GWr hUd ew Xw6C SY1b e3hD9P Kh a 1y0 SRw yxi AG zdCM VMmi JaemmP 8x r bJX bKL DYE 1F pXUK ADtF 9ewhNe fd 2 XRu tTl 1HY JV p5cA hM1J fK7UIc pk d TbE ndM 6FW HA 72Pg LHzX lUo39o W9 0 BuD eJS lnV Rv z8VD V48t Id4Dtg FO O a47 LEH 8Qw nR GNBM 0RRU LluASz jx x wGI BHm Vyy Ld kGww 5eEg HFvsFU nz l 0vg OaQ DCV Ez 64r8 UvVH TtDykr Eu F aS3 5p5 yn6 QZ UcX3 mfET Exz1kv qE p OVV EFP IVp zQ lMOI Z2yT TxIUOm 0f W L1W oxC tlX Ws 9HU4 EF0I Z1WDv3 TP 4 2LN 7Tr SuR 8u Mv1t Lepv ZoeoKL xf 9 zMJ 6PU In1 S8 I4KY 13wJ TACh5X l8 O 5g0 ZEQ55} \end{align} Since the right-hand sides of the above estimates all approach to zero as $n\to \infty$, we may extract a further subsequence, which for brevity we still denote by $\{u^{(n)}\}$, such that the convergence is exponentially rapid and thus   \begin{equation}   \int_0^t \indic_{[0,\tau_K\wedge T]}(s)(\pkn\sigma(\pkn u^{(n)}),\phi)\,d\WW_s   \rightarrow   \int_0^t \indic_{[0,\tau_K\wedge T]}(s)(\sigma(\uu),\phi)\,d\WW_s   \comma (t,\omega)\text{-a.e.}   \label{ERTWERTHWRTWERTSGDGHCFGSDFGQSERWDFGDSFGHSDRGTEHDFGHDSFGSDGHGYUHDFGSDFASDFASGTWRT147}   \end{equation} in $\bigcap_{i\geq N} \Omega_i^c$. Note that $\bigcap_{i\geq N} \Omega_i^c$ expands to the whole probability space as $N\to\infty$ and that $T$ is arbitrary. Combining \eqref{ERTWERTHWRTWERTSGDGHCFGSDFGQSERWDFGDSFGHSDRGTEHDFGHDSFGSDGHGYUHDFGSDFASDFASGTWRT141}--\eqref{ERTWERTHWRTWERTSGDGHCFGSDFGQSERWDFGDSFGHSDRGTEHDFGHDSFGSDGHGYUHDFGSDFASDFASGTWRT147} yields   \begin{align}   \begin{split}   &\indic_{[0,\tau_K]}(t)(u ( t),\phi)   = (\uu_0,\phi)   +\indic_{[0,\tau_K]}(t)\int_0^t    (u ,\Delta\phi)   \,ds+\indic_{[0,\tau_K]}(t)\int_0^t    \varphi^2(\sigma(u ),\phi)   \,d\WW_s   \\&\indeq\indeq   +\sum_{j}\indic_{[0,\tau_K]}(t)\int_0^t \varphi^2(\mathcal{P}( u_j u ),\partial_{j}\phi)\,ds   \comma  (t,\omega)\text{-a.e.}   ,   \end{split}   \llabel{ QA aCBu 2Lqw Gasprf IZ O iKV Vbu Vae 6a bauf y9Kc Fk6cBl Z5 r KUj htW E1C nt 9Rmd whJR ySGVSO VT v 9FY 4uz yAH Sp 6yT9 s6R6 oOi3aq Zl L 7bI vWZ 18c Fa iwpt C1nd Fyp4oK xD f Qz2 813 6a8 zX wsGl Ysh9 Gp3Tal nr R UKt tBK eFr 45 43qU 2hh3 WbYw09 g2 W LIX zvQ zMk j5 f0xL seH9 dscinG wu P JLP 1gE N5W qY sSoW Peqj MimTyb Hj j cbn 0NO 5hz P9 W40r 2w77 TAoz70 N1 a u09 boc DSx Gc 3tvK LXaC 1dKgw9 H3 o 2kE oul In9 TS PyL2 HXO7 tSZse0 1Z 9 Hds lDq 0tm SO AVqt A1FQ zEMKSb ak z nw8 39w nH1 Dp CjGI k5X3 B6S6UI 7H I gAa f9E V33 Bk kuo3 FyEi 8Ty2AB PY z SWj Pj5 tYZ ET Yzg6 Ix5t ATPMdl Gk e 67X b7F ktE sz yFyc mVhG JZ29aP gz k Yj4 cEr HCd P7 XFHU O9zo y4AZai SR O pIn 0tp 7kZ zU VHQt m3ip 3xEd41 By 7 2ux IiY 8BC Lb OYGo LDwp juza6i Pa k Zdh aD3 xSX yj pdOw oqQq Jl6RFg lO t X67 nm7 s1l ZJ mGUr dIdX Q7jps7 rc d ACY ZMs BKA Nx tkqf Nhkt sbBf2O BN Z 5pf oqS Xtd 3c HFLN tLgR oHrnNl wR n ylZ NWV NfH vO B1nU Ayjt xTWW4o Cq P Rtu Vua nMk Lv qbxp Ni0x YnOkcd FB d rw1 Nu7 cKy bL jCF7 P4dx j0Sbz9 fa V CWk VFo s9t 2a QIPK ORuE jEMtbS Hs Y eG5 Z7u MWW Aw RnR8 FwFC zXVVxn FU f yKL Nk4 eOI ly n3Cl I5HP 8XP6S4 KF f Il6 2Vl bXg ca uth8 61pU WUx2aQ TW g rZw cAx 52T kq oZXV g0QG rBrrpe iw u WyJ td9 ooD 8t UzAd LSnI tarmhP AW B mnm nsb xLI qX 4RQS TyoF DIikpe IL h WZZ 8ic JGa 91 HxRb 97kn Whp9sA Vz P o85 60p RN2 PS MGMM FK5X W52OnW Iy o Yng xWn o86 8S Kbbu 1Iq1 SyPkHJ VC v seV GWr hUd ew Xw6C SY1b e3hD9P Kh a 1y0 SRw yxi AG zdCM VMmi JaemmP 8x r bJX bKL DYE 1F pXUK ADtF 9ewhNe fd 2 XRu tTl 1HY JV p5cA hM1J fK7UIc pk d TbE ndM 6FW HA 72Pg LHzX lUo39o W9 0 BuD eJS lnV Rv z8VD V48t Id4Dtg FO O a47 LEH 8Qw nR GNBM 0RRU LluASz jx x wGI BHm Vyy Ld kGww 5eEg HFvsFU nz l 0vg OaQ DCV Ez 64r8 UvVH TtDykr Eu F aS3 5p5 yn6 QZ UcX3 mfET Exz1kv qE p OVV EFP IVp zQ lMOI Z2yT TxIUOm 0f W L1W oxC tlX Ws 9HU4 EF0I Z1WDv3 TP 4 2LN 7Tr SuR 8u Mv1t Lepv ZoeoKL xf 9 zMJ 6PU In1 S8 I4KY 13wJ TACh5X l8 O 5g0 ZGw Ddt u6 8wvr vnDC oqYjJ3 nF K WMA K8V OeG o4 DKxn EOyB wgmttc ES 8 dmT oAD 0YB Fl yGRB pBbo 8tQYBw bS X 2lc YnU 0fh At myR3 CKcU AQzzET Ng b ghH T64 KdO fL qFWu k07t DkzEQ148}   \end{align}    i.e.,   $\uu$ is a strong solution to the truncated model \eqref{ERTWERTHWRTWERTSGDGHCFGSDFGQSERWDFGDSFGHSDRGTEHDFGHDSFGSDGHGYUHDFGSDFASDFASGTWRT140} on $[0, \tau_{K}]$. Moreover, we have   \begin{align}   \begin{split}   \EE\biggl[   \sup_{0\leq s\leq \tau_K}\Vert u(s,\cdot)\Vert_p^p   +\int_0^{\tau_K}    \sum_{j}    \int_{\RR^d} | \nabla (|u_j(s,x)|^{p/2})|^2 \,dx\,ds   \biggr]   \leq    C\EE\bigl[\Vert\uu_0\Vert_p^p   +1   \bigr].   \label{ERTWERTHWRTWERTSGDGHCFGSDFGQSERWDFGDSFGHSDRGTEHDFGHDSFGSDGHGYUHDFGSDFASDFASGTWRT149}   \end{split}   \end{align} \par Note that the conclusion about the local existence and pathwise uniqueness is independent of the truncation level of the model~\eqref{ERTWERTHWRTWERTSGDGHCFGSDFGQSERWDFGDSFGHSDRGTEHDFGHDSFGSDGHGYUHDFGSDFASDFASGTWRT140}.  Namely, let $\varphi$ be a smooth function from $[0,\infty)$ to $[0,1]$ such that $\varphi\equiv 1$ on $[0,2]$ and  $\varphi\equiv 0$ on~$[4,\infty)$ such that \eqref{ERTWERTHWRTWERTSGDGHCFGSDFGQSERWDFGDSFGHSDRGTEHDFGHDSFGSDGHGYUHDFGSDFASDFASGTWRT146} holds. Then, for every $N\in\NNp$,  the model   \begin{align}   \begin{split}   &\partial_t\uu( t,x)   =\Delta \uu( t,x)   - \varphi\left(\frac{\Vert\uu(t)\Vert_{p}}{N}\right)^2\mathcal{P}\bigl(( \uu( t,x)\cdot \nabla)\uu( t,x)\bigr)   \\&\indeq\indeq\indeq\indeq\indeq   +\varphi\left(\frac{\Vert\uu(t)\Vert_{p}}{N}\right)^2\sigma(\uu( t,x))\dot{\WW}(t),   \\&   \nabla\cdot \uu( t,x) = 0   ,   \\&   \uu( 0,x)= \uu_0 (x)  \Pas   \commaone x\in\RR^d   \end{split}   \label{ERTWERTHWRTWERTSGDGHCFGSDFGQSERWDFGDSFGHSDRGTEHDFGHDSFGSDGHGYUHDFGSDFASDFASGTWRT150}   \end{align} has a unique local strong solution $(u_{(N)}, \bar\tau_N)$ satisfying \eqref{ERTWERTHWRTWERTSGDGHCFGSDFGQSERWDFGDSFGHSDRGTEHDFGHDSFGSDGHGYUHDFGSDFASDFASGTWRT149} if $\Vert u_0\Vert_p\leq N$~almost surely. Also note that \eqref{ERTWERTHWRTWERTSGDGHCFGSDFGQSERWDFGDSFGHSDRGTEHDFGHDSFGSDGHGYUHDFGSDFASDFASGTWRT01}--\eqref{ERTWERTHWRTWERTSGDGHCFGSDFGQSERWDFGDSFGHSDRGTEHDFGHDSFGSDGHGYUHDFGSDFASDFASGTWRT02} agrees with the model \eqref{ERTWERTHWRTWERTSGDGHCFGSDFGQSERWDFGDSFGHSDRGTEHDFGHDSFGSDGHGYUHDFGSDFASDFASGTWRT150} up to a positive time    \[   \eta_N:=\inf\left\{ t>0:    \sup_{0\leq s\leq t}\Vert \bar{u}_{(N)}(s)\Vert_p   \geq 2N   \right\}.   \]   Now, we remove the constraint~$\Vert u_0\Vert_p\leq N$ and define   \begin{equation}   u=\sum_{N=1}^{\infty} u_{(N)} \indic_{\{N-1\leq \Vert u_0\Vert_p< N\}}         \qquad{}\text{and}\qquad{}         \bar\tau=\sum_{N=1}^{\infty} (\bar{\tau}_N\wedge \eta_N) \indic_{\{N-1\leq \Vert u_0\Vert_p< N\}}.         \llabel{ f Qz2 813 6a8 zX wsGl Ysh9 Gp3Tal nr R UKt tBK eFr 45 43qU 2hh3 WbYw09 g2 W LIX zvQ zMk j5 f0xL seH9 dscinG wu P JLP 1gE N5W qY sSoW Peqj MimTyb Hj j cbn 0NO 5hz P9 W40r 2w77 TAoz70 N1 a u09 boc DSx Gc 3tvK LXaC 1dKgw9 H3 o 2kE oul In9 TS PyL2 HXO7 tSZse0 1Z 9 Hds lDq 0tm SO AVqt A1FQ zEMKSb ak z nw8 39w nH1 Dp CjGI k5X3 B6S6UI 7H I gAa f9E V33 Bk kuo3 FyEi 8Ty2AB PY z SWj Pj5 tYZ ET Yzg6 Ix5t ATPMdl Gk e 67X b7F ktE sz yFyc mVhG JZ29aP gz k Yj4 cEr HCd P7 XFHU O9zo y4AZai SR O pIn 0tp 7kZ zU VHQt m3ip 3xEd41 By 7 2ux IiY 8BC Lb OYGo LDwp juza6i Pa k Zdh aD3 xSX yj pdOw oqQq Jl6RFg lO t X67 nm7 s1l ZJ mGUr dIdX Q7jps7 rc d ACY ZMs BKA Nx tkqf Nhkt sbBf2O BN Z 5pf oqS Xtd 3c HFLN tLgR oHrnNl wR n ylZ NWV NfH vO B1nU Ayjt xTWW4o Cq P Rtu Vua nMk Lv qbxp Ni0x YnOkcd FB d rw1 Nu7 cKy bL jCF7 P4dx j0Sbz9 fa V CWk VFo s9t 2a QIPK ORuE jEMtbS Hs Y eG5 Z7u MWW Aw RnR8 FwFC zXVVxn FU f yKL Nk4 eOI ly n3Cl I5HP 8XP6S4 KF f Il6 2Vl bXg ca uth8 61pU WUx2aQ TW g rZw cAx 52T kq oZXV g0QG rBrrpe iw u WyJ td9 ooD 8t UzAd LSnI tarmhP AW B mnm nsb xLI qX 4RQS TyoF DIikpe IL h WZZ 8ic JGa 91 HxRb 97kn Whp9sA Vz P o85 60p RN2 PS MGMM FK5X W52OnW Iy o Yng xWn o86 8S Kbbu 1Iq1 SyPkHJ VC v seV GWr hUd ew Xw6C SY1b e3hD9P Kh a 1y0 SRw yxi AG zdCM VMmi JaemmP 8x r bJX bKL DYE 1F pXUK ADtF 9ewhNe fd 2 XRu tTl 1HY JV p5cA hM1J fK7UIc pk d TbE ndM 6FW HA 72Pg LHzX lUo39o W9 0 BuD eJS lnV Rv z8VD V48t Id4Dtg FO O a47 LEH 8Qw nR GNBM 0RRU LluASz jx x wGI BHm Vyy Ld kGww 5eEg HFvsFU nz l 0vg OaQ DCV Ez 64r8 UvVH TtDykr Eu F aS3 5p5 yn6 QZ UcX3 mfET Exz1kv qE p OVV EFP IVp zQ lMOI Z2yT TxIUOm 0f W L1W oxC tlX Ws 9HU4 EF0I Z1WDv3 TP 4 2LN 7Tr SuR 8u Mv1t Lepv ZoeoKL xf 9 zMJ 6PU In1 S8 I4KY 13wJ TACh5X l8 O 5g0 ZGw Ddt u6 8wvr vnDC oqYjJ3 nF K WMA K8V OeG o4 DKxn EOyB wgmttc ES 8 dmT oAD 0YB Fl yGRB pBbo 8tQYBw bS X 2lc YnU 0fh At myR3 CKcU AQzzET Ng b ghH T64 KdO fL qFWu k07t DkzfQ1 dg B cw0 LSY lr7 9U 81QP qrdf H1tb8k Kn D l52 FhC j7T Xi P7GF C7HJ KfXgrP 4K O Og1 8BM 001 mJ PTpu bQr6 1JQu6o Gr 4 baj 60k zdX oD gAOX 2DBk LymrtN 6T 7 us2 Cp6 eZm 1aEQ151}   \end{equation}   Clearly, $(u, \bar\tau)$ solves  \eqref{ERTWERTHWRTWERTSGDGHCFGSDFGQSERWDFGDSFGHSDRGTEHDFGHDSFGSDGHGYUHDFGSDFASDFASGTWRT01}--\eqref{ERTWERTHWRTWERTSGDGHCFGSDFGQSERWDFGDSFGHSDRGTEHDFGHDSFGSDGHGYUHDFGSDFASDFASGTWRT02} $\PP$-almost surely. We conclude that   \begin{align}   \begin{split}   \PP(\bar\tau>0)&=\sum_{N=1}^{\infty}   \PP(\bar{\tau}_N\wedge \eta_N>0\,| N-1\leq \Vert u_0\Vert_p< N)\,   \PP(N-1\leq \Vert u_0\Vert_p< N)   \\ &   =\sum_{N=1}^{\infty}   \PP(N-1\leq \Vert u_0\Vert_p< N)=1.   \end{split}    \llabel{2w77 TAoz70 N1 a u09 boc DSx Gc 3tvK LXaC 1dKgw9 H3 o 2kE oul In9 TS PyL2 HXO7 tSZse0 1Z 9 Hds lDq 0tm SO AVqt A1FQ zEMKSb ak z nw8 39w nH1 Dp CjGI k5X3 B6S6UI 7H I gAa f9E V33 Bk kuo3 FyEi 8Ty2AB PY z SWj Pj5 tYZ ET Yzg6 Ix5t ATPMdl Gk e 67X b7F ktE sz yFyc mVhG JZ29aP gz k Yj4 cEr HCd P7 XFHU O9zo y4AZai SR O pIn 0tp 7kZ zU VHQt m3ip 3xEd41 By 7 2ux IiY 8BC Lb OYGo LDwp juza6i Pa k Zdh aD3 xSX yj pdOw oqQq Jl6RFg lO t X67 nm7 s1l ZJ mGUr dIdX Q7jps7 rc d ACY ZMs BKA Nx tkqf Nhkt sbBf2O BN Z 5pf oqS Xtd 3c HFLN tLgR oHrnNl wR n ylZ NWV NfH vO B1nU Ayjt xTWW4o Cq P Rtu Vua nMk Lv qbxp Ni0x YnOkcd FB d rw1 Nu7 cKy bL jCF7 P4dx j0Sbz9 fa V CWk VFo s9t 2a QIPK ORuE jEMtbS Hs Y eG5 Z7u MWW Aw RnR8 FwFC zXVVxn FU f yKL Nk4 eOI ly n3Cl I5HP 8XP6S4 KF f Il6 2Vl bXg ca uth8 61pU WUx2aQ TW g rZw cAx 52T kq oZXV g0QG rBrrpe iw u WyJ td9 ooD 8t UzAd LSnI tarmhP AW B mnm nsb xLI qX 4RQS TyoF DIikpe IL h WZZ 8ic JGa 91 HxRb 97kn Whp9sA Vz P o85 60p RN2 PS MGMM FK5X W52OnW Iy o Yng xWn o86 8S Kbbu 1Iq1 SyPkHJ VC v seV GWr hUd ew Xw6C SY1b e3hD9P Kh a 1y0 SRw yxi AG zdCM VMmi JaemmP 8x r bJX bKL DYE 1F pXUK ADtF 9ewhNe fd 2 XRu tTl 1HY JV p5cA hM1J fK7UIc pk d TbE ndM 6FW HA 72Pg LHzX lUo39o W9 0 BuD eJS lnV Rv z8VD V48t Id4Dtg FO O a47 LEH 8Qw nR GNBM 0RRU LluASz jx x wGI BHm Vyy Ld kGww 5eEg HFvsFU nz l 0vg OaQ DCV Ez 64r8 UvVH TtDykr Eu F aS3 5p5 yn6 QZ UcX3 mfET Exz1kv qE p OVV EFP IVp zQ lMOI Z2yT TxIUOm 0f W L1W oxC tlX Ws 9HU4 EF0I Z1WDv3 TP 4 2LN 7Tr SuR 8u Mv1t Lepv ZoeoKL xf 9 zMJ 6PU In1 S8 I4KY 13wJ TACh5X l8 O 5g0 ZGw Ddt u6 8wvr vnDC oqYjJ3 nF K WMA K8V OeG o4 DKxn EOyB wgmttc ES 8 dmT oAD 0YB Fl yGRB pBbo 8tQYBw bS X 2lc YnU 0fh At myR3 CKcU AQzzET Ng b ghH T64 KdO fL qFWu k07t DkzfQ1 dg B cw0 LSY lr7 9U 81QP qrdf H1tb8k Kn D l52 FhC j7T Xi P7GF C7HJ KfXgrP 4K O Og1 8BM 001 mJ PTpu bQr6 1JQu6o Gr 4 baj 60k zdX oD gAOX 2DBk LymrtN 6T 7 us2 Cp6 eZm 1a VJTY 8vYP OzMnsA qs 3 RL6 xHu mXN AB 5eXn ZRHa iECOaa MB w Ab1 5iF WGu cZ lU8J niDN KiPGWz q4 1 iBj 1kq bak ZF SvXq vSiR bLTriS y8 Q YOa mQU ZhO rG HYHW guPB zlAhua o5 9 EQ152}   \end{align} In addition, using \eqref{ERTWERTHWRTWERTSGDGHCFGSDFGQSERWDFGDSFGHSDRGTEHDFGHDSFGSDGHGYUHDFGSDFASDFASGTWRT149} and the pathwise uniqueness, we obtain   \begin{align}   \begin{split}   &\EE\biggl[\sup_{0\leq s\leq \bar\tau}\Vert\uu(s,\cdot)\Vert_p^p   +\sum_{j}\int_0^{\bar\tau} \int_{\RR^3} | \nabla (|\uu_j(s,x)|^{p/2})|^2 \,dx\,ds   \biggr]   \\&\indeq   =\lim_{N\to\infty} \EE\biggl[\indic_{\{0\leq \Vert u_0\Vert_p< N+1\}}\biggl(   \sup_{0\leq s\leq \bar\tau}\Vert\uu(s,\cdot)\Vert_p^p   +\sum_{j}\int_0^{\bar\tau} \int_{\RR^3} | \nabla (|\uu_j(s,x)|^{p/2})|^2 \,dx\,ds   \biggr)\biggr]   \\&\indeq   \leq \lim_{N\to\infty} C\EE\bigl[\indic_{\{0\leq \Vert u_0\Vert_p< N+1\}}\Vert\uu_0\Vert_p^p\bigr]+C\leq  C\EE\bigl[\Vert\uu_0\Vert_p^p\bigr]+C   ,   \end{split}    \llabel{E V33 Bk kuo3 FyEi 8Ty2AB PY z SWj Pj5 tYZ ET Yzg6 Ix5t ATPMdl Gk e 67X b7F ktE sz yFyc mVhG JZ29aP gz k Yj4 cEr HCd P7 XFHU O9zo y4AZai SR O pIn 0tp 7kZ zU VHQt m3ip 3xEd41 By 7 2ux IiY 8BC Lb OYGo LDwp juza6i Pa k Zdh aD3 xSX yj pdOw oqQq Jl6RFg lO t X67 nm7 s1l ZJ mGUr dIdX Q7jps7 rc d ACY ZMs BKA Nx tkqf Nhkt sbBf2O BN Z 5pf oqS Xtd 3c HFLN tLgR oHrnNl wR n ylZ NWV NfH vO B1nU Ayjt xTWW4o Cq P Rtu Vua nMk Lv qbxp Ni0x YnOkcd FB d rw1 Nu7 cKy bL jCF7 P4dx j0Sbz9 fa V CWk VFo s9t 2a QIPK ORuE jEMtbS Hs Y eG5 Z7u MWW Aw RnR8 FwFC zXVVxn FU f yKL Nk4 eOI ly n3Cl I5HP 8XP6S4 KF f Il6 2Vl bXg ca uth8 61pU WUx2aQ TW g rZw cAx 52T kq oZXV g0QG rBrrpe iw u WyJ td9 ooD 8t UzAd LSnI tarmhP AW B mnm nsb xLI qX 4RQS TyoF DIikpe IL h WZZ 8ic JGa 91 HxRb 97kn Whp9sA Vz P o85 60p RN2 PS MGMM FK5X W52OnW Iy o Yng xWn o86 8S Kbbu 1Iq1 SyPkHJ VC v seV GWr hUd ew Xw6C SY1b e3hD9P Kh a 1y0 SRw yxi AG zdCM VMmi JaemmP 8x r bJX bKL DYE 1F pXUK ADtF 9ewhNe fd 2 XRu tTl 1HY JV p5cA hM1J fK7UIc pk d TbE ndM 6FW HA 72Pg LHzX lUo39o W9 0 BuD eJS lnV Rv z8VD V48t Id4Dtg FO O a47 LEH 8Qw nR GNBM 0RRU LluASz jx x wGI BHm Vyy Ld kGww 5eEg HFvsFU nz l 0vg OaQ DCV Ez 64r8 UvVH TtDykr Eu F aS3 5p5 yn6 QZ UcX3 mfET Exz1kv qE p OVV EFP IVp zQ lMOI Z2yT TxIUOm 0f W L1W oxC tlX Ws 9HU4 EF0I Z1WDv3 TP 4 2LN 7Tr SuR 8u Mv1t Lepv ZoeoKL xf 9 zMJ 6PU In1 S8 I4KY 13wJ TACh5X l8 O 5g0 ZGw Ddt u6 8wvr vnDC oqYjJ3 nF K WMA K8V OeG o4 DKxn EOyB wgmttc ES 8 dmT oAD 0YB Fl yGRB pBbo 8tQYBw bS X 2lc YnU 0fh At myR3 CKcU AQzzET Ng b ghH T64 KdO fL qFWu k07t DkzfQ1 dg B cw0 LSY lr7 9U 81QP qrdf H1tb8k Kn D l52 FhC j7T Xi P7GF C7HJ KfXgrP 4K O Og1 8BM 001 mJ PTpu bQr6 1JQu6o Gr 4 baj 60k zdX oD gAOX 2DBk LymrtN 6T 7 us2 Cp6 eZm 1a VJTY 8vYP OzMnsA qs 3 RL6 xHu mXN AB 5eXn ZRHa iECOaa MB w Ab1 5iF WGu cZ lU8J niDN KiPGWz q4 1 iBj 1kq bak ZF SvXq vSiR bLTriS y8 Q YOa mQU ZhO rG HYHW guPB zlAhua o5 9 RKU trF 5Kb js KseT PXhU qRgnNA LV t aw4 YJB tK9 fN 7bN9 IEwK LTYGtn Cc c 2nf Mcx 7Vo Bt 1IC5 teMH X4g3JK 4J s deo Dl1 Xgb m9 xWDg Z31P chRS1R 8W 1 hap 5Rh 6Jj yT NXSC UscEQ153}   \end{align} completing the proof. \end{proof} \par \section*{Acknowledgments} \rm IK was supported in part by the NSF grant DMS-2205493,
while FW was supported in part by the National Natural Science Foundation of China (No.~12101396 and~12161141004). \par \end{document}